\documentclass[12pt, a4paper, twoside]{report}

\usepackage[ngerman,english]{babel}	

\usepackage[utf8]{inputenc}
\usepackage[T1]{fontenc}
\usepackage{lmodern}

\usepackage{titling}
\newcommand{\subtitle}[1]{%
  \posttitle{%
    \par\end{center}
    \begin{center}\large#1\end{center}
    \vskip0.5em}%
}

\usepackage[headsepline,markcase=noupper]{scrlayer-scrpage}
\ohead[]{}
\ihead[]{}
\usepackage{graphicx}
\usepackage{amsmath} 
\usepackage{amssymb} 
\usepackage{amsfonts}
\usepackage{amsthm}
\usepackage{mathrsfs}
\usepackage{enumerate}
\usepackage{makeidx} 
\usepackage{verbatim}
\usepackage{epsfig}
\usepackage{dsfont}
\usepackage{caption}	
\usepackage{graphicx}
\usepackage{sidecap}
\usepackage{float}
\usepackage{stackrel}	
\usepackage{mathtools} 
\usepackage{ltablex}	
\usepackage{bbm}
\usepackage{stackengine,graphicx,scalerel} 

\usepackage{autobreak} 

\usepackage[thinlines]{easytable}

\usepackage[inline]{enumitem}
\makeatletter
\newcommand{\inlineitem}[1][]{%
\ifnum\enit@type=\tw@
    {\descriptionlabel{#1}}
  \hspace{\labelsep}%
\else
  \ifnum\enit@type=\z@
       \refstepcounter{\@listctr}\fi
    \quad\@itemlabel\hspace{\labelsep}%
\fi}
\makeatother

\usepackage[backref=page]{hyperref}

\makeindex

\makeindex      
     
\newlength{\fixboxwidth}     
\usepackage{microtype} 

\usepackage{bookmark}
\makeatletter 
  \providecommand*{\toclevel@author}{999}
  \providecommand*{\toclevel@title}{0}
\makeatother

\theoremstyle{plain}

\newtheorem{thm}{Theorem}[chapter]
\newtheorem{cor}[thm]{Corollary}
\newtheorem{lemma}[thm]{Lemma}

\newtheorem{prop}[thm]{Proposition}

\theoremstyle{definition}

\newtheorem{ex}[thm]{Example}
\newtheorem{rem}[thm]{Remark}
\newtheorem{defi}[thm]{Definition}
\newtheorem{alg}[thm]{Algorithm}

\numberwithin{equation}{chapter}   
\numberwithin{thm}{chapter}

\newcommand{\thistheoremname}{}
\newtheorem*{genericthm}{\thistheoremname}
\newenvironment{namedthm}[1]
  {\renewcommand{\thistheoremname}{#1}%
   \begin{genericthm}}
  {\end{genericthm}}

\renewcommand{\phi}{\varphi} 
\newcommand{\eps}{\varepsilon}

\newcommand{\IR}{\ensuremath{\mathbb{R}}}
\newcommand{\IN}{\ensuremath{\mathbb{N}}}
\newcommand{\IZ}{\ensuremath{\mathbb{Z}}}
\newcommand{\IQ}{\ensuremath{\mathbb{Q}}}
\newcommand{\IC}{\ensuremath{\mathbb{C}}}
\newcommand{\IK}{\ensuremath{\mathbb{K}}}
\newcommand{\IP}{\ensuremath{\mathbb{P}}}
\newcommand{\IE}{\ensuremath{\mathbb{E}}}

\newcommand{\C}{\mathcal{C}}

\newcommand{\lall}{\Lambda^{\rm all}}
\newcommand{\lstd}{\Lambda^{\rm std}}

\newcommand{\mix}{{\rm mix}}

\renewcommand{\d}{{\rm d}} 

\newcommand\varlesssim{\mathrel{\ensurestackMath{\ThisStyle{%
  \stackengine{-.4\LMex}{\SavedStyle<}{%
    \rotatebox{-25}{$\SavedStyle\sim$}}{U}{r}{F}{T}{S}}}}}
\newcommand\vargtrsim{\mathrel{\ensurestackMath{\ThisStyle{%
  \stackengine{-.4\LMex}{\SavedStyle>}{%
    \rotatebox{25}{$\SavedStyle\sim$}}{U}{l}{F}{T}{S}}}}}

\DeclareMathOperator{\cost}{cost}
\DeclareMathOperator{\err}{err}
\DeclareMathOperator{\comp}{n}
\DeclareMathOperator{\e}{e}
\DeclareMathOperator{\disp}{disp}

\DeclareMathOperator{\APP}{APP}
\DeclareMathOperator{\OPT}{OPT}
\DeclareMathOperator{\INT}{INT}

\newcommand{\diff}{\mathrm{D}}
\DeclareMathOperator\supp{supp}
\DeclareMathOperator*\esssup{ess\,sup}

\DeclareMathOperator{\diag}{diag}
\DeclareMathOperator{\rank}{rank}

\newcommand{\dist}{\mathrm{dist}}
\DeclareMathOperator\card{card}
\DeclareMathOperator{\vspan}{span}
\DeclareMathOperator\rad{rad}

\DeclareMathOperator\Var{Var}

\newcommand{\norm}[1]{\left\Vert#1\right\Vert}
\newcommand{\Xnorm}[2]{\left\Vert#1\right\Vert_{#2}}
\newcommand{\scalar}[2]{\left\langle#1,#2\right\rangle}
\newcommand{\Xscalar}[3]{\left\langle#1,#2\right\rangle_{#3}}

\newcommand{\set}[1]{\left\{#1\right\}}
\newcommand{\abs}[1]{\left|#1\right|}
\newcommand{\brackets}[1]{\left(#1\right)}

\begin{document}

\pagenumbering{alph}
\pagestyle{empty}


\begin{center}

\includegraphics[height = 4cm]{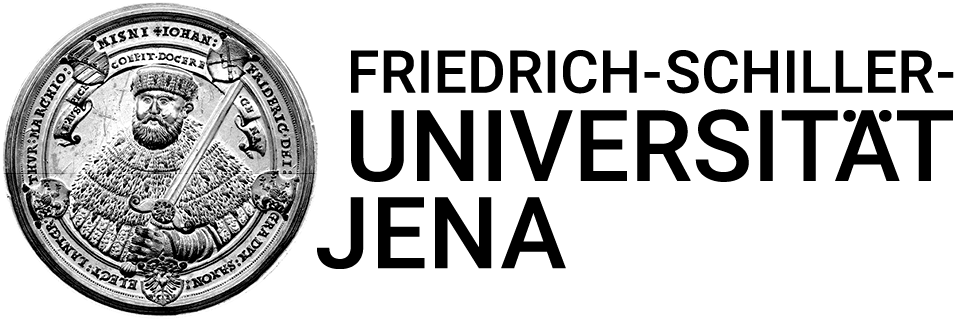}

\vspace{2cm}

{\textbf{\LARGE{%
Algorithms and Complexity\\
for some\\ 
\vspace*{3mm}
Multivariate Problems
}}}


\vspace{3cm}

\textbf{D I S S E R T A T I O N}\\
\textit{zur Erlangung des akademischen Grades\\
doctor rerum naturalium (Dr.\,rer.\,nat.)}\\

\vspace*{\fill}
vorgelegt dem Rat der\\
Fakult\"at f\"ur Mathematik und Informatik\\
der Friedrich-Schiller-Universit\"at Jena\\

\vspace*{\fill}
von David Krieg, M.\,Sc.\\
geboren am 8.\ Juli 1991 in Würzburg
\end{center}

\newpage

\vspace*{\fill}
\textbf{Gutachter}
\begin{enumerate}
	\item Prof.~Dr.~Aicke~Hinrichs, Linz
	\item Prof.~Dr.~Erich~Novak, Jena
	\item Prof.~Dr.~Henryk~Wo\'zniakowski, Warschau 
\end{enumerate}

\vspace{1cm}
\textbf{Tag der \"offentlichen Verteidigung}: 29.\ März 2019

\newpage

\subsection*{Acknowledgments}

First and foremost, I wish to express my deepest gratitude to
my supervisor Erich Novak for his valuable advice on so many topics,
including the innumerable hints and comments on this very thesis.
It is always a great pleasure to collaborate with Erich.
His well-aimed 
questions often led me
right to the heart of the matter at hand.
Furthermore, I am grateful to Aicke Hinrichs, Joscha Prochno, Daniel Rudolf, 
and Mario Ullrich for exciting collaborations
in the context of this thesis.
I thank Joscha,
Daniel, Mario, Glenn Byrenheid,
Marc Hovemann, and Winfried Sickel for their helpful
feedback on this thesis, and in particular
Robert Kunsch, who 
offered many insightful comments regarding the presentation
of the subject.
I also wish to thank numerous colleagues
for many great discussions during my time as a PhD student,
including, but not limited to, 
Stefan Heinrich, Therese Mieth, Christian Richter, Tino Ullrich,
and Henryk Wo\'zniakowski.
Finally, I cannot help but thank my parents, Mechthilde and Roland, 
for providing me with unfailing support
throughout my years of study.
The same holds for my two sisters, Anna and Judith,
and for Katharina, who never lost her patience with me.

\thispagestyle{empty}
\newpage
\phantom{empty page}
\thispagestyle{empty}
\newpage


\selectlanguage{ngerman}
\pagestyle{scrheadings}
\pagenumbering{roman}
\manualmark
\cleardoublepage
\chapter*{Zusammenfassung}
\addcontentsline{toc}{chapter}{Zusammenfassung}
\markboth{Zusammenfassung}{Zusammenfassung}

Auch mit den modernsten Computern
können wir in endlicher Zeit nur endlich viele Informationen 
über unsere Welt sammeln und verarbeiten.
Das macht das Finden exakter Lösungen 
für viele numerische Probleme unmöglich.
Beispiele hierfür sind die Frage nach der Abhängigkeit
einer beliebigen Größe von verschiedenen Parametern
(ein sogenanntes Approximationsproblem)
sowie die Berechnung eines Durchschnitts unter
unendlich vielen Werten
(ein sogenanntes Integrationsproblem).
In diesen Fällen müssen wir uns mit Näherungslösungen begnügen,
die wir auch mit endlich vielen Informationen bestimmen können.

Üblicherweise gibt es zwei Arten von Informationen:
das a priori Wissen und die empirischen Daten.
Das a priori Wissen ist bereits in der Problemstellung enthalten
und basiert in der Regel auf Modellannahmen.
Zum Beispiel wissen wir möglicherweise aus
theoretischen Vorbetrachtungen schon
etwas über die Regularität der
Funktion, die wir approximieren wollen.
Die empirischen Daten müssen wir dagegen erst durch 
Messungen, Umfragen, Programme oder andere Interaktionen
mit der Probleminstanz gewinnen.
Sie werden anschließend zu unserer
Näherungslösung verarbeitet.

Dieser Vorgang, also das Sammeln und das Verarbeiten der Daten,
kann durch einen Algorithmus beschrieben werden.
Jeder Algorithmus $A$ hat bestimmte Kosten und einen bestimmten Fehler, 
die wir mit $\cost(A)$ und $\err(A)$ bezeichnen.
Die Kosten messen den Aufwand, der zur Gewinnung
der Näherungslösung nötig ist.
Dieser ist oft proportional zu der Anzahl der gesammelten Informationen.
Der Fehler misst die zu erwartende Abweichung der Näherungslösung
von der exakten Lösung.
Für ein gegebenes Problem $\mathcal{P}$ stellen wir uns nun die Frage,
wie klein der Fehler eines Algorithmus mit vorgegebenen maximalen Kosten $n$ 
im besten Fall sein kann.
Wir fragen also nach dem $n$-ten minimalen Fehler 
$$
 \e(n,\mathcal{P}) = \inf\set{\err(A) \mid \cost(A)\leq n}.
$$
Umgekehrt fragen wir nach den minimal nötigen Kosten,
mit denen ein vorgegebener Fehler von höchstens $\varepsilon>0$
erreicht werden kann. Es geht also um die Größe
$$
 \comp(\varepsilon,\mathcal{P})= \min\set{\cost(A) \mid \err(A)\leq \varepsilon},
$$
die sogenannte $\varepsilon$-Komplexität des Problems.
Diese Größen sind invers zueinander und messen die Schwierigkeit des Problems. 
Zahlreiche klassische Untersuchungen beschäftigen sich
mit der Abfallgeschwindigkeit des $n$-ten minimalen Fehlers,
wenn $n$ gegen unendlich strebt.
Dies entspricht der Frage nach dem Verhalten der
$\varepsilon$-Komplexität des Problems,
wenn $\varepsilon$ gegen null strebt.

In vielen Fällen hat das Problem eine intrinsische Dimension $d\in\IN$. 
Beispielsweise ist die Probleminstanz häufig eine Funktion von $d$ Variablen.
Wir sprechen dann von einem multivariaten Problem,
das wir auch mit $\mathcal{P}_d$ bezeichnen. 
In diesem Fall interessieren wir uns für das Verhalten
der Komplexität $\comp(\varepsilon,\mathcal{P}_d)$
als Funktion in beiden Parametern $\varepsilon$ und $d$.
Viele Probleme unterliegen dem \emph{Fluch der Dimension}.
Die Komplexität wächst exponentiell mit der Dimension.
Solche Probleme sind für große Dimensionen praktisch unlösbar.
Man hofft also, dass die Komplexität nicht exponentiell von
$d$ oder $\varepsilon^{-1}$ abhängt.
In diesem Fall spricht man von \emph{Tractability}.
Noch besser ist es, wenn die Komplexität durch ein 
Polynom in $\varepsilon^{-1}$ und $d$ beschränkt ist.
Man spricht dann von polynomialer Tractability.

Für praktische Anwendungen reicht es allerdings nicht aus, 
die Komplexität des Problems zu studieren.
Diese gibt nur Auskunft darüber,
was der beste Algorithmus leisten kann.
Sie liefert uns nicht den besten Algorithmus.
Diesen zu finden, ist
im Allgemeinen eine unrealistische Hoffnung.
Es ist in der Regel bereits eine große Herausforderung,
einen Algorithmus zu finden, der den Fehler $\varepsilon$ erreicht
und dessen Kosten sich ähnlich wie die Komplexität des Problems verhalten.

In dieser Dissertation tragen wir Ergebnisse zu 
verschiedenen multivariaten Problemen bei.
Wir studieren die numerische Integration und Approximation
mit verschiedenen Arten von a priori Wissen.
Außerdem betrachten wir das Problem der globalen Optimierung und das Dispersionsproblem. 
In manchen Fällen erhalten wir neue Ergebnisse zur
Konvergenzordnung des Fehlers $\e(n,\mathcal{P}_d)$.
In anderen Fällen beweisen wir Ergebnisse bezüglich
der Tractability des Problems.
Aus der jeweiligen Sicht präsentieren wir optimale  
Algorithmen für die meisten dieser Probleme.
Diese Resultate finden sich in den Kapiteln~2--4.
In Kapitel~\ref{chap:problems and algorithms} stellen wir 
zunächst einige Grundlagen und Begrifflichkeiten
zur Verfügung.

\subsubsection{Zu Kapitel~\ref{chap:mixed smoothness}:
	Integration und Approximation von Funktionen gemischter Glattheit}

Dieses Kapitel beschäftigt sich mit der Integration
und der Approximation von Funktionen mit beschränkten gemischten Ableitungen,
wie sie beispielsweise im Zusammenhang
mit der elektronischen Schr\"odingergleichung und
verschiedenen Integralgleichungen auftreten \cite[Sec.\,9.1]{DTU18}.
Wir betrachten Funktionen aus der Klasse
$$
 F_d^r=\Bigg\{f\in L^2([0,1]^d) \,\,\Bigg\vert\,
 \sum\limits_{\alpha \in \set{0,\dots,r}^d}
 \norm{\diff^\alpha f}_2^2 \leq 1\Bigg\}.
$$

Wir beginnen mit dem Integrationsproblem.
Sei also $\mathcal{P}[\INT,F_d^r,\rm{det}]$ das Problem, 
Funktionen aus $F_d^r$ auf Basis von Funktionswerten
mithilfe deterministischer Algorithmen zu integrieren.
Die Konvergenzordnung des $n$-ten minimalen Fehlers 
ist für dieses Problem bekannt.
Ein optimaler Algorithmus wurde 1976 von Frolov vorgestellt \cite{Fr76}.
Es gilt
$$
 \e(n,\mathcal{P}[\INT,F_d^r,\mathrm{det}])
 \asymp
 n^{-r} \ln^{\frac{d-1}{2}} n.
$$
Mithilfe randomisierter Algorithmen lässt sich diese Konvergenzordnung verbessern.
Das Problem $\mathcal{P}[\INT,F_d^r,\rm{ran}]$,
Funktionen aus $F_d^r$ auf Basis von Funktionswerten
mithilfe randomisierter Algorithmen zu integrieren,
erfüllt die asymptotische \"Aquivalenz
$$
 \e(n,\mathcal{P}[\INT,F_d^r,\mathrm{ran}])
 \asymp
 n^{-r-1/2}.
$$
Insbesondere ist die Konvergenzordnung für letzteres Problem
unabhängig von der Dimension $d$ des Gebietes.
Dies ist eine Konsequenz von Satz~\ref{thm:main theorem},
welcher außerdem aufzeigt,
dass Frolovs Algorithmus in Kombination mit einer zufälligen
Verschiebung und Streckung der Menge der Knotenpunkte
optimal ist, siehe auch~\cite{KN17,Ul17}.

Die folgenden Abschnitte beschäftigen sich mit dem Problem
der $L^2$-Approximation.
Für dieses Problem ist es sinnvoll,
sowohl Algorithmen zu betrachten, deren Information
durch Funktionswerte gegeben ist,
als auch solche, die
beliebige lineare Information nutzen.
In Abschnitt~\ref{sec:tensorproduct} widmen wir uns dem
Fall der linearen Information.
In diesem Fall sind deterministische Algorithmen praktisch
genauso gut wie randomisierte Algorithmen \cite{No92}.
Wir studieren daher nur das Problem
$\mathcal{P}[\APP,F_d^r,\lall,\mathrm{det}]$,
die $L^2$-Approximation von
Funktionen aus $F_d^r$ auf Basis linearer Information
mithilfe deterministischer Algorithmen.
Es ist bereits sein 1960 bekannt, dass
$$
 \e(n,\mathcal{P}[\APP,F_d^r,\lall,\mathrm{det}])
 \asymp
 n^{-r} \ln^{r(d-1)} n
$$
im Sinne der schwachen asymptotischen \"Aquivalenz gilt \cite{Ba60}.
Ein optimaler Algorithmus ist anhand der Singul\"arwertzerlegung
der assoziierten Einbettung gegeben.
Wir wollen diese Fehlerzahlen hier jedoch etwas genauer betrachten.
Im Hinblick auf \cite{KSU15}, zeigen wir die starke asymptotische
\"Aquivalenz
$$
 \e(n,\mathcal{P}[\APP,F_d^r,\lall,\mathrm{det}])
 \sim
 (\pi^d (d-1)!\, n)^{-r} \ln^{r(d-1)} n,
$$
siehe Korollar~\ref{cor: strong equivalence mix all}.
Dies bedeutet, dass die Fehlerzahlen f\"ur großes $n$
sehr gut durch die rechte Seite der Gleichung beschrieben
werden können.
Da diese Ergebnisse nur f\"ur sehr große $n$ relevant sind,
stellen wir auch pr\"aasymptotische Absch\"atzungen bereit.
In Korollar~\ref{cor:preasymptotics mixed cube} beweisen wir die obere Schranke
$$
 \e(n,\mathcal{P}[\APP,F_d^r,\lall,\mathrm{det}])
 \leq\, 2\,n^{-c(d)}
  \quad\text{mit}\quad 
  c(d)=\frac{1.1929}{2+\ln d}
$$
f\"ur alle $n\in\IN$.
Weiter zeigen wir, dass diese Abschätzung für $n<2^d$ nicht wesentlich
verbessert werden kann, siehe Korollar~\ref{cor: preasymptotic lower bound mix all}.

In Abschnitt~\ref{sec:OptimalMC} wenden wir uns dem Fall zu,
dass die Informationen durch Funktionswerte gegeben sind.
Wir betrachten randomisierte Algorithmen.
Für das entsprechende Problem $\mathcal{P}[\APP,F_d^r,\lstd,\mathrm{ran}]$
beweisen wir die asymptotische Äquivalenz
$$
 \e(n,\mathcal{P}[\APP,F_d^r,\lstd,\mathrm{ran}])
 \asymp
 n^{-r} \ln^{r(d-1)} n,
$$
siehe Korollar~\ref{cor:order OptimalMC mixed}.
Wir geben einen Algorithmus an, dessen Fehler
sich in dieser Hinsicht optimal verhält,
siehe Algorithmus~\ref{alg:explicit alg}.
Außerdem beweisen wir die präasymptotische Abschätzung
$$
\e(n,\mathcal{P}[\APP,F_d^r,\lstd,\mathrm{ran}])
\leq\, 8\,n^{-c(d)}
$$
für alle $n\in\IN$ mit $c(d)$ wie oben, siehe \eqref{eq:preasymptotic ran std}.
Diese Abschätzungen zeigen, dass richtig gewählte Funktionswerte 
für das Approximationsproblem 
eine genauso große Aussagekraft haben
wie beliebige lineare Information,  
insofern randomisierte Algorithmen erlaubt sind.
Es ist ein ungelöstes Problem,
ob dieser Sachverhalt bestehen bleibt,
wenn wir nur deterministische Algorithmen betrachten.

An dieser Stelle wollen wir noch anmerken, 
dass die oben genannten Ergebnisse jeweils
für allgemeinere Fragestellungen formuliert werden können:
\begin{itemize}
 \item Abschnitt~\ref{sec:Frolov}: Frolovs Algorithmus und seine Randomisierung
 sind optimal für viele Klassen glatter Funktionen.
 \item Abschnitt~\ref{sec:tensorproduct}: Wir studieren optimale Algorithmen
 für beliebige Tensorproduktprobleme zwischen Hilberträumen.
 \item Abschnitt~\ref{sec:OptimalMC}: Wir präsentieren optimale
 Algorithmen für die $L^2$-Approximation von Funktionen
 aus der Einheitskugel von Hilberträumen, die kompakt in einen
 $L^2$-Raum eingebettet sind,
 vorausgesetzt die Singulärwerte dieser Einbettung
 erfüllen eine gewisse Abfallbedingung.
\end{itemize}

\subsubsection{Zu Kapitel~\ref{chap:tractability uniform approximation}:
	Tractability des Problems der gleichmäßigen Approximation}

In diesem Kapitel studieren wir die 
Leistungsfähigkeit deterministischer Algorithmen 
für das Problem, eine Funktion $f:[0,1]^d\to \IR$ 
gleichmäßig anhand endlich vieler Funktionswerte zu approximieren.
Um hier überhaupt etwas erreichen zu können,
ist a priori Wissen über die Funktion $f$ vonnöten,
sagen wir $f\in F_d$ für eine Menge $F_d$ von beschränkten Funktionen.
Sei $\mathcal{P}[F_d]$ das Problem der gleichmäßigen Approximation
mit a priori Wissen $F_d$.
Wir interessieren uns für die Tractability dieses Problems.
Insbesondere würden wir gerne mehr darüber wissen,
welche Art von a priori Wissen zu positiven
Ergebnissen in Hinblick auf die Tractability 
und damit zur praktischen Lösbarkeit des Problems in hohen Dimensionen führt.

Es ist bekannt, dass Glattheit alleine nicht ausreicht.
Selbst mit dem a priori Wissen
$$
 F_d = \set{f \in \C^\infty([0,1]^d) \,\big\vert\,
 \Vert \diff^\alpha f \Vert_\infty \leq 1 
 \text{\ für alle } \alpha \in \IN_0^d},
$$
unterliegt das Problem dem Fluch der Dimension~\cite{NW09}.
Selbstverständlich überträgt sich dieser Umstand auf 
den Fall endlicher Glattheit $r\in\IN$,
das heißt, auf den Fall von a priori Wissen
$$
 \C^r_d= \set{f \in \C^r([0,1]^d) \,\big\vert\,
 \Vert \diff^\alpha f \Vert_\infty \leq 1 
 \text{\ für alle } \alpha \in \IN_0^d \text{ mit } \abs{\alpha}\leq r}.
$$

Aber wie schlimm genau ist dieser Fluch?
Ab welcher Dimension hat man mit der Unlösbarkeit
des Problems zu rechnen?
Um diese Fragen dreht sich Abschnitt~\ref{sec:ck functions}.
Für gerade Zahlen $r$ stellen wir fest,
dass es positive Konstanten $c_r$, $C_r$ und $\varepsilon_r$
gibt, sodass
\begin{equation*}
 \brackets{c_r \sqrt{d}\, \varepsilon^{-1/r}}^d
 \leq \comp(\varepsilon,\mathcal{P}[\C^r_d]) \leq
 \brackets{C_r \sqrt{d}\, \varepsilon^{-1/r}}^d
\end{equation*}
für alle $d\in\IN$ und $\varepsilon\in (0,\varepsilon_r)$ gilt,
siehe Satz~\ref{main theorem even}.
Aus Ergebnissen von~\cite{Wa84} folgt,
dass selbige Abschätzungen auch für das Problem
der globalen Optimierung gelten,
da die Klasse $\C^r_d$ konvex und symmetrisch ist,
siehe Abschnitt~\ref{sec:optimization}.
Insbesondere wächst die Komplexität beider Probleme im Fall $r\geq 2$
für eine vorgegebene Fehlerschranke $\varepsilon>0$ wie $d^{d/2}$
und damit superexponentiell.

Andererseits wissen wir, dass zusätzliches Wissen über die 
Struktur der Funktion $f$ durchaus zu Tractability führen kann.
Beispiele hierfür sind folgende Annahmen:
\begin{itemize}
 \item Die Funktion ist eine Ridge-Funktion\cite{MUV15}. Das heißt,
 sie hat die Gestalt $f=g(\langle \cdot,\mathbf x_0\rangle)$
 für ein $\mathbf x_0\in\IR^d$ und ein $g:\IR\to\IR$.
 \item Die Funktion ist separierbar \cite{NR97,WW04}. 
 Sie lässt sich als Summe von Funktionen in $m$ Variablen
 schreiben, wobei die Ordnung $m$ unabhängig von der Dimension ist. 
 Man beachte, dass sich obige Paper
 nicht mit gleichmäßiger Approximation, 
 sondern mit $L^2$-Approximation und Integration beschäftigen.
 \item Die Funktion ist symmetrisch \cite{We12}. Das heißt, 
 $f(\mathbf x)$ ist invariant bezüglich Umordnungen 
 der Koordinaten von $\mathbf x\in[0,1]^d$. 
 Man beachte allerdings, dass die Funktionen in~\cite{We12} 
 nicht anhand von Funktionswerten, sondern anhand von
 anderen linearen Informationen approximiert werden.
\end{itemize}
Ein weiteres Beispiel studieren wir in Abschnitt~\ref{sec:rank one}.
Hier stellen wir uns vor, dass $f$ ein Rank-1-Tensor ist.
Das bedeutet, die $d$-dimensionale Funktion kann als Produkt von $d$ 
eindimensionalen Funktionen geschrieben werden.
Genauer gesagt nehmen wir an, dass $f$ ein Element der Klasse
$$
 F_{r,M}^d =
 \Big\{ \bigotimes_{i=1}^d  f_i \,\Big\vert\,
 f_i\colon [0,1]\to [-1,1], \
 \Vert f_i^{(r)} \Vert_\infty \le M \Big\}
$$
ist, wobei die Parameter $r\in\IN$ und $M>0$ die Glattheit
der Funktion beschreiben.
Die Funktion $\bigotimes_{i=1}^d  f_i$ heißt Tensorprodukt
der Funktionen $f_i$ und bildet $\mathbf x\in [0,1]^d$ 
auf das Produkt der Funktionswerte $f_i(x_i)$ ab. 
In Satz~\ref{tractability results} stellen wir fest,
dass das Problem der gleichmäßigen Approximation mit
a priori Wissen $F_{r,M}^d$ genau dann am Fluch der Dimension leidet, 
wenn $M\geq 2^r r!$.
Gilt dagegen $M<2^r r!$, so wächst die Komplexität nur polynomial mit der Dimension.
Falls $M\leq r!$ is der Grad dieser polynomialen Abhängigkeit
sogar unabhängig von der Fehlertoleranz $\varepsilon$
und wir erhalten polynomiale Tractability.
Andernfalls wächst der Exponent logarithmisch mit $\varepsilon^{-1}$.
In allen drei Fällen stellen wir einen Algorithmus vor,
dessen Kosten genau dieses Verhalten aufzeigen.
Der Algorithmus ist daher optimal im Hinblick auf die Tractability des Problems.
In Abschnitt~\ref{sec:optimization} beweisen wir außerdem,
dass die Komplexität des Problems der globalen Optimierung
auf $F_{r,M}^d$ dasselbe Verhalten aufweist.
Dies gilt, obwohl die Klasse $F_{r,M}^d$ nicht konvex ist.

Im Verlauf von Abschnitt~\ref{sec:rank one} wird klar,
dass das Problem der Approximation von Rank-1-Tensoren
eng mit dem geometrischen Problem der Dispersion zusammenhängt.
Die Dispersion einer Menge von Punkten im $d$-dimensionalen Einheitswürfel
ist das Volumen der größten achsenparallelen Box,
die keinen dieser Punkte enthält.
Diese Größe ist auch unabhängig vom obigen Approximationsproblem
von Interesse.
Man fragt nach möglichst kleinen Punktmengen,
die eine vorgegebene Dispersion $\varepsilon$ erreichen
oder unterbieten. 
In Abschnitt~\ref{sec:dispersion} geben wir eine solche
Punktmenge für alle $\varepsilon>0$ und jede Dimension $d\in\IN$ an.
Die Punktmenge ist ein dünnes Gitter und hat daher eine
besonders einfache Struktur.
Für viele Parameter $(\varepsilon,d)$ ist uns keine
kleinere Punktmenge mit der gewünschten Dispersion bekannt.

\subsubsection{Zu Kapitel~\ref{chap:random info}:
	Optimale Information versus zufällige Information}

Das letzte Kapitel unterscheidet sich wesentlich von den beiden
vorigen Kapiteln.
Bisher haben wir danach gestrebt,
optimale Algorithmen zu finden, welche optimale Information
über die Probleminstanz sammeln.
In Wirklichkeit haben wir jedoch oft keinen Zugriff auf 
optimale Information.
Das kann zum Beispiel daran liegen, dass wir nicht wissen,
wie wir die Parameter wählen müssen, um möglichst aussagekräftige
Messergebnisse zu erhalten.
Es kann auch sein, dass wir die Parameter für unsere
Messung nicht frei bestimmen können.
In diesem Kapitel nehmen wir an, dass die Parameter
dem Zufall unterliegen.
Wir erhalten also zufällige Information
und stellen uns die folgende Frage.
\begin{center}
Was ist die typische Güte
von zufälliger Information?
\end{center}
Selbstverständlich ist die zufällige Information
niemals besser als optimale Information,
aber es kann passieren, dass zufällige Information
nur unwesentlich schlechter ist.
In diesem Fall macht es wenig Sinn,
mühsam nach optimaler Information zu streben.

Um unsere Frage präzise formulieren zu können,
müssen wir klarstellen,
wie wir die Güte der Information messen
und welchem Zufall die Information unterliegt.
Die Güte der Information messen wir anhand ihres Radius.
Dies ist der Worst-Case-Fehler des besten Algorithmus,
der ausschließlich mit dieser Information und dem a priori Wissen arbeitet,
siehe Abschnitt~\ref{sec:errors}.
Unsere Information soll aus unabhängigen zufälligen Messungen
stammen, die alle derselben Verteilung genügen.
Sicher gibt es hier viele Verteilungen,
die es zu studieren wert sind.
Wir werden die obige Frage allerdings für zwei
Klassen von Beispielen betrachten,
bei denen wir jeweils eine Verteilung für
besonders natürlich und daher für besonders interessant halten.

Das erste Beispiel ist das Problem der $L^p$-Approximation
periodischer Lipschitz-Funktionen von $d$ Variablen
mithilfe von $n$ Funktionswerten.
Hier ist die optimale Information durch Funktionswerte
auf einem regulären Gitter gegeben.
Zufällige Information soll dagegen durch Funktionswerte an
$n$ unabhängigen, gleichverteilten Punkten gegeben sein.
Es stellt sich heraus, dass sich die Güte zufälliger Information
im Fall $p<\infty$ asymptotisch genauso verhält 
wie die Güte optimaler Information, 
siehe Korollar~\ref{cor:Lip p<infty}. 
Der Fall $p=\infty$ ist das Problem der gleichmäßigen Approximation
von Lipschitz-Funktionen.
Hier ist zufällige Information asymptotisch etwas schlechter
als optimale Information, jedoch nur wenig,
siehe Korollar~\ref{cor:Lip p=infty} sowie \cite{BDKKW17}.

Das zweite Beispiel ist das Problem der $\ell^2$-Approximation
von Punkten aus einem $m$-dimensionalen Ellipsoid
mithilfe von $n$ linearen Messungen,
wobei wir uns vorstellen, dass $m$ viel größer als $n$ ist,
beispielsweise $m=2^n$.
Optimale Information ist hier durch die Koordinaten
in Richtung der $n$ größten Halbachsen des Ellipsoids gegeben.
Zufällige Information ist dagegen durch Koordinaten
in $n$ zufällige Richtungen gegeben,
die unabhängig und gleichverteilt auf der Sphäre in $\IR^m$ sind.
Abhängig von der Dicke des Ellipsoids
erhalten wir sehr verschiedene Ergebnisse 
über die Güte zufälliger Information:
Wenn die geordnete Folge der Halbachsen des Ellipsoids
schnell genug abfällt, so ist zufällige Information
fast genauso gut wie optimale Information.
Fällt die Folge zu langsam,
so ist zufällige Information beinahe völlig nutzlos,
siehe Satz~\ref{thm:main result random info hilbert}.
Wir werden auch eine Version dieses Problems im Fall $m=\infty$
besprechen. Dieser Fall entspricht dem Problem der
$L^2$-Approximation von Funktionen aus einem kompakt
eingebetteten Hilbertraum.

\subsubsection{Veröffentlichungen}

Die meisten Ergebnisse dieser Dissertation wurden bereits veröffentlicht.
Es folgt eine Liste der relevanten Veröffentlichungen des Autors.
Die Reihenfolge entspricht der Reihenfolge der Abschnitte dieser Arbeit.
Der zweite Punkt der Liste ist die Masterarbeit des Autors.

\begin{enumerate}
 \item mit \textsc{E.\,Novak}.
      \newblock A universal algorithm
	    for multivariate integration.
      \newblock {\em Foundation of Computational Mathematics}, 17(4):895--916, 2017,
      siehe Abschnitt~\ref{sec:Frolov}.
 \item On the randomization of Frolov's algorithm for multivariate integration.
      \newblock Masterarbeit, Friedrich-Schiller-Universit\"at Jena, 2016,
      verfügbar als arXiv:1603.04637 [math.NA],
      siehe Abschnitt~\ref{sec:Frolov}.
 \item Tensor power sequences and the approximation of tensor product operators.
      \newblock {\em Journal of Complexity}, 44:30--51, 2018,
      siehe Abschnitt~\ref{sec:tensorproduct}.
 \item Optimal Monte Carlo methods for $L^2$-approximation.
      \newblock {\em Constructive Approximation}, 2018,
      \newblock https://doi.org/10.1007/s00365-018-9428-4,
      siehe Abschnitt~\ref{sec:OptimalMC}.
 \item Uniform recovery of high-dimensional $C^r$-functions.
      \newblock {\em Journal of Complexity}, 50:116--126, 2019,
      siehe Abschnitt~\ref{sec:ck functions}.
 \item mit \textsc{D.\,Rudolf}.
      \newblock Recovery algorithms for high-dimensional rank one tensors.
      \newblock {\em Journal of Approximation Theory}, 
      237:17--29, 2019,
      siehe Abschnitt~\ref{sec:rank one}.
 \item On the dispersion of sparse grids.
      \newblock {\em Journal of Complexity}, 45:115--119, 2018,
      siehe Abschnitt~\ref{sec:dispersion}.
 \item mit \textsc{A.\,Hinrichs, E.\,Novak, J.\,Prochno, and M.\,Ullrich}. 
      \newblock Random Abschnitts of ellipsoids and the power of random information.
      \newblock Preprint, verfügbar als arXiv:1901.06639 [math.FA],
      siehe Abschnitt~\ref{sec:random info hilbert}.
\end{enumerate}

\selectlanguage{english}
\cleardoublepage
\chapter*{Introduction and Results}
\addcontentsline{toc}{chapter}{Introduction and Results}
\markboth{Introduction and Results}{Introduction and Results}

Even with the help of modern computers, we 
cannot hope to collect or process 
more than a finite amount of information 
about the world in finite time.
This makes it impossible to find
exact solutions to many numerical problems,
such as
the question 
for the dependence
of a certain quantity 
upon several parameters 
(a so-called approximation problem)
or the computation of some average
of infinitely many values
(a so-called integration problem).
It is then necessary to settle for 
approximate solutions
which may be obtained from a finite amount of information.

The information usually consists of two parts:
the a priori knowledge and the empirical data.
The a priori knowledge
is inherent to the problem
or simply assumed by our model.
For example, we might have some knowledge about
the regularity of the function that 
we want to approximate.
The data has to be gained from 
measurements, surveys, programs, etc.\ We process the data
to generate the approximate solution.
%

The whole procedure is described by an algorithm.
Each algorithm $A$ has a certain cost, denoted by $\cost(A)$, 
and a certain error, denoted by $\err(A)$.
The cost measures the effort that is needed
to obtain the approximate solution.
It is often proportional to 
the amount of collected data.
The error measures the possible disparity 
of the approximate and the exact solution.
Given a problem $\mathcal{P}$,
the question is how small the error of an algorithm
with maximal cost $n$ can possibly be.
We ask for the $n^{\rm th}$ minimal error 
$$
 \e(n,\mathcal{P}) = \inf\set{\err(A) \mid \cost(A)\leq n}.
$$
Conversely, we 
ask for the minimal
cost that is needed to achieve an error of
at most $\varepsilon>0$, that is,
$$
 \comp(\varepsilon,\mathcal{P})= \min\set{\cost(A) \mid \err(A)\leq \varepsilon}.
$$
This quantity is called the $\varepsilon$-complexity of the problem.
We also talk about the $\varepsilon$-information complexity
if $\cost(A)$ is given by the amount of information
that is required by the algorithm.
The $n^{\rm th}$ minimal error and the $\varepsilon$-complexity 
are inverse to one another
and measure the difficulty of the problem. 

Many classical investigations
are concerned with the speed of decay of the 
$n^{\rm th}$ minimal error as $n$ tends to infinity,
or equivalently, with the behavior of the
$\varepsilon$-complexity as $\varepsilon$ tends to zero.
But quite often, the problem has some intrinsic dimension $d\in\IN$. 
For example, the problem instance may be a
function of $d$ variables.
We then talk about a multivariate problem,
which we denote by $\mathcal{P}_d$. 
In this case, we are interested in the
behavior of $\comp(\varepsilon,\mathcal{P}_d)$
as a function of both $\varepsilon$ and $d$.
At best, we hope that the problem is polynomially tractable,
that is, the complexity depends at most polynomially
on both $\varepsilon^{-1}$ and $d$.
However, many problems suffer from the curse of dimensionality:
the $\varepsilon$-complexity increases exponentially
with the dimension for some $\varepsilon$.
Of course, there are many shades of tractability
in between these extremes.
For instance, the problem is called quasi-polynomially tractable
if the $\varepsilon$-complexity increases at most
polynomially with the dimension for any fixed $\varepsilon$
and the polynomial order increases at most
logarithmically with $\varepsilon^{-1}$.

For practical purposes, however, 
it is not enough to know 
how much an algorithm can possibly achieve. 
One actually wants to get hold of optimal algorithms.
These are algorithms that achieve an error of at most $\varepsilon$
with (almost) minimal cost.

In this thesis,
we contribute to a collection of several multivariate problems.
We study numerical integration and approximation,
global optimization and the problem of dispersion.
In some cases, we present new results
on the speed of decay of $\e(n,\mathcal{P}_d)$.
In other cases, we give tractability results.
From the respective points of view, 
we provide optimal algorithms
for most of the problems.
These results can be found in Chapters~2--4.
The theoretical foundations 
are discussed in Chapter~\ref{chap:problems and algorithms}.

\subsubsection{On Chapter~\ref{chap:mixed smoothness}:
	Integration and Approximation of Functions with Mixed Smoothness}

This chapter is centered around the integration
and approximation problem for multivariate functions
having bounded mixed derivatives.
Such functions appear, for example, in the context of the
electronic Schr\"odinger equation or
certain integral equations \cite[Sec.\,9.1]{DTU18}.
More precisely, we consider functions from the class
$$
 F_d^r=\Bigg\{f\in L^2([0,1]^d) \,\,\Bigg\vert\,
 \sum\limits_{\alpha \in \set{0,\dots,r}^d}
 \norm{\diff^\alpha f}_2^2 \leq 1\Bigg\}.
$$

The first section is concerned with the integration problem.
Let $\mathcal{P}[\INT,F_d^r,\rm{det}]$
be the problem of integrating functions from $F_d^r$
with deterministic algorithms that use
function values as information.
For this problem, it is known that
$$
 \e(n,\mathcal{P}[\INT,F_d^r,\mathrm{det}])
 \asymp
 n^{-r} \ln^{\frac{d-1}{2}} n.
$$
An optimal algorithm was given by Frolov in 1976 \cite{Fr76}.
This order of convergence may be improved by randomized algorithms.
The problem $\mathcal{P}[\INT,F_d^r,\rm{ran}]$
of integrating such functions
with randomized algorithms satisfies
$$
 \e(n,\mathcal{P}[\INT,F_d^r,\mathrm{ran}])
 \asymp
 n^{-r-1/2}.
$$
In particular, the order is independent of the dimension $d$.
This is a consequence of Theorem~\ref{thm:main theorem} 
which states that a randomly shifted and dilated version
of Frolov's algorithm is optimal for this problem,
see also~\cite{KN17,Ul17}.

The remaining sections are concerned with the problem
of $L^2$-approximation.
For this problem it makes sense to study algorithms
that use function values as information as well as
algorithms that use arbitrary pieces of linear information.

Section~\ref{sec:tensorproduct} is concerned with the case of linear information.
In this case, deterministic algorithms are practically
as powerful as randomized algorithms \cite{No92}.
This means that it is enough to study the problem
$\mathcal{P}[\APP,F_d^r,\lall,\mathrm{det}]$
of approximating such functions in $L^2$
with deterministic algorithms that use 
linear information.
It is known since 1960 \cite{Ba60} that
$$
 \e(n,\mathcal{P}[\APP,F_d^r,\lall,\mathrm{det}])
 \asymp
 n^{-r} \ln^{r(d-1)} n.
$$
An optimal algorithm is given by the singular
value decomposition of the associated embedding.
Here, we go a little more into detail.
In the spirit of \cite{KSU15}, we show that
$$
 \e(n,\mathcal{P}[\APP,F_d^r,\lall,\mathrm{det}])
 \sim
 (\pi^d (d-1)!\, n)^{-r} \ln^{r(d-1)} n
$$
in the sense of strong equivalence of sequences,
see Corollary~\ref{cor: strong equivalence mix all}.
Since these results are only relevant for very
large $n$, we also provide explicit estimates
for small $n$, preasymptotic estimates. 
In Corollary~\ref{cor: preasymptotic lower bound mix all}
and Corollary~\ref{cor:preasymptotics mixed cube},
we prove that
$$
 \e(n,\mathcal{P}[\APP,F_d^r,\lall,\mathrm{det}])
 \leq\, 2\,n^{-c(d)}
  \quad\text{with}\quad 
  c(d)=\frac{1.1929}{2+\ln d}
$$
for all $n\in\IN$
and that this bound cannot be improved much
for $n<2^d$.

In Section~\ref{sec:OptimalMC}, we turn to the case of function values as information.
We provide an optimal randomized algorithm
for the respective problem $\mathcal{P}[\APP,F_d^r,\lstd,\mathrm{ran}]$
and show that
$$
 \e(n,\mathcal{P}[\APP,F_d^r,\lstd,\mathrm{ran}])
 \asymp
 n^{-r} \ln^{r(d-1)} n,
$$
see Corollary~\ref{cor:order OptimalMC mixed}.
Therefore, function values are as powerful
as arbitrary linear information, as long as randomized algorithms 
are allowed.
Also the preasymptotic estimates are similar.
We get
$$
\e(n,\mathcal{P}[\APP,F_d^r,\lstd,\mathrm{ran}])
\leq\, 8\,n^{-c(d)}
$$
for all $n\in\IN$ with $c(d)$ as above, see \eqref{eq:preasymptotic ran std}.
Note that the question for optimal algorithms
and the order of convergence is still unsolved
for deterministic algorithms that use function values
as information.
We remark that each section will cover
a more general setting:
\begin{itemize}
 \item Section~\ref{sec:Frolov}: Frolov's algorithm and its randomization 
 are optimal for many classes of smooth functions.
 \item Section~\ref{sec:tensorproduct}: We study optimal algorithms
 for the $L^2$-approximation of functions from the unit ball
 of any tensor product Hilbert space.
 \item Section~\ref{sec:OptimalMC}: We provide optimal randomized algorithms
 for the $L^2$-approximation of functions from 
 the unit ball of any Hilbert space that is
 compactly embedded in the respective $L^2$-space,
 provided that the singular values of this
 embedding satisfy a certain decay condition.
\end{itemize}

\subsubsection{On Chapter~\ref{chap:tractability uniform approximation}:
	Tractability of the Uniform Approximation Problem}

In this chapter,
we study the power of deterministic algorithms 
for the problem of recovering a
function $f:[0,1]^d\to \IR$ 
from a finite number of function values 
in the uniform norm.
In order to achieve anything at all,
it is necessary to have some a priori knowledge
about the function, say $f\in F_d$
for some $F_d\subset L^\infty([0,1]^d)$.
Let $\mathcal{P}[F_d]$ be the problem of uniform approximation
with a priori knowledge $F_d$.
We are interested in the tractability of this problem.
In particular, we want to know
what kind of a priori knowledge leads
to positive tractability results.

It is well known that smoothness alone is not enough.
Even if we have the a priori knowledge 
$$
 F_d = \set{f \in \C^\infty([0,1]^d) \,\big\vert\, 
 \Vert \diff^\alpha f \Vert_\infty \leq 1 
 \text{\ for all } \alpha \in \IN_0^d},
$$
the problem suffers from the curse of dimensionality~\cite{NW09}.
Of course, the curse stays present
if we only know about finite smoothness $r\in\IN$,
that is, if we have the a priori knowledge 
$$
 \C^r_d= \set{f \in \C^r([0,1]^d) \,\big\vert\,
 \Vert \diff^\alpha f \Vert_\infty \leq 1 
 \text{\ for all } \alpha \in \IN_0^d \text{ with  } \abs{\alpha}\leq r}.
$$
But how bad is the situation exactly?
This question is studied in Section~\ref{sec:ck functions}.
For even numbers $r$, we find that 
there are positive constants $c_r$, $C_r$ and $\varepsilon_r$
such that
\begin{equation*}
 \brackets{c_r \sqrt{d}\, \varepsilon^{-1/r}}^d
 \leq \comp(\varepsilon,\mathcal{P}[\C^r_d]) \leq
 \brackets{C_r \sqrt{d}\, \varepsilon^{-1/r}}^d
\end{equation*}
for all $d\in\IN$ and $\varepsilon\in (0,\varepsilon_r)$,
see Theorem~\ref{main theorem even}.
It follows from~\cite{Wa84}
that the same estimates hold for the problem
of global optimization on $\C^r_d$
since this class is convex and symmetric,
see Section~\ref{sec:optimization}.
In particular, the complexity of both problems
grows like $d^{d/2}$ for any fixed $\varepsilon>0$
and $r\geq 2$.
For odd numbers $r$, the precise behavior of the
complexity as a function of both $\varepsilon$ and $d$
is still unclear.

On the other hand, it is known that additional knowledge
about the structure of $f$ may lead to tractability of the
uniform approximation problem.
For example, we may assume that $f$ is
\begin{itemize}
 \item a ridge function \cite{MUV15}. That is, it
 can be written in the form
 $g(\langle \cdot,\mathbf x_0\rangle)$ for some $\mathbf x_0\in\IR^d$ and $g:\IR\to\IR$.
 \item separable \cite{NR97,WW04}. 
 It can be written as a sum of 
 $m$-variate functions, where $m$
 is independent of $d$. 
 Note that the above papers are not concerned
 with uniform approximation but with $L^2$-approximation
 and integration.
 \item symmetric \cite{We12}. That is,
 $f(\mathbf x)$ is invariant under a reordering of the
 coordinates of $\mathbf x\in[0,1]^d$. Note that
 the functions in~\cite{We12} are to be recovered from arbitrary linear information
 and not exclusively from function values.
\end{itemize}
Another example is studied in Section~\ref{sec:rank one}.
Here, we assume that $f$ is a rank one tensor.
That is, it can be written as a product of $d$ univariate functions.
More precisely, we assume that $f$
is contained in the class
$$
 F_{r,M}^d =
 \Big\{ \bigotimes_{i=1}^d  f_i \,\Big\vert\,
 f_i\colon [0,1]\to [-1,1], \
 \Vert f_i^{(r)} \Vert_\infty \le M \Big\}
$$
for some smoothness parameters $r\in\IN$ and $M>0$.
The function $\bigotimes_{i=1}^d  f_i$ is called 
the tensor product of the functions $f_i$ over $i\leq d$
and maps $\mathbf x\in [0,1]^d$ to the product
of all~$f_i(x_i)$. In Theorem~\ref{tractability results}
we find that uniform approximation with a priori
knowledge $F_{r,M}^d$ suffers from the curse
of dimensionality iff $M\geq 2^r r!$.
It is quasi-polynomially tractable iff $M<2^r r!$
and even polynomially tractable iff $M\leq r!$.
In every case we provide an optimal algorithm.
Moreover, we show that the same tractability results
hold for the problem of global optimization on the class $F_{r,M}^d$
which is symmetric but not convex, see Section~\ref{sec:optimization}.

It will become apparent that the uniform
approximation of rank one tensors
is closely related to the problem of dispersion.
The dispersion of a finite point set in $[0,1]^d$ 
is the volume of the largest empty axis-aligned box
amidst the point set.
This quantity is also of independent interest.
One asks for the minimal cardinality that is
necessary to achieve a dispersion of at most $\varepsilon$
in dimension $d$,
but also for explicit point sets with this property.
In Section~\ref{sec:dispersion} we provide such a point set
for every $\varepsilon>0$ and every $d\in\IN$.
In a vast range of the parameters $(\varepsilon,d)$,
we do not know any smaller point set with this property.
The point set is an instance of a sparse grid
and hence easy to handle.
It may be used for the algorithms from Section~\ref{sec:rank one}.

\subsubsection{On Chapter~\ref{chap:random info}:
	Optimal Information versus Random Information}

The last chapter is somewhat different.
In the previous chapters, we aimed at finding
optimal algorithms that use optimal information
about the problem instance.
However, quite often we do not have access to optimal information. 
The reason may be that 
we do not know which 
kind of measurements lead to optimal information
or that we do not even get to choose our measurements.
In this chapter, we assume that the information
comes in randomly
and ask the following question:
\begin{center}
What is the typical quality 
of random information?
\end{center}
Of course, random information cannot be better than
optimal information, but it
may turn out that typical random information 
is only slightly worse.
In this case, searching for optimal
information is rather pointless.

To make this more precise,
we need to clarify how we measure the
quality of our information 
and what we mean by random.
The first is done with the so-called
radius of information,
which is the worst case error of the best algorithm
that uses nothing but the given information
and the a priori knowledge about the
problem instance,
see Section~\ref{sec:errors}.
The random information, on the other hand, 
shall be obtained from
a certain number of independent measurements
that all follow the same law.
In general,
there is no right or wrong
in the choice of the distribution
that we want to investigate.
However, we study this question for two basic examples
for which there seems to be a natural
choice for this distribution.

The first example is the problem of $L^p$-approximation
of periodic Lipschitz functions on the $d$-dimensional unit cube
using $n$ function values.
While optimal information is given by
function values on a regular grid,
random information shall be given by function
values at $n$ random points that are chosen
independently and uniformly from the domain.
It turns out that typical random information
is asymptotically just as good as optimal information
if $p<\infty$, see Corollary~\ref{cor:Lip p<infty}. 
For $p=\infty$, it is only slightly worse,
see Corollary~\ref{cor:Lip p=infty} and \cite{BDKKW17}.

The second example is the problem of 
$\ell^2$-approximation
of a point from an $m$-dimensional ellipsoid
by means of $n$ linear measurements,
where we imagine that $m$ is much larger than $n$.
While optimal information is given by 
$n$ scalar products
in direction of the largest semi-axes,
random information shall be given by scalar products
in $n$ directions taken independently from the 
uniform distribution on the sphere in $\IR^m$.
We obtain very different results
depending on the shape of the ellipsoid:
If the ordered sequence of semi-axes decays fast enough,
typical random information is almost as good
as optimal information.
If it decays too slowly,
typical random information is practically useless,
see Theorem~\ref{thm:main result random info hilbert}.
We shall also present a variant of these results
for $m=\infty$, which corresponds to
the problem of $L^2$-approximation in
a Hilbert space.

\subsubsection{Publications}

Most of the results in this thesis are already published. 
Below, the relevant papers are listed 
in order of the corresponding sections.
The second item is the author's master thesis.

\begin{enumerate}
 \item with \textsc{E.\,Novak}.
      \newblock A universal algorithm
	    for multivariate integration.
      \newblock {\em Foundation of Computational Mathematics}, 17(4):895--916, 2017,
      see Section~\ref{sec:Frolov}.
 \item On the randomization of Frolov's algorithm for multivariate integration.
      \newblock Master thesis, Friedrich-Schiller-Universit\"at Jena, 2016,
      available on arXiv:1603.04637 [math.NA],
      see Section~\ref{sec:Frolov}.
 \item Tensor power sequences and the approximation of tensor product operators.
      \newblock {\em Journal of Complexity}, 44:30--51, 2018,
      see Section~\ref{sec:tensorproduct}.
 \item Optimal Monte Carlo methods for $L^2$-approximation.
      \newblock {\em Constructive Approximation}, 2018,
      \newblock https://doi.org/10.1007/s00365-018-9428-4,
      see Section~\ref{sec:OptimalMC}.
 \item Uniform recovery of high-dimensional $C^r$-functions.
      \newblock {\em Journal of Complexity}, 50:116--126, 2019,
      see Section~\ref{sec:ck functions}.
 \item with \textsc{D.\,Rudolf}.
      \newblock Recovery algorithms for high-dimensional rank one tensors.
      \newblock {\em Journal of Approximation Theory}, 
      237:17--29, 2019,
      see Section~\ref{sec:rank one}.
 \item On the dispersion of sparse grids.
      \newblock {\em Journal of Complexity}, 45:115--119, 2018,
      see Section~\ref{sec:dispersion}.
 \item with \textsc{A.\,Hinrichs, E.\,Novak, J.\,Prochno, and M.\,Ullrich}. 
      \newblock Random sections of ellipsoids and the power of random information.
      \newblock Preprint, available on arXiv:1901.06639 [math.FA],
      see Section~\ref{sec:random info hilbert}.
\end{enumerate}


\newpage

\pagestyle{scrplain}

\cleardoublepage
\tableofcontents
\newpage

\pagenumbering{arabic}
\setcounter{page}{1}
\pagestyle{scrheadings}
\automark[section]{chapter}

\cleardoublepage

\chapter{Problems and Algorithms}
\label{chap:problems and algorithms}

In most cases, a numerical problem is associated
with a solution operator $S:F\to G$. 
For example, we may think of the computation of integrals
$$
 S:F\to \IR, \quad S(f)=\int_0^1 f(x) ~dx
$$
for some input set $F$ of integrable functions on $[0,1]$.
Then a (deterministic) algorithm is just a particular
mapping $A:F\to G$, computing some
output $A(f)\in G$ for every input $f\in F$.
For example, a quadrature rule is a mapping
$$
 A:F\to \IR, \quad A(f)=\sum_{i=1}^n a_i f(x_i)
$$
for some number $n\in\IN$, weights $a_i\in\IR$
and nodes $x_i\in [0,1]$.
Each algorithm is assigned a cost and an error.
In one way or another, 
the error measures the distance
between the output $A(f)$ and the solution $S(f)$,
while the cost
measures the effort for computing the output.
In the above example, one could define
$$
 \cost(A)=n 
 \qquad\text{and}\qquad
 \err(A)=\sup_{f\in F} \abs{S(f)-A(f)}.
$$
We shall discuss various types of problems
that are defined via a solution operator
in Section~\ref{sec:usual problems}.
However, we also want to study the problem of dispersion
and the problem of finding a global maximizer, 
which are not associated with a solution operator.
For this reason, we first introduce an
abstract notion of a problem.

\section{General Notions}

\begin{defi}
\label{def:problem}
 A \emph{problem} $\mathcal{P}$ is a triple
 $(\mathcal{A},\err,\cost)$ consisting of a set $\mathcal{A}$
 and two functions
 $$
 \err: \mathcal{A}\to [0,\infty],
 \qquad
 \cost: \mathcal{A}\to \set{0,1,2,\hdots,\infty}.
 $$
 For $A\in\mathcal{A}$ the numbers $\err(A)$ and $\cost(A)$ are called the \emph{error}
 and the \emph{cost} of $A$.
 For every $n\in\IN$ the \emph{$n^{\rm th}$ minimal error} of $\mathcal{P}$
 is defined by
 $$
  \e(n,\mathcal{P})
  =\inf\set{\err(A) \mid A\in\mathcal{A}, \cost(A)\leq n}.
 $$
 For every $\varepsilon\geq 0$ the \emph{$\varepsilon$-complexity} of $\mathcal{P}$
 is defined by
 $$
  \comp(\varepsilon,\mathcal{P})
  =\min\set{\cost(A) \mid A\in\mathcal{A}, \err(A)\leq \varepsilon}.
 $$
\end{defi}

Many problems are inherited from a solution operator.
In this case, the set $\mathcal{A}$ consists of algorithms.
Before we turn to such problems,
let us have a look at an example of a geometric problem 
which is not defined via a solution operator,
the problem of dispersion.
A second example, 
the problem of finding a global maximizer,
will be described in Section~\ref{sec:optimization}.

\begin{ex}[The problem of dispersion, Part 1 of 3]
\label{ex:dispersion}
 For every $d\in \IN$ let $\mathcal{S}_d$ be the set of
 all finite subsets of $[0,1]^d$. Let $\mathcal{B}_d$ be
 the set of all boxes in $[0,1]^d$, that is, 
 $$
  \mathcal{B}_d
  =\set{\prod_{j=1}^d I_j\,\big|\, I_j\subset [0,1] \text{ interval}}.
 $$
 The dispersion of a point set $P\in\mathcal{S}_d$ 
 is the volume of the largest empty box amidst the point set,
 that is,
 $$
  \disp(P)
  =\sup\set{\lambda^d(B) \mid B\in \mathcal B_d, B\cap P =\emptyset }.
 $$
 We consider the problem $\mathcal{P}_d=(\mathcal{S}_d,\disp,\card)$.
 In this case, $\e(n,\mathcal{P}_d)$ is the minimal dispersion
 of $n$ points in $[0,1]^d$. 
 The complexity $\comp(\varepsilon,\mathcal{P}_d)$
 is the minimal cardinality of a $d$-dimensional point set
 achieving a dispersion of at most $\varepsilon$.
\end{ex}

A problem is called \emph{solvable}
if the $n^{\rm th}$ minimal error tends to 0
as $n$ tends to infinity.
Numerous classical investigations ask for
the speed of this convergence.


\begin{namedthm}{Example~\ref{ex:dispersion}}[Part 2 of 3]
 By dividing the unit cube into $(n+1)$ boxes
 of equal volume, we immediately see that the dispersion
 of $n$ points is at least $1/(n+1)$.
 On the other hand, Rote and Tichy \cite{RT96} 
 showed in 1996 that the dispersion of the
 first $n$ points of the Halton-Hammersely sequence
 is at most $2^{d-1}\pi_d /n$, where 
 $\pi_d$ is the product of the first $(d-1)$ primes.
 Hence,
 $$
  \e(n,\mathcal{P}_d) \asymp n^{-1},
 $$
 that is, the $n^{\rm th}$ minimal error of the problem
 of dispersion decays with polynomial order 1
 for all fixed $d\in\IN$.
 In particular, the problem of dispersion is solvable.
\end{namedthm}

The question for the speed of decay translates 
into the question
for the dependence of the complexity on $\varepsilon$.
Many problems, like the dispersion problem,
have some intrinsic dimension $d\in\IN$.
There is growing interest in the $d$-dependence
of the complexity.
Tractability asks for the behavior of 
the complexity as a function of both $\varepsilon$ and $d$.
We give some examples of tractability notions.
Note that the following list is far from complete.

\begin{defi}
 Consider a family of problems $\mathcal{P}_d$
 with index $d\in\IN$. The family
 \begin{itemize}
 \item is \emph{strongly polynomially tractable} 
  if there are constants $c,p>0$ such that 
  $\comp(\varepsilon,\mathcal{P}_d)\leq c\, \varepsilon^{-p}$ 
  for all $\varepsilon\in(0,1)$ and all $d\in\IN$;
 \item is \emph{polynomially tractable} 
  if there are constants $c,q,p>0$ such that 
  $\comp(\varepsilon,\mathcal{P}_d)\leq c \,\varepsilon^{-p} d^q$ 
  for all $\varepsilon\in(0,1)$ and all $d\in\IN$;
 \item is \emph{quasi-polynomially tractable} 
  if there are constants $c,t>0$ such that
  $$
  \comp(\varepsilon,\mathcal{P}_d)
  \leq c \exp\left(t(1+\ln(\varepsilon^{-1}))(1+\ln d)\right)
  $$
 for all $\varepsilon\in(0,1)$ and all $d\in\IN$;
 \item suffers from the \emph{curse of dimensionality} 
  if there is some $\varepsilon>0$, $c>0$ and $\alpha>1$ 
  such that $\comp(\varepsilon,\mathcal{P}_d)\geq c \alpha^d$
  for all $d\in\IN$.
 \end{itemize}
\end{defi}

Note that the term problem often refers to a whole family of problems.

\begin{namedthm}{Example~\ref{ex:dispersion}}[Part 3 of 3]
 The interest in the $d$-dependence of the 
 complexity of the problem of dispersion
 started much later.
 Aistleitner, Hinrichs, and Rudolf \cite{AHR17}
 were the first to show that the complexity
 increases with the dimension.
 In 2017, they proved
 $$
  \comp(\varepsilon,\mathcal{P}_d) 
  \geq (4\varepsilon)^{-1} (1-4\varepsilon)\log_2 d
 $$
 for all $d\in\IN$ and $\varepsilon\leq 1/4$.
 In 2018, Sosnovec \cite{So18} showed
 that this logarithmic dependence on $d$ is already optimal. 
 Not much later, Ullrich and Vybíral \cite{UV18} proved that
 $$
  \comp(\varepsilon,\mathcal{P}_d) 
  \leq \left\lceil 2^7\, 
  \varepsilon^{-2} \brackets{1 + \log_2\brackets{\varepsilon^{-1}}}^2  
  \,\log_2 d \right\rceil
 $$
 for all $d\geq 2$ and $\varepsilon<1/2$.
 In particular, the problem of dispersion is
 polynomially tractable, but not strongly
 polynomially tractable.
\end{namedthm}

\section{Important Types of Problems}
\label{sec:usual problems}

We now turn to problems
that are inherited from a solution operator.
In the following, let $F$ be a set
and let $(G,\dist)$ be a metric space.
We consider a mapping
$$
S: F \to G,
$$
which we call the \emph{solution operator}.
The set $F$ is called the \emph{input set},
$f\in F$ is called the \emph{input} or \emph{problem instance}
and $S(f)$ is called the \emph{solution}.

In this section, we discuss important types of algorithms,
error functions and cost functions
that are associated with $S$,
thereby defining various problems
in the sense of Definition~\ref{def:problem}.
We introduce basic concepts of information-based complexity.
For a detailed discussion
and a variety of further problems,
we refer the reader to \cite{TWW88} and the monographs
\cite{NW08,NW10,NW12}.

\subsection{Algorithms}
\label{sec:algorithms}

In this thesis, a (deterministic) algorithm
is nothing but a particular mapping $A:F\to G$.
It is described by the \emph{output} $A(f)$
belonging to each \emph{input} $f\in F$.

\begin{rem}
One may rightfully object that an actual algorithm
is not fully determined by its outputs.
However, we are only interested in the
error of the algorithm in comparison with its information cost.
These characteristics are already given by the 
input-output mapping $A$ itself,
see Section~\ref{sec:errors} and Section~\ref{sec:cost}.
If we wanted to talk about computational cost,
then we would have to describe an algorithm
by all the computational steps it performs.
\end{rem}

We assume that an algorithm can be decomposed into two parts.
The first is a mapping $N: F\to c_{00}$
that collects a finite amount of information
about the input.
Here $c_{00}$ is the union of all $\IR^n$ over $n\in\IN_0$.
The second is a mapping $\phi: N(F)\to G$
that uses this information to produce an output.
We now discuss these two parts.

The information mapping $N$ collects the information by taking several
measurements of the problem instance.
For different problems, different types of measurements may be executable.
Let $\Lambda$ be a class of real-valued functions $L:F\to \IR$. 
A functional $L\in\Lambda$ is called a \emph{measurement}, 
the number $L(f)\in \IR$ is called a \emph{piece of information} about $f$.\footnote{Analogously,
we may consider functionals $L:F\to\IC$ such that one piece of
of information is given by one complex number $L(f)\in\IC$.
For simplicity, we only discuss the $\IR$-valued case.}
We give two popular examples:
\begin{itemize}
 \item If $F$ consists of real-valued functions on a common domain $D$,
 we often consider the class $\Lambda=\lstd$ of function evaluations 
 $L(f)=f(x)$ for all $x\in D$,
 the class of \emph{standard information}.
 \item If $F$ is a subset of a normed space,
 we may allow the class $\Lambda=\lall$ of all continuous linear functionals,
 the class of \emph{linear information}.
\end{itemize}
A \emph{nonadaptive information mapping} based on $\Lambda$ 
is a mapping of the form
$$
 N_n: F\to \IR^n, 
 \quad 
 N_n(f)=(L_1(f),\hdots,L_n(f))
$$
for some $n\in\IN$ and measurements $L_1,\hdots,L_n\in\Lambda$.
That is, $N_n$ collects $n$ pieces of information 
about the problem instance.
We take the same measurements for every input.
In contrast, an adaptive information mapping 
may use the already collected pieces of information after each measurement 
to decide whether and how to take another measurement.
In general, a mapping $N:F\to c_{00}$ is called an 
\emph{information mapping} based on $\Lambda$ 
if there are
\begin{itemize}
 \item functionals $L_i:F\times \IR^{i-1} \to \IR$
 such that $L_i(\cdot,\mathbf y)\in\Lambda$ for all $\mathbf y\in\IR^{i-1}$, $i\in\IN$;
 \item a function $T:c_{00}\to\set{0,1}$,
 which we call the \emph{termination function};
\end{itemize}
such that for every $f\in F$ we have $N(f)=(y_1,\hdots,y_{n(f)})$
with
$$
 y_i=L_i(f,y_1,\hdots,y_{i-1})
 \quad\text{and}\quad 
 n(f)=\min\set{n\in\IN \mid T(y_1,\hdots,y_n)=0}.
$$
The family $((L_i)_{i\in\IN},T)$ is called a \emph{representation}
of the information mapping.
The information is called \emph{adaptive} if it is not nonadaptive.

To generate an output from the collected information, 
we allow any function $\phi: N(F)\to \IR$. 
Of course, this means that the computational cost to
obtain $\phi(\mathbf y)$ for $\mathbf y\in N(F)$ may be arbitrarily high.
For concrete algorithms, the function $\phi$
should be as simple as possible.

A mapping $A: F\to G$ is called a 
\emph{deterministic algorithm} based on $\Lambda$
if there is an information mapping 
$N:F\to c_{00}$ and a function $\phi:N(F)\to \IR$ 
such that $A=\phi \circ N$.
The pair $(\phi,N)$ is called a \emph{representation}
of the algorithm $A$.
It is said to be \emph{nonadaptive} if the information
mapping $N$ can be chosen to be nonadaptive.
Else, it is called \emph{adaptive}.
The class of all deterministic algorithms based on $\Lambda$
is denoted by 
$$
 \mathcal{A}[F,G,\Lambda,\mathrm{det}]. 
$$
The class of all nonadaptive deterministic algorithms based on $\Lambda$
is denoted by 
$$
 \mathcal{A}[F,G,\Lambda,\mathrm{det},\mathrm{nonada}].
$$
If $F$ and $G$ are subsets of linear spaces over $\IR$,
the algorithm may be \emph{linear}.
The class of all linear deterministic algorithms based on $\Lambda$
is denoted by 
$$
 \mathcal{A}[F,G,\Lambda,\mathrm{det},\mathrm{lin}].
$$
Note that every linear algorithm is nonadaptive and hence
$$
 \mathcal{A}[F,G,\Lambda,\mathrm{det},\mathrm{lin}]
 \subset
 \mathcal{A}[F,G,\Lambda,\mathrm{det},\mathrm{nonada}]
 \subset
 \mathcal{A}[F,G,\Lambda,\mathrm{det}].
$$

Let us turn to randomized algorithms.
Here, we assume that $F$ is equipped with a topology.
A \emph{randomized algorithm} based on $\Lambda$
is a family $(A^\omega)_{\omega\in\Omega}$ of
deterministic algorithms based on $\Lambda$ which is indexed
by a probability space $(\Omega,\mathcal{F},\IP)$ 
such that the mapping
$$
 F\times \Omega \to \IR,
 \quad
 (f,\omega) \mapsto \dist(S(f),A^\omega(f))
$$
is measurable.
The class of all such algorithms is denoted by
$$
 \mathcal{A}[F,G,\Lambda,\mathrm{ran}].
$$
The randomized algorithm is called \emph{nonadaptive} or \emph{linear}
if $A^\omega$ is nonadaptive or linear for almost every $\omega\in \Omega$.
A randomized algorithm is also referred to as
a \emph{Monte Carlo method}.
We use these terms interchangeably.
Randomized algorithms can be regarded as
a generalization of deterministic algorithms
since any deterministic algorithm may be viewed
as a family of deterministic algorithms $A^\omega$
that is independent of $\omega$.

\begin{rem}
 The definition of a randomized algorithm would still make sense
 if we skipped the property of measurability. 
 We prefer this definition since it simplifies the notion of
 the error of a randomized algorithm.
 Moreover, Theorem~\ref{thm:randomization useless in Hspace} 
 has only been proven for measurable algorithms.
\end{rem}

We now introduce different ways to measure
the error and cost of such algorithms.

\subsection{Errors}
\label{sec:errors}

We introduce the error criteria
that are used in this thesis.
The \emph{worst case error} of a deterministic algorithm $A:F\to G$
is defined as
$$
 \err(A,S,F,G,\mathrm{wc})
 = \sup_{f\in F}\, \dist(S(f),A(f)).
$$
It measures the maximal distance between
the output and the solution.
One may weaken this error criterion by
considering the average distance instead.
Given a probability measure $\mu$ on the input class $F$
we define the \emph{average case error}
$$
 \err(A,S,F,G,\mu)
 = \sqrt{\int_F \dist(S(f),A(f))^2~\d\mu(f)}
$$
provided that the error functional $f \mapsto \dist(S(f),A(f))$ is $\mu$-measurable.

The \emph{worst case error} of a randomized algorithm 
$(A^\omega)_{\omega\in\Omega}$ is defined as
$$
 \err((A^\omega)_{\omega\in\Omega},S,F,G,\mathrm{wc})
 = \sup_{f\in F}\sqrt{\IE\left[\dist(S(f),A(f))^2\right]}.
$$
That is, it measures the maximal root mean square of
the distance of the output and the solution.
Given a Borel probability measure $\mu$ on $F$
we also define the \emph{average case error} of
randomized algorithms as
$$
 \err((A^\omega)_{\omega\in\Omega},S,F,G,\mu)
 = \sqrt{\int_{F}\IE\left[\dist(S(f),A(f))^2\right]~\d\mu(f)}.
$$
One could easily replace the root mean square 
in the above definitions
by every other normalized moment of the distance.
Note that these errors coincide with the respective error
of a deterministic algorithm if the algorithm is independent of $\omega$.

To measure the quality of an information mapping,
we introduce the notion of the \emph{radius of information}.
The radius of an information mapping is the smallest error
which can be achieved with algorithms
that use this information mapping.
It can be defined for each of the above error criteria.
For instance, the radius of a deterministic
information mapping $N:F\to c_{00}$ in the worst case
setting is given by
$$
 \rad(N,S,F,G)
 = \inf_{\phi:c_{00}\to G} \err(\phi\circ N,S,F,G,\mathrm{wc}).
$$
Proposition~\ref{prop:radius global vs local} below explains 
why we call this quantity a radius.
Note that the radius of a 
subset $M$ of $G$ is given by
\begin{equation}
 \label{eq:radius of sets}
 \rad(M) = \adjustlimits\inf_{g\in G} \sup_{m\in M} \dist(g,m).
\end{equation}
An algorithm based on $N$ cannot distinguish
inputs with the same information.
Thus the optimal algorithm based on $N$
maps $f\in F$ to the center of 
the set $S(N^{-1}(\mathbf y))$ of all solutions
that are possible for the information $\mathbf y=N(f)$.
The radius of this set is called the
\emph{radius of information at $\mathbf y$}
and denoted by
$$
 r_{\mathbf y}(N,S,F,G)
 = \rad\brackets{S(N^{-1}(\mathbf y))}.
$$
For given information $\mathbf{y}$,
we cannot guarantee an error less than $r_{\mathbf y}(N,S,F,G)$.
This leads to the following result.

\begin{prop}[\cite{TW80}]
\label{prop:radius global vs local}
 Let $S:F\to G$ be a solution operator from a set $F$
 to a metric space $G$ and let $N:F\to c_{00}$
 be an information mapping. Then
 $$
  \rad(N,S,F,G)
  = \sup_{\mathbf y\in N(F)} r_{\mathbf y}(N,S,F,G).
 $$
\end{prop}


%
%

\subsection{Cost}
\label{sec:cost}

In this thesis, the cost of an algorithm is given
by the amount of information 
that the algorithm uses about the problem instance,
that is, we study the \emph{information cost} of an algorithm. 
We do not study their computational cost or other cost models.
This is based on the assumption that collecting information 
usually consumes much more time than processing it:
while the information may be obtained 
from complicated subroutines, physical measurements
or even surveys,
it is usually processed by basic arithmetic operations.
Although this assumption is fulfilled in many examples,
it may not always be adequate.
Then we would have to define problems $(\mathcal{A},\err,\cost)$
with other cost functions.

We first define the cost of information mappings.
The cost of a nonadaptive information mapping $N$ based on $\Lambda$
is simply given by the number $n$ of measurements.
The definition of the cost of adaptive information mappings 
is not quite as indisputable,
since it may take a different number $n(f)$ of measurements
for different inputs $f\in F$.
We study the \emph{worst case cost} of an information mapping.
That is, given an information mapping $N:F\to c_{00}$
as defined in Section~\ref{sec:algorithms}, 
we take the maximum of the number $n(f)$ of measurements
over all possible inputs $f\in F$.
Note that the number $n(f)$ does not depend on the 
representation $((L_i)_{i\in\IN},T)$ of the information mapping $N$.
Hence, we define
$$
 \cost(N,F,\Lambda,\mathrm{wc})= \sup_{f\in F}\, n(f).
$$
Another approach would be to consider the
average number $n(f)$ of measurements with respect to
some measure $\mu$ on $F$, the average case cost of $N$.

We now define the cost of algorithms.
The \emph{worst case cost} of a deterministic algorithm 
$A\in \mathcal{A}[F,G,\Lambda,\mathrm{det}]$
is the worst case cost of the information mapping
in an optimal representation of $A$, that is,
$$
 \cost(A,F,\Lambda,\mathrm{wc})=\min\set{\cost(N,F,\Lambda,\mathrm{wc}) \mid (\phi,N)
 \text{ representation of } A}.
$$
Moreover, we define the \emph{worst case cost} of a randomized algorithm 
$(A^\omega)_{\omega\in\Omega}$ 
by
$$
 \cost((A^\omega)_{\omega\in\Omega},F,\Lambda,\mathrm{wc})
 =\sup_{\omega\in\Omega}\, \cost(A^{\omega},F,\Lambda,\mathrm{wc}).
$$
This is the cost of computing $A^\omega(f)$
for the worst input $f\in F$ 
and the worst realization $A^\omega$ of $A$.
Note that it is also common to consider the 
expectation over all realizations instead of the maximum.

\subsection{Resulting Problems}

We may now formally define the problems of our interest
that are inherited from a solution operator $S$.

\begin{defi}
Let $S:F\to G$ be an operator from a topological space $F$ to a metric space $G$
and let $\Lambda$ be a class of real-valued functions on $F$.
Let $\star\in\set{\mathrm{det},\mathrm{ran}}$,
$\circ\in\set{\emptyset,\mathrm{nonada},\mathrm{lin}}$ 
and $\triangle\in\set{\mu,\mathrm{wc}}$,
where $\mu$ is some probability measure on $F$.
Then we define the problem
$$
 \mathcal{P}[S,F,G,\Lambda,\star,\circ,\triangle]
 =(\mathcal{A},\err,\cost)
$$
of approximating $S$ with (nonadaptive/linear) deterministic/randomized
algorithms based on $\Lambda$ in the worst/average case setting by
\begin{align*}
 \mathcal{A}&=\mathcal{A}[F,G,\Lambda,\star,\circ], \\
 \err&=\err(\cdot,S,F,G,\triangle), \\
 \cost&=\cost(\cdot,F,\Lambda,\mathrm{wc}).
\end{align*}
\end{defi}

\begin{rem}
 Note that we always consider the worst case cost.
 The term average case only refers to the error criterion.
\end{rem}

\begin{rem}
 The setting is determined by the parameters
 $S$, $F$, $G$, $\Lambda$, $\mathrm{det}$ or $\mathrm{ran}$, 
 $\mathrm{wc}$ or $\mathrm{\mu}$,
 and possibly $\mathrm{nonada}$ or $\mathrm{lin}$.
 So far, we put all relevant parameters in the 
 definition of the problems,
 the classes of algorithms, and the error and cost functions.
 In what follows, a part of the setting will often
 be clear from the context.
 For instance, a whole chapter may be concerned with the
 same solution operator $S$.
 We usually skip the respective parameters in this case.
\end{rem}

Let us discuss some basic relations between
the minimal errors in
the different settings.
Obviously, we have the relation
$$
 \e(n,\mathcal{P}[S,F,G,\Lambda,\star,\circ,\mu])
 \leq
 \e(n,\mathcal{P}[S,F,G,\Lambda,\star,\circ,\mathrm{wc}])
$$
since the worst case error of an algorithm
is always at least as large as the average case error.
Moreover,
we have
\begin{equation*}
 \e(n,\mathcal{P}[S,F,G,\Lambda,\mathrm{ran},\circ,\triangle])
 \leq
 \e(n,\mathcal{P}[S,F,G,\Lambda,\mathrm{det},\circ,\triangle])
\end{equation*}
since the class of randomized algorithms is larger
than the class of deterministic algorithms. 
In fact, we even have equality in the average case setting,
that is, if $\mu$ is a Borel probability measure on
the topological space $F$, we have
\begin{equation}
 \label{eq:randomization useless in ac}
 \e(n,\mathcal{P}[S,F,G,\Lambda,\mathrm{ran},\mu])
 =
 \e(n,\mathcal{P}[S,F,G,\Lambda,\mathrm{det},\mu]).
\end{equation}
This means that randomization has no effect in the average case setting.
This is a simple consequence of Tonelli's theorem:
if $(A^\omega)_{\omega\in\Omega}$ is a randomized algorithm
with worst case cost $n$ or less, we have
\begin{multline*}
  \err\brackets{(A^\omega)_{\omega\in\Omega},\mathrm{\mu}}^2
  = \int_F \IE\brackets{\dist(S(f),A^{\omega}(f))^2}~\d\mu(f)\\
  = \IE\brackets{ \int_F \dist(S(f),A^{\omega}(f))^2~\d\mu(f)}
  = \IE\brackets{\err\brackets{A^{\omega},\mu}^2}.
 \end{multline*}
This means that there is a realization $A^{\omega}$ of the
randomized algorithm such that
 $$
  \err\brackets{A^{\omega},\mu} 
  \leq \err\brackets{(A^\omega)_{\omega\in\Omega},\mu}.
 $$
Since $A^{\omega}$ is a deterministic algorithm with cost $n$ or less,
this proves \eqref{eq:randomization useless in ac}.

In particular, we obtain the following theorem, which links 
the worst case error
of randomized algorithms and the average case error of
deterministic algorithms.
It is called \emph{Bakhvalov's technique} and is essential 
for proving lower bounds for the error of randomized algorithms.
We refer to \cite[Section~4.3.3]{NW08} for more details.

\begin{thm}[\cite{NW08}]
\label{thm:Bakhvalov technique}
 Let $S:F\to G$ be an operator from a topological space $F$ to a metric space $G$
 and let $\Lambda$ be a class of real-valued functions on $F$.
 For any Borel probability measure $\mu$ on $F$ and any $n\in\IN_0$, we have
 $$
  \e(n,\mathcal{P}[S,F,G,\Lambda,\mathrm{ran},\mathrm{wc}])
  \geq
  \e(n,\mathcal{P}[S,F,G,\Lambda,\mathrm{det},\mu]).
 $$
\end{thm}


%

We finish this section with an example.
Note that many other examples are provided throughout this thesis.

\begin{ex}[An integration problem, Part 1 of 2]
\label{ex:Lipschitz}
 Assume that the function $f:[0,1]\to \IR$
 is known to be in the Lipschitz class
 $$
  F=\set{f:[0,1]\to \IR \mid 
  \forall (x,y)\in [0,1]^2: \abs{f(x)-f(y)}\leq \abs{x-y}}.
 $$
 The function itself, however, is unknown.
 We want to approximate the integral
 $$
  S(f)=\int_0^1 f(x)~\d x
 $$
 of the function up to a guaranteed error.
 To do so, we may request a finite number of function values
 using any deterministic scheme.
 The cost of an algorithm $A:F\to \IR$ is the maximal
 number of requested function values
 and its error is
 $$
  \err(A)=\sup_{f\in F} \abs{S(f)-A(f)}.
 $$
 In the above terms, we study the problem
 $
  \mathcal{P}[S,F,\IR,\lstd,\mathrm{det},\mathrm{wc}].
 $
\end{ex}
\noindent
We continue this example after gathering
some results on so-called linear problems.

\subsection{Linear Problems}
\label{sec:LPs}

We consider deterministic problems in the worst case setting.
Many of these problems are linear in the sense of
the following definition.

\begin{defi}
\label{def:LP}
 The problem $\mathcal{P}[S,F,G,\Lambda,\mathrm{det},\mathrm{wc}]$
 is called a \emph{linear problem} if
 \begin{itemize}
  \item $F$ is a nonempty, convex and symmetric subset of a normed space $\widetilde F$;
  \item $G$ is a normed space;
  \item $S:\widetilde F\to G$ is linear;
  \item $\Lambda$ is a class of continuous linear functionals.
 \end{itemize}
\end{defi}

We present some basic results on linear problems without proof.
We refer the reader to \cite[Section~4.2]{NW08} for further
details, proofs and references.
The first result says that the radius of a nonadaptive information
mapping $N:F\to \IR^n$
is already (almost) determined by its radius at zero.

\begin{thm}[\cite{NW08}]
\label{thm:radius at zero}
 Let $S,F,G$, and $\Lambda$ describe a linear problem
 and let $N$ be a nonadaptive
 information mapping based on $\Lambda$. Then
 $$
  r_{\mathbf 0}(N,S,F,G)
  \leq
  \rad(N,S,F,G)
  \leq
  2 r_{\mathbf 0}(N,S,F,G)
 $$
\end{thm}

It is easy to check that the radius of $N$ at zero satisfies
$$
 r_{\mathbf 0}(N,S,F,G)=\sup_{f\in F: N(f)=\mathbf{0}} \Xnorm{Sf}{G}.
$$
An important consequence of the previous theorem
is that adaption is not necessary for linear problems.
If $N:F\to c_{00}$ is an adaptive information mapping
as defined in Section~\ref{sec:algorithms},
we define a corresponding nonadaptive information mapping
$N^{\rm{non}}:F\to \IR^n$ by setting $n=n(0)$ and
$$
 N^{\rm{non}}(f)=(L_1(f),L_2(f,0),\hdots,L_n(f,0,\hdots,0)).
$$
This means that the nonadaptive information $N^{\rm{non}}$ 
takes the same measurements for every input
and these measurements are the same as for the
adaptive information $N$ for the input zero. 
By Proposition~\ref{prop:radius global vs local} and
Theorem~\ref{thm:radius at zero}, we obtain
$$
 \rad(N^{\rm{non}},S,F,G)
  \leq
  2 r_{\mathbf 0}(N^{\rm{non}},S,F,G) 
  = 2 r_{\mathbf 0}(N,S,F,G)
  \leq 2\rad(N,S,F,G).
$$
Clearly the worst case cost of $N^{\rm{non}}$
is bounded above by the worst case cost of $N$.
In particular, we may loose a factor 
of at most 2 if we study the error
of nonadaptive algorithms in comparison
to arbitrary algorithms \cite[Section~4.2.1]{NW08}.

\begin{cor}[\cite{NW08}]
\label{cor:adaption does not help}
 Let $S,F,G$ and $\Lambda$ describe a linear problem.
 For every information mapping $N:F\to c_{00}$
 the nonadaptive information mapping $N^{\rm{non}}:F\to\IR^n$
 satisfies
 $$
  \rad(N^{\rm{non}},S,F,G)
  \leq
  2 \rad(N,S,F,G).
 $$
 In particular, for all $n\in\IN$, we have
 $$
 \e\brackets{n,\mathcal{P}[S,F,G,\Lambda,\mathrm{det},\mathrm{wc},\mathrm{nonada}]}
 \leq 2 \e\brackets{n,\mathcal{P}[S,F,G,\Lambda,\mathrm{det},\mathrm{wc}]}.
 $$
\end{cor}


In many cases
we do not even loose the factor $2$.
In addition,
it turns out
that linear algorithms are optimal
in the very same cases \cite[Section~4.2.2]{NW08}.
In the following theorem, $\mathcal{B}(X)$ and $\mathcal{C}(X)$ are the spaces
of bounded respectively continuous real valued functions on $X$.

\begin{thm}[\cite{NW08}]
\label{thm:linear algorithms are optimal}
 Let $S,F,G$ and $\Lambda$ describe a linear problem.
 Assume that one of the following conditions is satisfied:
 \begin{itemize}
  \item $G=\IR$ or $G=\mathcal{B}(X)$ for some set $X$ or $G$ is some $L^\infty$-space;
  \item $F$ is the unit ball of a pre-Hilbert space $\widetilde F$;
  \item $G=\mathcal{C}(X)$ for some compact Hausdorff space $X$
  and $S$ is compact.
 \end{itemize}
 Then every information mapping $N:F\to c_{00}$
 yields a nonadaptive information mapping $N^{\rm{non}}:F\to\IR^n$ with
 $$
  \rad(N^{\rm{non}},S,F,G)
  \leq \rad(N,S,F,G).
 $$
 Moreover, the nonadaptive information satisfies
 $$
  \rad(N^{\rm{non}},S,F,G)
  = r_{\mathbf 0}(N^{\rm{non}},S,F,G)
  = \inf_{\phi\, \textrm{\emph{linear}}} \err\brackets{\phi\circ N^{\rm{non}}}.
 $$
\end{thm}

This leads to a very useful formula
for the $n^{\rm th}$ minimal worst case error.

\begin{thm}[\cite{NW08}]
\label{thm:error formula}
 Let $\mathcal{P}=(\mathcal{A},\err,\cost)$ be a linear problem
 given by $S,F,G$ and $\Lambda$
 such that one of the conditions in Theorem~\ref{thm:linear algorithms are optimal} holds.
 Then, for every $n\in\IN$,
 $$
  \e(n,\mathcal{P})
  = \inf_{\substack{A \text{ \emph{linear}}\\\cost(A)\leq n}} \err(A)
  =\adjustlimits\inf_{N\in\Lambda^n} \sup_{f\in F:\, N(f)=\mathbf{0}} \Xnorm{Sf}{G}.
 $$
\end{thm}

\begin{namedthm}{Example~\ref{ex:Lipschitz}}[Part 2 of 2]
The problem of integrating Lipschitz-functions is linear.
The target space is $\IR$.
Hence, linear algorithms are optimal
and we only need to consider algorithms of the form
$$
 A_n:F\to \IR, \quad A_n(f)=\sum_{i=1}^n a_i f(x_i)
$$
with some $n\in\IN$, weights $a_i\in\IR$
and nodes $x_i\in[0,1]$.
With the help of Theorem~\ref{thm:error formula}
it is easily verified that
$$
 \e(n,\mathcal{P})=\frac{1}{4n}
$$
and that the minimal error is achieved 
by the algorithm $A_n$ if we choose constant weights $a_i=1/n$
and equidistant nodes $x_i=\frac{2i-1}{2n}$ for $i=1,\hdots,n$.
\end{namedthm}

\subsubsection{Linear Problems over Hilbert Spaces}

We finish this section with linear problems
over Hilbert spaces based on $\lall$.
We assume that $F$ is the unit ball of a Hilbert space $H$
and that $G$ is another Hilbert space.
Let $S:H\to G$ be a compact linear operator.

The operator $W=S^*S:H\to H$ is positive and compact.
Hence, it admits a finite or countable orthonormal basis 
$\mathcal{B}$ of $\ker(S)^\perp$ consisting of eigenvectors
$b\in\mathcal{B}$ to eigenvalues
$$
 \lambda(b) = \Xscalar{Wb}{b}{H}= \Xnorm{Sb}{G}^2 > 0
.$$
For any $f\in H$ we have the relation
$$
 S(f) = \sum_{b\in\mathcal{B}} \Xscalar{f}{b}{H} Sb.
$$
The square-roots of the eigenvalues of $W$ are called singular values of $S$.
Let $\sigma_n$ be the $n^{\rm th}$ largest singular value if $n\leq\abs{\mathcal{B}}$. 
Else, let $\sigma_n=0$.
We consider the linear algorithm
$$
 A_n:F\to G, \quad
 A_n(f)= \sum_{b\in\mathcal{B}(n)} \Xscalar{f}{b}{H} Sb,
$$
where $\mathcal{B}(n)$ consists of all $b\in\mathcal{B}$ 
that satisfy $\Xnorm{S b}{G}>\sigma_{n+1}$.
This algorithm is optimal among all algorithms with
cost $n$ or less~\cite[Section~4.2.3]{NW08}.

\begin{thm}[\cite{NW08}]
 \label{thm:LPs over Hilbert spaces}
 The algorithm $A_n$ satisfies $\cost(A_n)\leq n$ and
 $$
  \err(A_n)=\e(n,\mathcal{P}[S,F,G,\lall,\mathrm{det},\mathrm{wc}])
  =\sigma_{n+1}.
 $$
\end{thm}

It is known from~\cite{No92} that
randomized algorithms cannot be much better than
deterministic algorithms in this setting:
up to a factor of at most $\sqrt{2}$, 
the algorithm $A_{2n-1}$ is as good as
any deterministic or randomized algorithm
with cost $n$ or less.

\begin{thm}[\cite{NW08}]
 \label{thm:randomization useless in Hspace}
 Let $H$ and $G$ be Hilbert spaces, let $F$ be the unit ball of $H$,
 and let $S:H\to G$ be compact.
 For any $n\in\IN$, we have
 $$
 \e(n,\mathcal{P}[S,F,G,\lall,\mathrm{ran},\mathrm{wc}])
 \geq 
 \frac{1}{\sqrt{2}}\,\e(2n-1,\mathcal{P}[S,F,G,\lall,\mathrm{det},\mathrm{wc}]).
 $$
\end{thm}

\chapter{Integration and Approximation of Functions with Mixed Smoothness}
\label{chap:mixed smoothness}

In this chapter, we study the following multivariate problems.
\begin{itemize}
 \item Section~\ref{sec:Frolov}: The integration of multivariate
 functions from different smoothness classes.
 We allow randomized algorithms based on $\lstd$.
 This section is based on~\cite{Kr16,KN17,Ul17}.
 \item Section~\ref{sec:tensorproduct}: The approximation of 
 a tensor product operator between Hilbert spaces.
 We allow deterministic algorithms based on $\lall$.
 This section is based on~\cite{Kr18}.
 \item Section~\ref{sec:OptimalMC}: The $L^2$-approximation of functions 
 from a Hilbert space
 that is compactly embedded into $L^2$.
 We allow randomized algorithms based on $\lstd$.
 This section is based on~\cite{Kr18c}.
\end{itemize}
We will focus on the rate of convergence of the $n^{\rm th}$ minimal error
and provide algorithms that achieve the optimal error rate. 
In Section~\ref{sec:tensorproduct} and \ref{sec:OptimalMC},
we will also discuss the error of these algorithms for small $n$.
All results can be applied for 
multivariate functions with mixed smoothness.


\section{A Universal Algorithm for Integration}
\label{sec:Frolov}

We want to approximate the integral 
$$
S_d(f)=\int_{[0,1]^d} f(\mathbf x)~\d\mathbf x
$$
of a multivariate function $f:[0,1]^d\to\IR$.
To compute an approximation,
we may request a certain amount $n$ of function values.
The function $f$ itself is not known.
We do, however, have some a priori knowledge about the function.
We assume that
the function is smooth
in the sense that certain weak derivatives $\diff^\alpha f$ 
exist and are square-integrable.
Which derivatives are known to be existent and
square-integrable is different in different applications. 

In several applications, $\alpha$ covers the range of all multi-indices
with $\vert\alpha\vert\leq r$ for some $r\in\IN$.
We say that $f$ has isotropic smoothness $r$.
For example, the solutions of elliptic partial differential
equations in general and Poisson's equation in particular have this 
type of smoothness~\cite{GT01,HT08}.
They typically appear in electrostatics or continuum mechanics.
With deterministic algorithms,
the integral of such functions can be computed up to 
an error of order $n^{-r/d}$,
but not with higher accuracy~\cite{Ba59,No88}.
The expected error may be smaller,
if randomness can be used.
With randomized algorithms,
we may achieve an expected error
of order $n^{-r/d-1/2}$~\cite{Ba59,Ba62,No88}.

In other applications, $\alpha$ covers the range of all multi-indices
with $\Vert\alpha\Vert_\infty\leq r$.
We say that $f$ has mixed smoothness $r$.
This is a stronger smoothness condition.
For example, solutions of the electronic Schrödinger equation 
have this type of smoothness~\cite{Ys10}.
With deterministic algorithms, the integral of such functions can be 
computed up to an error of order $n^{-r}(\ln n)^{(d-1)/2}$~\cite{Fr76}.
Using randomness, we may achieve an expected error
of order $n^{-r-1/2}$~\cite{Ba62,Ul17}.
These rates are much better than the rates in the isotropic case
if the number $d$ of variables is large.

In most applications, however, we do not really know 
how smooth our integrand is.
Thus, we would like to have an algorithm
which can be applied to any integrable function 
and automatically detects its smoothness.
That is,
whenever $f$ has isotropic or mixed smoothness $r$ for some $r\in\IN$,
the expected and guaranteed error should decay
with the above mentioned error rates.
We say that the algorithm is universal.
In this section, we will present a universal
algorithm for multivariate integration.

Let us formulate the main result of this section.
For every $r\in\IN$, let $H^r([0,1]^d)$
be the linear space of functions with isotropic smoothness $r$
and let $H^r_{\rm mix}([0,1]^d)$ be the linear space
of functions with mixed smoothness $r$.
We define norms on these spaces via the relations
\begin{align*}
 \Xnorm{f}{H^r([0,1]^d)}^2 &=
 \sum_{\abs{\alpha}\leq r}\,
 \Xnorm{\diff^\alpha f}{L^2([0,1]^d)}^2
 \quad&&\text{for}\quad 
 f\in H^r([0,1]^d),\\
 \Xnorm{f}{H^r_{\rm mix}([0,1]^d)}^2 &=
 \sum_{\Vert\alpha\Vert_\infty \leq r}\,
 \Xnorm{\diff^\alpha f}{L^2([0,1]^d)}^2
 \quad&&\text{for}\quad 
 f\in H^r_{\rm mix}([0,1]^d).
\end{align*}
For each $n\in\IN$, we define a
randomized algorithm $(A_n^\omega)_{\omega\in\Omega}$
of the form
$$
 A_n^\omega(f)=\sum_{j=1}^n a_j(\omega) f\brackets{\mathbf x^{(j)}(\omega)}
$$ 
for $f\in L^1([0,1]^d)$ and $\omega\in \Omega$,
where $(\Omega,\mathcal{F},\IP)$ is a probability space
and $\mathbf x^{(j)}:\Omega\to [0,1]^d$
and $a_j:\Omega\to\IR$ are random variables for each $j\leq n$,
see Algorithm~\ref{alg:final algorithm}.
These algorithms have the following properties.

\begin{thm}[\cite{KN17,Ul17}]
\label{thm:main theorem}
 There are positive constants $c,c_1,c_2,\hdots$
 such that the following holds for all $n\in\IN$ with $n\geq c$
 and $f\in L^1([0,1]^d)$.
 \begin{itemize}
  \item $\displaystyle \IE\brackets{ A_n(f)} = S_d(f)$.
 \end{itemize}
 If $f$ has mixed smoothness $r\in\IN$, then
  \begin{itemize}
   \item $\displaystyle \sqrt{\IE \abs{S_d(f)-A_n(f)}^2} \leq c_r\, n^{-r-1/2} 
   \Xnorm{f}{H^r_{\rm mix}([0,1]^d)}$,
   \item $\displaystyle \IP\brackets{\abs{S_d(f)-A_n(f)} 
   \leq c_r\, n^{-r}(\ln n)^{(d-1)/2} 
   \Xnorm{f}{H^r_{\rm mix}([0,1]^d)}}=1.$
  \end{itemize}
 If $f$ has isotropic smoothness $r\in\IN$ with $r>d/2$, then
 \begin{itemize}
  \item $\displaystyle \sqrt{\IE \abs{S_d(f)-A_n(f)}^2} \leq c_r\, n^{-r/d-1/2} 
  \Xnorm{f}{H^r([0,1]^d)}$,
  \item $\displaystyle
  \IP\brackets{\abs{S_d(f)-A_n(f)} \leq c_r\, n^{-r/d} 
  \Xnorm{f}{H^r([0,1]^d)}}=1$.
 \end{itemize}
\end{thm}

We remark that these constants may depend on $d$.
The condition $r>d/2$ ensures that the functions in $H^r([0,1]^d)$ are continuous.
The algorithm is a randomization of Frolov's algorithm~\cite{Fr76}.
It was first proposed in~\cite{KN17}.
The order of the expected error for functions with mixed smoothness 
was proven in~\cite{Ul17}.

In particular, we obtain the following result on the order of convergence.
Let $F_d^r$ be the unit ball of $H^r_{\rm mix}([0,1]^d)$
and let 
$$
 \mathcal{P}_d^r=
 \mathcal{P}[S_d,F_d^r,\IR,\lstd,\mathrm{ran},\mathrm{wc}]
$$ 
be the problem of integrating a function from $F_d^r$
with randomized algorithms based on $\lstd$ in the worst case 
setting.

\begin{cor}[\cite{Ba59,Ul17}]
 \label{cor:order of convergence}
 For any $r\in\IN$ and $d\in\IN$, we have
 $$
  \e\brackets{n,\mathcal{P}_d^r} \asymp n^{-r-1/2}.
 $$
\end{cor}

Section~\ref{sec:Frolov} is organized as follows.
In Section~\ref{functionclassessec} we define and characterize
the function classes of our interest.
In Section~\ref{frolovsrulesection} we introduce Frolov's deterministic
algorithm for the integration of functions with compact support.
In Section~\ref{randomdilationsection} and Section~\ref{randomshiftsection}
we discuss how this algorithm can be improved by introducing
a random dilation and a random shift to the set of nodes.
Section~\ref{transformationsection} shows how we can integrate
functions without compact support using a transformation of the unit cube.
Here we also give a proof of Theorem~\ref{thm:main theorem}
and Corollary~\ref{cor:order of convergence}.
We remark that our algorithm is optimal
for many other classes of smooth functions
in terms of the order of convergence of its error, see~\cite{Ul17}.

\subsection{The Function Classes}
\label{functionclassessec}

Let $r\in \IN$ and $d\in\IN$. The Sobolev space
of mixed smoothness $r$ is the vector space
\[
H^r_{\rm mix}(\IR^d) = \set{f\in L^2(\IR^d)\mid 
\diff^\alpha f \in L^2(\IR^d) \text{ for all } \alpha\in\IN_0^d 
\text{ with } \Vert\alpha\Vert_\infty \leq r}
\]
of $d$-variate, real-valued functions, equipped with the scalar product
$$
\Xscalar{f}{g}{H^r_{\rm mix}(\IR^d)}
= \sum_{\Vert\alpha\Vert_\infty \leq r} 
\Xscalar{\diff^\alpha f}{\diff^\alpha g}{L^2(\IR^d)}.
$$
It is known that $H^r_{\rm mix}(\IR^d)$ is a Hilbert space and its 
elements $f\in L^2(\IR^d)$ have continuous representatives.
The Fourier transform is the unique continuous linear operator
$\mathcal{F}: L^2(\IR^d)\to L^2(\IR^d)$
satisfying
\[
\mathcal{F}f(\mathbf y) = 
\int_{\IR^d}   f(\mathbf x)\, e^{-2\pi i \scalar{\mathbf x}{\mathbf y}}  \, \d \mathbf x
\]
for integrable $f:\IR^d \to \IR$ and almost all $\mathbf y\in\IR^d$.
The space ${H^r_{\rm mix}(\IR^d)}$ contains exactly those 
functions $f\in L^2(\IR^d)$ with 
${\mathcal{F}f \cdot h_r^{1/2} \in L^2(\IR^d)}$ 
for the weight function
\[
h_r: \IR^d \to \IR^+,\quad h_r(\mathbf x)= \sum\limits_{\Vert\alpha\Vert_\infty \leq r} 
\prod\limits_{j=1}^{d} |2\pi x_j|^{2\alpha_j} 
= \prod\limits_{j=1}^{d} \sum\limits_{k=0}^{r} |2\pi x_j|^{2k}.
\]
In terms of the Fourier transform, 
the scalar product in ${H^r_{\rm mix}(\IR^d)}$ 
is given by
$$
\Xscalar{f}{g}{H^r_{\rm mix}(\IR^d)}
= \Xscalar{\mathcal{F}f}{\mathcal{F}g}{L^2(\IR^d,h_r)},
$$
where $L^2(\IR^d,h_r)$ is the weighted $L^2$-space with weight $h_r$.
Analogously, the Sobolev space of isotropic
smoothness $r$ is
\[
{H^r(\IR^d)} = 
\set{f\in L^2(\IR^d)\mid \diff^\alpha f \in L^2(\IR^d) 
\text{ for all } \alpha \in \IN_0^d\text{ with } \vert\alpha\vert\leq r}
,\]
equipped with the scalar product
\[
\Xscalar{f}{g}{H^r(\IR^d)}= 
\sum_{\vert\alpha\vert\leq r} 
\Xscalar{\diff^\alpha f}{\diff^\alpha g}{L^2(\IR^d)}.
\]
This is again a Hilbert space.
If $r$ is greater than $d/2$,
then ${H^r(\IR^d)}$ also consists of continuous functions. 
The space contains exactly those functions 
$f\in L^2(\IR^d)$ with $\mathcal{F}f \cdot v_r^{1/2}\in L^2(\IR^d)$ 
for the weight function
\[
v_r: \IR^d \to \IR^+,\quad v_r(\mathbf x)= \sum\limits_{\abs{\alpha}\leq r} 
\prod\limits_{j=1}^{d} |2\pi x_j|^{2\alpha_j}
.\]
In terms of its Fourier transform, the scalar product in ${H^r(\IR^d)}$ is given by
$$
\Xscalar{f}{g}{H^r(\IR^d)}
= \Xscalar{\mathcal{F}f}{\mathcal{F}g}{L^2(\IR^d,v_r)},
$$
where $L^2(\IR^d,v_r)$ is the weighted $L^2$-space with weight $v_r$.
We refer to \cite{SU09} for an overview regarding these spaces
of mixed and isotropic smoothness.

Furthermore, let ${C_c(\IR^d)}$ be the real vector space of all 
continuous real valued functions 
with compact support in $\IR^d$. The spaces ${\mathring{H}^r_{\mix}([0,1]^d)}$ 
and ${\mathring{H}^r([0,1]^d)}$ of 
functions in ${H^r_{\rm mix}(\IR^d)}$ or ${H^r(\IR^d)}$ with compact support 
in the
unit cube are subspaces of ${C_c(\IR^d)}$.
They can also be considered as subspaces of the Hilbert space
\[
{H^r_{\rm mix}([0,1]^d)}= \set{f\in L^2([0,1]^d) \mid \diff^\alpha 
f \in L^2([0,1]^d) 
\text{ for all } \alpha \in \IN_0^d\text{ with } \vert\alpha\vert\leq r}
,\]
equipped with the scalar product
\[
\Xscalar{f}{g}{H^r_{\rm mix}([0,1]^d)}= \sum\limits_{\Vert\alpha\Vert_\infty \leq r} 
\Xscalar{\diff^\alpha f}{\diff^\alpha g}{L^2([0,1]^d)}  , 
\]
or the Hilbert space
\[
{H^r([0,1]^d)}= \set{f\in L^2([0,1]^d) 
\mid \diff^\alpha f \in L^2([0,1]^d) \text{ for all } \alpha 
\in \IN_0^d\text{ with } \abs{\alpha}\leq r}
,\]
with scalar product
\[
\Xscalar{f}{g}{H^r([0,1]^d)}= \sum\limits_{\abs{\alpha}\leq r} 
\Xscalar{\diff^\alpha f}{\diff^\alpha g}{L^2([0,1]^d)} 
.\]

\subsection{Frolov's Deterministic Algorithm}
\label{frolovsrulesection}

Our methods are based on the following family
of deterministic linear algorithms.

\begin{alg}
\label{alg:basic algorithm}
Let $B\in\IR^{d\times d}$ be invertible and $\mathbf v\in\IR^d$.
We define
\[
Q_B^{\mathbf v}(f)=\frac{1}{|\det B|} 
\sum\limits_{\mathbf m\in\IZ^d} f\left(B^{-\top}(\mathbf m+\mathbf v)\right)
\]
for any $f:\IR^d\to\IR$ such that the right hand side converges absolutely. 
The vector $\mathbf v$ is called shift parameter.
We write $Q_B=Q_B^{\mathbf 0}$.
\end{alg}


\begin{rem}
The value $Q_B^{\mathbf v}(f)$ can be thought of as a Riemann sum:
The nodes of the algorithm are the lower left corners of the parallelepipeds 
$$
B^{-\top}\brackets{\mathbf m+\mathbf v+[0,1]^d}, \quad \mathbf m\in\IZ^d,
$$
and the weight $\abs{\det B}^{-1}$ is the volume of this parallelepiped.
\end{rem}

The algorithm is well defined for functions with compact support.
To integrate these functions, the algorithm $Q_B^{\mathbf v}$ 
only uses the nodes $B^{-\top}(\mathbf m+ \mathbf v)$ for all 
$$
\mathbf m\in\IZ^d \cap \brackets{B^\top\left(\supp f\right)-\mathbf v}.
$$
The number of these nodes should be close to the
volume of the latter set.
In particular, the number of nodes of $Q_{aB}^{\mathbf v}$ 
should behave like $a^d$ as $a$ tends to infinity. 
The following lemma gives an exact upper bound, see~\cite{Sk94} 
for other bounds.

\begin{lemma}
\label{anlemma}
Assume that $f: \IR^d \to \IR$ is supported in an axis-parallel cube of edge length $l>0$. 
For any invertible matrix $B\in\IR^{d\times d}$, 
$\mathbf v\in\IR^d$, and $a\geq 1$, the algorithm $Q_{aB}^{\mathbf v}$ uses at most
$
\left(l\,\Vert B\Vert_1+1\right)^d a^d
$
function values of $f$.
\end{lemma}

\begin{proof}
The number of computed function values is 
given by the cardinality of
$$
 M=\set{\mathbf m\in\IZ^d \mid (aB)^{-\top}(\mathbf m+\mathbf v)\in \supp f}
$$
By assumption, $f$ has support in $[-l/2,l/2]^d+\mathbf z$ 
for some $\mathbf z\in\IR^d$. 
Thus, any $\mathbf m\in M$ satisfies 
$$
 \mathbf m+\left(\mathbf v-aB^\top \mathbf z\right)\in 
 \frac{al}{2} B^\top( [-1,1]^d).
$$
Since $\Vert B^\top x\Vert_\infty\leq \Vert B^\top\Vert_\infty 
=\Vert B\Vert_1$ for $\mathbf x\in[-1,1]^d$, we obtain
\[
M \subset \set{\mathbf m\in\IZ^d \mid m+\left(\mathbf v-aB^\top \mathbf z\right)\in 
\frac{al}{2}\left[-\Vert B\Vert_1,\Vert B\Vert_1\right]^d}
\]
and $\card(M) \leq \left(al\Vert B\Vert_1+1\right)^d$. 
Since $1\leq a$, we get the desired estimate.
\end{proof}

The error of this algorithm for integration on $C_c(\IR^d)$ can be 
expressed in terms of the Fourier transform.

\begin{lemma}
\label{errorlemma}
For any invertible matrix $B\in\IR^{d\times d}$, $\mathbf v\in\IR^d$, and 
$f\in C_c(\IR^d)$
\[
\left| Q_B^{\mathbf v}(f)-\int_{\IR^d} f(\mathbf x)~\d\mathbf x\right| 
\leq \sum\limits_{\mathbf m\in\IZ^d\setminus\set{0}} \left| \mathcal{F}f(B\mathbf m)\right|
.\]
\end{lemma}

\begin{proof}
The function $g=f\circ B^{-\top}(\cdot +\mathbf v)$ is continuous with compact support. 
Hence, the Poisson summation formula and an affine linear substitution 
$\mathbf x=B^\top \mathbf y-\mathbf v$ yield
\[\begin{split}
Q_B^{\mathbf v}(f)
&=\frac{1}{\abs{\det B}}\sum\limits_{\mathbf m\in\IZ^d} g(\mathbf m) 
= \frac{1}{\abs{\det B}}\sum\limits_{\mathbf m\in\IZ^d} \mathcal{F}g(\mathbf m)\\
&= \frac{1}{\abs{\det B}}\sum\limits_{\mathbf m\in\IZ^d}\, 
\int_{\IR^d} f\left(B^{-\top}(\mathbf x+\mathbf v)\right)\cdot 
e^{-2\pi i\langle\mathbf x,\mathbf m\rangle}
\, \d \mathbf x \\
&= \sum\limits_{\mathbf m\in\IZ^d}\, \int_{\IR^d} f\left(\mathbf y\right)
\cdot e^{-2\pi i\langle B^\top \mathbf y-\mathbf v,\mathbf m\rangle} \, \d \mathbf y \\
&= \sum\limits_{\mathbf m\in\IZ^d} \mathcal{F}f(B\mathbf m)
\cdot e^{2\pi i\langle \mathbf v,\mathbf m\rangle}
,\end{split}\]
if the latter series converges absolutely, see \cite[pp.\,356]{Ko00}. If not,
the stated inequality is obvious.
This proves the statement, 
since $\mathcal{F}f(B\mathbf m)e^{2\pi i\langle \mathbf v,\mathbf m\rangle}$
yields the integral of $f$ for $\mathbf m=0$.
\end{proof}

It is known how to choose the matrix $B$ in the rule $Q_B^{\mathbf v}$ to get a good 
deterministic quadrature rule on ${\mathring{H}^r_{\mix}([0,1]^d)}$. 
\begin{defi}
\label{def:Frolov}
We say that $B\in \IR^{d\times d}$ 
is a \emph{Frolov matrix} if the following holds:
\begin{itemize}
\item $B$ is invertible.
\item $\vert\prod\limits_{j=1}^{d}(B\mathbf m)_j\vert\geq 1$ 
for any $\mathbf m\in\IZ^d\setminus\set{0}$.
\item Any axis-aligned box of volume $c>0$
contains at most $c+1$ points of the lattice $B\IZ^d$.
\end{itemize}
If $B$ is a Frolov matrix, then
the algorithm $Q_{n^{1/d} B}$ for $n\in\IN$ (see Algorithm~\ref{alg:basic algorithm})
is referred to as \emph{Frolov's algorithm}.
\end{defi}

We first note that the number of nodes of
the Frolov algorithm is of order $n$.
To be precise, Lemma \ref{anlemma} says that 
$Q_{n^{1/d}B}$ uses at most
$\left(\Vert B\Vert_1+1\right)^d n$ function values
if the input function is supported in $[0,1]^d$.

It is known that one can construct a Frolov matrix $B$ in the following way. 
Let $p\in\IZ[x]$ be a polynomial of degree $d$ with leading 
coefficient 1 which is irreducible over $\IQ$ and has $d$ different 
real roots $\zeta_1,\hdots,\zeta_d$. Then the matrix
\[B=\left(\zeta_i^{j-1}\right)_{i,j=1}^d\]
has the desired properties, as shown in \cite[p.\,364]{Te93} and \cite{Ul16}. 
In arbitrary dimension $d$ we can choose 
$p(x)=(x-1)(x-3)\cdot\hdots\cdot(x-2d+1)-1$, 
see \cite{Fr76} or \cite{Ul16}.
In particular, there exists a $d$-dimensional Frolov matrix
for any $d\in\IN$. 
If $d$ is a power of two, we can also choose 
$p(x)=2\cos\left(d\cdot\arccos (x/2)\right)=2\,T_d(x/2)$, 
where $T_d$ is the Chebyshev 
polynomial of degree $d$, see \cite[p.\,365]{Te93}. 
Then the roots of $p$ are explicitly given by 
$\zeta_j=2\cos\left(\frac{2j-1}{2d}\pi\right)$ for $j=1,\hdots,d$
and the lattice $B\IZ^d$ is orthogonal~\cite{KOU17}.
We remark that an invertible matrix $B$ is a Frolov matrix 
iff there is some $c>0$ such that
$c B^{-\top}$ is a Frolov matrix~\cite{Sk94}. 

Geometrically speaking,
the second property of Definition~\ref{def:Frolov}
says that every point of the Frolov lattice $B\IZ^d$ 
but zero is contained in the complement of a hyperbolic cross.
We denote these sets by
$$
 D_t=\big\{\mathbf x\in\IR^d \mid \prod_{j=1}^{d}\abs{x_j}\geq t\big\}
 \quad\text{for}\quad t>0.
$$
This property is illustrated in Figure~\ref{A Frolov lattice}.


\begin{figure}[ht]
\begin{minipage}{.7\linewidth}
\includegraphics[width=\linewidth]{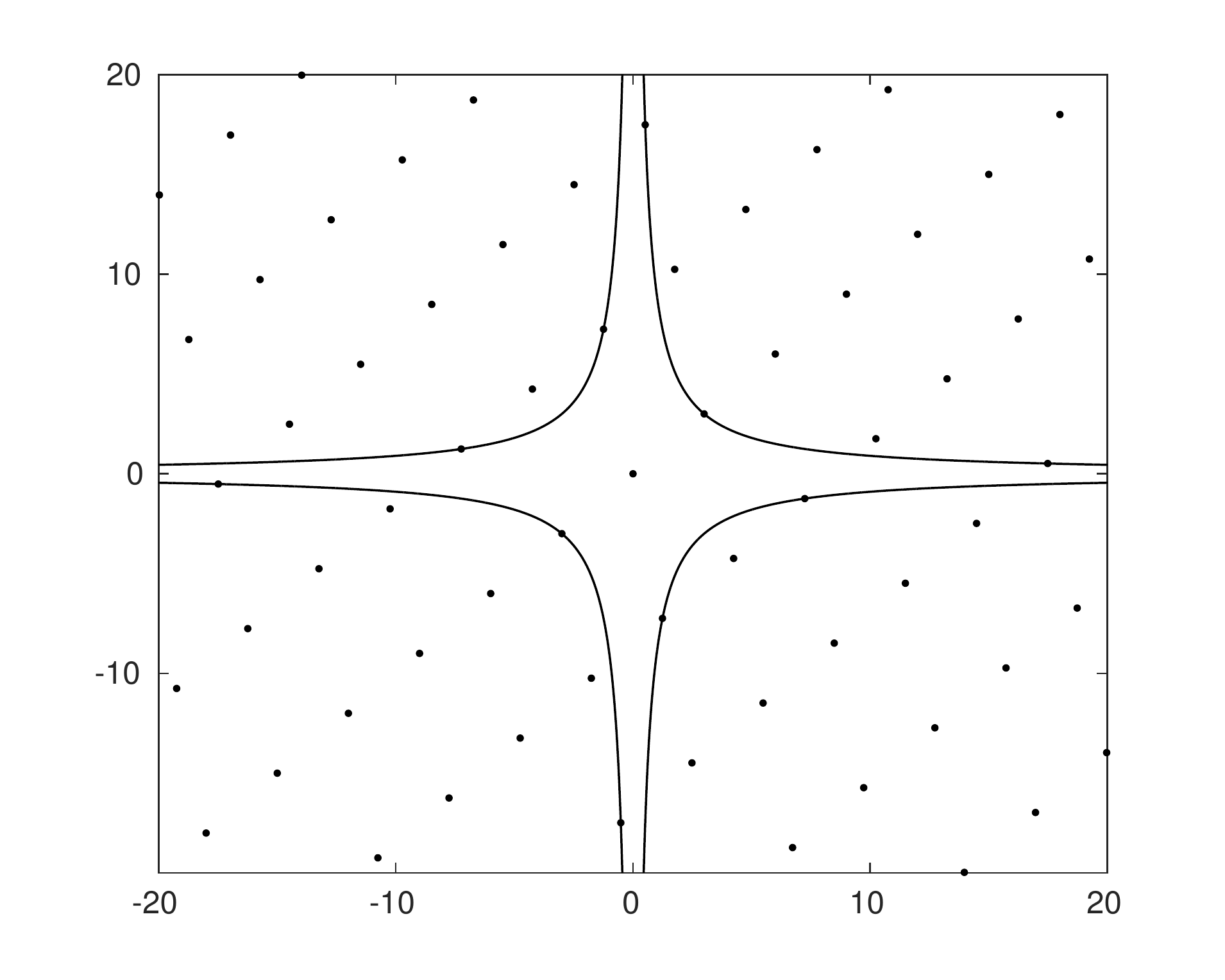}
\end{minipage}
\begin{minipage}{.29\linewidth}
\caption{}
\label{A Frolov lattice}
This figure shows the lattice $3B\IZ^d$ for $d=2$ and 
the Frolov matrix
\[B=\begin{pmatrix}
1 & 2-\sqrt{2}\\
1 & 2+\sqrt{2}
\end{pmatrix}.\]
Except the origin, every lattice point is
contained in $D_9$.
\end{minipage}
\end{figure}

In 1976, Frolov showed that this deterministic algorithm
has the optimal order of convergence on ${\mathring{H}^r_{\mix}([0,1]^d)}$
and that it satisfies the error bound below.
Note that the constant in this error bound
depends on the choice of the Frolov matrix.

\begin{thm}[\cite{Fr76}]
\label{frolovboundtheorem}
Let $B\in\IR^{d\times d}$ be a Frolov matrix and $r\in\IN$.
There is some $c_r>0$ such that for
every $n\geq 2$ and $f\in {\mathring{H}^r_{\mix}([0,1]^d)}$
\[
\abs{Q_{n^{1/d}B}(f)-S_d(f)} \leq\, c_r \,  n^{-r} 
\, (\ln n)^\frac{d-1}{2} \, \Xnorm{f}{H^r_{\rm mix}([0,1]^d)}
.\]
\end{thm}

For a proof of this error
bound and its optimality, we refer the reader to \cite{Ul16}.
In fact, this error bound holds uniformly for $Q_{n^{1/d}UB}^{\mathbf v}$ 
for any diagonal matrix $U\in\IR^{d\times d}$ with diagonal entries in $[1,2^{1/d}]$ 
and $\mathbf v\in\IR^d$,
which is the statement of Theorem~\ref{mixthmworstcase}.
We note that Frolov's algorithm also has the optimal
rate of convergence on ${\mathring{H}^r([0,1]^d)}$.
This is implied by Theorem~\ref{isothmworstcase}.

%
%

\subsection{Random Dilation}
\label{randomdilationsection}

We study the impact of random dilations on Frolov's algorithm.
We consider the method $Q_{n^{1/d}UB}^{\mathbf v}$
(see Algorithm~\ref{alg:basic algorithm})
for a Frolov matrix $B\in\IR^{d\times d}$, $n\in\IN$,
shift parameter $\mathbf v\in\IR^d$
and a random diagonal matrix $U\in\IR^{d\times d}$ whose
diagonal entries are independent and uniformly distributed 
in $[1,2^{1/d}]$.
This method computes at most
$2\left(\Vert B\Vert_1+1\right)^dn$ function values,
see Lemma~\ref{anlemma}.

\subsubsection{Guaranteed Errors}

With probability 1, the error 
has the same rate
of convergence as Frolov's algorithm.

\begin{thm}[\cite{KN17}]
\label{mixthmworstcase}
Let $B\in\IR^{d\times d}$ be a Frolov matrix and $r\in\IN$.
There is some $c_r>0$ such that for any
$n\geq 2$ and $f\in {\mathring{H}^r_{\mix}([0,1]^d)}$,
$$
\sup_{U,\mathbf v} 
\abs{Q_{n^{1/d}UB}^{\mathbf v}(f)-S_d(f)} \leq\, c_r \,  n^{-r} 
\, (\ln n)^\frac{d-1}{2} \, \Xnorm{f}{H^r_{\rm mix}([0,1]^d)}
,$$
where the supremum is taken over all
diagonal matrices $U\in\IR^{d\times d}$ with
diagonal entries in $[1,2^{1/d}]$ and $\mathbf v\in\IR^d$.
\end{thm}

\begin{proof}
Let us fix $U$ and $\mathbf v$ as above.
By Lemma~\ref{errorlemma} and Hölder's inequality,
\begin{multline}
\label{hoelder estimate}
\abs{Q_{n^{1/d}UB}^{\mathbf v}(f)-S_d(f)}^2
\leq \left(\sum\limits_{\mathbf m\in\IZ^d\setminus\set{0}} \left| 
\mathcal{F}f(n^{1/d}UB\mathbf m)\right|\right)^2\\
\leq \left(\sum\limits_{\mathbf m\in\IZ^d\setminus\set{0}} 
h_r(n^{1/d}UB\mathbf m)^{-1}\right)
\left(\sum\limits_{\mathbf m\in\IZ^d\setminus\set{0}} h_r(n^{1/d}UB\mathbf m)\cdot 
\abs{\mathcal{F}f(n^{1/d}UB\mathbf m)}^2\right)
.\end{multline}
We first prove that the first factor in this product is bounded 
above by a constant multiple of $n^{-2r}(\ln n)^{d-1}$,
where the constant is independent of $\mathbf v$ and $U$.
To that end, we consider the auxiliary set 
$$
 N(\beta)=\set{\mathbf x\in\IR^d \mid \forall1\leq j\leq d:
\lfloor 2^{\beta_j-1}\rfloor \leq|x_j|<2^{\beta_j}}
$$
for $\beta\in\IN_0^d$ and 
$$
G_n^\beta
=\set{\mathbf m\in\IZ^d\setminus\set{0}\mid n^{1/d}UB\mathbf m\in N(\beta)}.
$$
The domain $\IZ^d\setminus\set{0}$ of summation is the disjoint union 
of all $G_n^\beta$ over $\beta\in\IN_0^d$.

For $|\beta|\leq \log_2 n$, the points $\mathbf x$ in $N(\beta)$ satisfy
$\prod_{j=1}^{d}\abs{x_j}<2^{\abs{\beta}}\leq n$. But the second 
property of the Frolov matrix $B$ yields
$\prod_{j=1}^d \abs{n^{1/d}(UB\mathbf m)_j}\geq n$ for any 
$\mathbf m\in\IZ^d\setminus\set{0}$.
Hence, $G_n^\beta$ is empty for $|\beta|\leq \log_2 n$.
For $\abs{\beta} >\log_2 n$, any $\mathbf m\in G_n^\beta$ satisfies
\[
h_r(n^{1/d}UB\mathbf m)\geq \prod_{j=1}^d 
\left(1+\lfloor 2^{\beta_j-1}\rfloor^{2r}\right) 
\geq \prod_{j=1}^d 2^{2r(\beta_j-1)} = 2^{2r(|\beta|-d)}
\]
and hence $h_r(n^{1/d}UB\mathbf m)^{-1}\leq 2^{2r(d-|\beta|)}$.
Because of the third property of the Frolov matrix, we obtain
that the cardinality of $G_n^\beta$ is bounded above by
\[\card\brackets{\set{\mathbf m\in\IZ^d\setminus\set{0}\mid \abs{(B\mathbf m)_j}
< 2^{\beta_j}n^{-1/d}}} 
\leq 2^{d+\abs{\beta}} n^{-1} +1 \leq 2^{d+1+\abs{\beta}}
n^{-1}.\]
This shows that the first factor of~\eqref{hoelder estimate} satisfies
\begin{align}
\label{eq:needed in second proof}
\sum\limits_{|\beta|>\log_2 n}\, 
\sum\limits_{\mathbf m\in G_n^\beta} &h_r(n^{1/d}UB\mathbf m)^{-1}
\leq \sum\limits_{|\beta|>\log_2 n}\, 2^{2r(d-|\beta|)}
\cdot 2^{d+1+|\beta|}n^{-1}\\
\nonumber
&\leq \sum\limits_{k=\lceil \log_2 n\rceil}^\infty\, 2^{2r(d-k)}
\cdot 2^{d+1+k}n^{-1}\cdot \card\brackets{\set{\beta\in\IN_0^d\mid |\beta|=k}}.
\end{align}
The latter cardinality is bounded by $(k+1)^{d-1}$.
This yields the upper bound
\begin{multline*}
2^{2rd+d+1} n^{-1} \sum\limits_{k=\lceil \log_2 n\rceil}^{\infty}
2^{(1-2r)k} (k+1)^{d-1}\\
= 2^{2rd+d+1} n^{-1} 
\sum\limits_{k=0}^{\infty} 2^{(1-2r)(k+\lceil \log_2 n\rceil)}
\left(k+1+\lceil \log_2 n\rceil\right)^{d-1}\\
\leq 2^{2rd+d+1} n^{-1} \cdot n^{1-2r} 
\cdot \sum\limits_{k=0}^{\infty} 2^{(1-2r)k}\cdot 2^{d-1} 
\cdot (k+1)^{d-1}\cdot\lceil \log_2 n\rceil^{d-1}\\
\leq 2^{2rd+2d} \cdot n^{-2r} 
\cdot \sum\limits_{k=0}^{\infty} 2^{(1-2r)k} (k+1)^{d-1} 
\left( 2\cdot\frac{\ln n}{\ln 2}\right)^{d-1}\\
= \left( 2^{2rd+3d-1}\, (\ln 2)^{1-d} 
\sum\limits_{k=0}^{\infty} \left(2^{1-2r}\right)^k (k+1)^{d-1} \right)
\cdot n^{-2r}\, (\ln n)^{d-1},
\end{multline*}
which is the desired estimate since $2^{1-2r}<1$.

We now show that the second factor in the above 
inequality is bounded above by a constant 
multiple of $\Xnorm{f}{H^r_{\rm mix}([0,1]^d)}^2$,
where the constant is independent of $\mathbf v$ and $U$.
This will prove the theorem.
For $\mathbf x\in\IR^d$ we have
\[
h_r(\mathbf x)\cdot \abs{\mathcal{F}f(\mathbf x)}^{2}
=\sum\limits_{\Vert\alpha\Vert_\infty\leq r}\abs{\mathcal{F}\diff^\alpha f(\mathbf x)}^2
.\]
The function $g_\alpha=\diff^\alpha f\circ(n^{1/d}UB)^{-\top}$ has 
compact support
in $(n^{1/d}UB)^\top[0,1]^d$.
Consider the set
$J_n$ of all $\mathbf k\in\IZ^d$ for which $\brackets{\mathbf k+[0,1]^d}$ has 
nonempty intersection 
with $(n^{1/d}UB)^\top[0,1]^d$. The transformation
$\mathbf y=(n^{1/d}UB)^{-\top}\mathbf x$ yields
\begin{multline}
\label{eq:FDalphaf}
\abs{\mathcal{F}\diff^\alpha f \brackets{n^{1/d}UB\mathbf m}}^2
= \left| \int_{\IR^d} \diff^\alpha f(\mathbf y)\cdot 
e^{-2\pi i \langle n^{1/d}UB\mathbf m,\mathbf y\rangle} \d \mathbf y\right|^2\\
= \left|\frac{1}{\det(n^{1/d}UB)}\int_{\IR^d} g_\alpha(\mathbf x)\cdot 
e^{-2\pi i\langle \mathbf m,\mathbf x\rangle} \d \mathbf x\right|^2\\
= \left|\frac{1}{\det(n^{1/d}UB)}
\sum\limits_{\mathbf k\in J_n} \Xscalar{g_\alpha}{
e^{2\pi i\langle \mathbf m,\cdot\rangle}}{L^2\left(\mathbf k+[0,1]^d\right)}\right|^2\\
\leq \frac{\card(J_n)}{\abs{\det(n^{1/d}UB)}^2} 
\sum\limits_{\mathbf k\in J_n} 
\left|\Xscalar{g_\alpha}{
e^{2\pi i\langle \mathbf m,\cdot\rangle}}{L^2\left(\mathbf k+[0,1]^d\right)}\right|^2
.\end{multline}
Thus we obtain
\begin{multline*}
\sum\limits_{\mathbf m\in\IZ^d\setminus\set{0}} h_r(n^{1/d}UB\mathbf m)\cdot 
\abs{\mathcal{F}f(n^{1/d}UB\mathbf m)}^2
\leq \sum\limits_{\mathbf m\in\IZ^d} \sum\limits_{\Vert\alpha\Vert_\infty\leq r}
\abs{\mathcal{F}\diff^\alpha f(n^{1/d}UB\mathbf m)}^2\\
\leq \frac{\card(J_n)}{\abs{\det(n^{1/d}UB)}^2} 
\sum\limits_{\mathbf m\in\IZ^d}
\sum\limits_{\Vert\alpha\Vert_\infty\leq r}
\sum\limits_{\mathbf k\in J_n} \left|\Xscalar{g_\alpha}{e^{2\pi i\langle \mathbf m,
\cdot \rangle}}{L^2\left(\mathbf k+[0,1]^d\right)}\right|^2\\
= \frac{\card(J_n)}{\abs{\det(n^{1/d}UB)}^2} 
\sum\limits_{\Vert\alpha\Vert_\infty\leq r} 
\sum\limits_{\mathbf k\in J_n} \Xnorm{g_\alpha}{L^2\left(\mathbf k+[0,1]^d\right)}^2\\
= \frac{\card(J_n)}{\abs{\det(n^{1/d}UB)}^2} 
\sum\limits_{\Vert\alpha\Vert_\infty\leq r} \Xnorm{g_\alpha}{L^2(\IR^d)}^2
= \frac{\card(J_n)}{\abs{\det(n^{1/d}UB)}} 
\sum\limits_{\Vert\alpha\Vert_\infty\leq r} 
\Xnorm{\diff^\alpha f}{L^2(\IR^d)}^2
.\end{multline*}
Since both $\card(J_n)$ and $\abs{\det(n^{1/d}UB)}$ 
are of order $n$, their ratio is bounded
by a constant and the above inequality yields the statement.
\end{proof}

\begin{thm}[\cite{KN17}]
\label{isothmworstcase}
Let $B\in\IR^{d\times d}$ be any invertible matrix and $r>d/2$.
There is some $c_r>0$ such that, for any
$n\in\IN$ and $f\in {\mathring{H}^r([0,1]^d)}$,
\[
\sup_{U,\mathbf v} \abs{Q_{n^{1/d}UB}^{\mathbf v}(f)-S_d(f)} 
\leq\, c_r \,  n^{-r/d}
\, \Xnorm{f}{H^r([0,1]^d)}
,\]
where the supremum is taken over all
diagonal matrices $U\in\IR^{d\times d}$ with
diagonal entries in $[1,2^{1/d}]$ and $\mathbf v\in\IR^d$.
\end{thm}

\begin{proof}
Let $U$ and $\mathbf{v}$ be as above.
By Lemma~\ref{errorlemma} and Hölder's inequality,
\begin{multline}
\label{hoelder estimate 2}
\abs{Q_{n^{1/d}UB}^{\mathbf v}(f)-S_d(f)}^2
\leq \left(\sum\limits_{\mathbf m\in\IZ^d\setminus\set{0}} 
\left| \mathcal{F}f\brackets{n^{1/d}UB\mathbf m}\right|\right)^2\\
\leq \left(\sum\limits_{\mathbf m\in\IZ^d\setminus\set{0}} 
v_r\brackets{n^{1/d}UB\mathbf m}^{-1}\right)
\left(\sum\limits_{\mathbf m\in\IZ^d\setminus\set{0}} 
v_r\brackets{n^{1/d}UB\mathbf m}
\abs{\mathcal{F}f\brackets{n^{1/d}UB\mathbf m}}^2\right)
.\end{multline}
The first factor in this product is bounded by a constant multiple of 
$n^{-2r/d}$: since
\[
v_r\brackets{n^{1/d}UB\mathbf m} \geq \Vert n^{1/d}UB\mathbf m\Vert_2^{2r}
\geq n^{2r/d}  \Vert B\mathbf m\Vert_2^{2r}
\geq n^{2r/d} \Vert B^{-1}\Vert_2^{-2r}
\Vert\mathbf m\Vert_2^{2r}
,\]
we have
\[
\sum\limits_{\mathbf m\in\IZ^d\setminus\set{0}} v_r\brackets{n^{1/d}UB\mathbf m}^{-1}
\leq n^{-2r/d} \Vert B^{-1}\Vert_2^{2r}
\sum\limits_{\mathbf m\in\IZ^d\setminus\set{0}}\Vert \mathbf m\Vert_2^{-2r}
,\]
where this last series converges for $2r>d$.

We show that the second factor in \eqref{hoelder estimate 2} is bounded above 
by a constant multiple of $\Xnorm{f}{H^r([0,1]^d)}^2$. This will prove the theorem.
For any $\mathbf x\in\IR^d$ we have
\[
v_r(\mathbf x)\cdot \abs{\mathcal{F}f(\mathbf x)}^{2}
=\sum\limits_{\abs{\alpha}\leq r}\abs{\mathcal{F}\diff^\alpha f(\mathbf x)}^2
.\]

The function $g_\alpha=\diff^\alpha f\circ(n^{1/d}UB)^{-\top}$ has 
compact support
in the parallelepiped $(n^{1/d}UB)^\top[0,1]^d$.
Again consider the set
$J_n$ of all $\mathbf k\in\IZ^d$ for which $\brackets{\mathbf k+[0,1]^d}$ has a 
nonempty intersection 
with $(n^{1/d}UB)^\top[0,1]^d$.
With \eqref{eq:FDalphaf}, we obtain
\begin{multline*}
\sum\limits_{\mathbf m\in\IZ^d\setminus\set{0}} v_r(n^{1/d}UB\mathbf m)\cdot 
\abs{\mathcal{F}f(n^{1/d}UB\mathbf m)}^2
\leq \sum\limits_{\mathbf m\in\IZ^d} \sum\limits_{\vert\alpha\vert\leq r}
\abs{\mathcal{F}\diff^\alpha f(n^{1/d}UB\mathbf m)}^2\\
\leq \frac{\card(J_n)}{\abs{\det(n^{1/d}UB)}^2} 
\sum\limits_{\mathbf m\in\IZ^d}
\sum\limits_{\vert\alpha\vert\leq r}
\sum\limits_{\mathbf k\in J_n} \left|\Xscalar{g_\alpha}{e^{2\pi i\langle \mathbf m,
\cdot \rangle}}{L^2\left(\mathbf k+[0,1]^d\right)}\right|^2\\
= \frac{\card(J_n)}{\abs{\det(n^{1/d}UB)}^2} 
\sum\limits_{\vert\alpha\vert\leq r} 
\sum\limits_{\mathbf k\in J_n} \Xnorm{g_\alpha}{L^2\left(\mathbf k+[0,1]^d\right)}^2\\
= \frac{\card(J_n)}{\abs{\det(n^{1/d}UB)}^2} 
\sum\limits_{\vert\alpha\vert\leq r} \Xnorm{g_\alpha}{L^2(\IR^d)}^2
= \frac{\card(J_n)}{\abs{\det(n^{1/d}UB)}} 
\sum\limits_{\vert\alpha\vert\leq r} 
\Xnorm{\diff^\alpha f}{L^2(\IR^d)}^2
.\end{multline*}
Since both $\card(J_n)$ and $\abs{\det(n^{1/d}UB)}$ are of order $n$, 
their ratio is bounded
by a constant and the above inequality yields the statement.
\end{proof}

\subsubsection{Expected Errors}

In expectation, the random dilations improve the order of the error
of Frolov's algorithm by $1/2$ for
both ${\mathring{H}^r_{\mix}([0,1]^d)}$ and ${\mathring{H}^r([0,1]^d)}$.
These results are based on the following general error bound for
continuous functions with compact support. Recall that 
$D_n$ is the set of all $\mathbf x\in\IR^d$ 
with $\prod_{j=1}^{d}\abs{x_j}\geq n$.

\begin{thm}[\cite{KN17}]
\label{keyprop}
Let $B\in\IR^{d\times d}$ be a Frolov matrix
and let $U\in\IR^{d\times d}$ be a diagonal matrix
whose diagonal entries are independent and uniformly distributed in $[1,2^{1/d}]$.
There is a constant $c>0$ such that, for every $n\in\IN$, 
shift parameter $\mathbf v\in\IR^d$ and $f\in{\C_c(\IR^d)}$,
\[\IE \abs{Q_{n^{1/d}UB}^{\mathbf v}(f)-\int_{\IR^d} f(\mathbf x)~\d\mathbf x} 
\leq\, c \,  n^{-1} \int_{D_n} \abs{\mathcal{F}f(\mathbf x)}\, \d \mathbf x
.\]
\end{thm} 

\begin{proof}
Thanks to Lemma~\ref{errorlemma} 
and the monotone convergence theorem we have
\[\begin{split}
\IE \abs{Q_{n^{1/d}UB}^{\mathbf v}(f)-\int_{\IR^d} f(\mathbf x)~\d\mathbf x}
&\leq \IE 
\left(\sum\limits_{\mathbf m\in\IZ^d\setminus\set{0}} 
\abs{\mathcal{F}f\brackets{n^{1/d}UB\mathbf m}} \right)\\
&= \sum\limits_{\mathbf m\in\IZ^d\setminus\set{0}} 
\IE \abs{\mathcal{F}f\brackets{n^{1/d}UB\mathbf m}}
.\end{split}\]
Since each $n^{1/d}UB\mathbf m$ is uniformly distributed in the 
box $[n^{1/d}B\mathbf m,(2n)^{1/d}B\mathbf m]$ of volume 
$c_d\abs{\prod_{j=1}^dn^{1/d} (B\mathbf m)_j}$ with $c_d=(2^{1/d}-1)^d$, 
this series equals
\begingroup
  \allowdisplaybreaks
\begin{align*}
&\frac{1}{c_d} \sum\limits_{\mathbf m\in\IZ^d\setminus\set{0}}\, 
\int_{[n^{1/d}B\mathbf m,(2n)^{1/d}B\mathbf m]} \frac{\abs{\mathcal{F}f(\mathbf x)}}{\prod_{j=1}^d
\abs{n^{1/d}(B\mathbf m)_j}} \, \d \mathbf x \\
&\leq \frac{1}{c_d} 
\sum\limits_{\mathbf m\in\IZ^d\setminus\set{0}}\, \int_{[n^{1/d}B\mathbf m,(2n)^{1/d}B\mathbf m]} 
\frac{\abs{\mathcal{F}f(\mathbf x)}}{\prod_{j=1}^d 2^{-1/d}\abs{x_j}} \, \d \mathbf x\\
&= \frac{2}{c_d}\cdot \int_{\IR^d} 
\frac{\abs{\mathcal{F}f(\mathbf x)}}{\prod_{j=1}^d \abs{x_j}}\cdot 
\abs{\set{\mathbf m\in\IZ^d\setminus\set{0}\mid 
\mathbf x\in [n^{1/d}B\mathbf m,(2n)^{1/d}B\mathbf m]}} \, \d \mathbf x\\
&= \frac{2}{c_d}\cdot \int_{\IR^d} 
\frac{\abs{\mathcal{F}f(\mathbf x)}}{\prod_{j=1}^d \abs{x_j}}\cdot 
\abs{\set{\mathbf m\in\IZ^d\setminus\set{0}\mid B\mathbf m\in 
\left[\frac{\mathbf x}{(2n)^{1/d}},\frac{\mathbf x}{n^{1/d}}\right]}}\, \d \mathbf x  
.\end{align*}
\endgroup
Thanks to the properties of the Frolov matrix $B$, if 
$\prod_{j=1}^d \abs{x_j}<n$, 
the latter set is empty and otherwise contains no more 
than $\prod_{j=1}^d \abs{\frac{x_j}{n^{1/d}}}+1
\leq 2 n^{-1} \prod_{j=1}^d \abs{x_j}$ points.
Thus, we arrive at
\[
\IE \abs{Q_{n^{1/d}UB}^{\mathbf v}(f)-\int_{\IR^d} f(\mathbf x)~\d\mathbf x} \leq
\frac{4}{c_d}\cdot n^{-1} 
\int_{D_n} \abs{\mathcal{F}f(\mathbf x)} \, \d \mathbf x.
\]
\end{proof}

Additional differentiability properties of the function 
$f\in{\C_c(\IR^d)}$ result in decay 
properties of its Fourier transform $\mathcal{F}f$. 
This leads to estimates of the 
integral $\int_{D_n} \abs{\mathcal{F}f(\mathbf x)} \, \d \mathbf x$.
Hence, the general upper bound for the error of 
$Q_{n^{1/d}UB}^{\mathbf v}(f)$ in Theorem~\ref{keyprop} adjusts to the differentiability of $f$.
Two such examples are functions from 
${\mathring{H}^r_{\mix}([0,1]^d)}$ and ${\mathring{H}^r([0,1]^d)}$.

\begin{lemma}
\label{intlemmamix}
There is some $c_r>0$ such that,
for each $n\geq 2$ and $f\in{\mathring{H}^r_{\mix}([0,1]^d)}$,
\[
\int_{D_n} \abs{\mathcal{F}f(\mathbf x)} \, \d \mathbf x  \leq c_r \,  n^{-r+1/2}
\, \left(\ln n\right)^{\frac{d-1}{2}} \, \Xnorm{f}{H^r_{\rm mix}([0,1]^d)}
.\]
\end{lemma}

\begin{proof}
Applying Hölder's inequality and a linear substitution $\mathbf x=n^{1/d}B\mathbf y$ 
to the above integral, we get
\begin{multline*}
\left( \int_{D_n} \abs{\mathcal{F}f(\mathbf x)} \, \d \mathbf x  \right)^2
\leq \brackets{\int_{D_n} h_r(\mathbf x)^{-1} \, \d \mathbf x} \Xnorm{f}{H^r_{\rm mix}([0,1]^d)}^2\\
=  n \abs{\det B} \left(\int_{G} h_r(n^{1/d}B\mathbf y)^{-1} \, \d \mathbf y \right) 
\Xnorm{f}{H^r_{\rm mix}([0,1]^d)}^2
\end{multline*}
with $G=B^{-1}D_1$ being the set of all $\mathbf y\in\IR^d$ with 
$\prod_{j=1}^{d}\abs{(B\mathbf y)_j}\geq 1$. Hence, it is sufficient to prove that the 
integral $\int_{G} h_r(n^{1/d}B\mathbf y)^{-1} \, \d \mathbf y$ is bounded by a constant 
multiple of $n^{-2r}(\ln n)^{d-1}$.
We again consider the auxiliary set 
$$
 N(\beta)=\left\{\mathbf x\in\IR^d \mid \lfloor 2^{\beta_j-1}\rfloor
\leq|x_j|<2^{\beta_j},1\leq j\leq d\right\}
$$
for $\beta\in\IN_0^d$
and 
$$G_n^\beta=\set{\mathbf y\in G\mid n^{1/d}B\mathbf y\in N(\beta)}.$$
Similar to the proof of Theorem \ref{mixthmworstcase},
the domain $G$ of integration is the disjoint union of all 
$G_n^\beta$ over $\beta\in\IN_0^d$,
where $G_n^\beta=\emptyset$ if $\abs{\beta}\leq \log_2 n$, and otherwise
the integrand is bounded above by $2^{2r(d-\abs{\beta})}$ for $\mathbf y\in G_n^\beta$.
On the other hand,
\begin{multline*}
\lambda^d(G_n^\beta)\leq \lambda^d\left((n^{1/d}B)^{-1}N(\beta)\right) 
= n^{-1} |\det B|^{-1} \lambda^d(N(\beta)) \\
= n^{-1} |\det B|^{-1} 2^d \prod_{j=1}^d\left(2^{\beta_j}-
\lfloor 2^{\beta_j-1}\rfloor\right)
\leq n^{-1} |\det B|^{-1} 2^{d+|\beta|}
.\end{multline*}
Like in the proof of Theorem \ref{mixthmworstcase}, we obtain
\begin{multline*}
\int_{G} h_r(n^{1/d}B\mathbf y)^{-1} \, \d \mathbf y 
= \sum\limits_{|\beta|>\log_2 n}\ \int_{G_n^\beta} h_r(n^{1/d}B\mathbf y)^{-1} \, \d \mathbf y \\
\leq \sum\limits_{|\beta|>\log_2 n} 2^{2r(d-|\beta|)}
 n^{-1} |\det B|^{-1} 2^{d+\beta}
= |\det B|^{-1} 2^{-1} \sum\limits_{|\beta|>\log_2 n} 2^{2r(d-|\beta|)}
 n^{-1} 2^{d+1+\abs{\beta}}\\
\overset{\eqref{eq:needed in second proof}}{\leq}\left( 2^{2rd+3d-2} |\det B|^{-1} (\ln 2)^{1-d} 
\sum\limits_{k=0}^{\infty} \left(2^{1-2r}\right)^k (k+1)^{d-1} \right)
\, n^{-2r}\, (\ln n)^{d-1}
,\end{multline*}
where the constant is finite since $2^{1-2r}<1$.
\end{proof}

Combining Theorem~\ref{keyprop} and Lemma~\ref{intlemmamix} 
yields the following.

\begin{thm}[\cite{KN17}]
\label{mixthm}
Let $B\in\IR^{d\times d}$ be a Frolov matrix
and let $U\in\IR^{d\times d}$ be a diagonal matrix
whose diagonal entries are independent and uniformly distributed in $[1,2^{1/d}]$.
For all $r\in\IN$, there is a constant $c_r>0$ such that,
for every $n\geq 2$, shift parameter $\mathbf v\in\IR^d$, and 
$f\in {\mathring{H}^r_{\mix}([0,1]^d)}$,
\[
\IE \abs{Q_{n^{1/d}UB}^{\mathbf v}(f)-S_d(f)} \leq\, c_r \, n^{-r-1/2} 
\, (\ln n)^\frac{d-1}{2} \, \Xnorm{f}{H^r_{\rm mix}([0,1]^d)}
.\]
\end{thm}

If the integrand is from the space 
${\mathring{H}^r([0,1]^d)}$,
the following lemma holds.

\begin{lemma}
\label{intlemmaiso}
For $r>d/2$, there is some $c_r>0$
such that, for all $n\in\IN$ and $f\in{\mathring{H}^r([0,1]^d)}$,
\[
\int_{D_n} \abs{\mathcal{F}f(\mathbf x)} \, \d \mathbf x  \leq c_r \,  n^{-r/d+1/2} \, 
\Xnorm{f}{H^r([0,1]^d)}
.\]
\end{lemma}

\begin{proof}
Like in Lemma~\ref{intlemmamix}, we apply Hölder's inequality and get
\begin{multline*}
\left( \int_{D_n} \abs{\mathcal{F}f(\mathbf x)} \, \d \mathbf x  \right)^2
\leq \left(\int_{D_n} v_r(\mathbf x)^{-1} \, \d \mathbf x \right) \Xnorm{f}{H^r(\IR^d)}^2\\
\leq \left(\int_{D_n} \Vert\mathbf x\Vert_2^{-2r}~\d \mathbf x \right)
\Xnorm{f}{H^r([0,1]^d)}^2
.\end{multline*} 
Since $\Vert \mathbf x\Vert_2\geq \Vert \mathbf x\Vert_\infty \geq n^{1/d}$ 
for $\mathbf x\in D_n$,
the latter integral in the above relation is bounded above by
$$
\int_{\Vert\mathbf x\Vert_2\geq n^{1/d}} 
\Vert\mathbf x\Vert_2^{-2r}~\d \mathbf x
= \int_{n^{1/d}}^{\infty}\int_{\mathbb S_{d-1}} 
R^{-2r+d-1}~\d\sigma(\mathbf y) \, \d R
= \frac{\sigma\left(\mathbb S_{d-1}\right)}{2r-d}
n^{-2r/d+1}.
$$
Here, $\sigma$ is the surface measure on $\mathbb S_{d-1}$.
\end{proof}

In this case, combining Theorem~\ref{keyprop} and 
Lemma~\ref{intlemmaiso} yields the following, where we
recall that ${\mathring{H}^r([0,1]^d)}\subset \C_c(\IR^d)$
for $r>d/2$.

\begin{thm}[\cite{KN17}]
\label{isothm}
Let $B\in\IR^{d\times d}$ be a Frolov matrix
and let $U\in\IR^{d\times d}$ be a diagonal matrix
whose diagonal entries are independent and uniformly distributed in $[1,2^{1/d}]$.
For all $r\in\IN$ with $r>d/2$, there is a constant $c_r>0$ such that, for
every $n\in\IN$, shift parameter $\mathbf v\in\IR^d$, and 
$f\in {\mathring{H}^r([0,1]^d)}$,
\[\begin{split}
\IE \abs{Q_{n^{1/d}UB}^{\mathbf v}(f)-S_d(f)} 
\leq\, c_r \,  n^{-r/d-1/2} \, \Xnorm{f}{H^r([0,1]^d)}
.\end{split}\]
\end{thm}

We remark that the Frolov properties of the matrix $B$ 
are not needed to get this
estimate on ${\mathring{H}^r([0,1]^d)}$, although they are essential 
for the upper bound
on ${\mathring{H}^r_{\mix}([0,1]^d)}$ from Theorem~\ref{mixthm}. 
For example, also the 
identity matrix would do. But if $B$ is a Frolov matrix, 
$Q_{n^{1/d}UB}^{\mathbf v}$ works 
universally for ${\mathring{H}^r_{\mix}([0,1]^d)}$ and ${\mathring{H}^r([0,1]^d)}$. 
Furthermore, the Frolov properties
of $B$ prevent large jumps in the number of nodes of
$Q_{n^{1/d}UB}^{\mathbf v}$ for small changes of 
the dilation matrix $U$.

\subsection{Random Shift}
\label{randomshiftsection}

Now we also choose the
shift parameter $\mathbf v$ in $Q_{n^{1/d}UB}^{\mathbf v}$ randomly.
We choose it uniformly distributed in $[0,1]^d$. 
Note that the number of function values the algorithm
uses for functions with support in $[0,1]^d$ is still of order $n$.
The first advantage of this method is its unbiasedness.

\begin{prop}[\cite{KN17}]
\label{Munbiased}
Let $B\in\IR^{d\times d}$ be a random matrix which is
almost surely invertible.
Let $\mathbf v$ be uniformly distributed in $[0,1]^d$
and independent of $B$. 
For any $f\in L^1(\IR^d)$, 
the series $Q_B^{\mathbf v}(f)$ converges absolutely almost surely and
\[
\IE \brackets{Q_B^{\mathbf v}(f)}=\int_{\IR^d} f(\mathbf y) \, \d \mathbf y
.\]
\end{prop}

\begin{proof}
Let us first fix an invertible realization of $B$.
By the monotone convergence theorem,
we obtain
\begin{multline*}
\IE \brackets{\sum\limits_{\mathbf m\in\IZ^d} \frac{1}{\abs{\det B}} 
\abs{f\left( B^{-\top}(\mathbf m+{\mathbf v}) \right)}}
=\sum\limits_{\mathbf m\in\IZ^d} \IE \brackets{\frac{1}{\abs{\det B}} 
\abs{f\left( B^{-\top}(\mathbf m+{\mathbf v}) \right)}} \\
= \sum\limits_{\mathbf m\in\IZ^d} \frac{1}{\abs{\det B}} 
\int_{[0,1]^d} \abs{f\left( B^{-\top}(\mathbf m+\mathbf x) \right)} \, \d \mathbf x \\
= \sum\limits_{\mathbf m\in\IZ^d}\ \int_{B^{-\top}\brackets{\mathbf m+[0,1]^d}} 
\abs{f(\mathbf y)} \, \d \mathbf y
= \int_{\IR^d} \abs{f(\mathbf y)} \, \d \mathbf y\,
< \infty
.\end{multline*}
In particular, the series $Q_B^{\mathbf v}(f)$ 
converges absolutely almost surely
and is dominated by an integrable function.
Lebesgue's dominated convergence theorem yields
$$
\IE \brackets{Q_{B}^{\mathbf v}(f)}
= \sum\limits_{\mathbf m\in\IZ^d} \frac{1}{\abs{\det B}} 
\int_{[0,1]^d} f\left( B^{-\top}(\mathbf m+\mathbf x) \right) \, \d \mathbf x\\
= \int_{\IR^d} f(\mathbf y) \, \d \mathbf y.
$$
Fubini's theorem implies that the same equalities
hold if $B$ is a random matrix which is independent of $\mathbf v$ and 
almost surely invertible.
In particular, $Q_{B}^{\mathbf v}(f)$ still converges 
absolutely almost surely.
\end{proof}

%
%

The second advantage of this method is the slight improvement in the
order of convergence of the expected error 
on ${\mathring{H}^r_{\mix}([0,1]^d)}$.
If only $U$ is random, 
the expected error is of order $n^{-r-1/2}(\ln n)^\frac{d-1}{2}$,
see Theorem~\ref{mixthm}.
If both $U$ and $\mathbf v$ are random,
the expected error is of order $n^{-r-1/2}$,
as proven in \cite{Ul17}.
The proof even shows that the quantity 
$$
\brackets{\IE \abs{Q_{n^{1/d}UB}^{\mathbf v}(f)-S_d(f)}^2}^{1/2}
$$
satisfies this bound.
This is a stronger statement, as implied by Hölder's inequality.
We now turn to the proof.
Similar to Lemma~\ref{errorlemma}, 
the expected error of the randomized algorithm
for integration on ${\C_c(\IR^d)}$ can be expressed 
in terms of the Fourier transform.

\begin{lemma}
\label{errorlemma2}
Let $B\in\IR^{d\times d}$ be invertible and
$\mathbf v$ be uniformly distributed in $[0,1]^d$. 
For any $f\in \C_c(\IR^d)$, we have
\[
\IE \abs{Q_B^{\mathbf v}(f)-\int_{\IR^d} f(\mathbf x)~\d\mathbf x}^2 
= \sum\limits_{\mathbf m\in\IZ^d\setminus\set{0}} \abs{\mathcal{F}f(B\mathbf m)}^2
.\]
\end{lemma}

\begin{proof}
We first recall that
$$
 \IE\, Q_B^{\mathbf v}(f)=\int_{\IR^d} f(\mathbf x) ~\d\mathbf x
 =\mathcal{F}f(\mathbf 0).
$$
In particular, we obtain
\[
\IE \abs{Q_B^{\mathbf v}(f)-\int_{\IR^d} f(\mathbf x)~\d\mathbf x}^2
= \Var \brackets{Q_B^{\mathbf v}(f)}
= \IE \abs{Q_B^{\mathbf v}(f)}^2 - \abs{\IE\, Q_B^{\mathbf v}(f)}^2
.\]
The algorithm $Q_B^{\mathbf v}(f)$ considered as a function of $\mathbf v\in[0,1]^d$
is a finite sum of square-integrable functions
and hence square-integrable.
Parseval's identity states
\[
\IE \abs{Q_B^{\mathbf v}(f)}^2
= \Xnorm{Q_B^{\brackets{\cdot}}(f)}{L^2([0,1]^d)}^2
= \sum\limits_{\mathbf m\in\IZ^d}
\abs{\Xscalar{Q_B^{\brackets{\cdot}}(f)}{
e^{2\pi i \scalar{\mathbf m}{\cdot}}}{L^2([0,1]^d)}}^2
.\]
For each index $\mathbf m\in\IZ^d$ we have the equality
\begin{multline*}
\Xscalar{Q_B^{\brackets{\cdot}}(f)}{e^{2\pi i \scalar{\mathbf m}{\cdot}}}{L^2([0,1]^d)}
= \abs{\det B}^{-1} \sum\limits_{\mathbf k\in\IZ^d}
\int_{[0,1]^d} f\left(B^{-\top}(\mathbf k+\mathbf v)\right)\, 
e^{-2\pi i \scalar{\mathbf m}{\mathbf v}}\d \mathbf v\\
= \abs{\det B}^{-1}
\int_{\IR^d} f\left(B^{-\top}\mathbf v\right)\, e^{-2\pi i \scalar{\mathbf m}{\mathbf v}}\d \mathbf v
= \int_{\IR^d} f\left(\mathbf v\right)\, e^{-2\pi i \scalar{B\mathbf m}{\mathbf v}}\d \mathbf v
= \mathcal{F}f(B\mathbf m)
.\end{multline*}
Putting everything together, we obtain the stated identity.
\end{proof}

Now follows an analogue of Theorem~\ref{keyprop}
for expected quadratic errors.

\begin{thm}[\cite{Ul17}]
\label{keyprop2}
Let $B$ be a Frolov matrix, let
$U$ be a random diagonal matrix whose diagonal entries
are independent and uniformly distributed in $[1,2^{1/d}]$,
and let $\mathbf v$ be independent of $U$ and uniformly
distributed in $[0,1]^d$.
There is a constant $c>0$ 
such that, for every $n\in\IN$ and $f\in{\C_c(\IR^d)}$,
\[
\IE \abs{Q_{n^{1/d}UB}^{\mathbf v}(f)-\int_{\IR^d} f(\mathbf x)~\d\mathbf x}^2
\leq\, c \,  n^{-1} \, \Vert\mathcal{F}f\Vert_{L^2\brackets{D_n}}^2
.\]
\end{thm}

\begin{proof}
By Lemma~\ref{errorlemma2} 
and the monotone convergence theorem, we have
\begin{equation*}
\IE \abs{Q_{n^{1/d}UB}^{\mathbf v}(f)-\int_{\IR^d} f(\mathbf x)~\d\mathbf x}^2
= \sum\limits_{\mathbf m\in\IZ^d\setminus\set{0}} 
\IE_U \abs{\mathcal{F}f(n^{1/d}UB\mathbf m)}^2
.\end{equation*}
Since each $n^{1/d}UB\mathbf m$ is uniformly distributed in the 
box $[n^{1/d}B\mathbf m,(2n)^{1/d}B\mathbf m]$ of volume 
$c_d\abs{\prod_{j=1}^d n^{1/d} (B\mathbf m)_j}$
with $c_d=(2^{1/d}-1)^d$,
this series equals
\begin{multline*}
\frac{1}{c_d} 
\sum\limits_{\mathbf m\in\IZ^d\setminus\set{0}}\, 
\int_{[n^{1/d}B\mathbf m,(2n)^{1/d}B\mathbf m]} 
\frac{\abs{\mathcal{F}f(\mathbf x)}^2}{\prod_{j=1}^d
\abs{n^{1/d}(B\mathbf m)_j}} \, \d \mathbf x \\
\leq \frac{1}{c_d} 
\sum\limits_{\mathbf m\in\IZ^d\setminus\set{0}}\, 
\int_{[n^{1/d}B\mathbf m,(2n)^{1/d}B\mathbf m]} 
\frac{\abs{\mathcal{F}f(\mathbf x)}^2}{\prod_{j=1}^d 2^{-1/d}\abs{x_j}} \, \d \mathbf x
= \frac{2}{c_d} \int_{\IR^d} 
\frac{\abs{\mathcal{F}f(\mathbf x)}^2}{\prod_{j=1}^d \abs{x_j}} 
N(\mathbf x)\d \mathbf x,
\end{multline*}
where
\begin{align*}
  N(\mathbf{x}) &= \card\set{\mathbf m\in\IZ^d\setminus\set{0}\mid 
 \mathbf x\in [n^{1/d}B\mathbf m,(2n)^{1/d}B\mathbf m]}\\
 &= \card\set{\mathbf m\in\IZ^d\setminus\set{0}\mid B\mathbf m\in
 \left[\frac{\mathbf x}{(2n)^{1/d}},\frac{\mathbf x}{n^{1/d}}\right]}.
\end{align*}
Thanks to the properties of the Frolov matrix $B$, if 
$\prod_{j=1}^d \abs{x_j}<n$, 
the latter set is empty and otherwise contains no more 
than 
$$
 \prod_{j=1}^d \abs{\frac{x_j}{n^{1/d}}}+1
 \leq 2 n^{-1} \prod_{j=1}^d \abs{x_j}
$$
points. Thus, we arrive at the upper bound
\[
\frac{4}{c_d} n^{-1} 
\int_{D_n} \abs{\mathcal{F}f(\mathbf x)}^2 \, \d \mathbf x 
\]
and the theorem is proven.
\end{proof}

Like the upper bound of Theorem~\ref{keyprop},
the upper bound of Theorem~\ref{keyprop2}
adjusts to the smoothness of the function.
This leads to the previously mentioned result
on the rate of convergence on $\mathring{H}^r_{\mix}([0,1]^d)$.

\begin{thm}[\cite{Ul17}]
\label{mixthm2}
Let $B$ be a Frolov matrix, 
let $U$ be a random diagonal matrix whose diagonal entries
are independent and uniformly distributed in $[1,2^{1/d}]$,
and let $\mathbf v$ be independent of $U$ and uniformly
distributed in $[0,1]^d$.
For every $r\in\IN$,
there is some $c_r>0$ such that,
for every $n\geq 2$ and $f\in {\mathring{H}^r_{\mix}([0,1]^d)}$,
\[
\brackets{\IE \abs{Q_{n^{1/d}UB}^{\mathbf v}(f)-S_d(f)}^2}^{1/2} 
\leq\, c_r \, n^{-r-1/2} 
\, \Xnorm{f}{H^r_{\rm mix}([0,1]^d)}
.\]
\end{thm}

\begin{proof}
If $c$ is the constant of Theorem~\ref{keyprop2}, we have the upper bound
\begin{multline*}
\IE \abs{Q_{n^{1/d}UB}^{\mathbf v}(f)-S_d(f)}^2
\leq c \,  n^{-1} \, \Vert\mathcal{F}f\Vert_{L^2\brackets{D_n}}^2\\
= c\, n^{-1} \, \int_{D_n} h_r(\mathbf x)^{-1}
\left|\mathcal{F}f(\mathbf x)\right|^2\, h_r(\mathbf x)\, \d \mathbf x\\
\leq\, c \,  n^{-1} \, \Xnorm{h_r^{-1}}{L^\infty\brackets{D_n}} 
\int_{\IR^d}\left|\mathcal{F}f(\mathbf x)\right|^2\, h_r(\mathbf x)\, \d \mathbf x
\end{multline*}
for the expected quadratic error.
Since $h_r(\mathbf x)\geq n^{2r}$ for $\mathbf x\in D_n$, we get the estimate
\[
\IE \abs{Q_{n^{1/d}UB}^{\mathbf v}(f)-S_d(f)}^2
\leq c\, n^{-2r-1} 
\, \Xnorm{f}{H^r_{\rm mix}([0,1]^d)}^2
,\]
which proves the theorem.
\end{proof}

The error of the algorithm in Theorem~\ref{mixthm2} also has the optimal
order of convergence for 
${\mathring{H}^r([0,1]^d)}$.
This can be derived from Theorem~\ref{keyprop2} using the same argument
as in the proof of Theorem~\ref{mixthm2}.

\begin{thm}[\cite{KN17}]
\label{isothm2}
Let $B$ be a Frolov matrix, 
let $U$ be a random diagonal matrix whose diagonal entries
are independent and uniformly distributed in $[1,2^{1/d}]$,
and let $\mathbf v$ be independent of $U$ and uniformly
distributed in $[0,1]^d$.
For every $r\in\IN$ with $r>d/2$,
there is some $c_r>0$ such that, 
for every $n\in\IN$ and $f\in {\mathring{H}^r([0,1]^d)}$,
\[
\brackets{\IE \abs{Q_{n^{1/d}UB}^{\mathbf v}(f)-S_d(f)}^2}^{1/2} 
\leq\, c_r \, n^{-r/d-1/2} 
\, \Xnorm{f}{H^r([0,1]^d)}
.\]
\end{thm}

Note that the corresponding upper bound for the expected absolute error
(instead of the expected mean square error)
is a direct consequence of either Theorem~\ref{isothm}
or Theorem~\ref{isothm2}.


\subsection{Functions without Boundary Conditions}
\label{transformationsection}

We can transform the algorithm from Section~\ref{randomshiftsection}
such that its error satisfies the same upper bounds for 
every function in 
${H^r_{\rm mix}([0,1]^d)}$ and ${H^r([0,1]^d)}$, 
not only for those vanishing at the boundary. 
This is done by a standard method, which was already used in 
\cite[pp.\,359]{Te03} to transform Frolov's deterministic algorithm.
The transformation is independent of $r$ and 
preserves the unbiasedness of the algorithm.

To that end, let $\psi:\IR\to\IR$ be an infinitely differentiable function 
that is a diffeomorphism of $(0,1)$,
vanishes on $(-\infty,0)$, and equals 1 on $(1,\infty)$.
An example is given by the following definition for $x\in\IR$:
\[
h(x)=\begin{cases}
e^\frac{1}{(2x-1)^2-1} & \text{if } x\in(0,1),\\
0 & \text{else,}
\end{cases}
\quad\quad
\psi(x)=\frac{\int_{-\infty}^x h(t) \, \d t}{
\int_{-\infty}^{\infty} h(t) \, \d t}.
\]
Like $h$ also $\psi$ is infinitely differentiable and 
vanishes on $(-\infty,0)$ and equals 1 on $(1,\infty)$. 
Since the derivative of $\psi$ is strictly positive on $(0,1)$, 
it is strictly increasing and a bijection of $(0,1)$ and its 
inverse function is smooth.

\begin{figure}[ht]
 \begin{minipage}{.49\linewidth}
 \includegraphics[width=\linewidth]{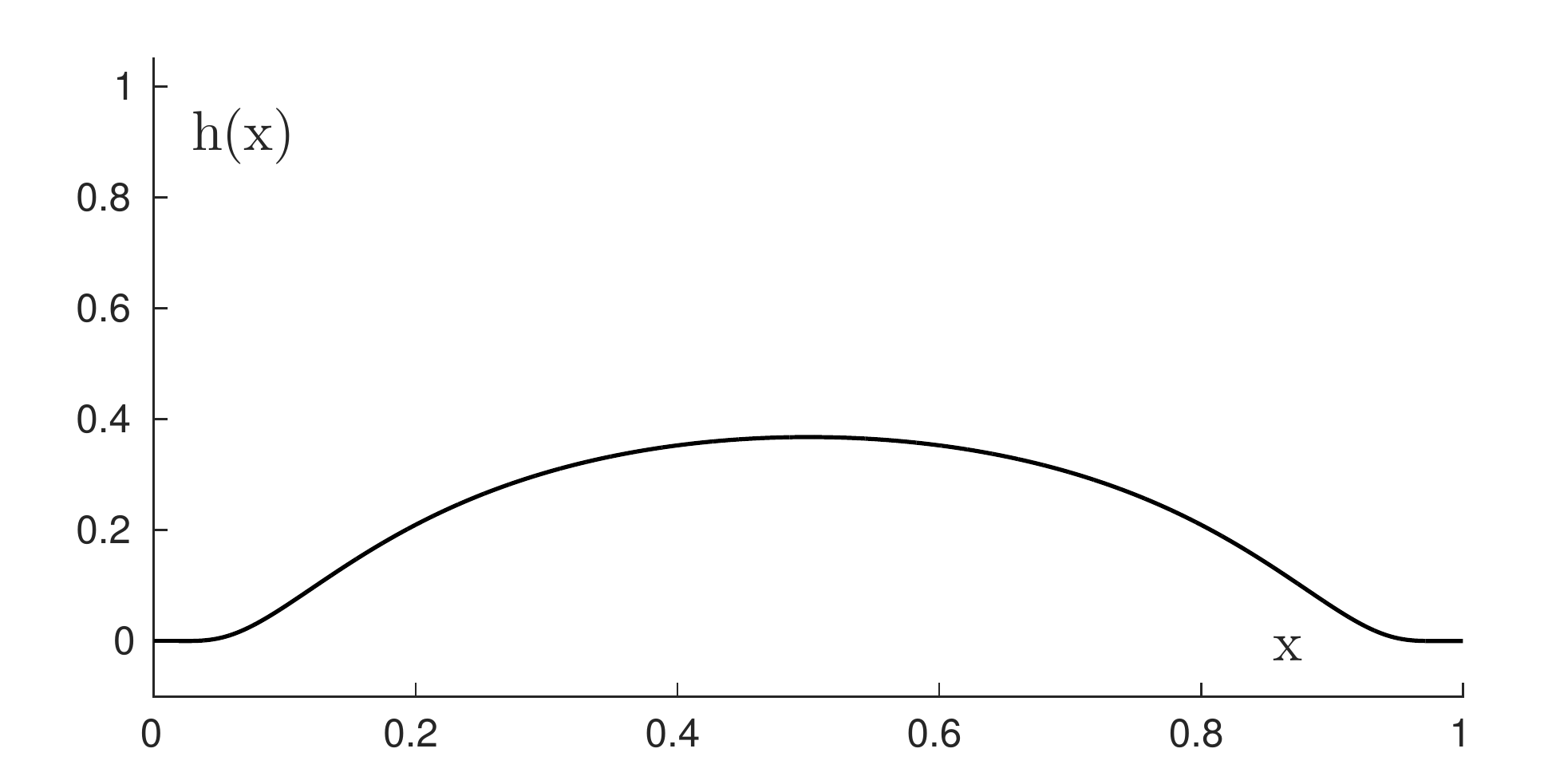}
 \end{minipage}
 \begin{minipage}{.49\linewidth}
 \includegraphics[width=\linewidth]{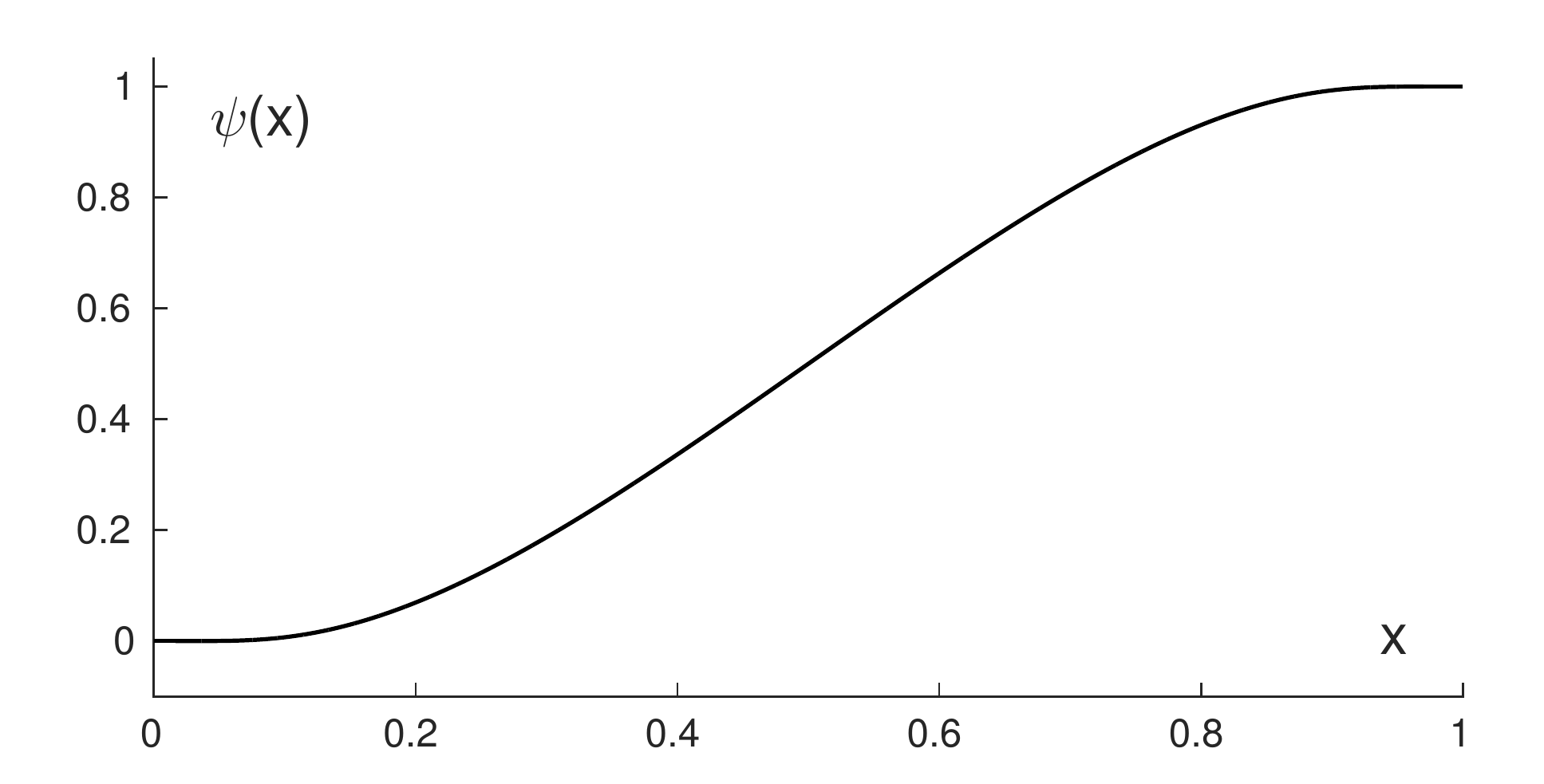}
 \end{minipage}
\end{figure}

Given such a function, 
the mapping
$$
 \Psi:\IR^d\to\IR^d, \quad
 \Psi(\mathbf x)=(\psi(x_1),\hdots,\psi(x_d))
$$ 
is a diffeomorphism of $(0,1)^d$ 
with inverse
$$
 \Psi^{-1}(\mathbf x)=(\psi^{-1}(x_1),\hdots,\psi^{-1}(x_d))^{\top}
$$ 
and Jacobian
$$
 |\diff \Psi(\mathbf x)| 
 =\prod\limits_{i=1}^{d}\psi'(x_i).
$$

If $A_n$ is any linear and deterministic quadrature formula 
with nodes $\mathbf x^{(j)}\in[0,1]^d$ and weights $a_j\in\IR$ for 
$j=1,\hdots,n$ we define the transformed quadrature formula 
$\widetilde{A}_n$ by choosing the new nodes and weights
\[
\widetilde{\mathbf x}^{(j)}=\Psi\brackets{\mathbf x^{(j)}} \text{\ \ \ \ and \ \ \ \ } 
\widetilde{a_j}=a_j \abs{\diff \Psi(\mathbf x^{(j)})}.
\]
Thus, $\widetilde{Q}_B^{\mathbf v}$ for $\mathbf v\in\IR^d$ and 
invertible $B\in\IR^{d\times d}$ takes the form
\[
\widetilde{Q}_B^{\mathbf v}(f)=\frac{1}{\abs{\det B}}
\sum\limits_{\mathbf m\in\IZ^d} f\left(\Psi\left(B^{-\top}(\mathbf m+\mathbf v)\right)\right)
\left|\diff \Psi\left(B^{-\top}(\mathbf m+\mathbf v)\right)\right|
\]
for any input function $f:[0,1]^d\to \IR$. Note that
the Jacobian is zero for any
$\mathbf m\in\IZ^d$ with $B^{-\top}(\mathbf m+\mathbf v)\not\in [0,1]^d$.
We now define the algorithm from Theorem~\ref{thm:main theorem}
in the introduction of this section.

\begin{alg}
 \label{alg:final algorithm}
 Let $B\in \IR^{d\times d}$ be a Frolov matrix
 and let $c= 2\left(\Vert B\Vert_1+1\right)^d$.
 For any $n\in\IN$ with $n\geq c$, we consider the randomized algorithm
 $$
  A_n=\widetilde{Q}_{(n/c)^{1/d}UB}^{\mathbf v},
 $$
 see Algorithm~\ref{alg:basic algorithm},
 where $U$ is a diagonal matrix whose diagonal entries
 are independent and uniformly distributed in $[1,2^{1/d}]$
 and $\mathbf v$ is independent of $U$ and uniformly
 distributed in $[0,1]^d$.
\end{alg}

By Lemma~\ref{anlemma}, the number of function values 
that this algorithm uses is bounded by $n$
for any $f:[0,1]^d\to \IR$.
We turn to the proof of Theorem~\ref{thm:main theorem}.

\begin{proof}[Proof of Theorem~\ref{thm:main theorem}]
For any $f\in L^1([0,1]^d)$, we define
$f_0=(f\circ \Psi) |\diff \Psi|$.
By the change of variables theorem, 
this function is integrable on $[0,1]^d$.
It satisfies
\[
S_d(f)=S_d(f_0)\quad\text{and}\quad 
\widetilde{Q}_{(n/c)^{1/d}UB}^{\mathbf v}(f) = Q_{(n/c)^{1/d}UB}^{\mathbf v}(f_0).
\]

\noindent
\emph{Part~1}.
Proposition~\ref{Munbiased} yields, for any $f\in L^1([0,1]^d)$,
\[
\IE \left(\widetilde{Q}_{(n/c)^{1/d}UB}^{\mathbf v}(f)\right) 
= \IE \left( Q_{(n/c)^{1/d}UB}^{\mathbf v}(f_0)\right) = S_d(f_0) = S_d(f).
\]

\noindent
\emph{Part~2}.
Since $\psi'$ vanishes outside $(0,1)$, 
all derivatives of $f_0$ vanish at the boundary
of $[0,1]^d$.
This implies $f_0\in {\mathring{H}^r_{\mix}([0,1]^d)}$ 
for all $f\in{H^r_{\rm mix}([0,1]^d)}$. By
Theorem~\ref{mixthm2},
\begin{multline*}
\IE \abs{\widetilde{Q}_{(n/c)^{1/d}UB}^{\mathbf v}(f)-S_d(f)}^2
= \IE \abs{ Q_{(n/c)^{1/d}UB}^{\mathbf v}(f_0)-S_d(f_0)}^2\\
\leq c_r^2\, n^{-2r-1} \Xnorm{f_0}{H^r_{\rm mix}([0,1]^d)}^2
\end{multline*}
and Theorem~\ref{mixthmworstcase} yields that
\begin{align*}
\sup\limits_{U,\mathbf v}
\abs{\widetilde{Q}_{(n/c)^{1/d}UB}^{\mathbf v}(f)-S_d(f)}
&= \sup\limits_{U,\mathbf v} 
\abs{ Q_{(n/c)^{1/d}UB}^{\mathbf v}(f_0)-S_d(f_0)}\\
&\leq c_r n^{-r} (\ln n)^\frac{d-1}{2} 
\Xnorm{f_0}{H^r_{\rm mix}([0,1]^d)}
,\end{align*}
if $c_r>0$ is the maximum of the constants of these theorems.
It remains to show that 
there is some $C_r>0$ such that every 
$f\in{H^r_{\rm mix}([0,1]^d)}$ 
satisfies 
$$
\Xnorm{f_0}{H^r_{\rm mix}([0,1]^d)}
\leq C_r\, \Xnorm{f}{H^r_{\rm mix}([0,1]^d)}.
$$
This is proven as follows.
The partial derivatives of $f_0$ take the form
\[
\diff^\alpha f_0(\mathbf x)=\frac{\partial^{\abs{\alpha}}}{\partial 
x_1^{\alpha_1}\cdots\partial x_d^{\alpha_d}}\ f(\Psi(\mathbf x))
\prod\limits_{i=1}^{d}\psi'(x_i) 
= \sum\limits_{\beta_1,\hdots,\beta_d=0}^{\alpha_1,\hdots,\alpha_d} 
\diff^\beta f(\Psi(\mathbf x)) R_{\alpha,\beta}(\mathbf x)
\]
for $\alpha\in\{0,1,\hdots,r\}^d$, where $R_{\alpha,\beta}(\mathbf x)$ is a finite 
sum of finite products of terms $\psi^{(j)}(x_i)$ with $i\in\{1,\hdots,d\}, 
j\in\{1,\hdots,rd+1\}$ and does not depend on $f$. It is therefore continuous 
and bounded by some $c_{\alpha,\beta}>0$.
We get
\begin{multline*}
 \Xnorm{\diff^\alpha f_0}{L^2([0,1]^d)}^2
 \leq \left(\sum\limits_{\beta_1,\hdots,\beta_d=0}^{\alpha_1,\hdots,\alpha_d} 
 \Xnorm{(\diff^\beta f \circ \Psi) \cdot R_{\alpha,\beta}}{L^2([0,1]^d)}\right)^2\\
 \leq \left(\sum\limits_{\beta_1,\hdots,\beta_d=0}^{\alpha_1,\hdots,\alpha_d} 
 c_{\alpha,\beta} \cdot\Xnorm{\diff^\beta f \circ \Psi}{L^2([0,1]^d)}\right)^2.
\end{multline*}
We proceed with H\"older's inequality and
the change of variables theorem for the
diffeomorphism $\Psi^{-1}$ of $(0,1)^d$ and
obtain the upper bound
\begin{multline*}
(r+1)^d\sum\limits_{\beta_1,\hdots,\beta_d=0}^{\alpha_1,
\hdots,\alpha_d} c_{\alpha,\beta}^2 \Xnorm{\diff^\beta f \circ \Psi}{
L^2([0,1]^d)}^2\\
= (r+1)^d\sum\limits_{\beta_1,\hdots,\beta_d=0}^{\alpha_1,\hdots,\alpha_d} 
c_{\alpha,\beta}^2 \int_{(0,1)^d} 
\diff^\beta f (\mathbf x)^2 
|\diff \Psi^{-1}(\mathbf x)| \, \d~\mathbf x\\
\leq (r+1)^d\sup_{\mathbf x\in(0,1)^d} |\diff \Psi^{-1}(\mathbf x)| 
\sum\limits_{\beta_1,\hdots,\beta_d=0}^{\alpha_1,\hdots,\alpha_d} 
c_{\alpha,\beta}^2 
\Xnorm{\diff^\beta f}{L^2([0,1]^d)}^2.
\end{multline*}
Summing over all $\alpha\in\{0,\hdots,r\}^d$
yields the desired estimate.

\noindent
\emph{Part~3}.
This is proven in the exact same manner, where we use
Theorem~\ref{isothm2} and Theorem~\ref{isothmworstcase}
instead of Theorem~\ref{mixthm2} and Theorem~\ref{mixthmworstcase}.
\end{proof}

We finish this section by showing how to 
arrive at Corollary~\ref{cor:order of convergence}.
Obviously,
Theorem~\ref{thm:main theorem} implies that
$$
 \e(n,\mathcal{P}_d^r)\preccurlyeq n^{-r-1/2}.
$$
On the other hand, it is proven in \cite{Ba59} that
$$
 \e(n,\mathcal{P}_1^r)\succcurlyeq n^{-r-1/2}.
$$
Moreover, we note that the function $f_d:[0,1]^d\to\IR$ with 
$f_d(\mathbf{x})=f_1(x_1)$
is contained in $F_d^r$ for any $f_1\in F_1^r$
and has the same integral.
If $A^d$ is a randomized algorithm on $F_d^r$,
we can define a
randomized algorithm $A^1$ on $F_1^r$ via
$A^1(f_1)=A^d(f_d)$.
The cost and error of $A^1$ are bounded above
by the cost and error of $A^d$.
This yields the relation
$$
 \e(n,\mathcal{P}_1^r) \leq \e(n,\mathcal{P}_d^r),
$$
which proves the corollary.

\section{Tensor Product Problems}
\label{sec:tensorproduct}

Let $H$ and $G$ be Hilbert spaces and 
let $S:H\to G$ be a compact linear operator.
Let $F$ be the unit ball of $H$.
The problem 
$$
 \mathcal{P}=\mathcal{P}[S,F,G,\lall,\mathrm{det},\mathrm{wc}]
$$ 
of approximating $S$ with deterministic algorithms based on $\lall$
in the worst case setting
is linear and was discussed in Section~\ref{sec:LPs}.
We study the corresponding problem
$$
 \mathcal{P}_d=\mathcal{P}[S_d,F_d,G_d,\lall,\mathrm{det},\mathrm{wc}],
$$
of approximating the $d^{\rm th}$ tensor product operator $S_d$.
This problem is linear as well.
See Section~\ref{sec:tensorproduct setting} for a more
detailed description of the problem.

The difficulty of the $d$-dimensional
problem $\mathcal{P}_d$ is completely determined
by the difficulty of the 1-dimensional problem $\mathcal{P}$.
In Section~\ref{asymptoticssection} we study the asymptotic behavior
of the $n^{\rm th}$ minimal error $\e(n,\mathcal{P}_d)$ for $n\to\infty$
based on the respective behavior of $\e(n,\mathcal{P})$.
In Section~\ref{preasymptoticssection} we do likewise
for the preasymptotic behavior of the minimal error.
Section~\ref{applicationssection} contains several examples.
The preasymptotic estimates also lead to
a tractability result in Section~\ref{tracsection}.

\begin{rem}
 It follows from Theorem~\ref{thm:LPs over Hilbert spaces} that
 the $n^{\rm th}$ minimal error satisfies
 $$
 \e(n,\mathcal{P}_d) = \inf\set{ \norm{S_d-A_n} \mid A_n\colon H_d\to G_d 
 \text{ linear, } \rank(A_n) \leq n}.
 $$
 This means that the error conincides with
 the $(n+1)^{\rm st}$ approximation number
 and all other $s$-numbers
 of the solution operator, see also~\cite[Section~11.3]{Pi78}
\end{rem}

\subsection{The Setting}
\label{sec:tensorproduct setting}

Let $H$ and $G$ be Hilbert spaces and 
let $S:H\to G$ be a compact linear operator.
Let $F$ be the unit ball of $H$.
We consider the problem 
$$
 \mathcal{P}=\mathcal{P}[S,F,G,\lall,\mathrm{det},\mathrm{wc}]
$$ 
of approximating $S$ with deterministic algorithms based on $\lall$
in the worst case setting.
From Section~\ref{sec:LPs},
we know that linear algorithms are optimal for this problem and that
optimal linear algorithms are given by the singular value decomposition
of $S$ in the following way,
see also \cite[Section~5.2]{NW08}.

Since $W=S^*S\in\mathcal{L}(H)$ is positive and compact,
it admits a finite or countable orthonormal basis $\mathcal{B}$
of $\ker(S)^\perp$ consisting of eigenvectors
$b\in\mathcal{B}$ to eigenvalues
$$
 \lambda(b) = \Xscalar{Wb}{b}{H}= \Xnorm{S b}{G}^2 > 0.
$$
We will refer to $\mathcal{B}$ as the orthonormal basis associated with $S$.
It can be characterized as the orthonormal basis of $\ker(S)^\perp$ 
whose image is an orthogonal basis of $\overline{S(H)}$.
It is unique up to the choice of orthonormal bases 
in the finite-dimensional eigenspaces of $W$.
We have
$$
 Sf = \sum_{b\in\mathcal{B}} \Xscalar{f}{b}{H} Sb
$$
for all $f\in H$. This representation is
called the \emph{singular value decomposition} or 
\emph{Schmidt decomposition} of $S$.
The square-roots of the eigenvalues of $W$ are called 
\emph{singular values} of $S$.
Let $\sigma_n$ be the $n^{\rm th}$ largest singular value of $S$
for all $n\leq\abs{\mathcal{B}}$. For $n>\abs{\mathcal{B}}$, let $\sigma_n=0$.
We consider the linear algorithm
$$
 A_n:F\to G, \quad
 A_n(f)= \sum_{b\in\mathcal{B}(n)} \Xscalar{f}{b}{H} Sb,
$$
where $\mathcal{B}(n)$ consists of all $b\in\mathcal{B}$ 
that satisfy $\Xnorm{S b}{G}>\sigma_{n+1}$.
We know that $A_n$ is optimal among all algorithms with cost $n$ or less,
see Theorem~\ref{thm:LPs over Hilbert spaces}.
It satisfies
$$
 \err(A_n)=\e(n,\mathcal{P})=\sigma_{n+1}.
$$
Moreover, we can easily verify the relation
\begin{equation}
\label{minmax}
 \sigma_{n+1} = \min\limits_{\substack{V\subset H\\ \dim(V)\leq n}}\, 
 \max\limits_{\substack{f\perp V\\ \Xnorm{f}{H}=1}} \Xnorm{Sf}{G},
\end{equation}
where equality is obtained for $V=\vspan(\mathcal{B}(n))$
and $f=b_{n+1}$.

We are concerned with tensor product problems, 
defined as follows.
Let $D$ be a set and let $\IK\in\set{\IR,\IC}$.
Let $D_d$ be the $d$-fold Cartesian product of $D$.
The tensor product of $\IK$-valued functions 
$f_1,\hdots,f_d$ on $D$ is the function
$$
 f_1\otimes\hdots\otimes f_d:\quad 
 D_d \to \IK, \quad \mathbf x \mapsto f_1(x_1)\cdot\hdots\cdot f_d(x_d).
$$
If $H$ is a Hilbert space of $\IK$-valued functions on $D$,
its $d^{\rm th}$ tensor product $H_d$ is the smallest Hilbert space of 
$\IK$-valued functions on $D_d$
that contains any tensor product of functions in $H$ and satisfies
$$
 \scalar{f_1\otimes\hdots\otimes f_d}{g_1\otimes\hdots\otimes g_d}
 = \scalar{f_1}{g_1}\cdot\hdots\cdot\scalar{f_d}{g_d}
$$
for any choice of functions $f_1,\hdots,f_d$ and $g_1,\hdots,g_d$ in $H$.
Let $G$ be another Hilbert space of $\IK$-valued functions 
with tensor product $G_d$
and let $S\in\mathcal{L}(H,G)$.
The $d^{\rm th}$ tensor product of $S$ is the unique 
operator $S_d\in\mathcal{L}(H_d,G_d)$ that satisfies
$$
 S_d\brackets{f_1\otimes\hdots\otimes f_d} = Sf_1\otimes\hdots\otimes Sf_d
$$
for any choice of functions $f_1,\hdots,f_d$ in $H$. 
If $S$ is compact, then so is $S_d$.
Finally, the $d^{\rm th}$ \emph{tensor product problem} is the problem
$$
 \mathcal{P}_d=\mathcal{P}[S_d,F_d,G_d,\lall,\mathrm{det},\mathrm{wc}],
$$
where $F_d$ is the unit ball of $H_d$.

Just like for the 1-dimensional problem $\mathcal{P}$,
optimal algorithms for $\mathcal{P}_d$ are linear and 
given by the singular value
decomposition of $S_d$.
Based on the singular value decomposition of $S$,
we easily obtain the singular value decomposition of $S_d$.
If $\mathcal{B}$ is the orthonormal basis associated with $S$, then
$$
 \mathcal{B}_d = \set{b_1\otimes\hdots\otimes b_d 
 \mid b_1,\hdots,b_d\in\mathcal{B}}
$$
is the orthonormal basis associated with the tensor product $S_d$.
In particular, the family of singular values of $S_d$ 
is given by
$$
 \sigma(\mathbf{n})=\sigma_{n_1}\cdot\hdots\cdot \sigma_{n_d}
 \quad\text{for}\quad \mathbf{n}\in\IN^d.
$$
Recall that $\e(n,\mathcal{P}_d)$ coincides with the
$(n+1)^{\rm st}$ largest singular value of $S_d$.
In particular,
$$
 \comp(\varepsilon,\mathcal{P}_d)
 =\card\set{\mathbf{n}\in\IN^d \mid \sigma(\mathbf{n})> \varepsilon}
$$
for all $\varepsilon\geq 0$.
The question for the difficulty of the
tensor product problem is thus of combinatorial nature.

\subsection{Asymptotic Behavior}
\label{asymptoticssection}

A classical result of Babenko \cite{Ba60} and Mityagin \cite{Mi62} 
is concerned with
the speed of decay of the $n^{\rm th}$ minimal error:

\begin{thm}[\cite{Ba60,Mi62}]
\label{babenko mityagin theorem}
 Let $\mathcal{P}_d$ be a tensor product problem
 as defined in Section~\ref{sec:tensorproduct setting}. 
 For any $r>0$ the following holds:
 \begin{itemize}
  \item[(i)] If\, $\e(n,\mathcal{P}) \preccurlyeq n^{-r}$,\, 
  then\, $\e(n,\mathcal{P}_d) \preccurlyeq  n^{-r}\brackets{\ln n}^{r(d-1)}$.
  \item[(ii)] If\, $\e(n,\mathcal{P}) \succcurlyeq n^{-r}$,\, 
  then\, $\e(n,\mathcal{P}_d) \succcurlyeq  n^{-r}\brackets{\ln n}^{r(d-1)}$. 
 \end{itemize}
\end{thm}

Of course, other decay assumptions on $\e(n,\mathcal{P})$ may be of interest.
For instance, Pietsch~\cite{Pi82} and K\"onig~\cite{Ko84} 
study the decay of $\e(n,\mathcal{P}_d)$
if $\e(n,\mathcal{P})$ lies in the 
Lorentz sequence space $\ell_{p,q}$ for positive indices $p$ and $q$,
which is a stronger assumption than $(i)$ for $r=1/p$, 
but weaker than $(i)$ for any $r>1/p$.
However, we are motivated by the example of Sobolev embeddings,
see Section~\ref{applicationssection}.
We will hence stick to the assumptions of 
Theorem~\ref{babenko mityagin theorem}.
However, this theorem does not provide explicit estimates for $\e(n,\mathcal{P}_d)$, 
even if $n$ is huge.
This is because of the constants hidden in the notation.
But Theorem~\ref{babenko mityagin theorem} can be sharpened.

\begin{thm}[\cite{KSU15,Kr18}]
\label{asymptotic theorem}
 Let $\mathcal{P}_d$ be a tensor product problem
 as defined in Section~\ref{sec:tensorproduct setting}. 
 For any $c>0$ and $r>0$, the following holds:
 \begin{itemize}
  \item[(i)] If\, $\e(n,\mathcal{P}) \varlesssim c\, n^{-r}$,\, 
  then\, $\e(n,\mathcal{P}_d) \varlesssim 
  \frac{c^d}{{(d-1)!}^r}\, n^{-r}\brackets{\ln n}^{r(d-1)}$.
  \item[(ii)] If\, $\e(n,\mathcal{P}) \vargtrsim c\, n^{-r}$,\,
  then\, $\e(n,\mathcal{P}_d) \vargtrsim 
  \frac{c^d}{{(d-1)!}^r}\, n^{-r}\brackets{\ln n}^{r(d-1)}$. 
 \end{itemize}
\end{thm}

In particular, we obtain that asymptotic equality 
$\e(n,\mathcal{P}) \sim c\, n^{-r}$ implies 
asymptotic equality
$\e(n,\mathcal{P}_d) \sim 
\frac{c^d}{{(d-1)!}^r}\, n^{-r}\brackets{\ln n}^{r(d-1)}$
for the tensor product problem.
Theorem~\ref{asymptotic theorem} is due to Theorem~4.3 in \cite{KSU15}. 
There, Kühn, Sickel, and Ullrich prove
this asymptotic equality in the special case
where $S_d$ is the embedding of the mixed order Sobolev space
$H^{r}_{\rm mix}(\mathbb{T}^d)$ on the $d$-torus 
$\mathbb{T}^d=[0,2\pi]^d$.
The general statement can be deduced from this special case 
with the help of their Lemma~4.14.
However, we prefer to give a direct proof by generalizing
the proof of Theorem~4.3 in \cite{KSU15}.

For the proof,
it will be essential to study the asymptotics of the cardinalities
\begin{equation}
\label{cardinality def}
 K_N(R,l)
 =\card\Big\{\mathbf{n}\in\set{N,N+1,\hdots}^l \,\Big\vert\, 
 \prod\nolimits_{j=1}^l n_j \leq R\Big\}
\end{equation}
for $l\in\set{1,\hdots,d}$ and $N\in\IN$ as $R\to \infty$.
In \cite[Lemma 3.2]{KSU15} it is shown that
\begin{equation}
\label{cardinalityasymptoticsexplicitformula}
 R \brackets{\frac{\brackets{\ln\frac{R}{2^l}}^{l-1}}{\brackets{l-1}!} 
 - \frac{\brackets{\ln\frac{R}{2^l}}^{l-2}}{\brackets{l-2}!}}
 \leq K_2(R,l) \leq
 R \frac{\brackets{\ln R}^{l-1}}{\brackets{l-1}!}
\end{equation}
for $l\geq 2$ and $R\in\{4^l,4^l+1,\hdots\}$, 
see also \cite[Theorem~3.4]{CD16}. Consequently we have
\begin{equation}
\label{cardinalityasymptoticsformula}
 \lim\limits_{R\to\infty} \frac{K_N(R,l)}{R \brackets{\ln R}^{l-1}} 
 = \frac{1}{\brackets{l-1}!}
\end{equation}
for $N=2$. In fact, (\ref{cardinalityasymptoticsformula}) holds true 
for any $N\in\IN$.
This can be derived from the case $N=2$,
but for the reader's convenience we give a complete proof.

\begin{lemma}
 \label{cardinalityasymptotics}
 \begin{equation*}
  \lim\limits_{R\to\infty} \frac{K_N(R,l)}{R \brackets{\ln R}^{l-1}} 
  = \frac{1}{\brackets{l-1}!}.
 \end{equation*}
\end{lemma}

\begin{proof}
 Note that for all values of the parameters,
 \begin{equation*}
  K_N(R,l+1) = \sum\limits_{k=N}^\infty K_N\brackets{\frac{R}{k},l},
 \end{equation*}
 where $K_N\brackets{\frac{R}{k},l}=0$ for $k > \frac{R}{N^l}$.
 This allows a proof by induction on $l\in\IN$.
 Like in estimate (\ref{cardinalityasymptoticsexplicitformula}), 
 we first show that
 \begin{equation}
 \label{cardinalitysharpupperbound}
   K_2(R,l) \leq R \frac{\brackets{\ln R}^{l-1}}{\brackets{l-1}!}
 \end{equation}
 for any $l\in\IN$ and $R\geq 1$. This is obviously true for $l=1$.
 On the other hand, if this relation holds for some $l\in\IN$ 
 and if $R\geq 1$, then
 \begin{equation*}
  \begin{split}
   &K_2(R,l+1)
   = \sum\limits_{k=2}^{\left\lfloor R\right\rfloor} 
   K_2\brackets{\frac{R}{k},l}
   \leq \sum\limits_{k=2}^{\left\lfloor R \right\rfloor} 
   \frac{R \brackets{\ln \frac{R}{k}}^{l-1}}{k \brackets{l-1}!}\\
   &\leq \frac{R}{\brackets{l-1}!} 
   \int_1^R \frac{\brackets{\ln \frac{R}{x}}^{l-1}}{x} ~\d x
   = \frac{R}{\brackets{l-1}!} 
   \left[-\frac{1}{l}\brackets{\ln \frac{R}{x}}^l \right]_1^R
   = R \frac{\brackets{\ln R}^l}{l!}
  \end{split}
 \end{equation*}
 and (\ref{cardinalitysharpupperbound}) is proven. In particular, we have
 \begin{equation}
 \label{cardinalitylimsup}
  \limsup\limits_{R\to\infty} \frac{K_N(R,l)}{R \brackets{\ln R}^{l-1}} 
  \leq \frac{1}{\brackets{l-1}!}
 \end{equation}
 for $l\in\IN$ and $N=2$. Clearly, the same holds for $N\geq 2$, 
 since $K_N(R,l)$ is decreasing in $N$.
 Relation (\ref{cardinalitylimsup}) for $N=1$ follows from the case $N=2$ 
 by the identity
 \begin{equation*}\begin{split}
  K_1(R,l) 
  &= \sum\limits_{m=0}^l \card\set{\mathbf{n}\in\IN^l
  \,\Big\vert\, \card\set{1\leq j\leq l \mid n_j\neq 1}=m \land 
  \prod\nolimits_{j=1}^d n_j \leq R}\\
  &= \mathbf{1}_{R\geq 1} + \sum\limits_{m=1}^l \binom{l}{m}\cdot K_2(R,m)
 .\end{split}\end{equation*}
 It remains to prove
 \begin{equation}
 \label{cardinalityliminf}
  \liminf\limits_{R\to\infty} 
  \frac{K_N(R,l)}{R \brackets{\ln R}^{l-1}} \geq \frac{1}{\brackets{l-1}!}
 \end{equation}
 for $N\in\IN$ and $l\in\IN$. Again, this is obvious for $l=1$.
 Suppose, (\ref{cardinalityliminf}) holds for some $l\in\IN$ and let $b<1$.
 Then there is some $R_0\geq 1$ such that
 \begin{equation*}
  K_N(R,l) \geq b R \frac{\brackets{\ln R}^{l-1}}{\brackets{l-1}!}
 \end{equation*}
 for all $R\geq R_0$ and hence
 \begin{equation*}
  \begin{split}
   &K_N(R,l+1)
   \geq \sum\limits_{k=N}^{\left\lfloor R/R_0\right\rfloor} 
   K_N\brackets{\frac{R}{k},l}
   \geq \sum\limits_{k=N}^{\left\lfloor R/R_0\right\rfloor} 
   \frac{b R \brackets{\ln \frac{R}{k}}^{l-1}}{k \brackets{l-1}!}\\
   &\geq \frac{b R}{\brackets{l-1}!} \int_N^{\frac{R}{R_0}} 
   \frac{\brackets{\ln \frac{R}{x}}^{l-1}}{x} ~\d x
   = \frac{b R}{l!} 
   \brackets{\brackets{\ln \frac{R}{N}}^l- \brackets{\ln R_0}^l}
   \geq b^2 R \frac{\brackets{\ln R}^l}{l!} 
  \end{split}
 \end{equation*}
 for large $R$. Since this is true for any $b<1$, the induction step is complete.
\end{proof}

We turn to the proof of Theorem~\ref{asymptotic theorem}.

\begin{proof}[Proof of Theorem~\ref{asymptotic theorem}]
 We first realize that changing the singular numbers $\sigma_n$
 by a multiplicative constant $c$ for all $n\in\IN$ 
 changes $\e(n,\mathcal{P}_d)$ by the the factor $c^d$.
 Moreover, raising the singular numbers to some fixed power
 changes $\e(n,\mathcal{P}_d)$
 by the same power.
 We can hence assume without loss of generality that $\sigma_1=1$ and $r=1$.
 
 Proof of $(i)$: Let $c_3>c_2>c_1>c$.
 Since $\sigma_n \varlesssim c\, n^{-r}$,
 there is some $N\in\IN$ such that for any $n\geq N$ we have
 \begin{equation}
 \label{asymptotictheoremupperbounddim1}
  \sigma_n\leq c_1\, n^{-1}.
 \end{equation}
 We want to prove
 \begin{equation}
 \label{asymptotictheoremupperbound}
  \limsup\limits_{n\to\infty}  
  \frac{\e(n,\mathcal{P}_d)\,n}{\brackets{\ln n}^{d-1}} \leq \frac{c^d}{(d-1)!}.
 \end{equation}
 Since $n/\brackets{\ln n}^{d-1}$ is eventually increasing,
 instead of giving an upper bound for $\e(n,\mathcal{P}_d)$ in terms of $n$,
 we can just as well give an upper bound for $n$ 
 in terms of $\e(n,\mathcal{P}_d)$ to obtain (\ref{asymptotictheoremupperbound}).
 Clearly, there are at least $n+1$ singular values
 of $S_d$ greater than or equal 
 to $\e(n,\mathcal{P}_d)$ and hence
 \begin{equation*}
  \begin{split}
   n &\leq \card\set{\mathbf{n}\in\IN^d \mid \sigma(\mathbf{n}) 
   \geq \e(n,\mathcal{P}_d)}\\
   &= \sum\limits_{l=0}^d \card\set{\mathbf{n}\in\IN^d
   \,\big\vert\, \card\set{1\leq j \leq d \mid n_j\geq N}=l \land \sigma(\mathbf{n}) 
   \geq \e(n,\mathcal{P}_d)}\\
   &\overset{\sigma_1=1}{\leq} \sum\limits_{l=0}^d \binom{d}{l} N^{d-l}\, 
   \card\set{\mathbf{n}\in\set{N,N+1,\hdots}^l
   \,\big\vert\, \sigma(\mathbf{n}) \geq \e(n,\mathcal{P}_d)}
  .\end{split}
 \end{equation*}
 For every $\mathbf{n}$ in the latter set, 
 relation~(\ref{asymptotictheoremupperbounddim1}) implies
 that $\prod_{j=1}^l n_j \leq c_1^l \e(n,\mathcal{P}_d)^{-1}$. 
 Thus,
 \begin{equation*}
   n\leq \sum\limits_{l=0}^d \binom{d}{l}\, N^{d-l}\, K_N\brackets{c_1^l\,\e(n,\mathcal{P}_d)^{-1},l}.
 \end{equation*}
 Lemma~\ref{cardinalityasymptotics} yields that,
 if $n$ and hence $c_1^l\,\e(n,\mathcal{P}_d)^{-1}$ is large enough,
 \begin{equation*}
  K_N\brackets{c_1^l\,\e(n,\mathcal{P}_d)^{-1},l}
  \leq \frac{c_2^l\,\e(n,\mathcal{P}_d)^{-1}}{\brackets{l-1}!} \brackets{\ln \brackets{c_2^l\,\e(n,\mathcal{P}_d)^{-1}}}^{l-1}
 \end{equation*}
 for $l\in\set{1,\hdots,d}$. Letting $n\to\infty$, the term for $l=d$ is dominant and hence
 \begin{equation*}
  n \leq \frac{c_3^d\,\e(n,\mathcal{P}_d)^{-1}}{\brackets{d-1}!} \brackets{\ln \brackets{c_3^d\,\e(n,\mathcal{P}_d)^{-1}}}^{d-1}
 \end{equation*}
 for large values of $n$. By the monotonicity of $n/\brackets{\ln n}^{d-1}$ we obtain
 \begin{equation*}
  \frac{\e(n,\mathcal{P}_d)\,n}{\brackets{\ln n}^{d-1}}
  \leq \frac{c_3^d}{\brackets{d-1}!} \cdot
  \brackets{\frac{\ln \brackets{c_3^d \e(n,\mathcal{P}_d)^{-1}}}{\ln \brackets{\e(n,\mathcal{P}_d)^{-1}\cdot \frac{c_3^d}{\brackets{d-1}!} \brackets{\ln \brackets{c_3^d \e(n,\mathcal{P}_d)^{-1}}}^{d-1}}}}^{d-1}
 .\end{equation*}
 The fraction in brackets tends to one as $n$ and hence $\e(n,\mathcal{P}_d)^{-1}$ tends to infinity and thus
 \begin{equation*}
  \limsup\limits_{n\to\infty} \frac{\e(n,\mathcal{P}_d)\,n}{\brackets{\ln n}^{d-1}} \leq \frac{c_3^d}{\brackets{d-1}!}
 .\end{equation*}
 Since this is true for any $c_3>c$, the proof of (\ref{asymptotictheoremupperbound}) is complete.
 
 Proof of $(ii)$: Let $0<c_3<c_2<c_1<c$. Since $\sigma_n \vargtrsim c\, n^{-r}$,
 there is some $N\in\IN$ such that for $n\geq N$,
 we have
 \begin{equation}
 \label{asymptotictheoremlowerbounddim1}
  \sigma_n\geq c_1\, n^{-1}.
 \end{equation}
 We want to prove
 \begin{equation}
 \label{asymptotictheoremlowerbound}
  \liminf\limits_{n\to\infty} \frac{\e(n,\mathcal{P}_d)\,n}{\brackets{\ln n}^{d-1}} \geq \frac{c^d}{{(d-1)!}^s}
 \end{equation}
 for any $d\in\IN$. Clearly, there are at most $n$ 
 singular values of $S_d$ greater than 
 $\e(n,\mathcal{P}_d)$ and hence
 \begin{equation*}
 \begin{split}
  n &\geq \card\set{\mathbf{n}\in\IN^d \mid \sigma(\mathbf{n}) 
  > \e(n,\mathcal{P}_d)}\\
  &\geq \card\set{\mathbf{n}\in\set{N,N+1,\hdots}^d \mid 
  \sigma(\mathbf{n}) > \e(n,\mathcal{P}_d)}
 .\end{split}\end{equation*}
 Relation (\ref{asymptotictheoremlowerbounddim1}) implies that every $\mathbf{n}\in\set{N,N+1,\hdots}^d$
 with $\prod_{j=1}^d n_j < c_1^d\, \e(n,\mathcal{P}_d)^{-1}$ is contained in the last set.
 This observation and Lemma~\ref{cardinalityasymptotics} yield that
 \begin{equation*}
  n \geq K_N\brackets{c_2^d\, \e(n,\mathcal{P}_d)^{-1},d} \geq \frac{c_3^d\, \e(n,\mathcal{P}_d)^{-1}}{\brackets{d-1}!} \brackets{\ln \brackets{c_3^d\, \e(n,\mathcal{P}_d)^{-1}}}^{d-1}
 \end{equation*}
 for sufficiently large $n$. By the monotonicity of $n/\brackets{\ln n}^{d-1}$ for large $n$ we obtain
 \begin{equation*}
  \frac{\e(n,\mathcal{P}_d)\,n}{\brackets{\ln n}^{d-1}}
  \geq \frac{c_3^d}{\brackets{d-1}!} \cdot \brackets{\frac{\ln \brackets{c_3^d \e(n,\mathcal{P}_d)^{-1}}}{\ln \brackets{\frac{c_3^d}{\brackets{d-1}!} \brackets{\ln \brackets{c_3^d \e(n,\mathcal{P}_d)^{-1}}}^{d-1} \e(n,\mathcal{P}_d)^{-1} }}}^{d-1}
 .\end{equation*}
 The fraction in brackets tends to 1 as $n$ and hence $\e(n,\mathcal{P}_d)^{-1}$ tends to infinity and thus
 \begin{equation*}
  \liminf\limits_{n\to\infty} \frac{\e(n,\mathcal{P}_d)\,n}{\brackets{\ln n}^{d-1}}
  \geq \frac{c_3^d}{\brackets{d-1}!}
 .\end{equation*}
 Since this is true for any $c_3<c$, the proof of (\ref{asymptotictheoremlowerbound}) is complete.
\end{proof}

We give an interpretation of Theorem~\ref{asymptotic theorem}.
For $r>0$ let us consider the quantities
\begin{align*}
 &C_1=\limsup\limits_{n\to\infty}\, \e(n,\mathcal{P}) n^r,
 &&c_1=\liminf\limits_{n\to\infty}\, \e(n,\mathcal{P}) n^r,\\
 &C_d=\limsup\limits_{n\to\infty} 
 \frac{\e(n,\mathcal{P}_d)\cdot n^r}{\brackets{\ln n}^{r(d-1)}},
 &&c_d=\liminf\limits_{n\to\infty} 
 \frac{\e(n,\mathcal{P}_d)\cdot n^r}{\brackets{\ln n}^{r(d-1)}}.
\end{align*}
These limits may be both infinite or zero.
They can be interpreted as asymptotic or optimal constants for the bounds
\begin{align}
\label{rateupperbound}
\e(n,\mathcal{P}_d) 
&\leq C \cdot n^{-r}\brackets{\ln n}^{r(d-1)}\quad\text{and}\\
\label{ratelowerbound}
\e(n,\mathcal{P}_d) 
&\geq c \cdot n^{-r}\brackets{\ln n}^{r(d-1)}
.\end{align}
For any $C>C_d$ respectively $c<c_d$ there is a threshold $n_0\in\IN$ such that
(\ref{rateupperbound}) respectively (\ref{ratelowerbound}) 
holds for all $n\geq n_0$,
whereas for any $C<C_d$ respectively $c>c_d$ there is no such threshold.
Note that our proof provides a possibility to track down admissible 
thresholds $n_0$
for any $C> \frac{C_1^d}{{(d-1)!}^r}$ respectively 
any $c<\frac{c_1^d}{{(d-1)!}^r}$.
Theorem~\ref{babenko mityagin theorem} states that $C_d$ is finite, 
whenever $C_1$ is finite,
whereas $c_d$ is positive, whenever $c_1$ is positive.
Theorem~\ref{asymptotic theorem} is more precise. It states that
\begin{equation}
\label{constantsrelations}
 \frac{c_1^d}{{(d-1)!}^r}
 \leq c_d \leq C_d
 \leq \frac{C_1^d}{{(d-1)!}^r}.
\end{equation}
Obviously, there must be equality in 
all the relations of \eqref{constantsrelations},
if the limit of the sequence $\e(n,\mathcal{P}) n^r$ 
for $n\to\infty$ exists,
that is, if $C_1=c_1$.
It is natural to ask, whether any of these equalities always holds true.
The answer is no, as shown by the following example.

\begin{ex}
 Consider a solution operator $S$ with singular values
 $\sigma_n=2^{-k}$ for $n\in\{2^k,\hdots,2^{k+1}-1\}$ and $k\in\IN_0$. 
 That is, the singular values decay linearly in $n$, 
 but are constant on segments of length $2^k$.
 They satisfy
 \begin{equation*}
  C_1=\limsup\limits_{n\to\infty} \sigma_{n+1} n =
  \lim\limits_{k\to\infty} 2^{-k}\cdot\brackets{2^{k+1}-2}=2
 \end{equation*}
 and
 \begin{equation*}
  c_1=\liminf\limits_{n\to\infty} \sigma_{n+1} n 
  =\lim\limits_{k\to\infty} 2^{-k}\cdot \brackets{2^k-1}=1
 .\end{equation*}
 Also the singular values $\sigma(\mathbf n)$ 
 of the tensor product operator $S_d$
 are of the form $2^{-k}$ for some $k\in\IN_0$, where
 \begin{equation*}
  \begin{split}
  &\card\set{\mathbf n\in\IN^d \mid \sigma(\mathbf n)=2^{-k}}
  = \sum\limits_{\abs{\mathbf{k}}=k} 
  \card\set{\mathbf{n}\in\IN^d\mid \sigma_{n_j}=2^{-k_j} 
  \text{ for } j=1\hdots d}\\
  &= \sum\limits_{\abs{\mathbf{k}}=k} 2^k
  = 2^k \cdot \binom{k+d-1}{d-1}
  = \frac{2^k}{(d-1)!}\cdot (k+1)\cdot\hdots\cdot(k+d-1)
  .\end{split}
 \end{equation*}
 Hence, $\e(n,\mathcal{P}_d)=2^{-k}$ for $N(k-1,d)\leq n< N(k,d)$ 
 with $N(-1,d)=0$ and
 \begin{equation*}
  N(k,d)= \sum\limits_{j=0}^k 
  \frac{2^j}{(d-1)!}\cdot (j+1)\cdot\hdots\cdot(j+d-1)
 \end{equation*}
 for $k\in\IN_0$. The monotonicity of $n / \brackets{\ln n}^{d-1}$ 
 for large $n$ implies
 \begin{equation}
  \label{limsupinfexample1}
   C_d=\limsup\limits_{n\to\infty} \frac{\e(n,\mathcal{P}_d)\cdot n}{\brackets{\ln n}^{d-1}} = \lim\limits_{k\to\infty} \frac{2^{-k}\cdot N(k,d)}{\brackets{\ln N(k,d)}^{d-1}}
  \end{equation}
  and
  \begin{equation}
  \label{limsupinfexample2}
   c_d=\liminf\limits_{n\to\infty} 
   \frac{\e(n,\mathcal{P}_d)\cdot n}{\brackets{\ln n}^{d-1}} = \lim\limits_{k\to\infty} \frac{2^{-k}\cdot N(k-1,d)}{\brackets{\ln N(k-1,d)}^{d-1}}
 .\end{equation}
 We insert the relations
 \begin{equation*}
   N(k,d) \leq \frac{(k+d)^{d-1}}{(d-1)!}\sum\limits_{j=0}^k 2^j \leq \frac{2^{k+1}\cdot(k+d)^{d-1}}{(d-1)!}
 \end{equation*}
 and
 \begin{equation*}
   N(k,d) \geq \frac{(k-l)^{d-1}}{(d-1)!}\sum\limits_{j=k-l+1}^k 2^j = \frac{2^{k+1}(k-l)^{d-1}}{(d-1)!}\brackets{1-2^{-l}}
 \end{equation*}
 for arbitrary $l\in\IN$ in (\ref{limsupinfexample1}) and (\ref{limsupinfexample2}) to obtain
  \begin{equation*}
   C_d = 2\cdot\frac{\brackets{\log_2 e}^{d-1}}{(d-1)!} \quad\quad\text{and}\quad\quad
   c_d = \frac{\brackets{\log_2 e}^{d-1}}{(d-1)!}.
 \end{equation*}
 In particular, 
 \begin{equation*}
  \frac{c_1^d}{\brackets{d-1}!}
  < c_d
  < C_d
  < \frac{C_1^d}{\brackets{d-1}!}
 \quad\text{for}\quad 
 d\neq 1.\end{equation*}
\end{ex}

More generally, one can define the tensor product $S_d$ of $d$ 
different compact operators $S^{(j)}$ between Hilbert spaces.
If the singular numbers of $S^{(j)}$ are given by 
$\sigma^{(j)}_n$ for $n\in\IN$,
then the singular numbers of $S_d$ are given by
$$
 \sigma(\mathbf n) = \prod_{j=1}^d\sigma^{(j)}_{n_j}
 \quad\text{for}\quad \mathbf n \in \IN^d.
$$
An example for $S_d$ is given by the $L^2$-embedding of
Sobolev functions on the $d$-torus 
with mixed smoothness $(r_1,\hdots,r_d)$,
where $r_j>0$ for all $j\leq n$.
It is the tensor product of the $L^2$-embeddings
of the univariate Sobolev spaces with smoothness $r_j$,
whose singular numbers are of order $n^{- r_j}$.
It is known that $\e(n,\mathcal{P}_d)$
has the order $n^{-r}\brackets{\ln n}^{r(l-1)}$ in this case
where $r$ is the minimum among all numbers $r_j$ and $l$ is its multiplicity.
This was proven by Mityagin~\cite{Mi62} for integer vectors 
$(r_1,\hdots,r_d)$ and by Nikol’skaya~\cite{Ni74} in the general case.
See \cite[pp.\,32, 36, 72]{Te86} and \cite{DTU18} for more details.
It is not hard to deduce that
the order of decay of $\e(n,\mathcal{P}_d)$ 
is at least (at most) $n^{-r}\brackets{\ln n}^{r(l-1)}$,
whenever the singular values of the factor operators
decay at least (at most) with order $n^{-r_j}$.
But in contrast to the case studied above, 
asymptotic constants of tensor product problems with different factors
are not determined by the asymptotic constants of the factor problems.

\begin{ex}
Consider solution operators $S$, $T$ and $\widetilde T$ 
with singular numbers
 \begin{equation*}
  \sigma_n=n^{-1}, \quad \mu(n)=n^{-2}, \quad \tilde\mu(n)=
  \left\{\begin{array}{lr}
        1, & \text{for } n\leq N,\\
        n^{-2}, & \text{for } n>N,
        \end{array}\right.
 \end{equation*}
for some $N\in\IN$ and all $n\in\IN$.
The tensor product $S_2$ of $S$ and $T$ has the 
singular values
\begin{equation*}
 \sigma(\mathbf n)= n_1^{-1} n_2^{-2}
 \quad\text{for}\quad \mathbf n\in\IN^2
\end{equation*}
which yields for the respective problem $\mathcal{P}_2$ 
and all $n\in\IN$ that
\begin{equation*}
 \begin{split}
  n &\leq \card\set{\mathbf n\in\IN^2 \mid 
  \sigma(\mathbf n) \geq \e(n,\mathcal{P}_2)}
  = \card\set{\mathbf n\in\IN^2 \mid {n}_1 n_2^2 \leq \e(n,\mathcal{P}_2)^{-1}}\\
  &\leq \sum\limits_{ n_2=1}^\infty \card\set{ n_1\in\IN \mid 
  n_1 \leq \e(n,\mathcal{P}_2)^{-1}  n_2^{-2}}
  \leq \e(n,\mathcal{P}_2)^{-1} \sum\limits_{ n_2=1}^\infty n_2^{-2}
  \leq 2 \e(n,\mathcal{P}_2)^{-1},
 \end{split}
\end{equation*}
and hence
\begin{equation*}
 \limsup\limits_{n\to\infty} \e(n,\mathcal{P}_2)n \leq 2.
\end{equation*}
The tensor product $\widetilde S_2$ of $S$ and $\widetilde T$ 
has the singular values
\begin{equation*}
 \tilde \sigma(\mathbf n)=
 \left\{\begin{array}{lr}
        n_1^{-1}, & \text{if } n_2\leq N,\\
        n_1^{-1} n_2^{-2}, & \text{else},
        \end{array}\right.
 \quad\text{for}\quad \mathbf n\in\IN^2.
\end{equation*}
For the respective problem $\mathcal{\widetilde P}_2$ 
and $n\in\IN$ we obtain
\begin{equation*}
\begin{split}
  n &\geq \card\set{\mathbf n\in\IN^2 \mid
  \tilde \sigma(\mathbf n) > \e(n,\mathcal{\widetilde P}_2)}
  \geq N \card\set{n_1\in\IN \mid n_1^{-1} 
  > \e(n,\mathcal{\widetilde P}_2)}\\
  &\geq N \brackets{ \e(n,\mathcal{\widetilde P}_2)^{-1} -1}
\end{split}
\end{equation*}
and thus
\begin{equation*}
 \liminf\limits_{n\to\infty} \e(n,\mathcal{\widetilde P}_2) n 
 \geq N.
\end{equation*}
Hence, matching asymptotic constants of the factor problems 
do not necessarily lead to
matching asymptotic constants of the tensor product problems.
\end{ex}

\subsection{Preasymptotic Behavior}
\label{preasymptoticssection}

Theorem \ref{asymptotic theorem} leads to a good 
understanding of the asymptotic behavior of the $n^{\rm th}$ minimal
error of $\mathcal{P}_d$ if the $n^{\rm th}$ minimal error
of $\mathcal{P}$ is of polynomial decay.
If $\e(n,\mathcal{P})$ is roughly $c\, n^{-r}$ for large $n$,
then $\e(n,\mathcal{P}_d)$ is roughly 
$c^d (d-1)!^{-r} n^{-r} (\ln n)^{r(d-1)}$
for $n$ larger than a certain threshold.
But even for modest dimensions, the size of this threshold may
go far beyond the scope of computational capabilities.
Indeed, while the minimal error decreases, 
the function $n^{-r}\brackets{\ln n}^{r(d-1)}$ grows rapidly 
as $n$ goes from 1 to $e^{d-1}$.
For this function to become less than 1, 
the number $n$ even has to be super-exponentially large with respect to
the dimension.
Thus, any estimate for the minimal error 
in terms of this function is useless
to describe its behavior in the range $n\leq 2^d$, 
its so called preasymptotic behavior.
As a replacement we present the following estimate.

\begin{thm}[\cite{Kr18}]
\label{preasymptoticstheorem1}
Let $\mathcal{P}_d$ be a tensor product problem
as defined in Section~\ref{sec:tensorproduct setting}.
Let $\sigma_1=1$ and $\sigma_2\in(0,1)$ and assume 
$\sigma_n \leq C\, n^{-r}$ for some $r,C>0$ and all $n\geq 2$.
Then
\begin{equation*}
 \sigma_2 
 \brackets{\frac{1}{n+1}}^{
 \frac{\ln \brackets{\sigma_2^{-1}}}{
 \ln\brackets{1 + \frac{d}{\log_{2}(n+1)}}
 }}
 \leq \e(n,\mathcal{P}_d) 
 \leq \brackets{\frac{\exp\brackets{C^{2/r}}}{n+1}}^{
 \frac{\ln \brackets{\sigma_2^{-1}}}{
 \ln\brackets{\sigma_2^{-2/r}\, d}
 }}
\end{equation*}
for any $n\in\{1,\hdots,2^d-1\}$.
\end{thm}

Let us assume that the dimension $d$ is large.
Then the $n^{\rm th}$ minimal error, 
which roughly decays like $n^{-r}$ for huge values of $n$,
roughly decays like $n^{-t_d}$ with 
$$
t_d=\frac{\ln \brackets{\sigma_2^{-1}}}{\ln d}
$$ for small values of $n$.
This is why we refer to $t_d$
as preasymptotic rate of the tensor product sequence.
The preasymptotic rate is much worse than the asymptotic rate.
This is not an unusual phenomenon for high-dimensional problems.
Comparable estimates for the case of
$S_d$ being the $L^2$-embedding of the mixed order Sobolev space
on the $d$-torus
are established in Theorems 4.9, 4.10, 4.17 and 4.20 of \cite{KSU15}.
See \cite{CW17,CW17b,KMU16} for other examples.

In order to obtain bounds on the 
$n^{\rm th}$ minimal error for small values of $n$,
we give explicit estimates for $K_2(R,l)$ from (\ref{cardinality def})
for $l\leq d$ and small values of $R$.
The right asymptotic behavior of these estimates 
is less important in this case.
Note that $K_2(R,l)=0$ for $R< 2^l$.

\begin{lemma}
\label{hyperboliccrosscountinglemma}
Let $R\geq 0$ and $l\in\IN$. For any $\delta>0$ we have
\begin{align*}
 &K_2(R,l) \leq \frac{R^{1+\delta}}{\delta^{l-1}}  	&&\text{and}\\
 &K_2(R,l) \geq \frac{R}{3\cdot 2^{l-1}}			&&\text{for } R\geq 2^l.
\end{align*}
\end{lemma}

\begin{proof}
 Both estimates hold in the case $l=1$, since
 \begin{equation*}
  K_2(R,1)
  = \left\{\begin{array}{lr}
        0, & \text{for } R< 2,\\
        \lfloor R \rfloor -1, & \text{for } R\geq 2
        .\end{array}\right.
 \end{equation*}
If they hold for some $l\in\IN$, then
\begin{equation*}
  K_2(R,l+1)
  = \sum\limits_{k=2}^\infty K_2\left(\frac{R}{k},l\right)
  \leq\ \frac{R^{1+\delta}}{\delta^{l-1}} 
  \sum\limits_{k=2}^\infty \frac{1}{k^{1+\delta}}
  \leq\ \frac{R^{1+\delta}}{\delta^{l-1}} \int_1^\infty \frac{1}{x^{1+\delta}} ~\d x
  =\ \frac{R^{1+\delta}}{\delta^l}
\end{equation*}
and for $R\geq 2^{l+1}$
\begin{equation*}
  K_2(R,l+1)
  \geq K_2\brackets{\frac{R}{2},l}
  \geq \frac{R/2}{3\cdot 2^{l-1}}
  = \frac{R}{3\cdot 2^l}.
\end{equation*}
We have thus proven Lemma~\ref{hyperboliccrosscountinglemma} by induction.
\end{proof}

We give a slight refinement of Theorem~\ref{preasymptoticstheorem1}.

\begin{thm}[\cite{Kr18}]
\label{preasymptoticstheorem}
Consider $\mathcal{P}_d$
as defined in Section~\ref{sec:tensorproduct setting}
with $\sigma_1>\sigma_2>0$.
\begin{itemize}
 \item[(i)] Suppose that $\sigma_n \leq C\, n^{-r}$ 
 for some $r,C>0$ and all $n\geq 2$ and let $\delta\in (0,1]$.
 For any $n\in\IN_0$,
 \begin{equation*}
  \e(n,\mathcal{P}_d) \leq \sigma_1^d
  \brackets{\frac{\tilde{C}(\delta)}{n+1}}^{\mathbf{\alpha}(d,\delta)}
  \quad\text{with}
  \end{equation*}
  \begin{equation*}                 
  \tilde{C}(\delta)=\exp{\brackets{\frac{(C/\sigma_1)^{(1+\delta)/r}}{\delta}}} 
  \quad\text{and}\quad
  \mathbf{\alpha}(d,\delta)
  =\frac{\ln(\sigma_1/\sigma_2)}{
  \ln \brackets{d (\sigma_1/\sigma_2)^{(1+\delta)/r}}
  }>0
  .\end{equation*}
 \item[(ii)] Let $v=\card\set{n \geq 2\mid \sigma_n= \sigma_2}$. 
 For any $n\in\{1,\hdots,(1+v)^d-1\}$,
 \begin{equation*}
   \e(n,\mathcal{P}_d) \geq \sigma_1^{d-1}
   \sigma_2 \brackets{\frac{1}{n+1}}^{\beta(d,n+1)}
  \quad\text{with}\quad
  \beta(d,n) 
  = \frac{\ln(\sigma_1/\sigma_2)}{
  \ln \brackets{1+\frac{v}{\log_{1+v} n}\cdot d}}
  >0.
 \end{equation*}
\end{itemize}
\end{thm}

Note that the assumption $\sigma_1>\sigma_2>0$ is in fact
the only interesting case. If $\sigma_2=\sigma_1$, we have
$\e(n,\mathcal{P}_d)=\sigma_1^d$\, for every $n< (1+v)^d$.
On the other hand, $\sigma_2=0$ implies 
$\e(n,\mathcal{P}_d)=0$ for every $n\in\IN$.

\begin{proof}
Recall that multiplying the singular numbers
with a constant factor $c$ scales the minimal errors of $\mathcal{P}_d$
with the factor $c^d$. Hence we may assume that
$\sigma_1=1$ without loss of generality.

Part $(i)$:
Let $n\in\IN$. There is some $L\geq 0$ with $\e(n-1,\mathcal{P}_d)=\sigma_2^L$.
If $\sigma(\mathbf{n}) \geq \e(n-1,\mathcal{P}_d)$, 
the number $l$ of components of $\mathbf{n}\in\IN^d$ that are
not equal to 1 is at most $\lfloor L\rfloor$ and hence
\begin{equation*}
 \begin{split}
  n &\leq \card\set{\mathbf{n}\in\IN^d \mid 
  \sigma(\mathbf{n}) \geq \e(n-1,\mathcal{P}_d)}\\
  &= \sum\limits_{l=0}^{\min\set{\lfloor L\rfloor,d}}  \card\set{\mathbf{n}\in\IN^d
  \mid \card\set{1\leq j \leq d \mid 
  n_j\neq 1}=l \land \sigma(\mathbf{n}) \geq \e(n-1,\mathcal{P}_d)}\\
  &= 1+ \sum\limits_{l=1}^{\min\set{\lfloor L\rfloor,d}} 
  \binom{d}{l} \card\set{\mathbf{n}\in\set{2,3,\hdots}^l
  \mid \sigma(\mathbf{n}) \geq \e(n-1,\mathcal{P}_d)}
 .\end{split}
\end{equation*}
Since $\sigma(\mathbf{n})\leq C^l\, \prod_{j=1}^l n_j^{-r}$ 
for $\mathbf{n}\in\set{2,3,\hdots}^l$,
Lemma~\ref{hyperboliccrosscountinglemma} yields 
for $l\leq\min\set{\lfloor L\rfloor,d}$,
\begin{equation*}
\begin{split}
  \card\set{\mathbf{n}\in\set{2,3,\hdots}^l \mid 
  \sigma(\mathbf{n}) \geq \e(n-1,\mathcal{P}_d)}
  &\leq K_2\brackets{C^{l/r} \e(n-1,\mathcal{P}_d)^{-1/r},l}\\
  &\leq C^{(1+\delta)l/r} \e(n-1,\mathcal{P}_d)^{-(1+\delta)/r} \delta^{-l}.
  \end{split}
\end{equation*}
Obviously,
\begin{equation*}
 1\leq \binom{d}{0}\cdot C^{0/r} \e(n-1,\mathcal{P}_d)^{-(1+\delta)/r}\delta^{0}.
\end{equation*}
A combination of these bounds yields
\begin{multline*}
  n \leq \sum\limits_{l=0}^{\min\set{\lfloor L\rfloor,d}} 
  \binom{d}{l}\cdot C^{(1+\delta)l/r} 
  \e(n-1,\mathcal{P}_d)^{-(1+\delta)/r} \delta^{-l} \\
  \leq \e(n-1,\mathcal{P}_d)^{-(1+\delta)/r} 
  \sum\limits_{l=0}^{\min\set{\lfloor L\rfloor,d}} 
  \frac{d^l}{l!} C^{(1+\delta)l/r} \delta^{-l}\\
  \leq \sigma_2^{-(1+\delta)L/r} d^L 
  \sum\limits_{l=0}^{\min\set{\lfloor L\rfloor,d}} 
  \frac{\brackets{\frac{C^{(1+\delta)/r}}{\delta}}^l}{l!}\\
  \leq \brackets{\sigma_2^{-(1+\delta)/r}\cdot d}^L 
  \exp{\brackets{\frac{C^{(1+\delta)/r}}{\delta}}}
\end{multline*}
and hence
\begin{equation*}
 L \geq \frac{\ln n - \frac{C^{(1+\delta)/r}}{\delta}}{
 \ln \brackets{\sigma_2^{-(1+\delta)/r}\cdot d}}.
\end{equation*}
Thus
\begin{equation}
\label{otherformofpreasymptotics1}
\e(n-1,\mathcal{P}_d)= \sigma_2^L
\leq \exp\brackets{\frac{\brackets{\frac{C^{(1+\delta)/r}}{\delta} -\ln n} 
\ln\sigma_2^{-1}}{\ln \brackets{\sigma_2^{-(1+\delta)/r}\cdot d}}}
= \brackets{\frac{\exp{\brackets{\frac{C^{(1+\delta)/r}}{\delta}}}}{n}}^{
\mathbf{\alpha}(d,\delta)}
\end{equation}
with
\begin{equation*}
 \mathbf{\alpha}(d,\delta)
 =\frac{\ln\sigma_2^{-1}}{\ln \brackets{\sigma_2^{-(1+\delta)/r}\cdot d}}.
\end{equation*}

Part $(ii)$:
Let $n\in\{2,\hdots,(1+v)^d\}$. 
Then $\sigma_2^d\leq \e(n-1,\mathcal{P}_d)\leq \sigma_2$.
If $\e(n-1,\mathcal{P}_d)$ equals $\sigma_2$, the lower bound is trivial.
Else, there is some $L\in\{1,\hdots,d-1\}$ such that $\e(n-1,\mathcal{P}_d)\in [\sigma_2^{L+1},\sigma_2^L)$.
Clearly,
\begin{equation}
\label{nononebasiclowerbound}
\begin{split}
n &> \card\set{\mathbf{n}\in\IN^d \mid \sigma(\mathbf{n})
> \e(n-1,\mathcal{P}_d) }\\
&\geq \sum_{l=1}^L \binom{d}{l} 
\card\set{\mathbf{n}\in\set{2,3,\hdots}^l \mid 
\sigma(\mathbf{n})> \e(n-1,\mathcal{P}_d)}.
\end{split}
\end{equation}
If $l\leq L$, we have $\sigma(\mathbf{n})> \e(n-1,\mathcal{P}_d)$ for every $\mathbf{n}\in\set{2,\hdots,1+v}^l$ and hence
\begin{equation}
 \label{lowerboundonn}
 n \geq \sum_{l=0}^L \binom{d}{l}\, v^l
 \geq \sum_{l=0}^L \binom{L}{l} \brackets{\frac{d}{L}}^l v^l
 =\brackets{1+\frac{vd}{L}}^L.
\end{equation}
Since $d/L$ is bigger than 1, this yields in particular that
$L \leq \log_{1+v} n$.
We insert this auxiliary estimate on $L$ in (\ref{lowerboundonn}) and get
\begin{equation*}
 n \geq \brackets{1+\frac{vd}{\log_{1+v} n}}^L,
\end{equation*}
or equivalently
\begin{equation*}
 L \leq \frac{\ln n}{\ln \brackets{1+\frac{vd}{\log_{1+v} n}}}.
\end{equation*}
We recall that $\e(n-1,\mathcal{P}_d)\geq \sigma_2^{L+1}$ and realize that the proof is finished.
\end{proof}

The bounds of Theorem~\ref{preasymptoticstheorem} are completely explicit, 
but complex.
One might be bothered by the dependence 
of the exponent in the lower bound on $n$.
This can be overcome, 
if we restrict the lower bound to the case $n < (1+v)^{d^a}$ for some $0<a<1$
and replace $\beta(d,n)$ by 
\begin{equation*}
 \tilde{\beta}(d) = \frac{\ln \sigma_2^{-1}}{\ln \brackets{1+v\cdot d^{1-a}}}.
\end{equation*}
Of course we throw away information this way.
Similarly we get a worse but still valid estimate, if we replace $v$ by 1.
Note that the lower bounds are valid for any zero sequence $\sigma$,
independent of its rate of convergence.
The additional parameter $\delta$ in the upper bound was introduced to
maximize the exponent $\mathbf{\alpha}(d,\delta)$. 
If $\delta$ tends to zero, $\mathbf{\alpha}(d,\delta)$ gets bigger,
but also the constant $\tilde C(\delta)$ explodes.


These kinds of estimates are also closely related 
to those in \cite[Section 3]{GW11}.
Using the language of generalized tractability, 
it is shown in \cite{GW11} that the supremum of all
$p>0$ such that there is a constant $c>0$ with
\begin{equation*}
 \e(n,\mathcal{P}_d)
 \leq e\sigma_1^d \brackets{\frac{c}{n+1}}^{\frac{p}{1+\ln d}}
\end{equation*}
for all $n\in\IN_0$ and $d\in\IN$ is the minimum of $r$ and $\ln(\sigma_2^{-1})$.

\subsection{Examples}
\label{applicationssection}

We apply our results to three different solution operators $S_d$.
The solution operators are embeddings of Hilbert spaces
of $d$-variate functions into $L^2$.
That is, we study the problem of $L^2$-approximation
of certain $d$-variate functions
using deterministic algorithms based on $\lall$.

\subsubsection{Approximation of Mixed Order Sobolev Functions on the Torus}

Let $\mathbb{T}$ be the 1-torus, the circle, represented by the interval $[a,b]$,
where the two end points $a<b$ are identified.
By $L^2\brackets{\mathbb{T}}$ we denote the Hilbert space of 
square-integrable, complex-valued functions on $\mathbb{T}$,
equipped with the scalar product
\begin{equation*}
 \scalar{f}{g}= \frac{1}{L} \int_{\mathbb{T}} f(x) \overline{g(x)} ~\d x
\end{equation*}
and the induced norm $\norm{\cdot}$ for some $L>0$.
Typical normalizations are $[a,b]\in\set{[0,1],[-1,1],[0,2\pi]}$ and $L\in\set{1,b-a}$.
The family $\brackets{b_{k}}_{k\in \IZ}$ with
\begin{equation*}
 b_{k}(x) = \sqrt{\frac{L}{b-a}} \exp\brackets{2\pi i k\,\frac{x-a}{b-a}}
\end{equation*}
is an orthonormal basis of $L^2\brackets{\mathbb{T}}$, its Fourier basis, and
\begin{equation*}
 \hat{f}(k) = \scalar{f}{b_{k}}
\end{equation*}
is the $k^{\rm th}$ Fourier coefficient of $f\in L^2\brackets{\mathbb{T}}$. 
By Parseval's identity,
\begin{equation*}
 \norm{f}^2 = \sum_{k\in\IZ} |\hat{f}(k)|^2 \quad \text{and} \quad \scalar{f}{g}=\sum_{k\in\IZ} \hat{f}(k)\cdot\overline{\hat{g}(k)}.
\end{equation*}
Let $w=\brackets{w_k}_{k\in\IN}$ be a nondecreasing sequence 
of nonnegative numbers with $w_0=1$
and let $w_{-k}=w_k$ for $k\in\IN$ and so let $\tilde w$. The univariate Sobolev space $H^{w}\brackets{\mathbb{T}}$
is the Hilbert space of functions $f\in L^2\brackets{\mathbb{T}}$ for which
\begin{equation*}
 \norm{f}_w^2=\sum_{k\in\IZ} w_k^2\cdot |\hat{f}(k)|^2
\end{equation*}
is finite, equipped with the scalar product
\begin{equation*}
 \scalar{f}{g}_w= \sum_{k\in\IZ} w_k\hat{f}(k)\cdot\overline{w_k\hat{g}(k)}.
\end{equation*}
Note that $H^{w}\brackets{\mathbb{T}}$ and 
$H^{\tilde{w}}\brackets{\mathbb{T}}$ coincide and their norms are equivalent,
if and only if $w_k \asymp \tilde{w}_k$. 
In case $w_k \asymp k^r$ for some $r\geq 0$, the space 
$H^{w}\brackets{\mathbb{T}}$ is the classical
Sobolev space of periodic univariate functions with 
fractional smoothness $r$, also denoted by $H^r \brackets{\mathbb{T}}$.
In particular, $H^w\brackets{\mathbb{T}}=L^2\brackets{\mathbb{T}}$ 
for $w_k\asymp 1$.

In accordance with previous notation, 
let $H=H^{w}\brackets{\mathbb{T}}$ and $G=H^{\tilde w}\brackets{\mathbb{T}}$.
The embedding $S$ of $H$ into $G$ is compact, 
if and only if $w_k / \tilde w_k$ tends to infinity as $k$ tends to infinity.
The Fourier basis $\brackets{b_{k}}_{k\in \IZ}$ is an orthogonal basis of $H$ 
consisting of eigenfunctions of $W=S^* S$
with corresponding eigenvalues
\begin{equation*}
 \lambda(b_k) = 
 \frac{\norm{b_k}_G^2}{\norm{b_k}_H^2} = \frac{\tilde w_k^2}{w_k^2}.
\end{equation*}
The $n^{\rm th}$ singular value $\sigma_n$ of this embedding 
is the square root of the $n^{\rm th}$ biggest eigenvalue.
Hence, replacing the Fourier weight sequences $w$ and $\tilde{w}$
by equivalent sequences does not affect the order of convergence of
the corresponding singular values,
but it may drastically affect their asymptotic constants 
and preasymptotic behavior.
If $G=L^2\brackets{\mathbb{T}}$, we obtain
\begin{equation*}
 \sigma_n = w_{k_n}^{-1}, \quad \text{where} \quad 
 k_n=(-1)^n\left\lfloor n/2\right\rfloor.
\end{equation*}
Note that $\sigma_1$, the norm of the embedding $S$, is always 1.

The $d^{\rm th}$ tensor product $H_d=H^w_{\rm mix}(\mathbb{T}^d)$ 
of $H$ is a space of mixed order Sobolev
functions on the $d$-torus. If $w_k\asymp k^r$ for some $r\geq 0$, 
this is the space $H^r_{\rm mix}(\mathbb{T}^d)$ of functions
with mixed smoothness $r$. If $r\in\IN_0$, this space consists of
all real-valued functions on the $d$-torus, which have a weak (or distributional) derivative of order $\mathbf{\alpha}$
in $L^2(\mathbb{T}^d)$ for any $\mathbf{\alpha}\in\set{0,1,\hdots,r}^d$. 
The same holds for the $d^{\rm th}$
tensor product $G_d=H^{\tilde w}_{\rm mix}(\mathbb{T}^d)$ of $G$.
The tensor product operator $S_d: H_d \to G_d$ is the compact embedding of 
$H^w_{\rm mix}(\mathbb{T}^d)$
into $H^{\tilde w}_{\rm mix}(\mathbb{T}^d)$.

If $\tilde w_k/w_k$ is of polynomial decay for $k\to\infty$, 
Theorem~\ref{asymptotic theorem}
and Theorem~\ref{preasymptoticstheorem} can be applied.
We formulate the results for the 
embedding of $H^r_{\rm mix}(\mathbb{T}^d)$
into $L^2(\mathbb{T}^d)$, which we denote by $\APP$. 
The space $H^r_{\rm mix}(\mathbb{T}^d)$
is equipped with different equivalent norms.
To indicate the norm, we write
$H^{r,\square,\gamma}_{\rm mix}(\mathbb{T}^d)$ with
$\square\in\set{\circ,*,+,\#}$ and $\gamma>0$.
The norms are given by the following weights.
\begin{center}
\renewcommand{\arraystretch}{2}
\begin{tabular}[h]{c|c|c|c|c}
 $\square$ & $\circ$ & $*$ & $+$ & $\#$\\
 \hline
 $w_k^2$ &
 $\sum_{l=0}^r \abs{\frac{2\pi k}{\gamma(b-a)}}^{2l}$ &
 $1+\abs{\frac{2\pi k}{\gamma(b-a)}}^{2r}$ &
 $\big(1+\abs{\frac{2\pi k}{\gamma(b-a)}}^{2}\big)^r$ &
 $\big(1+\abs{\frac{2\pi k}{\gamma(b-a)}}\big)^{2r}$
 \vspace*{2mm}
\end{tabular}
\renewcommand{\arraystretch}{1}
\end{center}
%
The last three norms are due to Kühn, Sickel and Ullrich \cite{KSU15}, 
who study all these norms for $\gamma=1$, $L=1$ and $[a,b]=[0,2\pi]$.
The last norm is also studied by Chernov and D\~ung in \cite{CD16} 
for $L=2\pi$, $[a,b]=[-\pi,\pi]$ and arbitrary values of $\gamma$.
If $r$ is a natural number, the first two scalar products take the form
\begin{equation*}
\begin{split}
 &\Xscalar{f}{g}{H^{r,\circ,\gamma}_{\rm mix}(\mathbb{T}^d)} 
 = \sum\limits_{\mathbf{\alpha}\in\set{0,\hdots,r}^d} \gamma^{-2r\abs{\mathbf{\alpha}}} 
 \scalar{\diff^\mathbf{\alpha} f}{\diff^\mathbf{\alpha} g},\\
 &\Xscalar{f}{g}{H^{r,*,\gamma}_{\rm mix}(\mathbb{T}^d)} 
 = \sum\limits_{\mathbf{\alpha}\in\set{0,r}^d} \gamma^{-2r\abs{\mathbf{\alpha}}}
 \scalar{\diff^\mathbf{\alpha} f}{\diff^\mathbf{\alpha} g}.
\end{split}
\end{equation*}
This is why $H^{r,\circ,1}_{\rm mix}(\mathbb{T}^d)$ and 
$H^{r,*,1}_{\rm mix}(\mathbb{T}^d)$ might be considered the most natural choice.
We now translate Theorems~\ref{asymptotic theorem} and \ref{preasymptoticstheorem}
for the approximation problem
$$
 \mathcal{P}_d^{r,\square,\gamma}
 =\mathcal{P}[\APP,F_d^{r,\square,\gamma},L^2,\lall,\mathrm{det},\mathrm{wc}],
$$
where $F_d^{r,\square,\gamma}$ is the unit ball of $H^{r,\square,\gamma}_{\rm mix}(\mathbb{T}^d)$.

\begin{cor}[\cite{Kr18}]
 For any $r>0$, $\gamma>0$ and $\square\in\set{\circ,*,+,\#}$,
 we have
 \begin{equation*}
  \e\brackets{n,\mathcal{P}^{r,\square,\gamma}_d}
  \sim
  \brackets{\frac{\brackets{\gamma(b-a)}^d}{\pi^d \brackets{d-1}!}}^r
  n^{-r} \brackets{\ln n}^{r(d-1)}
 .\end{equation*}
\end{cor}

This agrees with the limits
that are computed in \cite{KSU15} 
for the case $\gamma(b-a)/\pi=2$.
The limit for $\square=+$ ,
$[a,b]=[-\pi,\pi]$ and $L=2\pi$
can also be derived from \cite[Theorem~4.6]{CD16}.
The preasymptotic estimates take the following form.

\begin{cor}[\cite{Kr18}]
\label{cor:preasymptotics mixed torus}
 For any $r,\gamma>0$, $\square\in\set{\circ,*,+,\#}$,
 and $n<3^d$
 we have
 $$
  \sigma^\square_2 \brackets{\frac{1}{n+1}}^{\beta_\square(d,n+1)} \leq
  \e\brackets{n,\mathcal{P}^{r,\square,\gamma}_d}  
  \leq\, \brackets{\frac{\tilde{C}(\delta)}{n+1}}^{
  \mathbf{\alpha}_\square(d,\delta)}.
 $$
 The parameter $\delta\in (0,1]$ is arbitrary, $\tilde{C}(\delta) 
 = \exp\brackets{\brackets{3/\eta}^{1+\delta}/\delta}$
 for $\eta=\frac{2\pi}{\gamma(b-a)}$
 and the values $\sigma^\square$, $\mathbf{\alpha}_\square$
 and $\beta_\square$ are listed below.
 The upper bound holds for all $n\in\IN_0$.
\end{cor}

\begin{TAB}(5mm,1cm,1cm,1cm,1cm)[5pt]{c|c|c|c}{c|c|c|c|c}
$\square$	& $\sigma^\square_2$	& $\mathbf{\alpha}_\square(d,\delta)$		& $\beta_\square(d,n)$		\\
 $\circ$	& $\displaystyle \big(\sum_{l=0}^r \eta^{2l}\big)^{-1/2}$
		& $\displaystyle \frac{r\ln\brackets{\sum_{l=0}^r\eta^{2l}}}{2r\ln d + (1+\delta)\cdot\ln\brackets{\sum_{l=0}^r\eta^{2l}}}$
		& $\displaystyle \frac{\ln\brackets{\sum_{l=0}^r\eta^{2l}}}{2\ln\brackets{1+2d/\log_3 n}}$		\\
 $*$		& $\displaystyle \brackets{1+\eta^{2r}}^{-1/2}$
		& $\displaystyle \frac{r\ln\brackets{1+\eta^{2r}}}{2r\ln d + (1+\delta)\cdot\ln\brackets{1+\eta^{2r}}}$
		& $\displaystyle \frac{\ln\brackets{1+\eta^{2r}}}{2\ln\brackets{1+2d/\log_3 n}}$		\\
 $+$		& $\displaystyle \brackets{1+\eta^{2}}^{-r/2}$	
		& $\displaystyle \frac{r \ln\brackets{1+\eta^{2}}}{2\ln d + (1+\delta)\cdot\ln\brackets{1+\eta^{2}}}$		
		& $\displaystyle \frac{r \ln\brackets{1+\eta^{2}}}{2\ln\brackets{1+2d/\log_3 n}}$		\\
 $\#$		& $\displaystyle \brackets{1+\eta}^{-r}$	
		& $\displaystyle \frac{r \ln\brackets{1+\eta}}{\ln d + (1+\delta)\ln\brackets{1+\eta}}$		
		& $\displaystyle \frac{r \ln\brackets{1+\eta}}{\ln\brackets{1+2d/\log_3 n}}$
\end{TAB}


Let us discuss the setting of \cite{KSU15}, 
where $\gamma=1$ and $b-a=2\pi$ and hence $\eta=1$.
The exponents $\mathbf{\alpha}_\#(d,\delta)=r(\log_2 d + 1+\delta)^{-1}$ and
$\mathbf{\alpha}_+(d,\delta)=r(2\log_2 d + 1+\delta)^{-1}$ in our 
upper bounds are slightly better
than the exponents $r(\log_2 d + 2)^{-1}$ and $r(2\log_2 d + 4)^{-1}$ 
in Theorem 4.9, 4.10 and Theorem 4.17 of \cite{KSU15},
but almost the same. Also the lower bounds basically coincide.
Regarding $H^{r,*,1}_{\rm mix}(\mathbb{T}^d)$, Kühn, Sickel and 
Ullrich only studied the case $1/2\leq r\leq1$ in Theorem 4.20.
As we see now, there is a major difference between this natural norm 
and the last two norms:
For large $d$, the preasymptotic behavior of the 
singular values is roughly $n^{-t_{d,\square}}$, where
\begin{equation*}
 t_{d,\circ}=\frac{\log_2\brackets{r+1}}{2 \log_2 d}, \quad
 t_{d,*}=\frac{1}{2 \log_2 d}, \quad
 t_{d,+}=\frac{r}{2 \log_2 d}, \quad
 t_{d,\#}=\frac{2r}{2\log_2 d}.
\end{equation*}
This means that the smoothness of the space only has a minor or even 
no impact on the preasymptotic decay of the singular values
for
$H^{r,\circ,1}_{\rm mix}(\mathbb{T}^d)$ and for $H^{r,*,1}_{\rm mix}(\mathbb{T}^d)$.
This changes, however, if the value of $\eta$ 
changes. 
If we have $\eta>1$, then
also the exponents $t_{d,\circ}$ and $t_{d,*}$ get linear in $r$.
For the other two families of norms, the smoothness does show 
and the value of $\eta$ is less important.


\subsubsection{Approximation of Mixed Order Jacobi Functions on the Cube}

The above results also apply to problem of the $L^2$-approximation
of mixed order Jacobi functions on the 
$d$-cube as considered in \cite[Section 5]{CD16}.
For fixed parameters $\mathbf{\alpha},\beta>-1$ with $a:=\frac{\mathbf{\alpha}+\beta+1}{2}>0$,
the weighted $L^2$-space $G=L^2\brackets{[-1,1],w}$ is the Hilbert space of measurable,
real-valued functions on $[-1,1]$ with
\begin{equation*}
 \int_{-1}^1 f(x)^2 w(x) ~\d x < \infty,
\end{equation*}
with scalar product
\begin{equation*}
 \scalar{f}{g} = \int_{-1}^1 f(x) g(x) w(x) ~\d x
\end{equation*}
and the induced norm $\norm{\cdot}$, where $w:[-1,1]\to \IR$ is the Jacobi weight
\begin{equation*}
 w(x)=(1-x)^\mathbf{\alpha} (1+x)^\beta.
\end{equation*}
This reduces to the classical space of square-integrable functions, 
if both parameters are zero.
As $\mathbf{\alpha}$ respectively $\beta$ increases, 
the space grows, since we allow for
stronger singularities on the right respectively left endpoint
of the interval, and vice versa.

The family of Jacobi polynomials $\brackets{P_k}_{k\in\IN_0}$ is an orthogonal basis of $G$.
These polynomials can be defined as the unique solutions of the differential equations
\begin{equation*}
 \mathcal{L} P_k= k(k+2a) P_k
\end{equation*}
for the second order differential operator
\begin{equation*}
 \mathcal{L} = - w(x)^{-1} \frac{\d}{\d x}\brackets{\brackets{1-x^2}w(x)\frac{\d}{\d x}}
\end{equation*}
that satisfy
\begin{equation*}
 P_k(1)=\binom{k+\mathbf{\alpha}}{k}
 \quad\text{and}\quad
 P_k(-1)=(-1)^k \binom{k+\beta}{k}
.\end{equation*}
For more details on Jacobi polynomials we refer the reader to~\cite[Chapter~4]{Sz39}.
We denote the $k^{\rm th}$ Fourier coefficient of $f$ with respect to the normalized Jacobi basis by $f_k$.
The scalar product in $G$ hence admits the representation
\begin{equation*}
 \scalar{f}{g}=\sum\limits_{k=0}^\infty f_k g_k.
\end{equation*}
For $r>0$ let $H=K^r\brackets{[-1,1],w}$ be the Hilbert space 
of functions $f\in G$ with
\begin{equation*}
 \sum\limits_{k=0}^\infty \brackets{1+a^{-1} k}^{2r} f_k^2 < \infty,
\end{equation*}
equipped with the scalar product
\begin{equation*}
 \scalar{f}{g}_r = \sum\limits_{k=0}^\infty \brackets{1+a^{-1} k}^{2r} f_k g_k
\end{equation*}
and the induced norm $\norm{\cdot}_r$.
Obviously, $\brackets{P_k}_{k\in\IN_0}$ is an orthogonal basis of $H$, too.
In case $r$ is an even integer, this is the space of all functions 
$f\in L^2\brackets{[-1,1],w}$
such that $\mathcal{L}^j f\in L^2\brackets{[-1,1],w}$ 
for $j=1,\hdots,r/2$ and the scalar product
\begin{equation*}
 \scalar{f}{g}_{r,*} = 
 \sum\limits_{j=0}^{r/2} \scalar{\mathcal{L}^j f}{\mathcal{L}^j g}
\end{equation*}
is equivalent to the one above.
Hence the parameter $r$ can be interpreted as smoothness of 
the functions in $K^r\brackets{[-1,1],w}$.
The embedding $S$ of $H$ into $G$ is compact
and its $n^{\rm th}$ singular value is given by
\begin{equation*}
 \sigma_n
 =\frac{\norm{P_{n-1}}}{\norm{P_{n-1}}_r}
 =\brackets{1+a^{-1} \brackets{n-1}}^{-r}.
\end{equation*}

We can apply our theorems to study the singular values 
of the $d^{\rm th}$ tensor product $S_d$ of $S$.
This is the embedding of $H_d=K^r([-1,1]^d,w_d)$ 
into $G_d=L^2([-1,1]^d,w_d)$,
where $G_d$ is the weighted $L^2$-space on the $d$-cube with respect to the
Jacobi weight $w_d=w\otimes\hdots\otimes w$ 
and $H_d$ is the subspace of Jacobi functions of mixed order $r$.
Like in the univariate case, $H_d$ can be described via differentials
of mixed order $r$ and less, if $r$ is an even integer.
Let $\mathcal{P}_d^r$ be the respective approximation problem.

\begin{cor}[\cite{CD16,Kr18}] For any $d\in\IN$ and $r>0$ we have
\begin{equation*}
 \e\brackets{n,\mathcal{P}_d^r}
 \sim \brackets{\frac{a^d}{\brackets{d-1}!}}^r 
 n^{-r} \brackets{\ln n}^{-r(d-1)}
.\end{equation*}
\end{cor}

This result can also be derived from \cite[Theorem~5.5]{CD16}.
In addition we get the following preasymptotic estimates.

\begin{cor}[\cite{Kr18}]
 For any $\delta\in (0,1]$, $r>0$, 
 $d\in\IN$ and $n<2^d$ we have
 \begin{align*}
  &\brackets{\frac{a}{a+1}}^r \brackets{\frac{1}{n+1}}^{p_{r,a,d,n+1}}
  \leq
  \e\brackets{n,\mathcal{P}_d^r}  \leq\,
  \brackets{\frac{\exp\brackets{\frac{(2a)^{1+\delta}}{\delta}}}{n+1}}^{
  q_{r,a,d,\delta}}\\
  &\text{with} \quad
  p_{r,a,d,n}=\frac{r \ln\frac{a+1}{a}}{
  \ln\brackets{1+ \frac{d}{\log_2 n}}} \quad\text{and}\quad
  q_{r,a,d,\delta}=\frac{r \ln\frac{a+1}{a}}{\ln d + (1+\delta) 
  \ln\frac{a+1}{a}}.
 \end{align*}
 The upper bound holds for all $n\in\IN_0$.
\end{cor}

This means that for large dimension $d$, 
a preasymptotic decay of approximate order $t_d=r \ln\frac{a+1}{a} /\ln d$ 
in $n$ can be observed.

\subsubsection{Approximation of Mixed Order Sobolev Functions on the Cube}

Another example is the problem of approximating
mixed order Sobolev functions on the $d$-cube in $L^2$
with $n$ pieces of linear information.
We want to compare the difficulty of this problem
with the difficulty of the respective problem
for the subspace of periodic functions
as considered in the first example.
Of course, the nonperiodic problem can only be harder
than the periodic problem.
In fact we find that it is much harder if $n$ is small
but just about as hard if $n$ is large.

We consider an interval $[a,b]$ and the circle $\mathbb{T}$.
The latter shall also be represented by $[a,b]$, 
where $a$ and $b$ are identified.
For any $r\in\IN_0$, the vector space
\begin{equation*}
 H^r\brackets{[a,b]}=
 \set{f\in L^2\brackets{[a,b]} \mid f^{(l)}\in L^2\brackets{[a,b]} 
 \text{ for } 1\leq l \leq r}
,\end{equation*}
equipped with the scalar product
\begin{equation}
\label{scalarproductdefinition}
 \scalar{f}{g}_r=\sum\limits_{l=0}^r 
 \int_a^b f^{(l)}(x)\cdot \overline{g^{(l)}(x)} ~\d x
\end{equation}
and induced norm $\norm{\cdot}_r$, is a Hilbert space, 
the Sobolev space of order $r$ on $[a,b]$.
In case $r=0$, it coincides with $L^2\brackets{[a,b]}$.
The subset
\begin{equation*}
 H^r\brackets{\mathbb{T}} = 
 \set{f\in H^r\brackets{[a,b]} \mid f^{(l)}(a)=f^{(l)}(b) 
 \text{\,\, for\, } l=0,1,\hdots,r-1}
\end{equation*}
of periodic functions is a closed subspace with codimension $r$,
the Sobolev space of order $r$ on $\mathbb{T}$.
By means of Parseval's identity and integration by parts, 
the above norm can be rearranged to
\begin{equation}
\label{normrewritten}
 \norm{f}_r^2= \sum_{k\in\IZ} 
 \abs{\hat f (k)}^2 \sum_{l=0}^r \abs{\frac{2\pi k}{b-a}}^{2l} \quad
 \text{for}\quad f\in H^r\brackets{\mathbb{T}}
,\end{equation}
where
\begin{equation*}
 \hat f(k)=\sqrt{\frac{1}{b-a}} \int_a^b f(x)\cdot 
 \exp\brackets{-2\pi i k \,\frac{x-a}{b-a}} \d x
\end{equation*}
is the $k^{\rm th}$ Fourier coefficient of $f$.
In the limiting case $r=\infty$, the Sobolev space 
$H^\infty\brackets{[a,b]}$ shall be defined
as the Hilbert space
\begin{equation*}
 H^\infty\brackets{[a,b]}=\set{f\in \mathcal{C}^\infty\brackets{[a,b]}
 \,\big|\, \sum_{l=0}^\infty \norm{f^{(l)}}_0^2 < \infty}
,\end{equation*}
equipped with the scalar product (\ref{scalarproductdefinition}) for $r=\infty$.
It contains all polynomials and is hence infinite-dimensional.
The space $H^\infty\brackets{\mathbb{T}}$ shall be the closed subspace of periodic functions, i.e.
\begin{equation*}
 H^\infty\brackets{\mathbb{T}} 
 = \set{f\in H^\infty\brackets{[a,b]} \mid f^{(l)}(a)=f^{(l)}(b) 
 \text{ for any } l\in \IN_0}
.\end{equation*}
Note that (\ref{normrewritten}) also holds for $r=\infty$. Hence,
\begin{equation*}
 H^\infty\brackets{\mathbb{T}} = 
 \vspan\set{\exp\brackets{2\pi i k\,\frac{\cdot -a}{b-a}} \,\big|\, k\in\IZ 
 \text{ with } \abs{\frac{2\pi k}{b-a}} <1}
\end{equation*}
is finite-dimensional with dimension $2 \lceil \frac{b-a}{2\pi}\rceil -1$.
In case $b-a\leq 2\pi$, it consists of constant functions only.
If $r$ is positive, $H^r\brackets{[a,b]}$ is compactly embedded 
into $L^2\brackets{[a,b]}$.
Let $\sigma^{(r)}_n$ be the $n^{\rm th}$ singular value of this embedding
and let $\tilde \sigma^{(r)}_n$ be the $n^{\rm th}$ singular value of the 
embedding of the subspace $H^r\brackets{\mathbb{T}}$ into
$L^2\brackets{\mathbb{T}}$.
Recall from the first example of this subsection that
\begin{equation*}
 \tilde\sigma^{(r)}_n=\brackets{\sum_{l=0}^r 
 \abs{\frac{2\pi\left\lfloor n/2\right\rfloor}{b-a}}^{2l}}^{-1/2}
 \quad\text{for } n\in \IN \text{ and } r\in\IN.
\end{equation*}
The singular values $\sigma^{(r)}_n$ for nonperiodic functions 
are not known explicitly.
However, $\sigma^{(r)}_n$ and $\tilde \sigma^{(r)}_n$ interrelate as follows.

\begin{lemma}
\label{sigmanlemma}
 For any $n\in\IN$ and $r\in\IN$, 
 it holds that $\sigma^{(r)}_{n+r}\leq \tilde\sigma^{(r)}_n\leq \sigma^{(r)}_n$.
\end{lemma}

\begin{proof}
 The second inequality is obvious, since $H^r\brackets{\mathbb{T}}$ 
 is a subspace of $H^r\brackets{[a,b]}$.
 The first inequality is true, since the codimension of this subspace is $r$.
 Let $U$ be the orthogonal complement of of $H^r\brackets{\mathbb{T}}$ 
 in $H^r\brackets{[a,b]}$.
 By relation~(\ref{minmax}),
 \begin{equation*}
 \begin{split}
  \sigma^{(r)}_{n+r}\,=
  &\min\limits_{
  \substack{V\subset H^r\brackets{[a,b]}\\ \dim(V)\leq n+r-1}}\, 
  \max\limits_{\substack{f\in H^r\brackets{[a,b]}, f\perp V\\ \norm{f}_r=1}} \norm{f}_0
  \,\leq \min\limits_{
  \substack{\tilde V\subset H^r\brackets{\mathbb{T}}\\ \dim(\tilde V)\leq n-1}}\, 
  \max\limits_{
  \substack{f\in H^r\brackets{[a,b]}, \norm{f}_r=1\\f\perp(\tilde V\oplus U)}} 
  \norm{f}_0\\
  &= \min\limits_{
  \substack{\tilde V\subset H^r\brackets{\mathbb{T}}\\ \dim(\tilde V)\leq n-1}}\, 
  \max\limits_{\substack{f\in H^r\brackets{\mathbb{T}}, f \perp \tilde V\\\norm{f}_r=1}} 
  \norm{f}_0
  \,=\, \tilde\sigma^{(r)}_n
 ,\end{split}
 \end{equation*}
 as it was to be proven.
\end{proof}

Lemma~\ref{sigmanlemma} implies that the asymptotic constants of 
the singular values for the periodic and the nonperiodic
functions coincide in the univariate case:
\begin{equation*}
 \lim\limits_{n\to\infty} n^r\sigma^{(r)}_{n+1} 
= \lim\limits_{n\to\infty} n^r\tilde\sigma^{(r)}_{n+1}
 =\pi^{-r} (b-a)^r.
\end{equation*}
Let $H^r_{\rm mix}([a,b]^d)$ be
the $d^{\rm th}$ tensor product space of $H^r([a,b])$.
Note that this space satisfies the identity
\begin{equation*}
 H^r_{\rm mix}([a,b]^d) 
 = \set{f\in L^2([a,b]^d) \mid 
 \diff^\mathbf{\alpha} f \in L^2([a,b]^d)
 \text{ for all } \mathbf{\alpha}\in\set{0,\hdots,r}^d}
\end{equation*}
in the case that $r$ is finite,
and the scalar product is given by
\begin{equation}
\label{natural scalar product}
 \scalar{f}{g}_r=\sum_{\mathbf{\alpha}\in\set{0,\hdots,r}^d}
 \int_{[a,b]^d} \diff^\mathbf{\alpha} f(\mathbf x)\cdot 
 \overline{\diff^\mathbf{\alpha} g(\mathbf x)} ~\d \mathbf{x}
.\end{equation}
We want to study the tensor product problem
$$
\mathcal{P}^r_d = \mathcal{P}[\APP,F_d^r,L^2,\lall,\mathrm{det},\mathrm{wc}],
$$
where $\APP$ is the embedding of 
$H^r_{\rm mix}([a,b]^d)$ into $L^2([a,b]^d)$,
and $F_d^r$ is the unit ball of $H^r_{\rm mix}([a,b]^d)$.
Since we know the asymptotic behavior of the $n^{\rm th}$
minimal error for the case $d=1$,
Theorem~\ref{asymptotic theorem} yields
the asymptotic behavior in the general case.

\begin{cor}[\cite{Kr18}]
\label{cor: strong equivalence mix all}
 For any $d\in\IN$ and $r\in\IN$ we have
\begin{equation*}
  \e\brackets{n,\mathcal{P}^r_d}
  \sim
  \brackets{\frac{(b-a)^d}{\pi^d \brackets{d-1}!}}^r
  n^{-r} \brackets{\ln n}^{r(d-1)}
 .\end{equation*}
 In particular, the $n^{\rm th}$ minimal errors for the nonperiodic
 and the periodic problem are strongly equivalent as $n\to\infty$.
\end{cor}

We turn to preasymptotic estimates.
As depicted in Section~\ref{preasymptoticssection}, the singular values 
show a preasymptotic decay
of approximate order $\ln (1/\sigma^{(r)}_2)/\ln d$.
Lemma~\ref{sigmanlemma} gives no information on $\sigma^{(r)}_2$. 
However, relation~(\ref{minmax}) implies that
\begin{equation*}
 \sigma^{(\infty)}_2 = \max\limits_{ f\perp 1,\, f\neq 0} 
 \frac{\norm{f}_0}{\norm{f}_\infty}
 \geq \frac{\norm{2x-a-b}_0}{\norm{2x-a-b}_\infty}
 = \sqrt{\frac{(b-a)^2}{12+(b-a)^2}}.
\end{equation*}
If, for example, the length of the interval $[a,b]$ is 1, we obtain
\begin{equation*}
 \sigma^{(\infty)}_2 \geq 0.27735
.\end{equation*}
Since any lower bound on the singular values for $r=\infty$ 
is a lower bound for $r\in\IN$,
Theorem~\ref{preasymptoticstheorem} yields the following corollary.

\begin{cor}[\cite{Kr18}]
\label{cor: preasymptotic lower bound mix all}
 Let $b-a=1$.
 For any $d\in\IN$, any $r\in\IN\cup\set{\infty}$ and $d\leq n< 2^d$,
 we have
 \begin{align*}
  \e\brackets{n,\mathcal{P}^r_d}
   \geq\, 0.27 \cdot (n+1)^{-c(d,n+1)},\quad
 \text{where}\quad c(d,n) 
 = \frac{1.2825}{\ln \brackets{1+\frac{2 d}{\log_2 n}}}\leq 1.17
.\end{align*}
\end{cor}

On the other hand, any upper bound on the singular values for $r=1$ 
is an upper bound for $r\geq 1$.
The singular values $\sigma^{(r)}_n$ for $r=1$ are known.
It is shown in \cite{Th96} that the family $\brackets{b_k}_{k\in\IN_0}$ 
is a complete orthogonal system in $H^1\brackets{[a,b]}$,
where the function $b_k:[a,b]\to\IR$ with
\begin{equation*}
 b_k(x)=\cos\brackets{k\pi\cdot\frac{x-a}{b-a}}  \quad\text{for } k\in\IN_0
\end{equation*}
is an eigenfunction of $W=S^*S$ for $r=1$ with respective eigenvalue
\begin{equation*}
 \lambda_k=\brackets{1+\brackets{\frac{k\pi}{b-a}}^2}^{-1}.
\end{equation*}
In case $b-a=1$,
\begin{equation*}
 \sigma^{(1)}_2=\brackets{\sqrt{1+\pi^2}}^{-1}\leq 0.30332
\end{equation*}
and
\begin{equation*}
 \sigma^{(1)}_n\leq 0.607 \cdot n^{-1}
\end{equation*}
for $n\geq 2$. Theorem~\ref{preasymptoticstheorem} for $\delta=0.65$ 
yields the following upper bound.

\begin{cor}[\cite{Kr18}]
\label{cor:preasymptotics mixed cube}
 Let $b-a=1$.
 For any $d\in\IN$, $r\in\IN\cup\set{\infty}$ and $n\in\IN_0$,
 we have
 \begin{equation*}
  \e\brackets{n,\mathcal{P}^r_d}
  \leq\, \brackets{\frac{2}{n+1}}^{c(d)}
  \quad\text{with}\quad 
  c(d)=\frac{1.1929}{2+\ln d}.
  \end{equation*}
\end{cor}

Apparently, the upper bound for $r=1$ and the lower bound for $r=\infty$ 
are already close if $d$ is large.
The gap between the cases $r=2$ and $r=\infty$ is even smaller.

Let $c$ be the midpoint of $[a,b]$ and let $l$ be its radius.
Moreover, let $\hat\omega = \sqrt{1+\omega^2}$ for $\omega\in\IR$ and 
consider the countable sets
\begin{equation*}
 \begin{split}
  &I_1=\set{\omega\geq 0\mid \hat\omega^3\cosh(\hat\omega l)\sin(\omega l)
  + \omega^3\sinh(\hat\omega l)\cos(\omega l)=0},\\
  &I_2=\set{\omega> 0\mid \hat\omega^3\sinh(\hat\omega l)\cos(\omega l)
  - \omega^3\cosh(\hat\omega l)\sin(\omega l)=0}.
 \end{split}
\end{equation*}
It can be shown (with some effort) that the family 
$\brackets{b_\omega}_{\omega\in I_1\cup I_2}$
is a complete orthogonal system in $H^2\brackets{[a,b]}$,
where the function $b_\omega:[a,b]\to\IR$ with
\begin{equation*}\begin{split}
  &b_\omega(x)
  =\omega^2\cdot \frac{\cosh\brackets{\hat \omega (x-c)}}{
  \cosh\brackets{\hat\omega l}}
  +\hat\omega^2\cdot \frac{\cos\brackets{\omega (x-c)}}{
  \cos\brackets{\omega l}},\quad \text{if }\omega\in I_1,\\
  &b_\omega(x)
  =\omega^2\cdot \frac{\sinh\brackets{\hat \omega (x-c)}}{
  \sinh\brackets{\hat\omega l}}
  +\hat\omega^2\cdot \frac{\sin\brackets{\omega (x-c)}}{
  \sin\brackets{\omega l}},\quad \text{if }\omega\in I_2,
\end{split}
\end{equation*}
is an eigenfunction of $W=S^*S$ with respective eigenvalue
\begin{equation*}
 \lambda_\omega=\brackets{1+\omega^2+\omega^4}^{-1}
.\end{equation*}
In particular,
\begin{equation*}
 \sigma^{(2)}_2 = \brackets{\sqrt{1+\omega_0^2+\omega_0^4}}^{-1},
\end{equation*}
where $\omega_0$ is the smallest nonzero element of $I_1\cup I_2$.
If, for example, the interval $[a,b]$ has length 1, we obtain
\begin{equation*}
 \sigma^{(2)}_2\leq 0.27795 
\end{equation*}
and like before,
\begin{equation*}
 \sigma^{(2)}_n\leq 0.607 \cdot n^{-1}
\end{equation*}
for $n\geq 2$.
Theorem~\ref{preasymptoticstheorem} for $\delta=0.65$ yields 
the following upper bound.

\begin{cor}[\cite{Kr18}]
 Let $b-a=1$, $d\in\IN$, $n\in\IN_0$ and $r\geq 2$.
 Then
 \begin{equation*}
  \e\brackets{n,\mathcal{P}^r_d}
  \leq\, \brackets{\frac{2}{n+1}}^{c(d)}
  \quad\text{with}\quad 
  c(d)=\frac{1.2803}{2+\ln d}.
  \end{equation*}
\end{cor}

In short, the preasymptotic rate of the $n^{\rm th}$ minimal error
is $1.1929/\ln d$ for $r=1$,
and in between $1.2803/\ln d$ and $1.2825/\ln d$ 
for any other $r\in\IN\cup\set{\infty}$.
In contrast, the preasymptotic rate for the periodic
problem is roughly $1.8379\,r/\ln d$.
Thus, the nonperiodic problem is much harder than
the periodic problem if $n$ is small compared to $2^d$
and the smoothness $r$ is large compared to 1.

\subsection{A Tractability Result}
\label{tracsection}

A consequence of the preasymptotic estimates 
from Section~\ref{preasymptoticssection}
is the following tractability result.
For each $d\in\IN$, let $S^{(d)}$ be a compact norm-one operator 
between two Hilbert spaces with singular values $\sigma^{(d)}_n$.
Let $\mathcal{P}^{(d)}$ be the respective approximation problem
with deterministic algorithms based on $\lall$.
With $\mathcal{P}^{(d)}_d$ we denote the $d^{\rm th}$ tensor product
problem that belongs to the $d^{\rm th}$ tensor product operator $S^{(d)}_d$
of $S^{(d)}$.
Note that now the univariate problem $\mathcal{P}^{(d)}$ may be
different for every $d\in\IN$.
In fact,
it is shown in \cite[Theorem~5.5]{NW08} that the 
family of multivariate problems $\mathcal{P}^{(d)}_d$
is not polynomially tractable
if the univariate problem $\mathcal{P}^{(d)}$ does not
depend on $d$.
However, we may hope for tractability,
if the univariate problem gets easier as $d$ increases,
that is, if $\sigma^{(d)}_n$ is decreasing in $d$.
It turns out that we obtain polynomial tractability
and even strong polynomial tractability
if the second singular value $\sigma^{(d)}_2$ of $S^{(d)}$
decreases polynomially in $d$.
This condition is also necessary.

\begin{thm}[\cite{Kr18}]
\label{tractabilitytheorem}
 Let $\sigma^{(d)}_n$ be nonincreasing in $d$
 and polynomially decreasing in $n$ for $d=1$.
 The family of multivariate problems $\mathcal{P}^{(d)}_d$
 is strongly polynomially tractable,
 iff it is polynomially tractable, 
 iff $\sigma^{(d)}_2$ decays polynomially in $d$.
\end{thm}

\begin{proof}
 Clearly, strong polynomial tractability implies polynomial tractability.
 
 Let the problem be polynomially tractable and choose
 nonnegative numbers $C,p$ and $q$ such that
 \begin{equation*}
  \comp(\varepsilon,\mathcal{P}^{(d)}_d)
  \leq C\, \varepsilon^{-q}\, d^p
 \end{equation*}
 for all $\varepsilon>0$ and $d\in\IN$. In particular, 
 there is some $r\in\IN$ with
 \begin{equation*}
  \comp(d^{-1},\mathcal{P}^{(d)}_d) \leq d^{r} -1
 \end{equation*}
 for every $d\geq 2$.
 If $d$ is large enough, we can apply the second part of 
 Theorem~\ref{preasymptoticstheorem} for $n=d^r$ and the estimate
 \begin{equation*}
  \beta\brackets{d,d^r} = \frac{\ln (1/\sigma^{(d)}_2)}{
  \ln \brackets{1+\frac{v\cdot d}{r \log_{1+v} d}}}
  \leq \frac{2 \ln (1/\sigma^{(d)}_2)}{\ln d}
 \end{equation*}
 to obtain
 \begin{equation*}
  d^{-1}
  \geq \e\brackets{d^r-1,\mathcal{P}^{(d)}_d}
  \geq \sigma^{(d)}_2 \cdot d^{-r \beta\brackets{d,d^r}}
  \geq (\sigma^{(d)}_2)^{2r +1}.
 \end{equation*}
 Consequently, $\sigma^{(d)}_2$ decays polynomially in $d$.
 
 Now let $\sigma^{(d)}_2$ be of polynomial decay.
 Then there are $p>0$ and $d_0\in\IN$ such that 
 $\sigma^{(d)}_2$ is bounded above by $d^{-p}$ for any $d\geq d_0$.
 On the other hand, there are positive constants $C$ and $r$ such that
 \begin{equation*}
  \sigma^{(d)}_n \leq \sigma^{(1)}_n \leq C \, n^{-r}.
 \end{equation*}
 We apply the first part of 
 Theorem~\ref{preasymptoticstheorem} and the estimate
 \begin{equation*}
  \mathbf{\alpha}\brackets{d,1}= \frac{\ln (1/\sigma^{(d)}_2)}{
  \ln d + \frac{2}{r} \ln (1/\sigma^{(d)}_2)}
  \geq \frac{p}{1+\frac{2p}{r}} =: s >0
 \end{equation*}
 to obtain
 \begin{equation*}
  \e\brackets{n,\mathcal{P}^{(d)}_d} 
  \leq \brackets{\frac{\exp\brackets{C^{2/r}}}{n+1}}^s
 \end{equation*}
 for any $n\in\IN$ and $d\geq d_0$. Consequently,
 \begin{equation*}
  \comp\brackets{\varepsilon,\mathcal{P}^{(d)}_d}
  \leq \exp\brackets{C^{2/r}}\cdot \varepsilon^{-1/s}
 \end{equation*}
 for any $d\geq d_0$ and $\varepsilon>0$ and the problem 
 is strongly polynomially tractable.
\end{proof}

As an example we consider the embeddings
$$
 S^{(d)}_d: H^{r_d}_{\rm mix}([a,b]^d) \hookrightarrow L^2([a,b]^d),
 \qquad
 \widetilde S^{(d)}_d: H^{r_d}_{\rm mix}(\mathbb{T}^d) 
 \hookrightarrow L^2(\mathbb{T}^d),
$$
where the mixed order Sobolev spaces with smoothness $r_d\in\IN$ 
are equipped with the scalar product \eqref{natural scalar product},
see Section~\ref{applicationssection}.
Let $\mathcal{P}^{r_d}_d$ and $\mathcal{\widetilde P}^{r_d}_d$
be the respective approximation problems.
If the smoothness $r_d$ is independent of $d$
these problems are not polynomially tractable.
Can we achieve polynomial tractability by increasing 
the smoothness with the dimension?
We obtain the following result.

\begin{cor}[\cite{Kr18}]
\label{tractabilitycorollary}
 The problem $\mathcal{P}^{r_d}_d$ is not polynomially tractable
 for any choice of natural numbers $r_d$.
 The problem $\mathcal{\widetilde P}^{r_d}_d$
 is strongly polynomially tractable,
 iff it is polynomially tractable,
 iff $b-a<2\pi$ and $r_d$ grows at least logarithmically in $d$
 or $b-a=2\pi$ and $r_d$ grows at least polynomially in $d$.
\end{cor}

With regard to tractability, the $L^2$-approximation 
of mixed order Sobolev functions
is hence much harder for nonperiodic than for periodic functions.
The negative tractability result for nonperiodic 
functions can be explained by the difficulty of approximating $d$-variate
polynomials with degree 1 or less in each 
variable and $H^1_{\rm mix}$-norm less than 1.
The corresponding set of functions is contained in the unit ball
of the nonperiodic space $H^r_{\rm mix}$ for every $r\in\IN\cup\set{\infty}$.

\begin{rem}
Corollary~\ref{tractabilitycorollary} 
for cubes of unit length is in accordance with the
results of \cite{PW10},
where Papageorgiou and Woźniakowski prove the 
corresponding statement for the $L^2$-approximation in Sobolev spaces
of mixed smoothness $(r_1,\hdots,r_d)$ on 
the unit cube.
The smoothness of such functions increases 
from variable to variable,
but the smoothness with respect to a fixed 
variable does not increase with the dimension.
There, the authors raise the question for a 
characterization of spaces and their norms
for which increasing smoothness yields polynomial tractability.
Theorem~\ref{tractabilitytheorem} says that in 
the setting of uniformly increasing mixed smoothness,
polynomial tractability is achieved,
if and only if it leads to a polynomial decay 
of the second singular value of the univariate problem.
It would be interesting to verify whether the 
same holds in the case of variable-wise increasing smoothness
and to compute the exponents of strong polynomial tractability.
\end{rem}

\begin{rem}[Impact of the interval representation]
The reason for the great sensibility of the 
tractability results for the periodic spaces to the length of the interval
can be seen in the difficulty of approximating
trigonometric polynomials with frequencies in $\frac{2\pi}{b-a}\set{-1,0,1}^d$
that are contained in the unit ball of $H^\infty_{\rm mix}(\mathbb{T}^d)$.
The corresponding set of functions is nontrivial,
if and only if $\frac{2\pi}{b-a}$ is smaller than 1.

It may yet seem unnatural that the singular values are so
sensible to the representation $[a, b]^d$ 
of the $d$-torus or the $d$-cube.
This can only happen, since the above and common scalar products
\eqref{natural scalar product}
do not define a homogeneous family of norms on 
$H^{r}_{\rm mix}([a,b]^d)$.
To see that, let $S$ be the embedding of 
$H^{r}_{\rm mix}([a,b]^d)$
into $L^2([a,b]^d)$
and let $S_0$ be the embedding in the case 
$[a,b]=[0,1]$.
The dilation operation 
$Mf=f\brackets{\mathbf a+(b-a) \,\cdot}$ defines a
linear homeomorphism both
from $L^2([a,b]^d)$ into $L^2([0,1]^d)$
and from $H^{r}_{\rm mix}([a,b]^d)$ 
into $H^{r}_{\rm mix}([0,1]^d)$
and we have $S_0 = M S M^{-1}$.
The $L^2$-spaces satisfy the homogeneity relation
\begin{equation*}
 \Xnorm{Mf}{L^2([0,1]^d)}
 = (b-a)^{-d/2} 
 \cdot \Xnorm{f}{L^2([a,b]^d)}
 \quad \text{for} \quad f\in L^2([a,b]^d)
.\end{equation*}
If the chosen family of norms on $H^{r}_{\rm mix}(([a,b]^d))$ 
is also homogeneous, i.e.
\begin{equation*}
 \Xnorm{Mf}{H^{r}_{\rm mix}([0,1]^d)}
 = (b-a)^{-d/2} 
 \cdot \Xnorm{f}{H^{r}_{\rm mix}([a,b]^d)}
 \quad \text{for} \quad f\in H^{r}_{\rm mix}([a,b]^d)
,\end{equation*}
the singular values of $S$ and $S_0$ clearly must coincide.
The above scalar products do not yield a homogeneous family of norms.
An example of an equivalent and homogeneous family of norms 
on $H^{r}_{\rm mix}([a,b]^d)$
is given by the scalar products
\begin{equation*}
 \scalar{f}{g}=\sum_{\mathbf{\alpha}\in\set{0,\hdots,r}^d} 
 (b -a)^{2\abs{\mathbf{\alpha}}} \Xscalar{\diff^\mathbf{\alpha} f}{\diff^\mathbf{\alpha} g}{L^2([a,b]^d)}
.\end{equation*}
Hence, the singular values and tractability 
results with respect to this scalar product
do not depend on $a$ and $b$ at all,
both in the periodic and the nonperiodic case.
They coincide with the singular values with respect to the
scalar product \eqref{natural scalar product} 
for the case $[a,b]=[0,1]$.
\end{rem}



\section{Randomized Approximation \texorpdfstring{in $L^2$}{}}
\label{sec:OptimalMC}

We want to approximate an unknown real or complex valued function $f$
on a set $D$ based on a finite number $n$ of function values
which may be evaluated at randomly and adaptively chosen points.
In general, 
we cannot avoid to make an approximation error.
We measure this error in the space $L^2(D,\mathcal{A},\mu)$
of square integrable functions on $D$ with respect to some measure $\mu$.

If we want to say anything about this error,
we need to have some a priori knowledge of the function.
For example, $D$ might be a compact manifold and we might know that
$f$ is bounded with respect to some Sobolev norm on $D$.
More generally, we may assume that the function 
can be approximated well with respect
to some orthonormal system $\mathcal{B}=\{b_1,b_2,\hdots\}$ in $L^2$.
That is, there is a nonincreasing zero sequence $\eps:\IN_0\to(0,\infty)$
such that the function is contained in
\begin{equation*}
 F_{\mathcal{B}}^\eps
 =
 \Big\{ f\in L^2 \,\Big\vert\, \big\Vert f-\sum_{j=1}^m \scalar{f}{b_j}_2 b_j\big\Vert_2^2
 \leq \eps(m)
 \text{\ \ for all }
 m\in\IN_0\Big\}.
\end{equation*}

The approximation is described by a random mapping
$A_n: F_{\mathcal{B}}^\eps\to L^2$,
which we call algorithm.
The error of the algorithm $A_n$ is defined by
\begin{equation*}
 \err\brackets{A_n,F_{\mathcal{B}}^\eps}
 = \sup\limits_{f\in F_{\mathcal{B}}^\eps} 
 \brackets{\IE \norm{f-A_n(f)}_2^2 }^{1/2} 
.\end{equation*}
This is the worst mean squared error that
can occur for the given a priori knowledge.
The algorithm is called $n^{\rm th}$ optimal,
we write $A_n^*$ instead of $A_n$,
if it satisfies
\begin{equation*}
 \err\brackets{A_n^*,F_{\mathcal{B}}^\eps}
 = \inf\,
 \err\brackets{A_n,F_{\mathcal{B}}^\eps},
\end{equation*}
where the infimum is taken over all algorithms $A_n$
that require at most $n$ function values of the unknown function.

It seems to be an unrealistic hope 
to find such algorithms $A_n^*$.
Things look better if $n$ arbitrary pieces of linear information are allowed.
The optimal deterministic algorithm
that requires at most $n$ pieces of linear information is 
given by the orthogonal projection $P_n$
onto the span of the first $n$ functions in $\mathcal{B}$.
Its worst case error is the square-root of $\eps(n)$.
The algorithm $P_n$ asks for the first $n$ coefficients of $f$
with respect to the orthonormal system $\mathcal{B}$.

In most applications, however, it is not possible to sample these coefficients
and we may only make use of function values.
This leads to the following questions:
\begin{itemize}
 \item How does the error of
 $A_n^*$ compare to the error of $P_n$?
 \item Find an algorithm $A_n$ whose error is close to the error of $A_n^*$.
\end{itemize}
Note that $A_n^*$ cannot be much better than $P_n$.
In 1992, Novak~\cite{No92} proved that
\begin{equation}
\label{eq:novak lower ran}
 \err\brackets{A_n^*,F_{\mathcal{B}}^\eps}
 \geq \frac{1}{\sqrt{2}}\,
 \err\brackets{P_{2n-1},F_{\mathcal{B}}^\eps},
\end{equation}
see also Theorem~\ref{thm:randomization useless in Hspace}.
On the other hand, there are various examples
where the error of the algorithm $A_n^*$
behaves similarly to the error of $P_n$,
see for instance \cite{CDL13,CM17,He94,Ma91,TWW88}.
In 2006,
Wasilkowski and Woźniakowski~\cite{WW06}
proved for the general case that
\begin{equation*}
 \err\brackets{P_n,F_{\mathcal{B}}^\eps} 
 \preccurlyeq n^{-p} (\ln n)^q \quad
 \Rightarrow \quad
 \err\brackets{A_n^*,F_{\mathcal{B}}^\eps}
 \preccurlyeq n^{-p} (\ln n)^q (\ln \ln n)^{p+1/2} 
\end{equation*}
for all $p>0$ and $q\geq 0$.
In that sense, $A_n^*$ is almost as good as $P_n$.
The proof is constructive.
Of course, we immediately wonder whether
the additional power of the double logarithm is necessary.
In 2012, Novak and Woźniakowski showed
that this is not the case for $q=0$, that is,
\begin{equation*}
 \err\brackets{P_n,F_{\mathcal{B}}^\eps} \preccurlyeq n^{-p} \quad
 \Rightarrow \quad
 \err\brackets{A_n^*,F_{\mathcal{B}}^\eps} \preccurlyeq n^{-p}
\end{equation*}
for all $p>0$.
The proof of this result is not constructive.
Both proofs can be found in~\cite[Chapter~22]{NW12}.
In this section we prove the corresponding statement
for $q>0$.
This solves Open Problem~99 as posed in~\cite{NW12}.

More generally, we consider sequences with
the property
\begin{equation}
\label{eq:equivalence relation eps}
\eps(2n)\asymp\eps(n).
\end{equation}
For any such sequence and any orthonormal system $\mathcal{B}$,
we provide an algorithm $A_n$ and a constant $c_{\eps}>0$
such that, for all $n\in\IN$, we have
\begin{equation*}
 \err\brackets{A_n,F_{\mathcal{B}}^\eps} \leq 
 c_{\eps} \err\brackets{P_n,F_{\mathcal{B}}^\eps},
\end{equation*}
see Theorem~\ref{thm:explicit algorithm}.
Together with \eqref{eq:novak lower ran},
this answers both questions from above:
The errors of $A_n$ and $A_n^*$ only differ by a constant
and we have
\[
 \err\brackets{A_n^*,F_{\mathcal{B}}^\eps}
 \asymp
 \err\brackets{P_n,F_{\mathcal{B}}^\eps}.
\]
The algorithm is a refinement of the algorithm proposed in~\cite{WW06}.
Note that the constant $c_{\eps}$ only depends on the equivalence constant
of~\eqref{eq:equivalence relation eps}.
This constant is usually independent of the dimension
of the domain $D$.

These results are presented in Section~\ref{sec:general result}.
In Section~\ref{sec:examples OptimalMC} we consider several examples.
In particular, we study the problem of approximating
mixed order Sobolev functions in $L^2$ with
randomized algorithms based on $\lstd$ and
obtain the optimal order of convergence for the $n^{\rm th}$ minimal error,
see Corollary~\ref{cor:order OptimalMC mixed}.

In Section~\ref{int section}, we use these algorithms for the integration
of functions $f$ in $F_{\mathcal{B}}^\eps$ with respect 
to probability measures $\mu$.
We simply exploit the relation
\begin{equation*}
 \int_D f~\d\mu = \int_D A_n f~\d\mu + \int_D (f- A_n f)~\d\mu.
\end{equation*}
We compute the integral of $A_n f$ precisely
and use a direct simulation to approximate the integral of $f\,$--$\,A_n f$,
which has a small variance.
This technique is called control variates
or separation of the main part and is widely used
for Monte Carlo integration, see
\cite[Theorem~5.3]{He94} for another example.
The error of the resulting algorithm significantly improves 
on the error of a sole direct simulation,
even if the number of samples is small
and $D$ is a high-dimensional domain.

\subsection{The Algorithm and its Error} 
\label{sec:general result}

Let $(D,\mathcal{A},\mu)$ be a measure space and $\IK\in\set{\IR,\IC}$.
The space $L^2=L^2(D,\mathcal{A},\mu)$ is the space of square integrable
$\IK$-valued functions on $(D,\mathcal{A},\mu)$,
equipped with the scalar product
\begin{equation*}
 \scalar{f}{g}_2=\int_D f\, \overline{g}~\d\mu.
\end{equation*}
Let $\mathcal{B}=\set{b_1,b_2,\hdots}$ be an orthonormal system in $L^2$
and let $\mathcal{B}=\set{b_1,b_2,\hdots}$ be a nonincreasing zero-sequence.
We consider the set
\[
F_{\mathcal{B}}^\eps
 =
 \Big\{ f\in L^2 \,\Big\vert\, \big\Vert f-\sum_{j=1}^m \scalar{f}{b_j}_2 b_j\big\Vert_2^2
 \leq \eps(m)
 \text{\ \ for all }
 m\in\IN_0\Big\}.
\]
The functions in this set can be approximated well with respect to $\mathcal{B}$.
In other words, they can be approximated
well based on $m$ pieces of linear information.
The goal is to show that they can be approximated
just as well based on $n$ randomly chosen function values,
where $n$ is not much larger than $m$.

We introduce some further notation. Let $S$ be the identity on $L^2$.
For $m\in\IN_0$,
let $V_m$ be the span of the first $m$ elements of $\mathcal{B}$.
By $P_m$ we denote the orthogonal projection onto $V_m$ in $L^2$,
that is,
$$
 P_m: L^2\to L^2, \quad P_m(f)=\sum_{j=1}^m \scalar{f}{b_j}_2 b_j.
$$
The orthogonal projection onto the orthogonal complement of $V_m$
is denoted by $Q_m$. 
Note that $Q_m+P_m=S$.
Moreover, we define the function
\begin{equation*}
 u_m=\frac{1}{m} \sum\limits_{j=1}^m \abs{b_j}^2.
\end{equation*}
This is a probability density with respect to $\mu$.
We consider the probability measures
\begin{equation*}
 \mu_m: \mathcal{A}\to [0,1], \quad \mu_m(E)=\int_E u_m ~\d\mu.
\end{equation*}
on $(D,\mathcal{A})$.
We now define a family 
of randomized algorithms based on function evaluations.
Using the notions form Section~\ref{sec:usual problems},
we study algorithms
$$
 A\in \mathcal{A}[L^2,L^2,\lstd,\mathrm{ran}].
$$
Recall that the worst-case mean-square error 
of the randomized algorithm $A$ for the $L^2$-approximation of a function
from $F\subset L^2$ is defined by
$$
 \err\brackets{A,F}^2 =
 \err\brackets{A,S,F,L^2,\mathrm{wc}}^2 = 
 \sup_{f\in F}\,\IE \norm{f-A(f)}_2^2
$$
and that $\cost(A,F)=\cost(A,F,\lstd,\mathrm{wc})$
is the maximal number of function values
that the algorithm requests about a problem instance $f\in F$.

\begin{alg}
\label{alg:main alg}
Let $\mathbf n$ and $\mathbf m$ be sequences of nonnegative integers
such that $\mathbf m$ is nondecreasing.
For every nonnegative integer $k$, we define
$$
 M^{(k)}_{\mathbf n,\mathbf m}: L^2\to L^2,
$$
by the following recursive scheme.
\begin{itemize}
 \item For $f\in L^2$, we set $M^{(0)}_{\mathbf n,\mathbf m}(f) =0$.
 \item For $k\geq 1$ and $f\in L^2$, let
 $X_1^{(k)},\hdots,X_{n_k}^{(k)}$ be random variables with distribution $\mu_{m_k}$
 that are each independent of all the other random variables and set
 \begin{equation*}
  M^{(k)}_{\mathbf n,\mathbf m} f = M^{(k-1)}_{\mathbf n,\mathbf m} f +
  \sum_{j=1}^{m_k} \left[\frac{1}{n_k}\sum_{i=1}^{n_k} \frac{\brackets{f-M^{(k-1)}_{\mathbf n,\mathbf m} f}\overline{b_j}}{u_{m_k}}
  \brackets{X_i^{(k)}} \right] b_j.
 \end{equation*}
\end{itemize}
\end{alg}

Note that the expectation of each term in the inner sum is 
$$
 \langle f-M^{(k-1)}_{\mathbf n,\mathbf m} f,b_j\rangle_2.
$$
The algorithm hence approximates $f$ in $k$ steps.
In the first step, $n_1$ function values of $f$ are used for standard Monte Carlo type approximations
of its $m_1$ leading coefficients with respect to the orthonormal system $\mathcal B$.
In the second step, $n_2$ values of the residue are used for standard Monte Carlo type approximations
of its $m_2$ leading coefficients and so on.
In total, the algorithm uses 
\begin{equation}
\label{eq:cost equation}
 \cost\brackets{M^{(k)}_{\mathbf n,\mathbf m},L^2}=\sum_{j=1}^k n_j
\end{equation}
function values of $f$.
The total number of approximated coefficients is $m_k$.

Algorithms of this type have already been studied by Wasilkowski and Woźniakowski in \cite{WW06}.
The simple but crucial difference with the above algorithms
is the variable number $n_j$ of nodes in each approximation step.
Note that this stepwise approximation
is similar to several multilevel Monte Carlo methods
as introduced by Heinrich in 1998, see \cite{He01}.

The benefit from the $k^{\rm th}$ step
is controlled by $m_k$ and $n_k$ as shown by the following lemma.
The lemma corresponds to Theorem~22.14 in \cite{NW12}.
The setting here is
slightly more general, but the proof is
almost the same.

\begin{lemma}
\label{error lemma}
 Algorithm~\ref{alg:main alg} satisfies for all $k\in\IN$ that
 \begin{equation*}
  \eps(m_k) \leq
  \err\brackets{M^{(k)}_{\mathbf n,\mathbf m},F_{\mathcal{B}}^\eps}^2
  \leq \frac{m_k}{n_k} \err\brackets{M^{(k-1)}_{\mathbf n,\mathbf m},F_{\mathcal{B}}^\eps}^2 
  + \eps(m_k).
 \end{equation*}
\end{lemma}


\begin{proof}
 We start with the lower bound. We consider the function
 $$
  f=\sqrt{\eps(m_k)}\cdot b_{m_k+1} \in F_{\mathcal{B}}^\eps.
 $$
 Note that $M^{(k)}_{\mathbf n,\mathbf m}(f)$ is contained in $V_{m_k}$
 and hence
 $$
  \norm{f-M^{(k)}_{\mathbf n,\mathbf m}(f)}_2^2
  \geq\norm{f-P_{m_k} f}_2^2
  =\norm{f}_2^2
  =\eps(m_k)
 $$
 for any realization of $M^{(k)}_{\mathbf n,\mathbf m}$.
 This yields the lower bound.
 
 We turn to the upper bound.
 Let $f\in F_{\mathcal{B}}^\eps$.
 Let us first fix a realization of $M^{(k)}_{\mathbf n,\mathbf m}$.
 We have
 \begin{equation*}
  \norm{f-M^{(k)}_{\mathbf n,\mathbf m}(f)}_2^2
  = \norm{P_{m_k}(f)-M^{(k)}_{\mathbf n,\mathbf m}(f)}_2^2
  + \norm{Q_{m_k}(f)}_2^2.
 \end{equation*}
 Recall that the second term is bounded by $\varepsilon(m_k)$.
 The first term satisfies
 $$
  \norm{P_{m_k}(f)-M^{(k)}_{\mathbf n,\mathbf m}(f)}_2^2
  = \sum_{j=1}^{m_k} \abs{\scalar{f-M^{(k)}_{\mathbf n,\mathbf m}f}{b_j}_2}^2.
 $$
 We turn back to the randomized setting.
 For $j\leq m_k$, we use the abbreviation
 \begin{equation*}
  g_j = \frac{1}{u_{m_k}} \brackets{f-M^{(k-1)}_{\mathbf n,\mathbf m}f}\overline{b_j}.
 \end{equation*}
 Note that $u_{m_k}=0$ implies $b_j=0$ and we set $g_j=0$ in this case.
 Let $\IE_k$ denote the expectation with respect to
 the random variables $X_i^{(k)}$ for $i\leq n_k$.
 We obtain
 \begin{align*}
   \IE_k &\abs{\scalar{f-M^{(k)}_{\mathbf n,\mathbf m}f}{b_j}_2}^2
   = \IE_k \abs{\scalar{f-M^{(k-1)}_{\mathbf n,\mathbf m}f}{b_j}_2
   - \frac{1}{n_k}\sum_{i=1}^{n_k} g_j\brackets{X_i^{(k)}}}^2\\
   &= \IE_k \abs{\int_D g_j(x) ~\d\mu_{m_k}(x) - \frac{1}{n_k}\sum_{i=1}^{n_k} g_j\brackets{X_i^{(k)}}}^2
   \leq \frac{1}{n_k} \int_D \abs{g_j(x)}^2 ~\d\mu_{m_k}(x)\\
   &= \frac{1}{n_k} \int_D \abs{g_j(x)}^2 u_{m_k}(x) ~\d\mu(x)
 \end{align*}
 and hence
 \begin{multline*}
  \IE_k \sum_{j=1}^{m_k} \abs{\scalar{f-M^{(k)}_{\mathbf n,\mathbf m}f}{b_j}_2}^2
  \leq \frac{1}{n_k} \int_D \sum_{j=1}^{m_k} \abs{g_j(x)}^2 u_{m_k}(x)~\d\mu(x)\\
  = \frac{m_k}{n_k} \int_D  \abs{\brackets{f-M^{(k-1)}_{\mathbf n,\mathbf m} f}(x)}^2 ~\d\mu(x)
  = \frac{m_k}{n_k} \norm{f-M^{(k-1)}_{\mathbf n,\mathbf m} f}_2^2.
 \end{multline*}
 With Fubini's theorem this yields that
 \begin{equation*}
   \IE \norm{f-M^{(k)}_{\mathbf n,\mathbf m}f}_2^2
   \ \leq\, \frac{m_k}{n_k}\, \IE \norm{f-M^{(k-1)}_{\mathbf n,\mathbf m} f}_2^2 
   + \eps(m_k)
 \end{equation*}
 and the upper bound is proven.
\end{proof}

Based on this error formula,
we now tune the parameters of Algorithm~\ref{alg:main alg}.

\begin{prop}[\cite{Kr18c}]
 \label{thm:fundamental theorem}
 Let $m_j=2^{j-1}$ and
 $n_j=2^j \left\lceil 
 \varepsilon(\lfloor 2^{j-2}\rfloor)/\varepsilon(2^{j-1})
 \right\rceil$
 for all $j\in\IN$.
 Then Algorithm~\ref{alg:main alg} satisfies
 for all $k\in\IN_0$ that
 \begin{itemize}
  \item $\displaystyle 
  \err\brackets{M^{(k)}_{\mathbf n,\mathbf m},F_{\mathcal{B}}^\eps}^2
  \leq 2\,\eps\brackets{\lfloor 2^{k-1}\rfloor}$.
  \item $\displaystyle \cost\brackets{M^{(k)}_{\mathbf n,\mathbf m},L^2}
  \leq 2^{k+1} \max\limits_{0\leq j < k}
  \left\lceil\frac{\varepsilon(\lfloor 2^{j-1}\rfloor)}{\varepsilon(2^j)}\right\rceil$.
 \end{itemize}
\end{prop}

\begin{proof}
 The second estimate is obvious from~\eqref{eq:cost equation}.
 The first estimate follows from Lemma~\ref{error lemma}
 by induction on $k$.
 For $k=0$, we have $M^{(k)}_{\mathbf n,\mathbf m}=0$ and hence
 $$
 \err\brackets{M^{(k)}_{\mathbf n,\mathbf m},F_{\mathcal{B}}^\eps}^2
 = \sup_{f\in F_{\mathcal{B}}^\eps} \norm{f}_2^2
 = \varepsilon(0).
 $$
 If the statement holds for all $k<k'$ with some $k'\in\IN$,
 Lemma~\ref{error lemma} yields
 \begin{multline*}
  \err\brackets{M^{(k')}_{\mathbf n,\mathbf m},F_{\mathcal{B}}^\eps}^2
  \leq \frac{m_{k'}}{n_{k'}} 
  \err\brackets{M^{(k'-1)}_{\mathbf n,\mathbf m},F_{\mathcal{B}}^\eps}^2 
  + \eps(m_{k'})\\
  \leq \frac{\eps\brackets{2^{k'-1}}}{2\,\eps\brackets{\lfloor 2^{k'-2}\rfloor}} 
  2\,\eps\brackets{\lfloor 2^{k'-2}\rfloor} 
  + \eps\brackets{2^{k'-1}}
  =2\,\eps\brackets{2^{k'-1}}
  =2\,\eps\brackets{\lfloor 2^{k'-1}\rfloor}
 \end{multline*}
 and the proof by induction is complete.
\end{proof}

For many sequences
the maximum in the cost bound
of Proposition~\ref{thm:fundamental theorem}
is bounded by a constant,
that is,
\begin{equation}
 \label{eq:constant in cost bound}
 \sup\limits_{j\in\IN_0}
 \frac{\varepsilon(\lfloor 2^{j-1}\rfloor)}{\varepsilon(2^j)}
 < \infty.
\end{equation}
In this case,
Proposition~\ref{thm:fundamental theorem} says that
we may achieve an error of order $\eps(n)$
for every $n\in\IN$ using only $n$
function values of the target function.
To make this precise,
we define the following instance of Algorithm~\ref{alg:main alg}.

\begin{alg}
 \label{alg:explicit alg}
 Let $\mathcal{B}$ be an orthonormal system in $L^2$
 and let $\varepsilon:\IN_0\to(0,\infty)$ be a nonincreasing zero
 sequence.
 For any $n\in\IN_0$,
 we consider the algorithm $A_n=M^{(k)}_{\mathbf n,\mathbf m}$
 as defined in Algorithm~\ref{alg:main alg}
 for the following parameters: 
 \begin{itemize}
  \item $m_j=2^{j-1}$ for all $j\in\IN$.
  \item $n_j=2^j \left\lceil 
 \varepsilon(\lfloor 2^{j-2}\rfloor)/\varepsilon(2^{j-1})
 \right\rceil$ for all $j\in\IN$.
  \item $k\in\IN_0$ maximal such that $\sum_{j=1}^k n_j \leq n$.
 \end{itemize}
\end{alg}

Note that
the randomized algorithm $A_n$ 
requires at most $n$ function values
of every input,
see~\eqref{eq:cost equation}.
We obtain the following.

\begin{thm}[\cite{Kr18c}]
 \label{thm:explicit algorithm}
 Let $\varepsilon:\IN_0\to(0,\infty)$ be a nonincreasing zero sequence
 that satisfies \eqref{eq:constant in cost bound}
 and let $\mathcal{B}$ be an orthonormal system in $L^2$.
 We put
 \begin{equation*}
 C = 2^{\ell^2+3\ell+1},
 \quad\text{with}\quad
 \ell=\min\Big\{ r\in\IN_0 \,\Big\vert\, 
 \frac{\varepsilon(\lfloor 2^{j-1}\rfloor)}{\varepsilon(2^j)}
 \leq 2^r
 \text{\ for all } j\in\IN_0
 \Big\}.
 \end{equation*}
 Then Algorithm~\ref{alg:explicit alg} satisfies for all $n\in\IN$
 that $\cost\brackets{A_n}\leq n$ and
 $$
  \err\brackets{A_n,F_{\mathcal B}^\eps}^2 \leq C\, \eps(n).
 $$
\end{thm}

\begin{proof}
 Proposition~\ref{thm:fundamental theorem} and our assumption yield that
 $$
  \err\brackets{A_n,F_{\mathcal B}^\eps}^2
  \leq 2\cdot \eps\brackets{\lfloor 2^{k-1}\rfloor}
 \leq 2 \cdot 2^{\ell(\ell+3)} \eps\brackets{2^{k+\ell+2}}
 \leq C \eps(n),
 $$
 where the last inequality follows from 
 $n< \sum_{j=1}^{k+1} n_j \leq 2^{k+\ell+2}$.
\end{proof}

Note that the constant $C$ in 
Theorem~\ref{thm:explicit algorithm}
is usually rather harmless.
We will consider several examples in Section~\ref{sec:examples OptimalMC}.
In all these examples, $F_{\mathcal B}^\eps$ will be a class of $d$-variate functions
and the constant will be independent of $d$.
Let us consider two sequences $\eps$ that satisfy 
the assumption of Theorem~\ref{thm:explicit algorithm}.

\begin{ex}
\label{ex:sequences with modest decay}
Let $c\geq 1$, $p>0$ and $q\geq 0$.
Property~\eqref{eq:constant in cost bound} is satisfied by
the sequence $\eps:\IN_0\to (0,\infty)$
with $\varepsilon(0)=1$ and
$$
 \eps(n) = \min\set{1, c\,n^{-p} \brackets{1+\log_2 n}^q},
 \quad n\geq 1.
$$
Another example is given by $\varepsilon(0)=1$ and
$$
 \eps(n) = \min\set{1, c \brackets{1+\log_2 n}^{-p}},
 \quad n\geq 1.
$$
In both cases we have the estimate
$$
 \sup\limits_{j\in\IN_0}
 \frac{\varepsilon(\lfloor 2^{j-1}\rfloor)}{\varepsilon(2^j)}
 \leq 2^p
 \leq 2^{\lceil p \rceil}.
$$
Thus Theorem~\ref{thm:explicit algorithm}
can be applied with $\ell=\lceil p \rceil$.
Property~\eqref{eq:constant in cost bound} is not satisfied
if the sequence
decays exponentially or if it has big jumps.
\end{ex}

\begin{rem}[Less a priori knowledge]
\label{knowledge remark}
We assumed that our target function $f$ satisfies
\begin{equation}
\label{eq:our assumption}
 \norm{f-P_nf}_2^2 \leq \eps(n)
\end{equation}
for all $n\in\IN_0$ for some $\eps:\IN_0\to(0,\infty)$
with $\eps(2n)\asymp\eps(n)$
and proved that randomized algorithms can achieve
a squared error of order $\eps(n)$ with a sample
size of order $n$.
Our algorithm depends on $\varepsilon$.
If we do not know an admissible upper bound $\varepsilon$,
we can still achieve a squared error of order $\norm{f-P_nf}_2^2$ using
a weighted least squares method,
see~\cite[Theorem~2.1~(iv)]{CM17}.
The sample size of this method, however, is
at least of order $n\ln n$.
Here it is assumed that
$D$ is a Borel subset of $\IR^d$
with positive Lebesgue measure,
$\mathcal{A}$ is the Borel sigma algebra on $D$
and $\mu$ is a probability measure on $(D,\mathcal{A})$.
Again, the involved proportionality constants
are independent of the dimension of the domain $D$.
\end{rem}

\subsection{Approximation of Functions from a Hilbert Space}
\label{sec:examples OptimalMC}

An important application of 
Theorem~\ref{thm:explicit algorithm}
is the $L^2$-approximation of functions
from a Hilbert space.
Let $\widetilde F$ be an infinite-dimensional Hilbert space 
that is compactly embedded into $L^2=L^2(D,\mathcal{A},\mu)$.
From Section~\ref{sec:LPs}, we know that
there is a countable orthogonal basis 
$$
 \mathcal{B}=\{b_1,b_2,\hdots\}
$$ 
of $\widetilde F$ such that $\mathcal{B}$ is orthonormal in $L^2$
and the sequence $\eps:\IN_0\to(0,\infty)$ with
$$
 \eps(n) = \Xnorm{b_{n+1}}{\widetilde F}^{-2}
$$
is a nonincreasing zero sequence.\footnote{We
point to the fact that the elements of $\mathcal{B}$ 
are normalized in $L^2$. In Section~\ref{sec:LPs},
we normalized the functions in $\widetilde F$.}
Let $F$ be the unit ball of $\widetilde F$.
For every $n\in\IN_0$, the linear algorithm
$$
 A_n:F\to L^2, \quad
 A_n(f)= \sum_{j=1}^{n} \scalar{f}{b_j}_2 b_j,
$$
is the optimal deterministic algorithm 
for the problem of approximating
functions from $F$ with $n$ pieces of linear information.
It satisfies
\begin{equation}
\label{eq:error of projection}
 \err\brackets{A_n,F}^2
 =\varepsilon(n).
\end{equation}
In other words, we have
\begin{equation}
 F\subset F_{\mathcal B}^\eps,
\end{equation}
and we can apply 
Algorithm~\ref{alg:explicit alg}.
In particular, we obtain the
following result on the order of convergence.

\begin{thm}[\cite{Kr18c}]
\label{thm:order thm}
 Let $F$ be the unit ball of an infinite-dimensional
 Hilbert space that is compactly embedded into 
 some $L^2$-space.
 If the singular values $\sigma:\IN\to (0,\infty)$ 
 of the embedding $\APP$ satisfy $\sigma(2n)\asymp\sigma(n)$
 then
 \begin{multline*}
 \e\brackets{n,\mathcal{P}[\APP,F,L^2,\lall,\mathrm{ran},\mathrm{wc}]}
 \asymp \e\brackets{n,\mathcal{P}[\APP,F,L^2,\lall,\mathrm{det},\mathrm{wc}]}\\
 \asymp \e\brackets{n,\mathcal{P}[\APP,F,L^2,\lstd,\mathrm{ran},\mathrm{wc}]}.
 \end{multline*}
\end{thm}

\begin{proof}
 The first relation follows from Theorem~\ref{thm:randomization useless in Hspace}.
 The second relation follows from Theorem~\ref{thm:explicit algorithm}.
 Note that we use the condition $\sigma(2n)\asymp\sigma(n)$
 for both these relations.
\end{proof}

This means that for the problem of approximating functions
from a Hilbert space in $L^2$,
randomized algorithms based on function values
are just as powerful as randomized or deterministic algorithms
based on arbitrary linear information.
Note that deterministic algorithms based on function values
may perform much worse, as shown by Hinrichs, Novak
and Vybíral \cite{HNV08},
see also \cite[Section~26.6.1]{NW12}.
We do not know whether the condition on the decay
of the singular values can be relaxed.

\begin{rem}
 We point to the fact that the error bounds of 
 Proposition~\ref{thm:fundamental theorem} and Theorem~\ref{thm:explicit algorithm}
 do not only hold for the class $F$
 but for the whole class $F_{\mathcal B}^\eps$.
 In general, the second class is strictly larger than the first.
 For example, if $\eps(m)=1/(m+1)^2$ for $m\in\IN_0$, then 
 \begin{equation*}
  f=\sum_{m\in\IN} (\eps(m-1)-\eps(m))^{1/2}\, b_m
 \end{equation*}
 belongs to $F_{\mathcal B}^\eps$
 but is not even contained in the space $\widetilde F$.
\end{rem}


We now consider several examples.
In each example, we first determine 
the order of convergence of
the $n^{\rm th}$ minimal error.
We then discuss explicit upper bounds.

\subsubsection{Functions with Mixed Smoothness on the Torus}

Let $D$ be the $d$-dimensional torus $\mathbb T^d$, 
represented by the unit cube $[0,1]^d$,
where opposite faces are identified.
Let $\mathcal{A}$ be the Borel $\sigma$-algebra on $\mathbb T^d$ and $\mu$ the Lebesgue measure.
Let $\widetilde F$ be the Sobolev space
of complex valued functions on $D$ with mixed smoothness $r\in\IN$,
equipped with the scalar product
\begin{equation}
\label{mix product}
 \Xscalar{f}{g}{\widetilde F} = \sum_{\norm{\alpha}_\infty \leq r} 
 \scalar{\diff^\alpha f}{\diff^\alpha g}_2
.\end{equation}
A classical result by Babenko \cite{Ba60} and Mityagin \cite{Mi62}
states that
\begin{equation}
\label{eq:order BaMi}
 \e\brackets{n,\mathcal{P}[\APP,F,L^2,\lall,\mathrm{det},\mathrm{wc}]}
 \asymp n^{-r}\ln^{r(d-1)} n.
\end{equation}
We remark that the same can be proven for fractional smoothness $r>0$,
see~\cite{Mi62}.
Theorem~\ref{thm:order thm} yields that
$$
 \e\brackets{n,\mathcal{P}[\APP,F,L^2,\lstd,\mathrm{ran},\mathrm{wc}]}
 \asymp n^{-r}\ln^{r(d-1)}n.
$$
This result is new.
The optimal order is achieved by Algorithm~\ref{alg:explicit alg} with 
$\eps(0)=1$ and
$$
\eps(n)=\min\set{1, c^2\,n^{-2r} \brackets{1+\log_2 n}^{2r(d-1)}}
$$
for $n\geq 1$, where $c$ is the constant in the upper bound
of~\eqref{eq:order BaMi}.
We do not know any other algorithm with this property.
It is still an open problem whether the same rate can be achieved
with deterministic algorithms based on function values.
So far, it is only known that
\begin{equation*}
 n^{-r}\ln^{r(d-1)}n
 \preccurlyeq
 \e\brackets{n,\mathcal{P}[\APP,F,L^2,\lstd,\mathrm{det},\mathrm{wc}]}
 \preccurlyeq
 n^{-r}\ln^{(r+1/2)(d-1)} n.
\end{equation*}
The upper bound is achieved by
Smolyak's algorithm, see~\cite{SU10}.

We now turn to explicit estimates.
We know that there is some $C_{r,d}>0$ such that
\begin{equation}
\label{asymptotic bound mix}
 \e\brackets{n,\mathcal{P}[\APP,F,L^2,\lstd,\mathrm{ran},\mathrm{wc}]} 
 \leq C_{r,d}\, n^{-r}\ln^{r(d-1)} n
\end{equation}
for all $n\geq 2$.
This upper bound is optimal as $n$ tends to infinity.
However, it is not useful to describe the error
numbers for small values of $n$.
Simple calculus shows that the right hand side in
\eqref{asymptotic bound mix} is increasing for $n\leq e^{d-1}$.
The error numbers, on the other hand, are decreasing.
Moreover,
the right hand side attains its minimum for $n=2$
if restricted to $n\leq (d-1)^{d-1}$
and is hence larger than the error for $n=2$.
This means that the trivial upper bound
\begin{equation}
\label{trivial bound}
 \e\brackets{n,\mathcal{P}[\APP,F,L^2,\lstd,\mathrm{ran},\mathrm{wc}]}
 \leq 
 \e\brackets{2,\mathcal{P}[\APP,F,L^2,\lstd,\mathrm{ran},\mathrm{wc}]}
\end{equation}
is better than \eqref{asymptotic bound mix} for all $n\in\{2,\hdots,(d-1)^{d-1}\}$
for any valid constant~$C_{r,d}$.
For these reasons, it is important 
to consider different error bounds,
if the dimension $d$ is large.
Based on~\cite{KSU15}, we already proved that the upper bound
\begin{equation*}
 \e\brackets{n,\mathcal{P}[\APP,F,L^2,\lstd,\mathrm{ran},\mathrm{wc}]} 
 \leq \brackets{2/n}^p
 \qquad\text{with}\quad
 p=\frac{r}{2+(\ln d)/(\ln 2\pi)}
\end{equation*}
holds for all $n\in\IN$. See~Corollary~\ref{cor:preasymptotics mixed torus}
for the parameters $[a,b]=[0,1]$, $\gamma=1$, $\square=\circ$ and $\delta=1$.
By Theorem~\ref{thm:explicit algorithm},
Algorithm~\ref{alg:explicit alg} with 
$\eps(n)=\min\set{1, 2^{2p} n^{-2p}}$ satisfies
\begin{equation}
\label{preasymptotic bound mix}
 \err\brackets{A_n,F}
  \leq 2^{(\ell^2+4\ell+1)/2}\cdot n^{-p} 
\end{equation}
for all $n\in\IN$, where $\ell=\lceil 2p\rceil$
is nonincreasing in $d$.
See Example~\ref{ex:sequences with modest decay}
for details on the constant.
For instance, let $r=6$ and $d=2000$.
Then \eqref{asymptotic bound mix} is useless up to $n=10^{6598}$,
whereas \eqref{preasymptotic bound mix} yields
the upper bound
$$
 \err(A_n,F)\leq 91\cdot n^{-0.97},
$$
which is useful for all $n\geq 105$. 

\subsubsection{Functions with Mixed Smoothness on the Cube}

Let $D$ be the $d$-dimensional unit cube $[0,1]^d$ equipped
with the Borel $\sigma$-algebra $\mathcal{A}$
and the Lebesgue measure $\mu$.
Let $\widetilde F$ be the Sobolev space
of complex valued functions on $[0,1]^d$ with mixed smoothness $r\in\IN$,
equipped with the scalar product \eqref{mix product}.
Just like on the torus, we have
$$
 \e\brackets{n,\mathcal{P}[\APP,F,L^2,\lall,\mathrm{det},\mathrm{wc}]}
 \asymp n^{-r}\ln^{r(d-1)}n
$$
and Theorem~\ref{thm:explicit algorithm} leads to the following result.

\begin{cor}[\cite{Kr18c}]
 \label{cor:order OptimalMC mixed}
 Let $F$ be the unit ball of the Sobolev space of
 mixed smoothness $r$ on the $d$-torus or on the $d$-cube. Then
 \begin{equation*}
 \e\brackets{n,\mathcal{P}[\APP,F,L^2,\lstd,\mathrm{ran},\mathrm{wc}]}
 \asymp n^{-r}\ln^{r(d-1)} n.
 \end{equation*}
\end{cor}

Again, the optimal rate can be achieved with Algorithm~\ref{alg:explicit alg}.
Also in this case, 
the corresponding upper bounds are bad for $n\leq (d-1)^{d-1}$.
In this range, we need different estimates for the error. 
By Corollary~\ref{cor:preasymptotics mixed cube}, we know that
\begin{equation*}
 \e\brackets{n,\mathcal{P}[\APP,F,L^2,\lall,\mathrm{det},\mathrm{wc}]} 
 \leq (2/n)^{p} \qquad\text{with}\quad
 p=\frac{1.1929}{2+\ln d}
\end{equation*}
for all $n\in\IN$.
This estimate cannot be improved significantly for $n\leq 2^d$.
By Theorem~\ref{thm:explicit algorithm},
Algorithm~\ref{alg:explicit alg}
with $\eps(n)=\min\set{1, 2^{2p} n^{-2p}}$ satisfies
for all $n\in\IN$ and $d\geq 2$ that
\begin{equation}
\label{eq:preasymptotic ran std}
 \err\brackets{A_n,F}
  \leq 8\,n^{-p}. 
\end{equation}

\subsubsection{Functions from Tensor Product Spaces}

 This example is more general than the previous examples.
 By $H_1\otimes H_2$ we denote the tensor product of 
 two Hilbert spaces $H_1$ and $H_2$.
 For $j=1,\hdots,d$ let $(D_j,\mathcal{A}_j,\nu_j)$ be a $\sigma$-finite
 measure space and let $\widetilde F_j$ be a Hilbert space of $\IK$-valued functions
 with unit ball $F_j$ such that the embedding $\APP_j$
 of $\widetilde F_j$ into $L^2(D_j,\mathcal{A}_j,\nu_j)$ is compact.
 The $\sigma$-finity of the measure spaces ensures that
 \begin{equation*}
  L^2(D_1,\mathcal{A}_1,\nu_1) \otimes \dots \otimes L^2(D_d,\mathcal{A}_d,\nu_d) = L^2(D,\mathcal{A},\mu),
 \end{equation*}
 where $D$ is the Cartesian product of the sets $D_j$
 and $\mu$ is the unique product measure of the measures $\nu_j$
 on the tensor product $\mathcal{A}$ of the $\sigma$-algebras $\mathcal{A}_j$.
 The embedding $\APP$ of the tensor product space
 $\widetilde F=\widetilde F_1 \otimes \dots\otimes \widetilde F_d$
 into $L^2(D,\mathcal{A},\mu)$ is compact, too.
 Assuming that the approximation numbers
 of the univariate embeddings 
 $\APP_j$ 
 are of polynomial decay, that is,
 \begin{equation*}
  \e\brackets{n,\mathcal{P}[\APP_j,F_j,L^2(D_j,\mathcal{A}_j,\nu_j),
  \lall,\mathrm{det},\mathrm{wc}]}
  \asymp n^{-r_j}
 \end{equation*}
 for some $r_j>0$, it can be derived from
 \cite{Mi62,Ni74} that
 \begin{equation*}
  \e\brackets{n,\mathcal{P}[\APP,F,L^2(D,\mathcal{A},\mu),
  \lall,\mathrm{det},\mathrm{wc}]}
  \asymp n^{-r}\ln^{r(d_0-1)}n,
 \end{equation*}
 where $F$ is the unit ball of $\widetilde F$, 
 $r$ is the minimum among all numbers $r_j$ and $d_0$ is its multiplicity.
 Theorem~\ref{thm:order thm} implies
\begin{equation*}
 \e\brackets{n,\mathcal{P}[\APP,F,L^2(D,\mathcal{A},\mu),
  \lstd,\mathrm{ran},\mathrm{wc}]}
  \asymp n^{-r}\ln^{r(d_0-1)}n,
\end{equation*}
where the optimal order can be achieved with Algorithm~\ref{alg:explicit alg}.
We do not discuss explicit estimates in this general setting.

\subsubsection{Functions with Isotropic Smoothness on the Torus}

Our algorithm may also be used for functions with isotropic smoothness.
Let $D$ again be the $d$-torus,
this time represented by $[0,2\pi]^d$.
Let $\widetilde F$ be the Sobolev space
of complex valued functions on $D$ with isotropic smoothness $r\in\IN$,
equipped with the scalar product
\begin{equation*}
 \Xscalar{f}{g}{\widetilde F} = \sum_{\norm{\alpha}_1 \leq r} \scalar{\diff^\alpha f}{\diff^\alpha g}_2
.\end{equation*}
This is not a tensor product problem.
For this classical problem, it is known that
\begin{multline*}
\e\brackets{n,\mathcal{P}[\APP,F,L^2,\lstd,\mathrm{det},\mathrm{wc}]}
\asymp
\e\brackets{n,\mathcal{P}[\APP,F,L^2,\lstd,\mathrm{ran},\mathrm{wc}]}\\
\asymp
\e\brackets{n,\mathcal{P}[\APP,F,L^2,\lall,\mathrm{det},\mathrm{wc}]}
\asymp
\e\brackets{n,\mathcal{P}[\APP,F,L^2,\lall,\mathrm{ran},\mathrm{wc}]}
\asymp
n^{-r/d}
\end{multline*}
for $r>d/2$. In the case $r\leq d/2$, where function values
are only defined almost everywhere, the last three relations stay valid. 
We refer to \cite{He08,Je67,Ma91,Tr05}.
In the range $n\leq 2^d$, however, the function $n^{-r/d}$ is not suited to
describe the behavior of the errors.
It has been proven by Kühn, Mayer and Ullrich \cite{KMU16}
that there are positive constants $b_r$ and $B_r$ that do not
depend on $d$ such that
\begin{multline}
\label{counting}
 b_r \brackets{\frac{\log_2\brackets{1+d/\log_2 n}}{\log_2 n}}^{r/2}
 \leq \e\brackets{n-1,\mathcal{P}[\APP,F,L^2,\lall,\mathrm{det},\mathrm{wc}]} \\
 \leq 
 B_r \brackets{\frac{\log_2\brackets{1+d/\log_2 n}}{\log_2 n}}^{r/2}
\end{multline}
for all $d>1$ and $n\in\IN$ with $d\leq n \leq 2^d$.
If we apply Theorem~\ref{thm:randomization useless in Hspace}
and Theorem~\ref{thm:explicit algorithm},
we obtain the existence of $d$-independent positive constants $\tilde b_r$
and $\widetilde B_r$ such that
\begin{multline*}
 \tilde b_r \brackets{\frac{\log_2\brackets{1+d/\log_2 n}}{\log_2 n}}^{r/2}
 \leq \e\brackets{n-1,\mathcal{P}[\APP,F,L^2,\lstd,\mathrm{ran},\mathrm{wc}]} \\ 
 \leq 
 \widetilde B_r \brackets{\frac{\log_2\brackets{1+d/\log_2 n}}{\log_2 n}}^{r/2}
\end{multline*}
for all $d>1$ and $n\in\IN$ with $d\leq n \leq 2^{d-1}$.

\subsubsection{Implementation of these algorithms}

We are able to implement Algorithm~\ref{alg:explicit alg}
if we know the basis $\mathcal B$ that is
associated with the embedding of $\widetilde F$ into $L^2$ 
and if we can sample from the probability measures $\mu_m$.
These tasks may be very hard.
In the case of Sobolev functions on the torus,
however, it is not.
Here, $\mathcal B$ is the Fourier basis of $L^2$
and all the random variables are independent and uniformly 
distributed on the unit cube.
Also the case of general tensor product spaces 
can be handled
if the orthonormal bases $\mathcal{B}_j$ that are
associated with the univariate embeddings 
are known.
Then the basis $\mathcal{B}$ is given by 
\begin{equation*}
 \mathcal{B}=\set{b^{(1)}\otimes\dots\otimes b^{(d)} \mid b^{(j)}\in\mathcal{B}_j \text{ for }j=1\dots d}
\end{equation*}
and the probability measure $\mu_m$ is the average of $m$ product measures, 
that is,
\begin{equation*}
 \mu_m=\frac{1}{m} \sum_{i=1}^m \bigotimes_{j=1}^d \eta_{i,j},
\end{equation*}
where $\d\eta_{i,j}=|b_{i,j}|^2\d\nu_j$ with some $b_{i,j}\in\mathcal{B}_j$.
A random sample $\mathbf x$ from this distribution can be obtained as follows:
\begin{itemize}
 \item[(1)] Get $i$ from the uniform distribution on $\set{1,\dots,m}$.
 \item[(2)] Get $x_1,\dots,x_d$ independently from the probability distributions $\eta_{i,1},\dots,\eta_{i,d}$.
\end{itemize}
The second step can for example be done by rejection sampling,
if the measures $\eta_{i,j}$ have a bounded Lebesgue density.
This way, the total sampling costs are linear in $d$.
Another method of sampling from $\mu_m$ is proposed
in~\cite[Section 5]{CM17}.

\subsection{Integration via Separation of the Main Part}
\label{int section}

We use the notation of Section~\ref{sec:general result}.
In this section, we require the measure $\mu$ to be finite.
This ensures that the integral operator
\begin{equation*}
 \INT: L^2 \to \IK, \quad \INT(f)=\int_D f d\mu
\end{equation*}
is well defined and continuous on $L^2$.
Let us assume that $\mu$ is a probability measure.
We want to approximate $\INT(f)$ for an unknown function $f\in F_{\mathcal B}^\eps$
by a randomized algorithm $Q_n$ which evaluates at most $n$ function values
of $f$. 
Recall that the worst case error of $Q_n$ is the quantity
\begin{equation*}
 \err\brackets{Q_n,\INT,F_{\mathcal B}^\eps}
 =\sup\limits_{f\in F_{\mathcal B}^\eps} \brackets{\IE\abs{\INT(f)-Q_n(f)}^2}^{1/2}
.\end{equation*}
The minimal worst case error among such algorithms is denoted by
\begin{equation*}
 \e\brackets{n,\mathcal{P}[\INT,F_{\mathcal B}^\eps,\lstd,\mathrm{ran},\mathrm{wc}]}
 = \inf\limits_{Q_n} \err\brackets{Q_n,\INT,F_{\mathcal B}^\eps}
.\end{equation*}
Like any approximation method,
Algorithm~\ref{alg:explicit alg}
can also be used for integration.

\begin{alg}
\label{alg:explicit alg int}
 Let $\mathcal{B}$ be an orthonormal system in $L^2$
 and let $\varepsilon:\IN_0\to(0,\infty)$ be a nonincreasing zero
 sequence.
 For all $n\in\IN$ and $f\in L^2$, let
 \begin{equation*}
  Q_{2n}(f)=\INT(A_n f) + \frac{1}{n}\sum_{j=1}^n \brackets{f-A_n f}(X_j),
 \end{equation*}
 where $A_n$ is defined in Algorithm~\ref{alg:explicit alg}
 and $X_1,\dots,X_n$ are random variables with distribution $\mu$ which are
 independent of each other and the random variables in $A_n$.
\end{alg}

It is easy to verify that $Q_{2n}$ is unbiased, evaluates at most 
$2n$ function values of $f$
and satisfies
\begin{equation*}
  \IE\abs{\INT(f)-Q_{2n}(f)}^2
  \leq \frac{1}{n}\, \IE \norm{f-A_n f}_2^2
\end{equation*}
for each $f$ in $L^2$. 
Thus we obtain the following corollary.

\begin{cor}[\cite{Kr18c}]
\label{integration cor}
 Let $\mathcal{B}$ be an orthonormal system in $L^2$
 and let $\varepsilon:\IN_0\to(0,\infty)$ be a nonincreasing zero sequence.
 We assume that
 $$
  \sup\limits_{j\in\IN_0}
 \left\lceil\frac{\varepsilon(\lfloor 2^{j-1}\rfloor)}{\varepsilon(2^j)}\right\rceil
 \leq 2^\ell
 $$
 for some $\ell\in\IN_0$.
 For any $n\in\IN$, Algorithm~\ref{alg:explicit alg int} satisfies
 $$
  \err\brackets{Q_{2n},\INT,F_{\mathcal B}^\eps}^2 
  \leq 2^{\ell^2+3\ell+1}\, \eps(n) n^{-1}.
 $$
 In particular, if $F$ is the unit ball
 of a Hilbert space that is compactly embedded in $L^2$,
 we obtain for all $p>0$ and $q\geq 0$ that
 \begin{align*}
 \e\brackets{n,\mathcal{P}[\APP,F,L^2,\lall,\mathrm{det},\mathrm{wc}]}
  &\preccurlyeq n^{-p} \ln^q n\\
 \Rightarrow\quad
 & \e\brackets{n,\mathcal{P}[\INT,F,\lstd,\mathrm{ran},\mathrm{wc}]}
  \preccurlyeq n^{-p-1/2} \ln^q n.
 \end{align*}
\end{cor}

The result on the order of convergence is quite general but not always optimal.
An example is given by integration with respect to the Lebesgue measure $\mu$
on the Sobolev space $\widetilde F$ with mixed smoothness $r$
on the $d$-dimensional unit cube, 
as treated in Section~\ref{sec:Frolov}. In this case, we have
\begin{align*}
 \e\brackets{n,\mathcal{P}[\APP,F,L^2,\lall,\mathrm{det},\mathrm{wc}]}
  &\asymp n^{-r} \ln^{r(d-1)} n,\\
 \e\brackets{n,\mathcal{P}[\INT,F,\lstd,\mathrm{ran},\mathrm{wc}]}
  &\asymp n^{-r-1/2},
\end{align*}
see~\cite{Ba60,Mi62,Ul17},
respectively \eqref{eq:order BaMi} and Corollary~\ref{cor:order of convergence}.
The main strength of Corollary~\ref{integration cor} is
that it provides an unbiased algorithm
achieving a reasonable integration error
with a modest number of function values
in high dimensions.

\begin{ex}[Functions with mixed smoothness on the torus]
\label{mix example torus integration}
Like in the first example of Section~\ref{sec:examples OptimalMC},
let $\widetilde F$ be the Sobolev space of mixed smoothness $r$ on the $d$-torus
and let $\mu$ be the Lebesgue measure.
Among all randomized algorithms for multivariate integration in $\widetilde F$
the randomized Frolov algorithm $Q_n^*$ is known to have the optimal error rate,
see Theorem~\ref{thm:main theorem}.
It is shown by Ullrich \cite{Ul17} that there is some constant $c>2^d$ such that
\begin{equation}
\label{Frolov bound}
 \err\brackets{Q_n^*,\INT,F}
  \leq c\,n^{-r-1/2}
\end{equation}
for all $n\in\IN$.
However, this estimate is trivial
for $n\leq 2^{d/(r+1/2)}$.
In this range, an error less than one is guaranteed 
by the direct simulation 
\begin{equation*}
 S_n(f)=\frac{1}{n} \sum_{j=1}^n f(X_j),
\end{equation*}
with independent and uniformly distributed random variables $X_j$.
It satisfies
\begin{equation}
\label{Standard bound}
 \err\brackets{S_n,\INT,F}
  \leq n^{-1/2}
\end{equation}
for all $n\in\IN$.
However, this error bound converges only slowly, as $n$ tends to infinity.
It does not reflect the smoothness of the integrands at all.
Algorithm~\ref{alg:explicit alg int}
also guarantees nontrivial error bounds for smaller values of $n$,
but converges faster than~$S_n$.
Relation~(\ref{preasymptotic bound mix}) immediately yields
that
\begin{equation}
\label{preasymptotic bound mix integration}
 \err\brackets{Q_{2n},\INT,F}
  \leq 2^{(\ell^2+4\ell+1)/2}\cdot n^{-p-1/2} 
\end{equation}
for all $n\in\IN$, where
$p=r/(2+(\ln d)/(\ln 2\pi))$ and $\ell=\lceil 2p\rceil$.
For the example $d=2000$ and $r=6$,
we obtain
$$
 \err(Q_{2n},\INT,F)\leq 91\cdot n^{-1.47}. 
$$
For one million samples, the estimate~(\ref{Frolov bound})
for Frolov's algorithm is larger than one,
the estimate~(\ref{Standard bound}) for the direct simulation 
gives the error $10^{-3}$
and the estimate~(\ref{preasymptotic bound mix integration}) for our new algorithm
gives an error smaller than $4\times10^{-7}$. 
\end{ex}

\begin{rem}[Implementation]
We are able to implement Algorithm~\ref{alg:explicit alg int}
under the following assumptions:
\begin{itemize}
 \item We can implement Algorithm~\ref{alg:explicit alg}.
 This issue is discussed in Section~\ref{sec:examples OptimalMC}.
 \item We know the integrals $\INT(b_j)$ of the eigenfunctions $b_j\in\mathcal{B}$
 for all $j\leq n$.
 \item We can sample from the probability distribution $\mu$.
\end{itemize}
In the above example, the implementation is comparably easy,
since $\mathcal{B}$ is the Fourier basis and
all the random variables are independent and uniformly distributed on the unit cube.
\end{rem}

\chapter{Tractability of the Uniform Approximation Problem}
\label{chap:tractability uniform approximation}

We study the task of approximating a
function $f:[0,1]^d\to \IR$ in the uniform norm
with a deterministic scheme
based on a finite number of function values.
As a priori knowledge, we assume that the function is contained in
some class $F$ of smooth functions. 
In the notation of Chapter~\ref{chap:problems and algorithms}
we consider the problem
$$
 \mathcal{P}[\APP,F]=\mathcal{P}[\APP,F,\mathcal{B}([0,1]^d),\lstd,\mathrm{det},\mathrm{wc}],
$$
where $\APP:F\to \mathcal{B}([0,1]^d)$ is given by $\APP(f)=f$
and $\mathcal{B}([0,1]^d)$ is the set of bounded real valued functions on $[0,1]^d$.
We first study the classes $F=\C^r_d$
of real-valued functions on $[0,1]^d$ 
whose partial derivatives up to order $r\in\IN$
are continuous and bounded by~1,
see Section~\ref{sec:ck functions}.
We derive new results on the complexity of the approximation problem 
and compare with known results on the complexity of
the corresponding integration problem
$$
 \mathcal{P}[\INT,F]=\mathcal{P}[\INT,F,\IR,\lstd,\mathrm{det},\mathrm{wc}],
$$
where $\INT:F\to \IR$ is given by $\INT(f)=\int_{[0,1]^d} f(\mathbf x)~\d\mathbf{x}$.
For both problems, the complexity grows super-exponentially with the dimension. 
In particular, the problems suffer from the curse of dimensionality.
This section is based on~\cite{Kr19}.

Section~\ref{sec:rank one} is based on~\cite{KR19}. 
We show that the curse of dimensionality can
be avoided if $F$ is a class of rank one tensors.
The same observations hold for the problem of global optimization
on $F$, as explained in Section~\ref{sec:optimization}.
Finally, Section~\ref{sec:dispersion} is concerned with the problem of dispersion,
which is closely related to the uniform approximation of rank one tensors.
This section is based on~\cite{Kr18b}.

\section{Smooth Functions}
\label{sec:ck functions}

It is known that the integration of functions from the class
\begin{equation*}
 \C^r_d= \set{f \in \C^r\brackets{[0,1]^d} \,\big\vert\,
 \Vert \diff^\beta f \Vert_\infty \leq 1 
 \text{\ for all } \beta \in \IN_0^d \text{ with  } \abs{\beta}\leq r}
\end{equation*}
suffers from the curse of dimensionality.
In fact, the minimal number $\comp(\varepsilon,\mathcal{P}[\INT,\C^r_d])$
of function values that is needed to guarantee an integration error
$\varepsilon \in (0,1/2)$ for any function from the class $\C^r_d$
grows super-exponentially with the dimension.
It is proven in \cite{HNUW17} that there are positive constants $c_r$
and $C_r$ such that
\begin{equation*}
 \brackets{c_r\, d^{1/r} \varepsilon^{-1/r}}^d
 \leq \comp(\varepsilon,\mathcal{P}[\INT,\C^r_d])
 \leq \brackets{C_r\, d^{1/r} \varepsilon^{-1/r}}^d
 \end{equation*}
for all $\varepsilon \in (0,1/2)$ and $d\in\IN$.
Roughly speaking $\comp(\varepsilon,\mathcal{P}[\INT,\C^r_d])$
is of order $(d/ \varepsilon)^{d/r}$.
Since an $\varepsilon$-approximation of the function
immediately yields an $\varepsilon$-approximation of its 
integral, the uniform recovery
of functions from $\C^r_d$ can only be harder.
But how hard is the uniform recovery problem?
Is it significantly harder than the integration problem?
These questions were recently posed in \cite[Section~6]{Wo18}.

In the case $r=1$ the answer is known.
In this case the minimal number 
$\comp(\varepsilon,\mathcal{P}[\APP,\C^r_d])$
of function values that is needed to guarantee an approximation error
$\varepsilon>0$ for any function from $\C^r_d$
in the uniform norm behaves similarly to 
$\comp(\varepsilon,\mathcal{P}[\INT,\C^r_d])$.
There are positive constants $c$
and $C$ such that
\begin{equation*}
 \brackets{c\, d\, \varepsilon^{-1}}^d
 \leq \comp(\varepsilon,\mathcal{P}[\APP,\C^1_d]) \leq
 \brackets{C\, d\, \varepsilon^{-1}}^d
\end{equation*}
for all $\varepsilon \in (0,1/2)$ and $d\in\IN$.
This result is basically contained in \cite{Su78}.
Nonetheless, we will present its proof.
If $r\geq 2$ is even, we obtain the following result.

\begin{thm}[\cite{Kr19}]
\label{main theorem even}
Let $r\in\IN$ be even. 
Then there are constants $c_r,C_r,\varepsilon_r>0$
such that
\begin{equation*}
 \brackets{c_r \sqrt{d}\, \varepsilon^{-1/r}}^d
 \leq \comp(\varepsilon,\mathcal{P}[\APP,\C^r_d]) \leq
 \brackets{C_r \sqrt{d}\, \varepsilon^{-1/r}}^d
\end{equation*}
for all $d\in\IN$ and $\varepsilon\in (0,\varepsilon_r)$.
The upper bound holds for all $\varepsilon>0$.
\end{thm}

Roughly speaking $\comp(\varepsilon,\mathcal{P}[\APP,\C^r_d])$
is of order $(d^{r/2} /\varepsilon)^{d/r}$.
If the error tolerance $\varepsilon$ is fixed,
the complexity grows like $d^{d/2}$.
This is in contrast to the case $r=1$,
where we have a growth of order $d^d$.
If $r\geq 3$ is odd, 
we only have a partial result.

\begin{thm}[\cite{Kr19}]
\label{main theorem even odd}
Let $r\geq 3$ be odd. 
Then there are constants $c_r,C_r,\varepsilon_r>0$
such that
\begin{equation*}
 \brackets{c_r \sqrt{d}\, \varepsilon^{-1/r}}^d
 \leq \comp(\varepsilon,\mathcal{P}[\APP,\C^r_d]) \leq
 \brackets{C_r\, d^{\frac{r+1}{2r}} \varepsilon^{-1/r}}^d
\end{equation*}
for all $d\in\IN$ and $\varepsilon\in (0,\varepsilon_r)$.
The upper bound holds for all $\varepsilon>0$.
\end{thm}

We point to the fact that 
$\comp(\varepsilon,\mathcal{P}[\APP,\C^r_d]) 
\leq \comp(\varepsilon,\mathcal{P}[\APP,\C^{r-1}_d])$
since the upper bound resulting 
from Theorem~\ref{main theorem even}
may improve on the upper bound 
of Theorem~\ref{main theorem even odd}
for $d\succ \varepsilon^{-2/(r-1)}$
if $r\geq 3$ is odd.
In~this case, we do not know the exact behavior of
$\comp(\varepsilon,\mathcal{P}[\APP,\C^r_d])$ as a function
of both $d$ and $\varepsilon$.
If regarded as a function of $\varepsilon$,
the complexity is of order $\varepsilon^{-d/r}$.
If regarded as a function of $d$, 
it is of order $d^{d/2}$.

Altogether, our results justify the following comparison.

\begin{cor}[\cite{Kr19}]
 The uniform recovery problem on the class $\C^r_d$ is
 significantly harder than the integration problem 
 if and only if $r\geq 3$.
\end{cor}

Aside from the case $r=1$, the lower bounds in 
Theorem~\ref{main theorem even} and Theorem~\ref{main theorem even odd}
even hold for the smaller class
\begin{equation*}
 \widetilde\C^r_d= \set{f \in \C^r\brackets{[0,1]^d} \,\big\vert\,
 \Vert \partial_{\theta_1}\cdots\partial_{\theta_\ell} f \Vert_\infty \leq 1 
 \text{\ for all } \ell \leq r \text{ and  } \theta_i\in \mathbb S_{d-1}}
\end{equation*}
of functions whose directional derivatives up to order $r\in\IN$
are bounded by one.
For this class, we obtain sharp bounds on the $\varepsilon$-complexity
of the uniform recovery problem for any $r\in\IN$.
The minimal number $\comp(\varepsilon,\mathcal{P}[\APP,\widetilde\C^r_d])$ of
function values that is needed to guarantee
an approximation error $\varepsilon$ for every function
from $\widetilde\C^r_d$ in the uniform norm satisfies the following.

\begin{thm}[\cite{Kr19}]
\label{side theorem}
Let $r\in\IN$. 
There are constants $c_r,C_r,\varepsilon_r>0$
such that
\begin{equation*}
 \brackets{c_r \sqrt{d}\, \varepsilon^{-1/r}}^d
 \leq \comp(\varepsilon,\mathcal{P}[\APP,\widetilde\C^r_d]) \leq
 \brackets{C_r \sqrt{d}\, \varepsilon^{-1/r}}^d
\end{equation*}
for all $d\in\IN$ and $\varepsilon\in (0,\varepsilon_r)$.
The upper bound holds for all $\varepsilon>0$.
\end{thm}

These theorems also imply new results on the
complexity of global optimization.
We shortly discuss this problem in Section~\ref{sec:optimization}. 
In Sections~\ref{sec:upper bounds ck} and \ref{sec:lower bounds ck}
we prove the upper and lower bounds of Theorems~\ref{main theorem even},
\ref{main theorem even odd} and \ref{side theorem}.
Before we turn to the proofs,
we comment on some related problems.

\begin{rem}[Infinite smoothness]
 It is proven in \cite{NW09}
 that even the uniform recovery of functions from
 \begin{equation*}
  \C^\infty_d = 
  \set{f \in \C^\infty\brackets{[0,1]^d} \,\big\vert\,
 \Vert D^\beta f \Vert_\infty \leq 1 
 \text{\ for all } \beta \in \IN_0^d}
 \end{equation*}
 suffers from the curse of dimensionality.
 This cannot be avoided even if we allow randomized
 algorithms that may evaluate arbitrary
 continuous linear functionals~\cite[Section~2.4.2]{Ku17}.
 In fact, we have seen that the complexity $\comp(\varepsilon,\mathcal{P}[\APP,\C^r_d])$
 depends super-exponentially on the dimension for $r\in\IN$.
 It would be interesting to verify whether
 this is also true for $r=\infty$ and randomized algorithms.
 We remark that the uniform recovery problem does not suffer from the curse
 if the target function lies within the modified class
 \begin{equation*}
  \overline{\C^\infty_d} = 
  \bigg\{f \in \C^\infty\brackets{[0,1]^d} \,\Big\vert\,
 \sum_{\abs{\beta}=k} \frac{\Vert D^\beta f \Vert_\infty}{\beta !} \leq 1 
 \text{\ for all } k \in \IN_0\bigg\}
 \end{equation*}
 of smooth functions. This is proven in \cite{Vy14}.
\end{rem}

\begin{rem}[Algorithms]
 This section is not concerned with explicit algorithms.
 Nonetheless, our proof shows that there are optimal algorithms in the sense
 of Theorem~\ref{main theorem even}, \ref{main theorem even odd} and 
 \ref{side theorem} whose information is given by function values
 at a regular grid and small clouds around the grid points.
 This information can be used for a subcubewise Taylor approximation
 of the target function around the grid points,
 where the partial derivatives of order less than $r$ 
 are replaced by divided differences.
 The resulting algorithm is indeed optimal for the class $\widetilde \C^r_d$.
 However, the author does not know whether it is also optimal for $\C^r_d$.
\end{rem}

\begin{rem}[Other domains]
 Our lower bounds are still valid, 
 if the domains $[0,1]^d$ are replaced
 by any other sequence of domains $D_d\subset\IR^d$ 
 that satisfies $\lambda^d(D_d)\geq a^d$ for some $a>0$
 and all $d\in\IN$.
 The upper bounds, however, heavily exploit the geometry of the unit cube.
 The curse of dimensionality for 
 the integration problem on general domains
 is studied in the recent paper~\cite{HPU18}.
\end{rem}

\begin{rem}[Integration on the class $\widetilde\C^r_d$]
 The precise behavior 
 of the $\eps$-complexity 
 of the integration problem on $\widetilde\C^r_d$
 as a function of both $d$ and $\varepsilon$ is still open.
\end{rem}

\subsection{Upper bounds}
\label{sec:upper bounds ck}

Let $F\in\{\C^r_d,\widetilde \C^r_d\}$.
These classes are convex and symmetric
and $\mathcal{P}[\APP,F]$ is a linear problem.
Since we measure the error in $\mathcal{B}([0,1]^d)$,
Theorem~\ref{thm:error formula} yields that
\begin{equation}
\label{eq:bakhvalov for APP}
 \e(n,\mathcal{P}[\APP,F]) =
 \inf_{\substack{P\subset [0,1]^d\\ \card(P)\leq n}}
 \sup_{\substack{f\in F\\ f\mid_P=0}} \norm{f}_\infty.
\end{equation}
Therefore, if we want
to derive an upper bound on the $n^{\rm th}$ minimal error,
we may choose any point set $P$
with cardinality at most $n$ and 
give an upper bound on the 
maximal value of a function $f\in F$
that vanishes on $P$.
In fact, we can choose any 
point set $Q$ with cardinality at most $n/(d+1)^{r-1}$
and assume that not only $f$ but all its derivatives
of order less than $r$ are arbitrarily 
small on $Q$.
We start with the case $F=\C^r_d$.
More precisely, for any $\delta>0$, any $r\in\IN_0$, $d\in\IN$ and $Q\subset [0,1]^d$,
we define the subclasses
\begin{equation*}
 \C^r_d(Q,\delta) =
 \set{f\in\C^r_d \,\big\vert\, \abs{D^\alpha f(\mathbf x)}\leq \delta^{2^{r-\abs{\alpha}-1}}
 \text{ for all } \mathbf x\in Q \text{ and } \abs{\alpha} < r}
\end{equation*}
and the auxiliary quantities
\begin{equation*}
 E\brackets{Q,\C^r_d,\delta}=
 \sup\limits_{f\in \C^r_d(Q,\delta)} \norm{f}_\infty
 \qquad\text{and}\qquad
 E\brackets{Q,\C^r_d}= 
 \lim\limits_{\delta\downarrow 0} E\brackets{Q,\C^r_d,\delta}
\end{equation*}
and obtain the following.

\begin{lemma}
\label{small derivatives}
Let $d\in \IN$, $r\in\IN$ and $n\in \IN_0$.
If the cardinality of $Q\subset [0,1]^d$
is at most $n/(d+1)^{r-1}$, then
 \begin{equation*}
  \e(n,\mathcal{P}[\APP,\C^r_d]) \leq E\brackets{Q,\C^r_d}.
 \end{equation*}
\end{lemma}


\begin{proof}
 Let $\delta \in (0,1)$.
 We will construct a point set $P\subset [0,1]^d$
 with cardinality at most $n$ such that
 any $f\in\C^r_d$ with $f\vert_P=0$ is contained
 in $\C^r_d(Q,\delta)$.
 Then equation~\eqref{eq:bakhvalov for APP} yields
 \begin{equation*}
  \e(n,\mathcal{P}[\APP,\C^r_d]) \leq
  \sup\limits_{f\in \C^r_d:\, f\vert_P=0} \norm{f}_\infty
  \leq \sup\limits_{f\in \C^r_d(Q,\delta)} \norm{f}_\infty
  = E\brackets{Q,\C^r_d,\delta}.
 \end{equation*}
 Taking the limit for $\delta\to 0$ yields the statement.
 
 If $r=1$, we can choose $P=Q$.
 Let us start with the case $r=2$.
 Given a set $M\subset [0,1]^d$ and $h\in (0,1/2]$, we define
 \begin{equation*}
  M[h]= M 
  \cup\bigcup_{\substack{(\mathbf x,j)\in M\times\set{1,\dots ,d}\\ \mathbf x+h\mathbf e_j\in [0,1]^d}} \set{\mathbf x+h\mathbf e_j}\
  \cup\bigcup_{\substack{(\mathbf x,j)\in M\times\set{1,\dots ,d}\\ \mathbf x+h\mathbf e_j\not\in [0,1]^d}} \set{\mathbf x-h\mathbf e_j}
 .\end{equation*}
 Obviously, the cardinality of $M[h]$ is at most $(d+1)\abs{M}$.
 Furthermore, we have 
 \begin{equation}
 \label{mean value}
   f\in\C^2_d\text{ with }\abs{f}\leq h^2\text{ on }M[h]\quad
   \Rightarrow\quad
    \abs{\frac{\partial f}{\partial x_j}} \leq 3h\text{ on }M\text{ for }j=1,\dots ,d
 .\end{equation}
 This is a simple consequence of the mean value theorem:
 For any $j\in\set{1,\dots ,d}$ and $\mathbf x\in M$ 
 with $\mathbf x+h\mathbf e_j\in[0,1]^d$ there is some $\eta\in (0,h)$ with
 \begin{equation*}
  \abs{\frac{\partial f}{\partial x_j}\brackets{\mathbf x+\eta \mathbf e_j}}
  = \abs{\frac{f\brackets{\mathbf x+h \mathbf e_j}-f(\mathbf x)}{h}}
  \leq 2h.
 \end{equation*}
 The same estimate holds for some $\eta\in(-h,0)$, if $\mathbf x+h\mathbf e_j\not\in[0,1]^d$.
 The fundamental theorem of calculus yields
 \begin{equation*}
  \abs{\frac{\partial f}{\partial x_j}\brackets{\mathbf x}}
  \leq \abs{\frac{\partial f}{\partial x_j}\brackets{\mathbf x+\eta \mathbf e_j}}
  + \abs{\eta}\cdot \max\limits_{\abs{t}\leq\eta} \abs{\frac{\partial^2 f}{\partial x_j^2}\brackets{\mathbf x+t\mathbf e_j}}
  \leq 3h
 .\end{equation*}
 This means that we can choose $P=Q\left[\delta/3\right]$.
 
 For $r>2$ we repeat this procedure $r-1$ times.
 We use the notation
 \begin{equation*}
  M\left[h_1,\dots,h_{i}\right]=M\left[h_1,\dots,h_{i-1}\right]\left[h_i\right]
 \end{equation*}
 for $i>1$. We choose the point set
 \[
 P=Q\left[h_1,\dots,h_{r-1}\right], 
 \quad\text{where}\quad
 h_i=3(\delta/9)^{2^{i-1}}
 \]
 for $i=1,\hdots ,r-1$.
 Note that $3h_i=h_{i-1}^2$ for each $i\geq 2$.
 Clearly, the cardinality of $P$ is at most $(d+1)^{r-1}\abs{Q}$
 and hence bounded by $n$.
 Let $f\in\C^r_d$ vanish on $P$ and let 
 $\frac{\partial^\ell f}{\partial x_{j_1}\dots\partial x_{j_\ell}}$
 be any derivative of order $\ell <r$.
 Fact (\ref{mean value}) yields:
 \begin{align*}
   &f\in C^r_d
   &&\text{ with }\abs{f}=0\leq h_{r-1}^2
   &&\text{ on }Q\left[h_1,\dots,h_{r-1}\right]\\
   \Rightarrow\
   &\frac{\partial f}{\partial x_{j_1}}\in C^{r-1}_d
   &&\text{ with }\abs{\frac{\partial f}{\partial x_{j_1}}}\leq 3h_{r-1} = h_{r-2}^2
   &&\text{ on }Q\left[h_1,\dots,h_{r-2}\right]\\
   \Rightarrow\
   &\frac{\partial^2 f}{\partial x_{j_1}\partial x_{j_2}}\in C^{r-2}_d
   &&\text{ with }\abs{\frac{\partial^2 f}{\partial x_{j_1}\partial x_{j_2}}}\leq 3h_{r-2}
   &&\text{ on }Q\left[h_1,\dots,h_{r-3}\right]\\
   \Rightarrow\ & &&\hdots &&\\
   \Rightarrow\
   &\frac{\partial^\ell f}{\partial x_{j_1}\dots\partial x_{j_\ell}}\in C^{r-\ell}_d
   &&\text{ with }\abs{\frac{\partial^\ell f}{\partial x_{j_1}\dots\partial x_{j_\ell}}}\leq 3h_{r-\ell}
   &&\text{ on }Q\left[h_1,\dots,h_{r-\ell-1}\right]
 .\end{align*}
 Since $Q\subset Q\left[h_1,\dots,h_{r-\ell-1}\right]$ and 
 $3h_{r-\ell}\leq \delta^{2^{r-\ell-1}}$, the lemma is proven.
\end{proof}

We can prove the desired upper bounds on $e\brackets{n,\C^r_d}$
by choosing $Q$ as a regular grid. 
We set
\begin{equation*}
 Q_m^d=\set{0,1/m,2/m,\hdots,1}^d
\end{equation*}
for $m\in\IN$. The following recursive formula is crucial.

\begin{lemma}
 \label{two step induction}
 Let $m\in \IN$, $d\geq 2$ and $r\geq 2$. Then
 \begin{equation*}
  E\brackets{Q_m^d,\C^r_d}
  \leq E\brackets{Q_m^{d-1},\C^r_{d-1}} + 
  \frac{1}{8m^2}\, E\brackets{Q_m^d,\C^{r-2}_d}.
 \end{equation*}
\end{lemma}

\begin{proof}
 We will prove for any $\delta>0$ that
 \begin{equation}
 \label{recursive formula}
  E\brackets{Q_m^d,\C^r_d,\delta}
  \leq E\brackets{Q_m^{d-1},\C^r_{d-1},\delta} + 
  \frac{1}{8m^2}\, E\brackets{Q_m^d,\C^{r-2}_d,\delta}.
 \end{equation}
 Taking the limit for $\delta\to 0$ yields the statement.
  
 Let $f\in \C^r_d\brackets{Q_m^d,\delta}$.
 We need to show that $\norm{f}_\infty$ is bounded
 by the right hand side of \eqref{recursive formula}.
 Since $f$ is continuous, there is some $\mathbf z\in[0,1]^d$
 such that $\abs{f(\mathbf z)}=\norm{f}_\infty$.
 We distinguish two cases.
 
 If $z_d\in \set{0,1}$,
 the restriction $f\vert_{H}$ of $f$ to the hyperplane
 \begin{equation*}
  H=\set{\mathbf x\in[0,1]^d \mid x_d=z_d}
 \end{equation*}
 is contained in $\C^r_{d-1}\brackets{Q_m^{d-1},\delta}$.
 This implies that
 \begin{equation*}
  \abs{f(\mathbf z)}=\norm{f\vert_H}_\infty \leq E\brackets{Q_m^{d-1},\C^r_{d-1},\delta}
 \end{equation*}
 and the statement is proven.
 
 Let us now assume that $z_d\in (0,1)$.
 Then we have $\frac{\partial f}{\partial x_d}(\mathbf z)=0$.
 We choose $\mathbf y\in [0,1]^d$ such that $y_j=z_j$ for $j<d$
 and $y_d\in Q_m$ with $\abs{y_d-z_d}\leq 1/(2m)$.
 The restriction $f\vert_{H^\prime}$ of $f$ to the hyperplane
 \begin{equation*}
  H^\prime=\set{\mathbf x\in[0,1]^d \mid x_d=y_d}
 \end{equation*}
 is contained in $\C^r_{d-1}\brackets{Q_m^{d-1},\delta}$.
 This implies that
 \begin{equation*}
  \abs{f(\mathbf y)}=\norm{f\vert_{H^\prime}}_\infty \leq E\brackets{Q_m^{d-1},\C^r_{d-1},\delta}.
 \end{equation*}
 Moreover, the second derivative $\frac{\partial^2 f}{\partial x_d^2}$
 is contained in $\C^{r-2}_d\brackets{Q_m^d,\delta}$ and hence
 \begin{equation*}
  \norm{\frac{\partial^2 f}{\partial x_d^2}}_\infty \leq 
  E\brackets{Q_m^d,\C^{r-2}_d,\delta}.
 \end{equation*}
 By Taylor's theorem, there is some $\mathbf a$ on the line segment between $\mathbf y$ and $\mathbf z$
 such that 
 \begin{equation*}
  f(\mathbf y)= f(\mathbf z) + \frac{1}{2} \frac{\partial^2 f}{\partial x_d^2}(\mathbf a)
  \cdot (y_d-z_d)^2.
 \end{equation*}
 We obtain
 \begin{equation*}
 \begin{split}
  \abs{f(\mathbf z)}\, &\leq\, \abs{f(\mathbf y)} \,+\,
  \frac{(y_d-z_d)^2}{2}\cdot \Big\Vert\frac{\partial^2 f}{\partial x_d^2}\Big\Vert_\infty \\
  &\leq\, E\brackets{Q_m^{d-1},\C^r_{d-1},\delta} \,+\, 
  \frac{1}{8m^2}\cdot E\brackets{Q_m^d,\C^{r-2}_d,\delta},
 \end{split}
 \end{equation*}
 as it was to be proven.
\end{proof}

By a double induction on $r$ and $d$ we obtain the following
result for even $r$.

\begin{lemma}
 \label{even r lemma}
 Let $d\in\IN$, $m\in\IN$ and $r\in\IN_0$ be even. Then
 \begin{equation*}
  E\brackets{Q_m^d,\C^r_d} \leq
  \frac{e d^{r/2}}{(2m)^r}.
 \end{equation*}
\end{lemma}

\begin{proof}
 We give a proof by induction on $d$.
 Let $\delta>0$ and $f\in \C^r_1\brackets{Q_m,\delta}$
 for some even number $r$.
 Since $f$ is continuous, there is some $z\in[0,1]$
 such that $\abs{f(z)}=\norm{f}_\infty$.
 Let $y\in Q_m$ with $\abs{y-z}<1/(2m)$.
 By Taylor's theorem, there is some $\xi$
 between $y$ and $z$ such that
 $$
 f(z)=\sum_{k=0}^{r-1}\frac{f^{(k)}(y)}{k!} (z-y)^k 
 + \frac{f^{(r)}(\xi)}{r!} (z-y)^r .
 $$
 Using that $\vert f^{(k)}(y)\vert\leq \delta^{2^{r-k-1}} \leq \delta^{r-k}$,
 we obtain for $\delta \leq 1/(2m)$ that
 $$
 \norm{f}_\infty \leq \sum_{k=0}^r\frac{\delta^{r-k}}{k!} \brackets{\frac{1}{2m}}^k 
 \leq \brackets{\frac{1}{2m}}^r \sum_{k=0}^r\frac{1}{k!}
 \leq \frac{e}{(2m)^r}.
 $$
 Since this is true for any such $f$ and any $\delta\leq 1/(2m)$,
 this proves the case $d=1$.
 
 Now let $d\geq 2$. We assume that the statement holds for every dimension
 smaller than $d$.
 To show that it also holds in dimension $d$,
 we use induction on $r$.
 For $r=0$ the statement is trivial since $E(Q_m^d,\C^0_d)=1$.
 Let $r\geq 2$ be even and assume that the statement
 holds in dimension $d$ for any even smoothness smaller than $r$.
 Lemma~\ref{two step induction} yields
 \begin{align*}
  E\brackets{Q_m^d,\C^r_d}
  \leq \frac{e (d-1)^{r/2}}{(2m)^r} + 
  \frac{1}{8m^2}\, \frac{e d^{r/2-1}}{(2m)^{r-2}}\\
  = \frac{e d^{r/2}}{(2m)^r}
  \brackets{\brackets{1-\frac{1}{d}}^{r/2}+\frac{1}{2d}}
  \leq \frac{e d^{r/2}}{(2m)^r} ,
 \end{align*}
 which completes the inner and therefore the outer induction.
\end{proof}

This immediately yields the upper bound of 
Theorem~\ref{main theorem even}.

\begin{proof}[Proof of Theorem~\ref{main theorem even} (Upper Bound)]
 Let $d\in\IN$, $r\in\IN$ be even and $\varepsilon >0$. We set
 $$
 n=(d+1)^{r-1} (m+1)^d, \quad\text{where}\quad
 m=\left\lceil\frac{e^{1/r}}{2} \sqrt{d} \varepsilon^{-1/r}\right\rceil.
 $$
 Lemmas~\ref{small derivatives} and \ref{even r lemma} yield
 $$
 \e(n,\mathcal{P}[\APP,\C^r_d]) 
 \leq E\brackets{Q_m^d,\C^r_d}
 \leq \frac{e d^{r/2}}{(2m)^r}
 \leq \varepsilon.
 $$
 Hence,
 $$
 \comp(\varepsilon,\mathcal{P}[\APP,\C^r_d]) \leq n
 $$
 and this implies the result.
\end{proof}

To derive the upper bounds for odd $r$,
we use the following recursive formula.

\begin{lemma}
 \label{odd r lemma}
 Let $m\in \IN$, $d\in \IN$ and $r\in\IN$. Then
 \begin{equation*}
  E\brackets{Q_m^d,\C^r_d}
  \leq \frac{d}{2m}\,
  E\brackets{Q_m^d,\C^{r-1}_d}.
 \end{equation*}
\end{lemma}

\begin{proof}
 It suffices to show for any $\delta>0$ that
 \begin{equation*}
  E\brackets{Q_m^d,\C^r_d,\delta}
  \leq \delta^{2^{r-1}} + \frac{d}{2m}\,
  E\brackets{Q_m^d,\C^{r-1}_d,\delta} .
 \end{equation*}
 Taking the limit for $\delta\to 0$ yields the statement.
 Let $f\in\C^r_d(Q_m^d,\delta)$ and
 let $\mathbf z\in[0,1]^d$ such that $\abs{f(\mathbf z)}=\norm{f}_\infty$.
 There is some $\mathbf y\in Q_m^d$ such that $\mathbf y$ and $\mathbf z$
 are connected by an axis-parallel polygonal chain
 of length at most $d/(2m)$.
 For every $j\in\set{1,\dots,d}$, the partial derivative
 $\partial f/\partial x_j$ is 
 contained in $\C^{r-1}_d(Q_m^d,\delta)$.
 Integrating along the curve yields
 $$
 \abs{f(\mathbf z)} \leq 
 \abs{f(\mathbf y)} + \frac{d}{2m} \max\limits_{j=1\hdots d} 
 \Big\Vert\frac{\partial f}{\partial x_j}\Big\Vert_\infty
 \leq \delta^{2^{r-1}} + \frac{d}{2m} E\brackets{Q_m^d,\C^{r-1}_d,\delta}.
 $$
 This proves the lemma.
\end{proof}

Now the upper bounds of Theorem~\ref{main theorem even odd}
follow from the results for even $r$.
Note that the upper bound for $r=1$ is included.

\begin{proof}[Proof of Theorem~\ref{main theorem even odd} (Upper Bound)]
 Let $d\in\IN$, $r\in\IN$ be odd and $\varepsilon>0$.
 For any $m\in\IN$,
 Lemma~\ref{even r lemma} and \ref{odd r lemma} yield
 $$
 E\brackets{Q_m^d,\C^r_d} \leq
  \frac{e d^{(r+1)/2}}{(2m)^r}.
 $$
 We set
 $$
 n=(d+1)^{r-1} (m+1)^d, \quad\text{where}\quad
 m=\left\lceil\frac{e^{1/r}}{2} d^{\frac{r+1}{2r}} \varepsilon^{-1/r}\right\rceil.
 $$
 We obtain
 $$
 \e(n,\mathcal{P}[\APP,\C^r_d]) 
 \leq E\brackets{Q_m^d,\C^r_d}
 \leq \varepsilon
 $$
 and hence
 $$
 \comp(\varepsilon,\mathcal{P}[\APP,\C^r_d]) \leq n,
 $$
 as it was to be proven.
\end{proof}

We proceed similarly
to prove of the upper bound of Theorem~\ref{side theorem}.
For any $\delta>0$, any $r\in\IN_0$, $d\in\IN$ and $Q\subset [0,1]^d$,
we define the subclasses
\begin{equation*}
 \widetilde\C^r_d(Q,\delta) =
 \set{f\in\widetilde\C^r_d \,\big\vert\,
 \abs{\partial_{\theta_1}\cdots\partial_{\theta_\ell} f(\mathbf x)}
 \leq \delta^{2^{r-\ell-1}}
 \text{ for } \mathbf x\in Q,  \ell < r, \theta_1\hdots\theta_\ell\in \mathbb S_{d-1}}
\end{equation*}
and the auxiliary quantities
\begin{equation*}
 E\brackets{Q,\widetilde\C^r_d,\delta}=
 \sup\limits_{f\in \widetilde\C^r_d(Q,\delta)} \norm{f}_\infty
 \qquad\text{and}\qquad
 E\brackets{Q,\widetilde\C^r_d}= 
 \lim\limits_{\delta\downarrow 0} E\brackets{Q,\widetilde\C^r_d,\delta}
\end{equation*}
and obtain the following.

\begin{lemma}
\label{small derivatives 2}
Let $d,r\in \IN$ and $n\in\IN_0$.
If the cardinality of $Q\subset [0,1]^d$
is at most $n/(d+1)^{r-1}$, then
\begin{equation*}
 \e(n,\mathcal{P}[\APP,\C^r_d]) 
 \leq E\brackets{Q,\widetilde\C^r_d}.
\end{equation*}
\end{lemma}

\begin{proof}
 Let $\delta \in (0,1)$.
 In the proof of Lemma~\ref{small derivatives}
 we constructed a point set $P$
 with cardinality at most $n$ such that
 any $f\in\C^r_d$ with $f\vert_P=0$ is contained
 in $\C^r_d(Q,\delta)$.
 In particular, any $f\in\widetilde\C^r_d$ with $f\vert_P=0$
 satisfies $\abs{D^\alpha f(\mathbf x)}\leq \delta^{2^{r-\abs{\alpha}-1}}$
 for all $\mathbf x\in Q$ and $\abs{\alpha} < r$.
 Taking into account that for $\mathbf x\in [0,1]^d$ and $\ell < r$ we have
 $$
 \abs{\partial_{\theta_1}\cdots\partial_{\theta_\ell} f(\mathbf x)}
 \leq d^{\ell/2} \max\limits_{\abs{\alpha}=\ell} \abs{D^\alpha f(\mathbf x)} ,
 $$
 we obtain that $f\in \widetilde\C^r_d(Q,d^{\frac{r-1}{2}} \delta)$ and hence
 $$
 \e(n,\mathcal{P}[\APP,\C^r_d]) \leq
  \sup\limits_{f\in \widetilde\C^r_d:\, f\vert_P=0} \norm{f}_\infty
  \leq \sup\limits_{f\in \widetilde\C^r_d(Q,d^{(r-1)/2} \delta)} \norm{f}_\infty
  = E\brackets{Q,\widetilde\C^r_d,d^{\frac{r-1}{2}}\delta}.
 $$
 Taking the limit for $\delta\to 0$ yields the statement.
\end{proof}

For these classes, 
it is enough to consider the following
single-step recursion.

\begin{lemma}
 \label{recursion lemma 2}
 Let $m\in \IN$, $d\in\IN$ and $r\in\IN$. Then
 \begin{equation*}
  E\brackets{Q_m^d,\widetilde\C^r_d}
  \leq \frac{\sqrt{d}}{2m}\,
  E\brackets{Q_m^d,\widetilde\C^{r-1}_d}.
 \end{equation*}
\end{lemma}

\begin{proof}
 It suffices to show for any $\delta>0$ that
 \begin{equation*}
  E\brackets{Q_m^d,\widetilde\C^r_d,\delta}
  \leq \delta^{2^{r-1}} + \frac{\sqrt{d}}{2m}\,
  E\brackets{Q_m^d,\widetilde\C^{r-1}_d,\delta} .
 \end{equation*}
 To this end, let $f\in\widetilde\C^r_d(Q_m^d,\delta)$ and
 let $\mathbf z\in[0,1]^d$ such that $\abs{f(\mathbf z)}=\norm{f}_\infty$.
 There is some $\mathbf y\in Q_m^d$ such that $\mathbf y$ and $\mathbf z$
 are connected by a line segment
 of length at most $\sqrt{d}/(2m)$.
 Let $\theta = \mathbf z-\mathbf y/\norm{\mathbf z-\mathbf y}_2$.
 Then we have $\partial_\theta f\in\widetilde\C^{r-1}_d(Q_m^d,\delta)$.
 Integrating along the line yields
 $$
 \abs{f(\mathbf z)} \leq 
 \abs{f(\mathbf y)} + \frac{\sqrt{d}}{2m}
 \norm{\partial_\theta f}_\infty
 \leq \delta^{2^{r-1}} + \frac{\sqrt{d}}{2m} E\brackets{Q_m^d,\widetilde\C^{r-1}_d,\delta}.
 $$
 Taking the limit for $\delta\to 0$ yields the statement.
\end{proof}

The upper bound of Theorem~\ref{side theorem}
can now be proven by induction on $r$.

\begin{proof}[Proof of Theorem~\ref{side theorem} (Upper Bound)]
Lemma~\ref{recursion lemma 2} and $E(Q_m^d,\widetilde\C^0_d)=1$ yield
$$
E\brackets{Q_m^d,\widetilde\C^r_d} \leq 
\brackets{\frac{\sqrt{d}}{2m}}^r
$$
for any $m\in\IN$, $d\in\IN$ and $r\in\IN_0$.
Now let $d\in\IN$, $r\in\IN$ and $\varepsilon>0$.
We set
 $$
 n=(d+1)^{r-1} (m+1)^d, \quad\text{where}\quad
 m=\left\lceil\frac{1}{2} \sqrt{d} \varepsilon^{-1/r}\right\rceil.
 $$
 Lemma~\ref{small derivatives 2} yields
 $$
 \e(n,\mathcal{P}[\APP,\C^r_d]) 
 \leq E\brackets{Q_m^d,\widetilde\C^r_d}
 \leq \varepsilon
 $$
 and hence
 $$
 \comp(\varepsilon,\mathcal{P}[\APP,\widetilde\C^r_d]) \leq n,
 $$
 as it was to be proven.
\end{proof}

\subsection{Lower bounds}
\label{sec:lower bounds ck}

By equation~\eqref{eq:bakhvalov for APP}, we can estimate 
$\e(n,\mathcal{P}[\APP,\widetilde\C^r_d])$ from below as follows.
For any point set $P$ with cardinality at most $n$,
we construct a function $f\in \widetilde\C^r_d$ that vanishes on $P$
but has a large maximum in $[0,1]^d$,
a so-called fooling function.
We will use the following lemma.
Note that
$$
 \norm{f}_{r,d} =
 \sup\limits_{\ell \leq r, \theta_i\in \mathbb S_{d-1}}
 \Vert \partial_{\theta_1}\cdots\partial_{\theta_\ell} f \Vert_\infty
$$
defines a norm on the space of smooth functions $f:\IR^d\to \IR$ 
with compact support.

\begin{lemma}
\label{radial function}
 There exists a sequence $\brackets{g_d}_{d\in\IN}$ of infinitely
 differentiable functions 
 $g_d:\IR^d\to \IR$
 with support in the Euclidean unit ball that satisfy $g_d(\mathbf 0) =1$ and
 \begin{equation*}
  \sup\limits_{d\in\IN}\, \norm{g_d}_{r,d}\, < \infty
  \qquad \text{for all } \quad r\in\IN_0.
 \end{equation*}
\end{lemma}

\begin{proof}
 Take any function $h\in\C^\infty(\IR)$ which 
 equals 1 on $(-\infty,0]$ and 0 on $[1,\infty)$.
 Then the radial functions
 \begin{equation*}
  g_d: \IR^d \to \IR, \quad g_d(x)=h\brackets{\norm{x}_2^2}
 \end{equation*}
 for $d\in\IN$ have the desired properties.
 This follows from the fact that the directional derivative
 $\partial_{\theta_1}\cdots\partial_{\theta_r} g_d(\mathbf x)$
 only depends on the length of $\mathbf x$ and the angles
 between each pair of vectors $\theta_1,\hdots,\theta_r\in \mathbb S_{d-1}$ and $\mathbf x\in\IR^d$. 
 As soon as $d$ is large enough such that 
 all constellations of lengths and angles are possible,
 the norm $\norm{g_d}_{r,d}$ is independent of
 the dimension $d$.
\end{proof}

To obtain a suitable fooling function for a given point set $P$,
it is enough to shrink and shift the support of $g_d$
to the largest euclidean ball that does not intersect with $P$.
The radius of this ball can be estimated by a simple volume argument.

\begin{lemma}
\label{radius lower bound}
 Let $P\subset[0,1]^d$ be of cardinality $n\in\IN$.
 There exists $\mathbf z\in[0,1]^d$ such
 that for all $\mathbf x\in P$ we have
 \begin{equation*}
  \norm{\mathbf z-\mathbf x}_2 \geq \frac{\sqrt{d}}{5 n^{1/d}}.
 \end{equation*}
\end{lemma}

\begin{proof}
 The set
 \begin{equation*}
  B_R^2(P) = \bigcup_{\mathbf x \in P} B_R^2(\mathbf x)
 \end{equation*}
 of points within a distance $R>0$ of $P$ has the volume
 \begin{equation*}
  \lambda^d\brackets{B_R^2(P)}
  \leq n R^d\, \lambda^d\brackets{B_1^2(\mathbf 0)}
  = \frac{n R^d\, \pi^{d/2}}{\Gamma\brackets{\frac{d}{2}+1}}.
 \end{equation*}
 By Stirling's Formula, this can be estimated from above by
  \begin{equation*}
  \lambda^d\brackets{B_R^2(P)}
  \leq \frac{n R^d\, \pi^{d/2}}{\frac{\sqrt{2\pi}}{e} \brackets{\frac{d}{2e}}^{d/2}}
  \leq \brackets{ \frac{n^{1/d} e^{3/2}}{\sqrt{d}}\, R }^d
 .\end{equation*}
 If $R=\sqrt{d}/(5 n^{1/d})$, the volume is less than 1
 and $[0,1]^d\setminus B_R^2(P)$ must be nonempty.
\end{proof}

We are ready to prove the lower bound of Theorem~\ref{side theorem}.

\begin{proof}[Proof of Theorem~\ref{side theorem} (Lower Bound)]
Let $r\in\IN$, $d\in\IN$ and $n\in\IN$.
Let $P$ be any subset of $[0,1]^d$ with cardinality at most $n$.
Let $g_d$ be like in Lemma~\ref{radial function} and set
 \begin{equation*}
  K_r=\sup\limits_{d\in\IN}\, \norm{g_d}_{r,d}
  \qquad\text{and}\qquad
  R=\min\set{1,\frac{\sqrt{d}}{5 n^{1/d}}}.
 \end{equation*}
By Lemma~\ref{radius lower bound} there is a point $\mathbf z\in[0,1]^d$
such that $B_R^2(\mathbf z)$ does not contain any element of $P$.
Hence, the function
 \begin{equation*}
  f_*: [0,1]^d \to \IR, \quad f_*(\mathbf x) = \frac{R^r}{K_r}\, g_d\brackets{\frac{\mathbf x-\mathbf z}{R}}
 \end{equation*}
is an element of $\widetilde\C^r_d$ and vanishes on $P$. 
We obtain
\begin{equation*}
 \sup\limits_{f\in \widetilde\C^r_d:\, f\vert_P=0} \norm{f}_\infty
 \geq \norm{f_*}_\infty
 \geq f_*(\mathbf z)
 = \frac{R^r}{K_r}
 = \min\set{\frac{1}{K_r},\frac{d^{r/2}}{5^r K_r n^{r/d}}}
.\end{equation*}
Since this is true for any such $P$, equation~\eqref{eq:bakhvalov for APP} yields 
\begin{equation}
\label{error numbers lower bound}
 \e(n,\mathcal{P}[\APP,\C^r_d]) \geq
 \min\set{\frac{1}{K_r},\frac{d^{r/2}}{5^r K_r n^{r/d}}}.
\end{equation}
We set $\varepsilon_r=1/K_r$.
Given $\varepsilon\in(0,\varepsilon_r)$, 
the right hand side in \eqref{error numbers lower bound}
is larger than $\varepsilon$
for any $n$ smaller than $d^{d/2}/(5^d K_r^{d/r}\varepsilon^{d/r})$.
This yields
\begin{equation*}
 \comp(\varepsilon,\mathcal{P}[\APP,\widetilde\C^r_d])
 \geq \brackets{(5^rK_r)^{-1/r} \sqrt{d}\, \varepsilon^{-1/r}}^d
\end{equation*}
as it was to be proven.
\end{proof}

In the same way, we obtain lower bounds
for the case that the domains $[0,1]^d$ are replaced
by other domains $D_d\subset\IR^d$ that satisfy
$\lambda^d(D_d)\geq a^d$ for some $a>0$ and all $d\in\IN$.
We simply have to multiply the radii in the previous proofs by $a$.

We now turn to the lower bounds of Theorem~\ref{main theorem even}
and \ref{main theorem even odd}.

\begin{proof}[Proof of Theorem~\ref{main theorem even} and \ref{main theorem even odd} (Lower Bounds)]
Note that $\C^r_d$ contains $\widetilde\C^r_d$ and hence
\begin{equation*}
 \comp(\varepsilon,\mathcal{P}[\APP,\C^r_d])
 \geq
 \comp(\varepsilon,\mathcal{P}[\APP,\widetilde\C^r_d]).
\end{equation*}
Furthermore, any $\varepsilon$-approximation
of a function on $[0,1]^d$ immediately yields an
$\varepsilon$-approximation of its integral and hence
\begin{equation*}
 \comp(\varepsilon,\mathcal{P}[\APP,\C^r_d])
 \geq
 \comp(\varepsilon,\mathcal{P}[\INT,\C^r_d]).
\end{equation*}
With these relations at hand,
the desired lower bounds for $r\geq 2$ immediately
follow from Theorem~\ref{side theorem}.
The lower bound for $r=1$ follows from
the complexity of numerical integration as
studied in~\cite{HNUW17}.
\end{proof}

\section{Rank One Tensors}
\label{sec:rank one}


The uniform approximation of smooth functions $f:[0,1]^d\to \IR$
suffers from the curse of dimensionality~\cite{NW09}.
The number of function values that we need to
capture $f$ up to some error $\varepsilon\in(0,1)$
in the uniform norm grows exponentially with the dimension.
But suppose we know that $f$
is the tensor product of $d$ univariate functions.
How many function values do we need then? 
This question has first been
posed and investigated 
in the recent work of Bachmayr, Dahmen, DeVore and Grasedyck \cite{BDDG14}.
More precisely, 
it is assumed that $f$ is contained in 
a class of rank one tensors that is given by
\begin{equation*} 
F_{r,M}^d = \big\{ \bigotimes_{i=1}^d  f_i \mid
f_i:[0,1]\to [-1,1], \
\Vert f_i^{(r)} \Vert_\infty \le M \big\}
\end{equation*} 
for smoothness parameters $r\in \IN$ and $M\geq 0$, 
where the function
$$
\bigotimes_{i=1}^d  f_i:[0,1]^d\to \IR,
\quad
f(\mathbf{x})= \prod_{i=1}^d f_i(x_i)
$$
is called a rank one tensor.
Note that $f_i^{(r)}$ denotes the $r^{\rm th}$ weak derivative of $f_i$.
In particular, 
it is assumed that $f_i$ is contained in
the Sobolev class $W_\infty^r([0,1])$ of univariate functions
that have $r$ weak derivatives in $L^\infty([0,1])$.

It is proven in \cite{NR16} that for $M\geq2^r r!$ 
also this problem
suffers from the curse of dimensionality.
Even for randomized methods, the curse is present.
For $M<2^r r!$ however,
a randomized algorithm is constructed
that does not require exponentially many function values
with respect to the dimension $d$.
We show that the same is possible with a deterministic algorithm.
In fact, we construct algorithms for every constellation
of the smoothness parameters
that are optimal in terms of tractability. 
%

\subsection{Results}

A deterministic algorithm for the uniform
recovery of rank one tensors is already constructed in \cite{BDDG14}.
It achieves the worst case error $\varepsilon$
while using at most
\begin{equation*}
 C_{r,d}\, M^{d/r} \varepsilon^{-1/r}
\end{equation*}
function values of $f$, see \cite[Theorem~5.1]{BDDG14}.
This number behaves optimally as a function of $\varepsilon$.
However, the constant $C_{r,d}$ and hence
the number of function values grows super-exponentially with $d$
for any $M>0$ and $r\in\IN$.
The algorithm uses the following observation. 
If we know a nonzero $\mathbf z$ of $f$,  
we can essentially recover every factor $f_i$ separately
by sampling along the line 
$$
 E_i=\set{\mathbf x \in [0,1]^d \mid \forall j\in\set{1,\hdots,d}\setminus\set{i}:
 x_j=z_j}.
$$
This results in a deterministic algorithm $I_{m}(\mathbf z,\cdot)$
that requests function values at
\begin{equation} \label{eq: n_chosen_for_err_eps}
   m= \left\lfloor C_{r,M}\, d^{1+1/r} \varepsilon^{-1/r} \right\rfloor
\end{equation}
points and satisfies
\begin{equation} \label{eq: known_z_star}
 \norm{I_{m}(\mathbf z,f)-f}_\infty \leq \varepsilon
\end{equation}
for any $f\in F_{r,M}^d$ with $f(\mathbf z)\neq 0$.
Here, the constant $C_{r,M}$ is positive and 
depends only on $r$ and $M$. 
See~\cite{BDDG14} for details.
Roughly speaking, the knowledge of a nonzero of $f$
allows us to reduce the problem
to $d$ univariate approximation problems
which can, for example, be treated by the use of polynomial interpolation.
With this observation at hand, the authors of \cite{BDDG14}
use an approximation scheme of the following type:
\begin{alg}
\label{generic algorithm}
   Given $m\in \mathbb{N}$, a finite point set $P\subset[0,1]^d$ 
   and a function $f\in F_{r,M}^d$,
   obtain $A_{P,m}(f)$ as follows:
\begin{enumerate}
 \item 
 For any $\mathbf x\in P$ check whether $f(\mathbf x)= 0$.
 \item If we found some $\mathbf z\in P$ with $f(\mathbf z)\not =0$ then
       call $I_{m}(\mathbf z,f)$ from \eqref{eq: known_z_star}. 
       If $f\vert_P=0$,
       then return the zero function.
\end{enumerate}
\end{alg}
The idea behind this algorithm is to choose a point set $P$ such that
every $f$ that vanishes on $P$
must also be small on the whole domain,
and thus the zero function is a good approximation of $f$.
This property is characterized by the notion of detectors.
We call a finite point set $P$ in $[0,1]^d$ 
an $\varepsilon$-detector for the class $F_{r,M}^d$
if it contains (detects) a nonzero of every function $f\in F_{r,M}^d$
with uniform norm greater than $\varepsilon$.
If $P$ is an $\varepsilon$-detector 
and $m$ is chosen as in (\ref{eq: n_chosen_for_err_eps}),
it is easy to see 
that Algorithm~\ref{generic algorithm} satisfies
\begin{equation*}
 \err(A_{P,m},F_{r,M}^d) \leq \varepsilon \quad\text{and}\quad
 \cost(A_{P,m},F_{r,M}^d) \leq \card(P) + m,
\end{equation*}
see also Lemma~\ref{detector lemma}.
The authors of \cite{BDDG14} use a point set $P$ that contains
a finite Halton sequence $H$.
They obtain that $P$ is an $\eps$-detector if
\begin{equation*}  
 \card(H) \geq 
 2^{d+d/r} M^{d/r} \varepsilon^{-d/r} \pi_d,
\end{equation*}
where $\pi_d$ is the product of the first $d$ primes.
However, this number increases super-exponentially with the dimension
for all parameters $M$ and $r$.
Here, we want to construct smaller $\varepsilon$-detectors for $F_{r,M}^d$.

In the range $M\geq 2^r r!$
we know that the problem suffers from the curse of dimensionality
such that we cannot expect 
to find an $\varepsilon$-detector with small cardinality. 
However, we provide a detector 
whose cardinality depends merely
exponentially on the dimension and not super-exponentially.
In the range $M<2^r r!$
we give a detector whose cardinality only grows polynomially with the dimension.
The order of growth is proportional to $\ln(\varepsilon^{-1})$.
There even exists
a detector whose cardinality grows at most quadratically 
with the dimension for all $\varepsilon$ if we have $M\leq r!$.
Altogether, this yields the following theorem.

\begin{thm}[\cite{KR19}]
\label{thm: thm_upp_bnd}
 For any $r\in\mathbb{N}$ and $M>0$, there are positive constants
 $c_1,\hdots,c_4$ such that the following holds.
 For any $d\in \mathbb{N}$ and $\varepsilon\in (0,1)$,  
 there is a finite point set $P\subset[0,1]^d$
 and a natural number $m$ such that $\err(A_{P,m},F_{r,M}^d)\leq \varepsilon$ and
 \begin{equation*}
  \cost(A_{P,m},F_{r,M}^d) \leq
  \left\{\begin{array}{lr}
        c_1^d\, \varepsilon^{-1/r}
        \quad
        & \text{if } M\in(0,\infty),\\ 
        c_2 \exp\left(c_3(1+\ln(\varepsilon^{-1}))(1+\ln d)\right)\quad
        & \text{if } M\in(0,2^r r!),\\
        c_4\, d^{2} \varepsilon^{-1/r} \ln(\varepsilon^{-1/r})
        & \text{if } M\in(0,r!].
        \end{array}\right.
 \end{equation*}
\end{thm}
We always choose $m$ as in \eqref{eq: n_chosen_for_err_eps}.
The point sets $P$ and the constants $c_i$ can be found in Section~\ref{algorithms section}.
In each of these ranges
we also give a lower bound on the complexity of the problem,
which is the reason for
us to call the resulting algorithms optimal.
In particular, we obtain the following tractability results~\cite{KR19}.

\begin{thm}[\cite{KR19}]
\label{tractability results}
 The problem $\mathcal{P}[\APP,F_{r,M}^d]$ of uniform approximation on $F_{r,M}^d$
 with deterministic standard information in the worst case setting
\begin{itemize}
 \item suffers from the curse of dimensionality
 iff $M\geq 2^r r!$.
 \item is quasi-polynomially tractable iff $M<2^r r!$.
 \item is polynomially tractable iff $M\leq r!$.
 \item is strongly polynomially tractable iff $M=0$ and $r=1$.
\end{itemize}
\end{thm}

We also show that the first three statements of Theorem~\ref{tractability results} 
do not change for
randomized algorithms.
In this sense, randomization does not help for the problem
of recovering high-dimensional rank one tensors.
However, we do not know whether the last statement has to be modified
for randomized algorithms.

Before we turn to the proofs,
let us introduce some further notation. 
For any $k\in\IN$ 
we write $[k]=\{1,\hdots,k\}$.
If $x_i\in\IR$ for each $i\in J$ with some finite index set $J$,
we set $\mathbf x_J=(x_i)_{i\in J}$.
If $I_i$ is an interval for each $i\in J$, then $I_J$ 
denotes the Cartesian product of these intervals.
The term \emph{box} will always
refer to a product of nonempty subintervals of $[0,1]$.
If we are given functions $f_i:I_i\to \IR$ for each $i\in J$,
their tensor product is denoted by $f_J:I_J\to \IR$.
We recall that the \emph{dispersion}
of a finite subset $P$ of $[0,1]^d$ is the minimal
number $\eta>0$ such that 
$P$ has non-empty intersection with
every box of volume greater than~$\eta$,
see also Example~\ref{ex:dispersion}.
Throughout this section, we always assume that $r,d\in\IN$,
$\varepsilon\in(0,1)$ and $M>0$.

\subsection{Algorithms}
\label{algorithms section}

This section contains the proof of Theorem~\ref{thm: thm_upp_bnd}.
Here we always assume that $M>0$.
We start with the observation 
that the construction of an $\varepsilon$-detector
is sufficient to achieve the worst case error $\varepsilon$
with the algorithm $A_{P,m}$.
Recall that a point set $P$ in $[0,1]^d$ is called
an \mbox{$\varepsilon$-detector} for $F_{r,M}^d$,
if it contains a nonzero of any function
$f\in F_{r,M}^d$ with $\norm{f}_\infty > \varepsilon$.
Any such function is of the following form:
\begin{equation}
\label{target function}
\begin{split}
 f= \bigotimes\limits_{i=1}^d f_i,
 \quad &\text{where} \quad f_i:[0,1]\to[-1,1]
 \quad \text{with} \quad \Vert f_i^{(r)}\Vert_\infty \leq M \\
 &\text{and} \quad \norm{f}_\infty = \prod\limits_{i=1}^d \norm{f_i}_\infty > \varepsilon.
\end{split}
\end{equation}
Note that this representation of $f\in F_{r,M}^d$
is usually not unique, since
we may rescale the factors $f_i$ without changing the product.

\begin{lemma}
\label{detector lemma}
 Let $r\in\IN$, $d\in\IN$ and $M>0$.
 If $P$ is an $\varepsilon$-detector for $F_{r,M}^d$
 and $m$ is chosen as in (\ref{eq: n_chosen_for_err_eps}), 
 then
 Algorithm~\ref{generic algorithm} satisfies
 \begin{equation*}
  \err(A_{P,m},F_{r,M}^d) \leq \varepsilon \quad\text{and}\quad 
  \cost(A_{P,m},F_{r,M}^d) \leq \card(P) + m.
 \end{equation*}
\end{lemma}
\begin{proof}
Let $f\in F_{r,M}^d$.
If $P$ contains a nonzero of $f$,
Algorithm~\ref{generic algorithm}
returns an \mbox{$\varepsilon$-approximation} of $f$
due to relation (\ref{eq: known_z_star}).
If not, the output is zero.
But since $P$ is a detector, we necessarily have
$\Vert f\Vert_\infty\leq \varepsilon$ and zero
is an $\varepsilon$-approximation of $f$, as well.
The second statement is obvious.
\end{proof}

Furthermore, we will use the following formula
for polynomial interpolation.

\begin{lemma}
 \label{interpolation lemma}
 Let $a<b$, $r\in\IN$ and $g\in W_\infty^r([a,b])$.
 Let $x_1,\hdots,x_r\in[a,b]$ be distinct
 and $p$ be the unique polynomial with degree less than $r$
 such that $p(x_i)=g(x_i)$ for all $i\in[r]$.
 For every $x\in[a,b]$, 
 there exist $\xi_1,\xi_2\in[a,b]$ such that
 $$
 g(x)-p(x) =
 \frac{1}{r!} \cdot
 \frac{g^{(r-1)}\brackets{\xi_2}-g^{(r-1)}\brackets{\xi_1}}{\xi_2-\xi_1} \cdot
 \prod\limits_{i=1}^r \brackets{x-x_i} .
 $$
\end{lemma}

Lemma~\ref{interpolation lemma} is well known for $g\in\mathcal{C}^r([a,b])$.
In this case, the second fraction can be 
replaced by $g^{(r)}(\xi)$ for some $\xi\in [a,b]$.
We refer to \cite[Theorem~2, Section~6.1]{CK91}.
Under the more general assumption that $g\in W_\infty^r([a,b])$,
we have to modify the proof of the mentioned theorem.

\begin{proof}
 If $x$ coincides with one of the nodes, the statement is trivial.
 Hence, let $x$ be distinct from all the nodes.
 We consider
 $$
 w: [a,b]\to \IR, \quad w(y)=\prod\limits_{i=1}^r \brackets{y-x_i}
 $$
 and set
 $$
 \lambda = \frac{g(x)-p(x)}{w(x)}.
 $$
 The function $\phi=g-p-\lambda w$
 vanishes at the points $x_1,\hdots,x_r$ and $x$.
 Since $g$ and $\phi$ are elements of $W_\infty^r([a,b])$,
 their $(r-1)^{\text{st}}$ derivatives are absolutely continuous.
 If we apply Rolle's Theorem $(r-1)$ times,
 we obtain that $\phi^{(r-1)}$ has at least 2 distinct zeros $\xi_1$
 and $\xi_2$ in $[a,b]$ and hence
 $$
 0=\int_{\xi_1}^{\xi_2} \phi^{(r)}(y)~\d y
 =\int_{\xi_1}^{\xi_2} g^{(r)}(y)-\lambda r! ~\d y
 = g^{(r-1)}\brackets{\xi_2}-g^{(r-1)}\brackets{\xi_1}
 - \lambda r! \brackets{\xi_2-\xi_1}.
 $$
 This is the stated identity in disguise.
\end{proof}
If $g\in W_\infty^r([0,1])$ has $r$ distinct zeros $x_1,\dots,x_r\in [0,1]$, 
and $x$ is a maximum point of $\abs{g}$, we get
\begin{equation} \label{eq: polyn_interpol}
 \norm{g}_\infty \leq \frac{\Vert g^{(r)}\Vert_\infty}{r!} \prod_{i=1}^r \abs{x-x_i}.
\end{equation}
This follows from Lemma~\ref{interpolation lemma} since the unique polynomial
$p$ with degree less than $r$ and $p(x_i)=g(x_i)$ for $i\in [r]$ is the zero polynomial
and
\[
\abs{g^{(r-1)}\brackets{\xi_2}-g^{(r-1)}\brackets{\xi_1}} =
\abs{\int_{\xi_1}^{\xi_2} g^{(r)}(y)~\d y}
\leq \Vert g^{(r)}\Vert_\infty\cdot \abs{\xi_2-\xi_1}.
\]

The rest of this section is devoted to
the construction of small $\varepsilon$-detectors for $F_{r,M}^d$.
Thanks to Lemma~\ref{detector lemma},
this is sufficient to prove Theorem~\ref{thm: thm_upp_bnd}.
We will use three different strategies for
three different ranges of the parameter $M$.

\subsubsection{Detectors for large derivatives}

In this section, the smoothness parameter $M$ can be arbitrarily large.
It is shown in \cite{NR16} that
the cost of any algorithm with worst case error smaller than $1$
is at least $2^d$ if $M\geq 2^r r!$.
In particular, the cardinality of 
any \mbox{$\varepsilon$-detector} must grow exponentially with the dimension.
Yet, it does not get any worse:
We construct an \mbox{$\varepsilon$-detector} whose cardinality
``only'' grows exponentially with the dimension
but not super-exponentially.
We use the following lemma.

\begin{lemma}
\label{empty interval lemma}
 For each $g\in W_\infty^r([0,1])$ with $\Vert g^{(r)}\Vert_\infty \leq M$
 there is a subinterval of $[0,1]$ with length 
 \begin{equation*}
  L(g)=\min\set{\frac{1}{r},\brackets{\frac{\norm{g}_\infty}{M}}^{1/r}}
 \end{equation*}
 that does not contain any zero of $g$.
\end{lemma}

\begin{proof}
 The function $\abs{g}$ attains its maximum, say for $x\in [0,1]$.
 We choose an interval $I\subset [0,1]$ of length $rL(g)$ that contains $x$.
 There are $r$ open and disjoint 
 subintervals of $I$ with length $L(g)$.
 We label these intervals $I_1,\dots,I_r$ such that the distance 
 of $x$ and $I_i$ is increasing with $i$.
 Assume that every interval $I_i$ contains a zero $x_i$ of $g$.
 Then we have $\abs{x-x_i}< iL(g)$ for all $i\in [r]$
 and \eqref{eq: polyn_interpol} leads to
 \[
  \norm{g}_\infty
  \leq \frac{M}{r!} \prod_{i=1}^r \abs{x-x_i} 
  < M L(g)^r
  \leq \norm{g}_\infty.
 \]
 This is a contradiction and the assertion is proven.
\end{proof}

If, in addition, the uniform norm of $g$ is bounded by $1$, we have
\begin{equation*}
 L(g)\geq \varrho^{-1}\norm{g}_\infty^{1/r}  \quad\text{for}\quad
 \varrho=\max\set{r,M^{1/r}}.
\end{equation*}
Hence, for every $f$ satisfying \eqref{target function}
there is a box $B$ in $[0,1]^{d}$ with volume
\begin{align*}
 &\prod_{i\in [d]} L\brackets{f_i}
 \geq \varrho^{-d} \prod_{i\in [d]} \norm{f_i}_\infty^{1/r}
 = \varrho^{-d} \norm{f}_\infty^{1/r}
 > \varrho^{-d} \varepsilon^{1/r}
\end{align*}
such that $f$ does not vanish anywhere on $B$.
Hence, any point set $P$ in $[0,1]^{d}$ with dispersion
$\varrho^{-d} \varepsilon^{1/r}$ or less 
is an $\varepsilon$-detector for $F_{r,M}^d$.
We know from the estimate of Larcher, see \cite{AHR17}, 
that we can choose $P$ as a $(t,s,d)$-net with cardinality
\begin{equation*}
 \card(P) = \left\lceil 2^{7d+1} \varrho^d \varepsilon^{-1/r}\right\rceil.
\end{equation*}
By Lemma~\ref{detector lemma},
the resulting algorithm achieves the worst case error $\varepsilon$ with 
the cost
\begin{equation*}
 \cost\brackets{A_{P,m},F_{r,M}^d}\leq
 \left\lceil 2^{7d+1} \varrho^d \varepsilon^{-1/r} \right\rceil
 + C_{r,M}\, d^{1+1/r} \varepsilon^{-1/r}.
\end{equation*}
This proves the first statement of Theorem~\ref{thm: thm_upp_bnd} 
with $c_1=2^8\rho+C_{r,M}$.
Note that the cost of this algorithm
has the minimal order of growth with respect to $\varepsilon$.
It grows like $\varepsilon^{-1/r}$ if $d$ is fixed
and $\varepsilon$ tends to zero.

\subsubsection{Detectors for moderately large derivatives}

In this section, we assume that $M < 2^r r!$.
In this case, we construct detectors $P$
with a cardinality that only grows polynomially with $d$
for any fixed $\varepsilon$.
The construction of $P$
is based on the observation that for any
function $f$ from $\eqref{target function}$
only some of the factors $f_i$
can have more than $(r-1)$ zeros close to $1/2$.
This is 
an essential difference
to the case $M\in [2^r r!,\infty)$,
where all factors $f_i$ may have infinitely many zeros
in any neighborhood of $1/2$.
We are going to specify this statement in Lemma~\ref{harmless functions},
but first we need the following observation.
For $\delta\in (0,1/2]$, we consider the interval
$I_\delta := [1/2-\delta,1/2+\delta]$.

\begin{lemma}
\label{zeros at 1/2}
 Let $g\in W_\infty^r([0,1])$ with $\Vert g^{(r)}\Vert_\infty \leq M$.
  Assume that $g$ has $r$ distinct zeros in
  $I_\delta$. Then
  \[
  \norm{g}_\infty\leq C_\delta := \frac{M(1+2\delta)^r}{2^r r!}.
  \]
\end{lemma}

\begin{proof}
 Let $x_1,\dots,x_r$ be those zeros.
 The function $\abs{g}$ attains its maximum, say for $x\in [0,1]$.
 By \eqref{eq: polyn_interpol} we have
 \[
  \norm{g}_\infty
  \leq \frac{\Vert g^{(r)}\Vert_\infty}{r!} \prod_{i=1}^r \abs{x-x_i}.
 \]
 This yields the desired inequality since $\abs{x-x_i}\leq 1/2+\delta$ for each $i\in [r]$.
\end{proof}

Since $M < 2^r r!$, we can choose $\delta\in(0,1/2]$ such that $C_\delta <1$.
We define the pseudo-dimension
$d_0$ as the largest number in $[d]\cup\{0\}$ that satisfies
$C_\delta^{d_0} >\varepsilon$, that is,
\begin{equation*}
 d_0 := \min\set{\left\lceil \frac{\ln \varepsilon}{\ln C_\delta}\right\rceil -1, d}.
\end{equation*}
Obviously, the pseudo-dimension is bounded above independently of $d$.
We can now specify the statement from the beginning of this section.

\begin{lemma}
\label{harmless functions}
 Let $f$ be given as in (\ref{target function}).
 Then there are at most $d_0$ coordinates $i\in [d]$
 such that $f_i$ has more than $(r-1)$ zeros in $I_\delta$.
\end{lemma}

\begin{proof}
 Let $k$ be the number of coordinates $i\in[d]$
 for which $f_i$ has more than $(r-1)$ zeros in $I_\delta$.
 Lemma~\ref{zeros at 1/2} yields that $\varepsilon < \Vert f\Vert_\infty \leq C_\delta^k$.
 The maximality of $d_0$ yields that $k\leq d_0$.
\end{proof}

This means that there is a subset $J^*$ of $[d]$ with cardinality $d_0$
such that $f_i$ has at most $(r-1)$ zeros in $I_\delta$ 
for all $i\in [d]\setminus J^*$.
Let us suppose for the moment that we know this set $J^*$.
Then we can find a nonzero of $f$ by solving the following tasks:
\begin{enumerate}
 \item Find a nonzero of $f_{J^*}$.
 \inlineitem Find a nonzero of $f_{[d]\setminus J^*}$.
\end{enumerate}

We can easily solve the first task
since this problem is only $d_0$-dimensional.
By Lemma~\ref{empty interval lemma}, 
there is a box $B$ in $[0,1]^{d_0}$ with volume
\begin{align*}
 &\prod_{i\in J^*} L\brackets{f_i}
 \geq \prod_{i\in J^*} \varrho^{-1} \norm{f_i}_\infty^{1/r}
 \geq \varrho^{-d_0} \prod_{i\in [d]} \norm{f_i}_\infty^{1/r}
 = \varrho^{-d_0} \norm{f}_\infty^{1/r}
 > \varrho^{-d_0} \varepsilon^{1/r}
\end{align*}
such that $f_{J^*}$ does not vanish on $B$.
Hence, any point set $P_1$ in $[0,1]^{d_0}$ with dispersion
$\varrho^{-d_0} \varepsilon^{1/r}$ or less contains a nonzero of $f_{J^*}$.
Again by the result of Larcher, see \cite{AHR17}, 
we know that we can choose $P_1$ as a $(t,s,d)$-net
of cardinality $2^{7d_0+1} \varrho^{d_0} \varepsilon^{-1/r}$.

We can also cope with the second task
since $f_i$ has at most $(r-1)$ zeros in $I_\delta$
for all $i\in [d]\setminus J^*$.
We use the following observation.

\begin{lemma}
\label{diagonal lemma}
 Let $J$ be a subset of $[d]$ with $\ell$ elements
 and, for all $i\in J$, let $f_i$ be a function with at most
 $k$ zeros on some interval $I_i$.
 Then every set in $I_J$ with $(\ell k+1)$ elements
 that are pairwise distinct in every coordinate
 contains a nonzero of~$f_J$.
\end{lemma}

\begin{proof}
 Let $P$ be a set in $I_J$ with $(\ell k+1)$ elements
 that are pairwise distinct in every coordinate 
 and suppose that $f_J$ vanishes everywhere on $P$.
 For each $i\in J$ let
 $P_i=\set{\mathbf x_J\in P \mid f_i(x_i)=0}$.
 Since $f_J(\mathbf x_J)=0$ implies that there is some $i\in J$ with $f_i(x_i)=0$, we have
 $P=\bigcup_{i\in J} P_i$.
 This can only be true, if one of the sets $P_i$ has more than $k$ elements.
 But since $x_i$ is different for every $\mathbf x_J\in P_i$,
 this means that the corresponding function $f_i$ has more than $k$ zeros, a contradiction.
\end{proof}
\noindent
Applying this lemma
for the functions $f_i$ in~\eqref{target function},
for the index set $J=[d]\setminus J^*$ and
for $k=r-1$, we obtain that the diagonal set
\begin{equation*}
 P_2=\set{\brackets{\frac{1}{2} -\delta + \frac{2\delta j}{(r-1)(d-d_0)}}\cdot \mathbf 1
 \mid j\in\IN_0 \text{ with } j\leq (r-1)(d-d_0)}
\end{equation*}
in $[0,1]^{d-d_0}$ contains a nonzero of $f_{[d]\setminus J^*}$.
Together, this yields that there must be 
at least one nonzero of $f$ in the set
$$
 \set{\mathbf x\in [0,1]^d \mid \mathbf x_{J^*}\in P_1,\, 
 \mathbf x_{[d]\setminus J^*}\in P_2 }.
$$

This would solve our problem if we knew the set $J^*$.
Since we do not know this set,
we simply try all sets $J\subset [d]$ of cardinality $d_0$.
This is OK since number of such sets only depends polynomially on $d$.
Altogether, we obtain the $\varepsilon$-detector
\begin{equation*}
P=\bigcup_{J\subset [d]:\, \card(J)=d_0}
 \set{\mathbf x\in [0,1]^d \mid \mathbf x_J\in P_1,\, \mathbf x_{[d]\setminus J}\in P_2 }.
\end{equation*}
In fact, we have seen that for any $f$ as in~\eqref{target function}
there must be some $J^*\subset [d]$ with cardinality $d_0$,
a nonzero $\mathbf y\in P_1$ of $f_{J^*}$
and a nonzero $\mathbf z\in P_2$ of $f_{[d]\setminus J^*}$.
The point $\mathbf x\in [0,1]^d$ with $\mathbf x_{J^*}=\mathbf y$ and 
$\mathbf x_{[d]\setminus J^*}=\mathbf z$
is contained in the set $P$ and a nonzero of $f$.
The cardinality of the detector is given by
\begin{equation*}
\card(P)
 = \binom{d}{d_0} \card(P_1) \card(P_2)
 = \binom{d}{d_0} \left[(r-1)(d-d_0)+1\right] 2^{7d_0+1} \varrho^{d_0} \varepsilon^{-1/r}.
\end{equation*}
This number grows like $d^{d_0+1}$
if $\varepsilon$ is fixed and $d$ tends to infinity.
Together with Lemma~\ref{detector lemma},
this proves the second statement of Theorem~\ref{thm: thm_upp_bnd} with
\[
 c_2=2r+C_{r,M},
 \quad\text{and}\quad
 c_3=\ln(2^7\rho)\left(1+1/\ln(C_\delta^{-1})\right).
\]
Note that 
$d_0$ equals $d$ 
if $\varepsilon$ is small enough.
Hence, the cardinality of $P$ and the cost 
of the algorithm grows like $\varepsilon^{-1/r}$
if $d$ is fixed and $\varepsilon$ tends to zero,
which is optimal.

\subsubsection{Detectors for small derivatives}

In this section, we assume that $M\leq r!$.
In this case, each function $f$ satisfying~\eqref{target function} 
does not vanish almost everywhere on a box whose size is independent of $d$.
This is due to the following fact.

\begin{lemma}
\label{almost empty interval lemma}
 For each $g\in W_\infty^r([0,1])$ with 
 $\Vert g^{(r)}\Vert_\infty \leq r!$ 
 there is an interval in $[0,1]$ with length 
 $\min\{1,\Vert g\Vert_\infty^{1/r}\}$
 that contains at most $(r-1)$ zeros of $g$.
\end{lemma}

\begin{proof}
 The function $\abs{g}$ attains its maximum, say for $x\in [0,1]$.
 We choose an open interval $I\subset [0,1]$ 
 of length $\min\{1,\Vert g\Vert_\infty^{1/r}\}$ whose closure contains $x$.
 Assume that $I$ contains $r$ distinct zeros $x_1,\hdots,x_r$ of $g$.
 Then $\abs{x-x_i}< \norm{g}_\infty^{1/r}$ for all $i\in [r]$
 and \eqref{eq: polyn_interpol} yields
 \[
  \norm{g}_\infty
  \leq \frac{\Vert g^{(r)}\Vert_\infty}{r!} \prod_{j=1}^r \abs{x-x_j} 
  \leq \prod_{j=1}^r \abs{x-x_j} 
  < \norm{g}_\infty.
 \]
 This is a contradiction and the assertion is proven.
\end{proof}

We now construct an $\varepsilon$-detector for any $\varepsilon\in(0,1)$. 
To this end, let
\begin{equation*}
 \gamma=(1-2^{-1/d})\,\varepsilon^{1/r}.
\end{equation*}
Note that $\gamma$ is smaller than 1/2.
We choose a point set $P_0$ in $[0,1]^d$ whose dispersion is
at most $\varepsilon^{1/r}/2$ and consider
the point set $P$ in $[0,1]^d$, given by
 \begin{equation*}
 P= \Big\{(1-\gamma)\cdot x+\frac{\gamma j}{(r-1)d}\cdot \mathbf 1 
 \,\Big\vert\, \mathbf x\in P_0 \text{ and }
 j\in\IN_0 \text{ with } j\leq (r-1)d\Big\}.
\end{equation*}

\begin{lemma}
\label{detector for small M}
 The point set $P$ is an \mbox{$\varepsilon$-detector} for $F_{r,M}^d$.
\end{lemma}

\begin{proof}
Let $f$ be given as in \eqref{target function}.
By Lemma~\ref{almost empty interval lemma},
there are intervals $(a_i,b_i)$ in $[0,1]$ with length $\norm{f_i}_\infty^{1/r}$
containing at most $(r-1)$ zeros of $f_i$.
By the choice of $\gamma$, we have
\begin{equation*}
 \gamma 
 \leq (1-2^{-1/d}) \norm{f_i}_\infty^{1/r}.
\end{equation*}
In particular, the box
\begin{equation*}
 \widetilde B= \prod_{i\in [d]} (a_i,b_i-\gamma)
\end{equation*}
is well defined.
In fact, the volume of this box satisfies
\begin{equation*}
 \vert \widetilde B \vert
 = \prod_{i\in [d]} \brackets{\norm{f_i}_\infty^{1/r}-\gamma}
 \geq \prod_{i\in [d]} \brackets{2^{-1/d} \norm{f_i}_\infty^{1/r}}
 = \frac{\norm{f}_\infty^{1/r}}{2}
 > \frac{\varepsilon^{1/r}}{2}.
\end{equation*}
The box $\widetilde B/(1-\gamma)$ is contained in $[0,1]^d$
and even larger than $\widetilde B$.
It hence contains some $\mathbf x\in P_0$.
Consequently, we have $(1-\gamma) \mathbf x\in \widetilde B$
and all the points
\begin{equation*}
 x^{(j)}=(1-\gamma)\cdot x+\frac{\gamma j}{(r-1)d}\cdot \mathbf 1  \quad\quad\text{for }
 j\in\IN_0 \text{ with } j\leq (r-1)d
\end{equation*}
are elements of $P$.
These are $(r-1)d+1$ points
that are pairwise distinct in every coordinate
and that are all contained in the larger box 
\begin{equation*}
 B= \prod_{i\in [d]} (a_i,b_i).
\end{equation*}
Recall that each function $f_i$ has at most $(r-1)$
zeros in $(a_i,b_i)$.
By Lemma~\ref{diagonal lemma}, one of the points $x^{(j)}$
must be a nonzero of $f$.
As an example, Figure~\ref{box figure} illustrates the
case $d=2$ and $r=3$.
\end{proof}

\begin{figure}[ht]
\centering
  \includegraphics[width=0.8\textwidth]{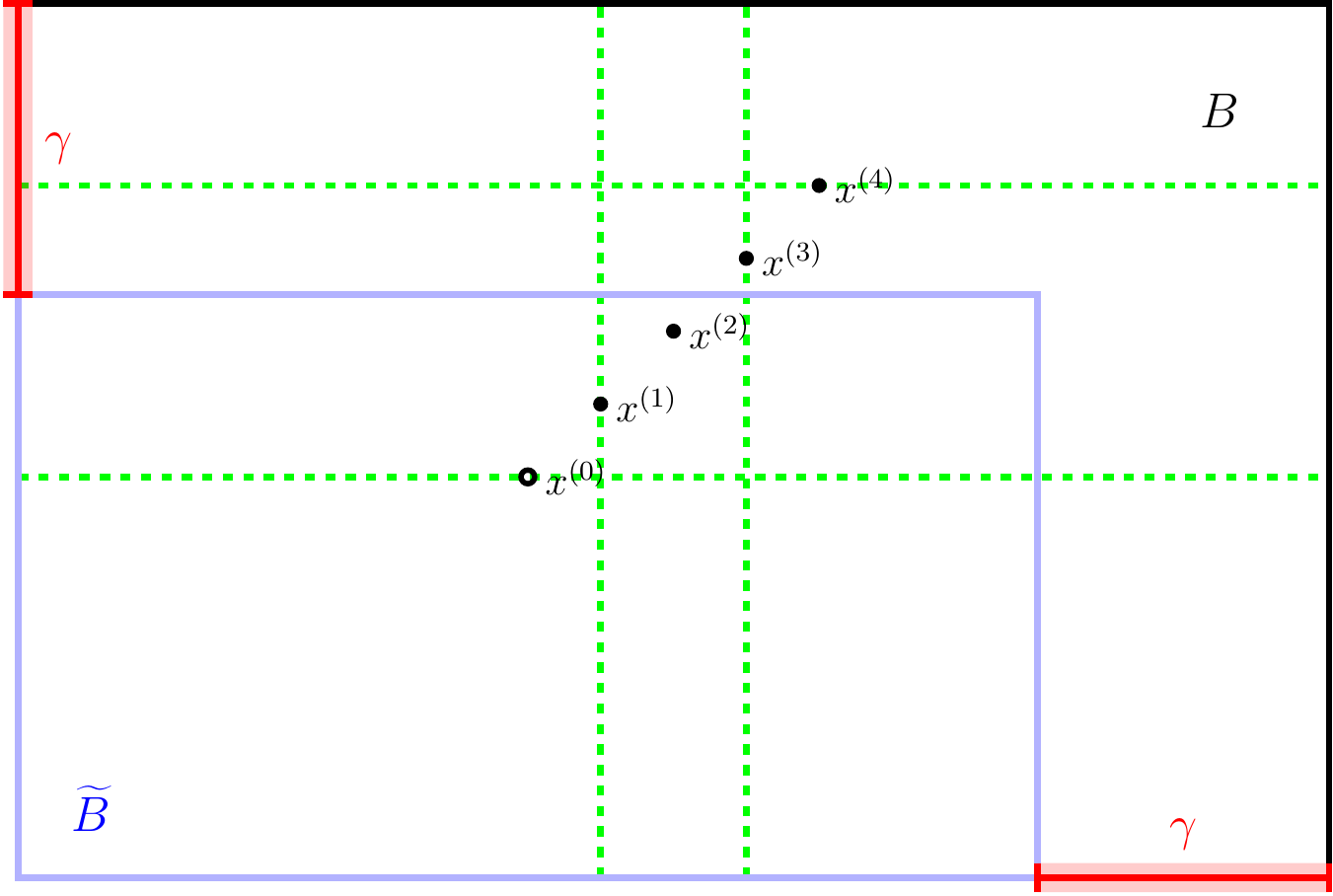}
  \caption{The box $B$ for $(r,d)=(3,2)$.
  The dashed lines indicate the zeros of $f$ in $B$.
  Since $f$ only vanishes there,
  one of the points $x^{(j)}$ must be a nonzero.}
  \label{box figure}
\end{figure}
This means that we have an $\varepsilon$-detector for $F_{r,M}^d$ with the cardinality
\begin{equation*}
 \card(P)=((r-1)d+1) \card(P_0),
\end{equation*}
where $P_0$ is a point set with dispersion $\varepsilon^{1/r}/2$ or less.
For example, we know from \cite{Ru18} that this can be achieved with
\begin{equation} \label{eq: pointset_Ru16}
 \card(P_0)=\left\lceil 2^4 d\, \varepsilon^{-1/r} 
 \ln\brackets{66\varepsilon^{-1/r}} \right\rceil
\end{equation}
points.
In particular, Lemma~\ref{detector for small M} 
and Lemma~\ref{detector lemma} 
give the last statement of Theorem~\ref{thm: thm_upp_bnd}
with the constant
$c_4=85r+C_{r,M}$.

\begin{rem}[Alternative choices of $P_0$]
If follows from \cite{UV18}
that the point set $P_0$ with dispersion $\varepsilon^{1/r}/2$ 
can also be chosen such that
\begin{equation}
\label{eq:probabilistic bound UV18}
 \card(P_0)\leq C \log_2(d)\,\brackets{1+\log_2\brackets{\varepsilon^{-1/r}}}^2\,\varepsilon^{-2/r}
\end{equation}
with an absolute constant $C>0$.
This number is smaller than \eqref{eq: pointset_Ru16} if $d$ is large, 
but the dependence on $\varepsilon$ is worse.
In both cases, however, the resulting algorithm $A_{P,m}$ is not completely explicit
since we do not know how to construct 
the point sets $P_0$ from~\eqref{eq: pointset_Ru16} and \eqref{eq:probabilistic bound UV18}.
We only know that they exist.
What we do know from~\cite{UV19} is how to construct $P_0$ such that
$$
 \card(P_0)\leq 
 C\, \varepsilon^{-6/r} \brackets{1 + \log_2\brackets{\varepsilon^{-1/r}}}^6  
  \,\log_2 d^*
$$
with $d^*=\max\{d,4\varepsilon^{-1/r}\}$, see also Section~\ref{Comparison Section}.
Using this construction of $P_0$, the algorithm $A_{P,m}$ from above
is completely explicit
and although the $\varepsilon$-dependence is far from optimal, 
its information cost is still polynomial in both $d$ and $\varepsilon^{-1}$.
\end{rem}

\subsection{Lower bounds}
\label{lower bounds section}

This section contains the proof of Theorem~\ref{tractability results}.
Of course, the positive tractability results are implied
by Theorem~\ref{thm: thm_upp_bnd},
where the case $M=0$ follows from the case $M>0$.
The only exception is the case $M=0$ and $r=1$.
Here, the functions in $F^d_{r,M}$ are constant
and can be recovered from a single function value.

We now provide lower bounds on the
complexity of the uniform approximation problem
which imply the negative tractability results.
Note that the first result of the following
lemma is already contained in~\cite[Theorem~2]{NR16}.

\begin{thm}[\cite{KR19}]
 \label{thm: lower bounds rank1}
 Let $r,d\in\IN$, $\varepsilon\geq 0$ and $M\geq 0$.
 \begin{itemize}
  \item If $M\geq 2^r r!$, then\, $\displaystyle\comp(\eps,\mathcal{P}[\APP,F^d_{r,M}])\geq 
  2^d$\, for any $\varepsilon<1$.
  \item If $M>r!$, then the problem $\mathcal{P}[\APP,F^d_{r,M}]$
  is not polynomially tractable.
  \item If $r\geq 2$, then\, $\displaystyle\comp(\eps,\mathcal{P}[\APP,F^d_{r,M}])> d$\,
  for any $\varepsilon<1$.
  \item If $r=1$ and $M>0$, then\, $\displaystyle\comp(\eps,\mathcal{P}[\APP,F^d_{r,M}])> 
  \lfloor \log_2 d\rfloor$\, for any $\varepsilon<M^2/4$.
 \end{itemize}
\end{thm}

\begin{proof} \textit{Part 1.} Let $M\geq 2^r r!$.
The function
\begin{equation*}
 g(x)=2^r \vert x-1/2\vert^r\cdot \mathbbm{1}_{[0,1/2]}(x),
 \quad
 x \in [0,1]
\end{equation*}
is $r$-times differentiable with $\norm{g}_\infty=1$ and 
$\Vert g^{(r)}\Vert_\infty \leq M$.
The same holds for the function
\begin{equation*}
 h(x)=2^r \vert x-1/2\vert^r\cdot \mathbbm{1}_{[1/2,1]}(x),
 \quad
 x \in [0,1].
\end{equation*}
Hence, all functions
$f=f_{[d]}$ with $f_i \in \set{g,h}$
for $i\in[d]$
are contained in $F_{r,M}^d$ and satisfy $\norm{f}_\infty=1$.
We obtain a set $E$ of 
$2^d$ functions with pairwise disjoint support.

Let $A$ be an algorithm and let $\mathbf x_1,\dots,\mathbf x_n$ 
be the sample points the algorithm uses for
the input $f_0=0$.
If $n<2^d$, there is at least one function $f\in E$ 
that vanishes at all these points.
The algorithm cannot distinguish this function $f$ from
--$f$ and $f_0$ such that
\begin{equation*}
 A(f)=A(f_0)=A(-f)
\end{equation*}
and we obtain the error bound
 \begin{equation*}
  \err(A,F_{r,M}^d) \geq \max\set{\norm{A(f_0)-f}_\infty,\norm{A(f_0)+f}_\infty}
  \geq \norm{f}_\infty =1.
 \end{equation*}
\textit{Part 2}. Let $M>r!$ and $\eps\in(0,1)$.
Note that the point $x_0=\brackets{r!/M}^{1/r}$ is contained
in $(1/2,1)$. The function
\begin{equation*}
 g(x)=\frac{M\vert x-x_0\vert^r}{r!} \cdot \mathbbm{1}_{[0,x_0]}(x),
 \quad
 x \in [0,1]
\end{equation*}
is $r$-times differentiable with $\norm{g}_\infty=1$
and $\Vert g^{(r)}\Vert_\infty \leq M$.
The function
\begin{equation*}
 h(x)=\frac{M\vert x-x_0\vert^r}{r!}\cdot \mathbbm{1}_{[x_0,1]}(x),
 \quad
 x \in [0,1]
\end{equation*}
is also $r$-times differentiable with $\Vert h^{(r)} \Vert_\infty \leq M$
and $\Vert h\Vert_\infty=h(1)$ is in $(0,1)$.
Let $k(\varepsilon,d)$ be the largest number in $\{0,1,\hdots,d\}$  
such that $h(1)^{k(\varepsilon,d)} > \varepsilon$.
Namely, let
\begin{equation*}
 k(\varepsilon,d)=\min\set{\kappa(\varepsilon), d}
 \quad\text{with}\quad
 \kappa(\varepsilon):=\left\lceil
 \frac{\ln(\varepsilon^{-1})}{\ln(h(1)^{-1})}
 \right\rceil -1.
\end{equation*}
For every subset $J$ of $[d]$ with cardinality $k(\varepsilon,d)$,
the function $f=f_{[d]}$
with $f_i = g$ for $i\in J$ and
$f_i = h$ for $i\in [d]\setminus J$ is contained in $F_{r,M}^d$
and satisfies $\norm{f}_\infty>\varepsilon$.
We obtain a set $E$ of
$$
 N(\eps,d)= \displaystyle \binom{d}{k(\varepsilon,d)}
$$ 
functions with pairwise disjoint support.

Let $A$ be an algorithm and let $\mathbf x_1,\dots,\mathbf x_n$ 
be the sample points the algorithm uses for
the input $f_0=0$.
If $n<N(\eps,d)$, there is at least one function $f\in E$
that vanishes at all these points.
The algorithm cannot distinguish this function $f$ from
--$f$ and $f_0$, such that its error satisfies
 \begin{equation*}
  \err(A,F_{r,M}^d) \geq \max\set{\norm{A(f_0)-f}_\infty,\norm{A(f_0)+f}_\infty}
  \geq \norm{f}_\infty > \varepsilon.
 \end{equation*}
We obtain 
\begin{equation*}
 \comp(\eps,\mathcal{P}[\APP,F^d_{r,M}]) \geq N(\eps,d) 
 \geq \brackets{\frac{d}{k(\varepsilon,d)}}^{k(\varepsilon,d)}.
\end{equation*}
This implies that the problem is not polynomially tractable.
In fact, let us assume that the problem is polynomially tractable.
Then there are $c,q,p>0$ such that 
\begin{equation}
\label{trac assumption}
 \comp(\eps,\mathcal{P}[\APP,F^d_{r,M}]) \leq c\, \varepsilon^{-p} d^q
\end{equation}
for all $\varepsilon\in(0,1)$ and all $d\in\IN$.
We can, however, choose $\varepsilon\in(0,1)$
such that $\kappa(\varepsilon) > q$ and hence
\begin{equation*}
 \lim\limits_{d\to \infty} \frac{n(\varepsilon,d)}{d^q}
 \geq \lim\limits_{d\to \infty}
 \frac{d^{\kappa(\varepsilon)-q}}{\kappa(\varepsilon)^{\kappa(\varepsilon)}}
 = \infty,
\end{equation*}
which contradicts the assumption \eqref{trac assumption}.\newline
\textit{Part 3}. Let $r\geq 2$.
Let $A$ be an algorithm and let $\mathbf x_1,\dots,\mathbf x_n$ 
be the sample points the algorithm uses for
the input $f_0=0$.
Let us assume that $n\leq d$.
For each $i\in[n]$, there is a linear function $f_i$
on $[0,1]$ that vanishes at the $i^{\rm th}$ coordinate of $\mathbf x_i$
and satisfies $\Vert f_i\Vert_\infty = 1$.
For $i\in[d]\setminus [n]$ we set $f_i=1$.
The function $f=f_{[d]}$ is in $F_{r,M}^d$
and vanishes at all sample points.
Hence, $f$ and --$f$ cannot be distinguished from $f_0$
and the error of $A$ is at least $\norm{f}_\infty=1$.\newline
\textit{Part 4}.
Let $r=1$ and $M>0$.
The previous argument does not work in this case,
since the first derivative of $f_i$ is not
necessarily bounded by $M$.
Here, we assume that 
the number of sample points of the algorithm $A$ for the input $f_0=0$
is at most $\lfloor \log_2 d \rfloor$.
By the proof of \cite[Lemma 2]{AHR17}, we know that there
are two distinct coordinates $j,\ell \in [d]$ such that
the box $I_{[d]}$ does not contain any of these points,
where $I_j = [0,1/2)$, $I_\ell=(1/2,1]$
and $I_i=[0,1]$ otherwise.
The function $f=f_{[d]}$ with 
\begin{equation*}
 f_i(x)=M (x-1/2) \cdot \mathbbm{1}_{I_i}(x),
 \quad x\in[0,1]
\end{equation*}
for $i\in\set{j,\ell}$ and $f_i=1$ otherwise,
is contained in $F_{r,M}^d$ and vanishes at all sample points.
Therefore, the algorithm cannot distinguish $f$ and
--$f$ from $f_0$ such that its error is at least $\norm{f}_\infty=M^2/4$.
\end{proof}

The first two statements of the previous lemma
can be extended to randomized algorithms based on $\lstd$.
We use Bakhvalov's proof technique, see Theorem~\ref{thm:Bakhvalov technique}.
The first statement of the following theorem is again contained
in~\cite[Theorem~3]{NR16}.

\begin{thm}
 \label{thm: lower bounds rank1 ran}
 Let $r,d\in\IN$, $\varepsilon\geq 0$ and $M\geq 0$.
 \begin{itemize}
  \item If $M\geq 2^r r!$, then\, $\displaystyle\e(n,\mathcal{P}[\APP,F^d_{r,M},\mathrm{ran}])>
  2^{d-1}$ for all $\varepsilon< 2^{-1/2}$.
  \item If $M>r!$, then the problem $\mathcal{P}[\APP,F^d_{r,M},\mathrm{ran}]$
  is not polynomially tractable.
 \end{itemize}
\end{thm}

\begin{proof} \textit{Part 1}. Let $M\geq 2^r r!$.
In the first part of the proof of Lemma~\ref{thm: lower bounds rank1}
we defined a set $E$ consisting of $2^d$ functions in $F^d_{r,M}$.
Let $\mu$ be the uniform distribution on $E\cup(-E)$.
Let $A$ be a deterministic algorithm and let $\mathbf x_1,\dots,\mathbf x_n$ 
be the sample points the algorithm uses for
the input $f_0=0$.
If $n\leq 2^{d-1}$, there is a subset $E_0$ of 
$E$ with cardinality at least $2^{d-1}$ such that any $f\in E_0$
vanishes at all the sample points.
The algorithm cannot distinguish $f$ from
--$f$ and the triangle inequality yields
\begin{equation*}
 \norm{A(f)-f}_\infty + \norm{A(-f)-(-f)}_\infty \geq 2\norm{f}_\infty
 =2,
\end{equation*}
and hence
\begin{equation*}
 \norm{A(f)-f}_\infty^2 + \norm{A(-f)-(-f)}_\infty^2 \geq 2.
\end{equation*}
We obtain the error bound
 \begin{equation*}
  \err(A,\mu)^2 =
  \frac{1}{2^{d+1}} \sum_{f\in E} 
  \brackets{\norm{A(f)-f}_\infty^2 + \norm{A(-f)-(-f)}_\infty^2}
  \geq
  \frac{2\card(E_0)}{2^{d+1}}
  \geq \frac{1}{2}.
 \end{equation*}
Together with Theorem~\ref{thm:Bakhvalov technique},
this yields the statement.\newline
\textit{Part 2}. Let $M>r!$ and $\eps\in(0,1)$.
In the second part of the proof of Lemma~\ref{thm: lower bounds rank1}
we defined a set $E$ consisting of $N(\eps,d)$ functions in $F^d_{r,M}$.
Let $\mu$ be the uniform distribution on $E\cup(-E)$.
Let $A$ be a deterministic algorithm and let $\mathbf x_1,\dots,\mathbf x_n$ 
be the sample points the algorithm uses for
the input $f_0=0$.
If $n\leq N(\eps,d)/2$, there is a subset $E_0$ of 
$E$ with cardinality at least $N(\eps,d)/2$ such that any $f\in E_0$
vanishes at all the sample points.
The algorithm cannot distinguish $f$ from
--$f$ and the triangle inequality yields
\begin{equation*}
 \norm{A(f)-f}_\infty + \norm{A(-f)-(-f)}_\infty \geq 2\norm{f}_\infty
 > 2\eps,
\end{equation*}
and hence
\begin{equation*}
 \norm{A(f)-f}_\infty^2 + \norm{A(-f)-(-f)}_\infty^2 > 2\eps^2.
\end{equation*}
We obtain the error bound
\begin{equation*}
  \err(A,\mu)^2 =
  \frac{1}{2 N(\eps,d)} \sum_{f\in E} 
  \brackets{\norm{A(f)-f}_\infty^2 + \norm{A(-f)-(-f)}_\infty^2}
  > \frac{\varepsilon^2}{2}.
\end{equation*}
Together with Theorem~\ref{thm:Bakhvalov technique},
we obtain that
\begin{equation*}
  \comp(\varepsilon,\mathcal{P}[\APP,F^d_{r,M},\mathrm{ran}])
  > \frac{N(\sqrt{2}\eps,d)}{2}.
\end{equation*}
Like above,
this implies that the problem is not polynomially tractable.
\end{proof}

\section{Global Optimization}
\label{sec:optimization}



Let $F$ be a class of bounded real-valued functions on $[0,1]^d$.
We study the problem $\mathcal{P}[\OPT,F]$ of global optimization on $F$
in the worst case setting.
That is, given any function $f\in F$,
we want to find a point $\mathbf{x}^*\in [0,1]^d$
such that $f(\mathbf{x}^*)$ is almost maximal.
We emphasize that we want want to find
the maximizer and not just the maximum.
In order to find $\mathbf x^*$, we
may request $n$ function values of the unknown function
at adaptively and deterministically chosen points.

In the sense of Definition~\ref{def:problem},
we define $\mathcal{P}[\OPT,F]=(\mathcal{A},\err,\cost)$,
where
$$
 \mathcal{A}=\mathcal{A}[F,[0,1]^d,\lstd,\mathrm{det}]
$$ 
is the class of all deterministic algorithms
based on standard information
with input 
$f$ in $F$ and output $\mathbf x^*=A(f)$ in $[0,1]^d$
(see Section~\ref{sec:algorithms})
and each algorithm $A\in \mathcal{A}$ is assigned
the cost
$$
 \cost(A)=\cost(A,F,\lstd,\mathrm{wc}),
$$
which is the maximal number
of function values that $A$ requests of the input function
(see Section~\ref{sec:cost}), and the error
$$
 \err(A)=\sup_{f\in F}\brackets{\sup_{\mathbf{x}\in[0,1]^d} f(\mathbf x) - f(\mathbf x^*)},
$$
which is the residual error in the worst case.
Note that this problem is not described by
solution operator $\OPT:F\to [0,1]^d$.
The reason is that we usually cannot assign a unique maximizer
to every function $f\in F$.


We now show that the results from Section~\ref{sec:ck functions}
and Section~\ref{sec:rank one} for the problem $\mathcal{P}[\APP,F]$ of
uniform approximation in the class $F\in\{\C^r_d,\widetilde\C^r_d,F_{r,M}^d\}$  also hold for
the problem $\mathcal{P}[\OPT,F]$ of global optimization.
On the one hand, it is easy to see that
global optimization is never harder than
uniform approximation.
On the other hand, it is known that global optimization
is practically as hard as uniform approximation
if $F$ is convex and symmetric~\cite{Wa84,No88}.
Note that the classes $\C^r_d$ and $\widetilde\C^r_d$ are convex and symmetric
but the class $F_{r,M}^d$ is not convex.
In this case, we may still apply the following 
comparison statement,
which is implicitly contained
in the proof of~\cite[Proposition~1.3.2]{No88}.

\begin{prop}
\label{prop:opt and app}
 Let $F\subset \mathcal{B}([0,1]^d)$ be symmetric 
 with $0\in F$. Then
 $$
 \inf_{\substack{P\subset [0,1]^d\\ \card(P)\leq n+1}}
 \sup_{\substack{f\in F\\ f\mid_P=0}} \norm{f}_\infty
 \leq \e(n,\mathcal{P}[\OPT,F]) \leq 
 2\e(n,\mathcal{P}[\APP,F]).
 $$
\end{prop}

\begin{proof}
 \emph{Upper Bound.}
 Let $A$ 
 be an algorithm for uniform approximation.
 For every $\delta>0$,
 we define an algorithm 
 $$
  Q_\delta:F\to [0,1]^d, \quad
  f \mapsto \mathbf x^*= Q_\delta(f)
 $$ 
 for global optimization as follows.
 Let $g=A(f)$ be our approximation of $f\in F$.
 We choose $\mathbf x^* \in[0,1]^d$ such that 
 $g(\mathbf x^*)\geq \sup g-\delta$.
 Then
 $$
  \sup f - f(\mathbf x^*) 
  \leq \sup f - \sup g + g(\mathbf x^*) - f(\mathbf x^*) + \delta
  \leq 2\norm{f-g}_\infty + \delta.
 $$
 We obtain
 $$
  \cost(Q_\delta)\leq \cost(A), \qquad
  \err(Q_\delta)\leq 2\err(A)+\delta.
 $$
 The statement is obtained as $\delta$ tends to zero.
 
 \emph{Lower Bound.}
 Let $Q$ be an algorithm for global optimization with cost $n$ or less.
 Then there is a point set $P\subset [0,1]^d$ with cardinality $n+1$
 that contains all nodes of the algorithm for the input $f_0=0$
 and the point $Q(f_0)$.
 If $f\in F$ vanishes on $P$, the algorithm cannot distinguish
 $f$ from $f_0$ and hence $Q(f)=Q(f_0)$.
 This yields $f(Q(f))=0$ and hence
 the error of the algorithm $Q$ is at least $\sup f$.
 With the symmetry of $F$, we obtain
 $$
  \err(Q) \geq \sup_{\substack{f\in F\\ f\mid_{P}=0}} \sup f
  = \sup_{\substack{f\in F\\ f\mid_{P}=0}} \norm{f}_\infty
 \geq \inf_{\substack{P\subset [0,1]^d\\ \card(P)\leq n+1}}
 \sup_{\substack{f\in F\\ f\mid_P=0}} \norm{f}_\infty,
 $$
 as it was to be proven.
\end{proof}

Theorem~\ref{thm:error formula}
implies that the lower bound of Proposition~\ref{prop:opt and app} 
coincides with 
the $(n+1)^{\rm st}$ minimal error of the approximation problem
if $F$ is convex and symmetric.
Thus optimization is just as hard as uniform approximation in this case.
Since $\C^r_d$ and $\widetilde\C^r_d$ are convex and symmetric,
we can translate Theorems~\ref{main theorem even},
\ref{main theorem even odd} and \ref{side theorem}
for the problem of global optimization.
For example, we obtain the following result.

\begin{cor}
\label{cor of main theorem even}
Let $r\in\IN$ be even. 
Then there are positive constants $c_r$, $C_r$ and $\varepsilon_r$
such that
\begin{equation*}
 \brackets{c_r \sqrt{d}\, \varepsilon^{-1/r}}^d
 \leq \comp(\varepsilon,\mathcal{P}[\OPT,\C^r_d]) \leq
 \brackets{C_r \sqrt{d}\, \varepsilon^{-1/r}}^d
\end{equation*}
for all $d\in\IN$ and $\varepsilon\in (0,\varepsilon_r)$.
The upper bound holds for all $\varepsilon>0$.
\end{cor}

The class $F_{r,M}^d$ of rank one tensors
is not convex. However, the lower bounds of
Theorem~\ref{thm: lower bounds rank1} were in fact
proven for the left hand side in Proposition~\ref{prop:opt and app}.
This yields the following.

\begin{cor}
\label{cor:tractability results}
 The problem $\mathcal{P}[\OPT,F_{r,M}^d]$ of 
 global optimization on $F_{r,M}^d$
 with deterministic standard information in the worst case setting
\begin{itemize}
 \item suffers from the curse of dimensionality
 iff $M\geq 2^r r!$.
 \item is quasi-polynomially tractable iff $M<2^r r!$.
 \item is polynomially tractable iff $M\leq r!$.
 \item is strongly polynomially tractable iff $M=0$ and $r=1$.
\end{itemize}
\end{cor}

Thus, in the sense of tractability, 
global optimization is just as hard
as uniform approximation also for the non-convex
class $F=F_{r,M}^d$.

\section{Dispersion}
\label{sec:dispersion}

Let $d\in\IN$ and let $\mathcal{S}_d$ be the set of all finite
subsets of $[0,1]^d$.
The dispersion of a point set $P\in \mathcal{S}_d$
is the volume of the largest empty box
amidst the point set, that is, 
\begin{equation*}
 \disp(P)=\sup\set{\lambda^d(B) \mid B\in\mathcal B_d, B\cap P =\emptyset },
\end{equation*}
where $\mathcal B_d$ is the set of all axis-aligned boxes inside $[0,1]^d$.
Point sets with small dispersion already proved to be useful
for the uniform recovery of rank one tensors,
see Section~\ref{sec:dispersion},
and for the discretization of the uniform norm of
trigonometric polynomials~\cite{Te17}.
Recently great progress has been made in the question
for the minimal cardinality for which there \emph{exists} a point set
whose dispersion is at most $\varepsilon$,
$$
 \comp(\eps,\mathcal{P}_d)=
 \min\set{\card(P) \mid P\in\mathcal{S}_d,\, \disp(P)\leq \varepsilon },
$$
see 
\cite{DJ13,AHR17,Ru18,So18,UV18}.
This is the $\varepsilon$-complexity of
the problem $\mathcal{P}_d=(\mathcal{S}_d,\disp,\card)$
as already considered in Example~\ref{ex:dispersion}.
Here we shall \emph{provide} a point set
with small cardinality
that achieves the desired dispersion.
This point set has a simple geometric structure.
It is generated by the one-dimensional sets
\begin{equation*}
 M_j=\set{\frac{1}{2^{j+1}},\frac{3}{2^{j+1}},\hdots,\frac{2^{j+1}-1}{2^{j+1}}}
\end{equation*}
for $j\in\IN_0$. The $d$-dimensional point set of order $k\in\IN_0$
is defined as
\begin{equation*}
 P(k,d)=\bigcup_{\abs{\mathbf j}=k}
 M_{j_1} \times\dots\times M_{j_d}.
\end{equation*}
These point sets are particular instances of a sparse grid
as widely used for high-dimensional numerical integration
and approximation. We refer to
Novak and Woźniakowski~\cite{NW10} and the references therein.
A picture of the set of order $3$ in dimension $2$ 
can be found in Figure~\ref{Point Set}.
Here we prove the following result.

\begin{figure}[ht]
\begin{minipage}{.76\linewidth}
\includegraphics[width=\linewidth]{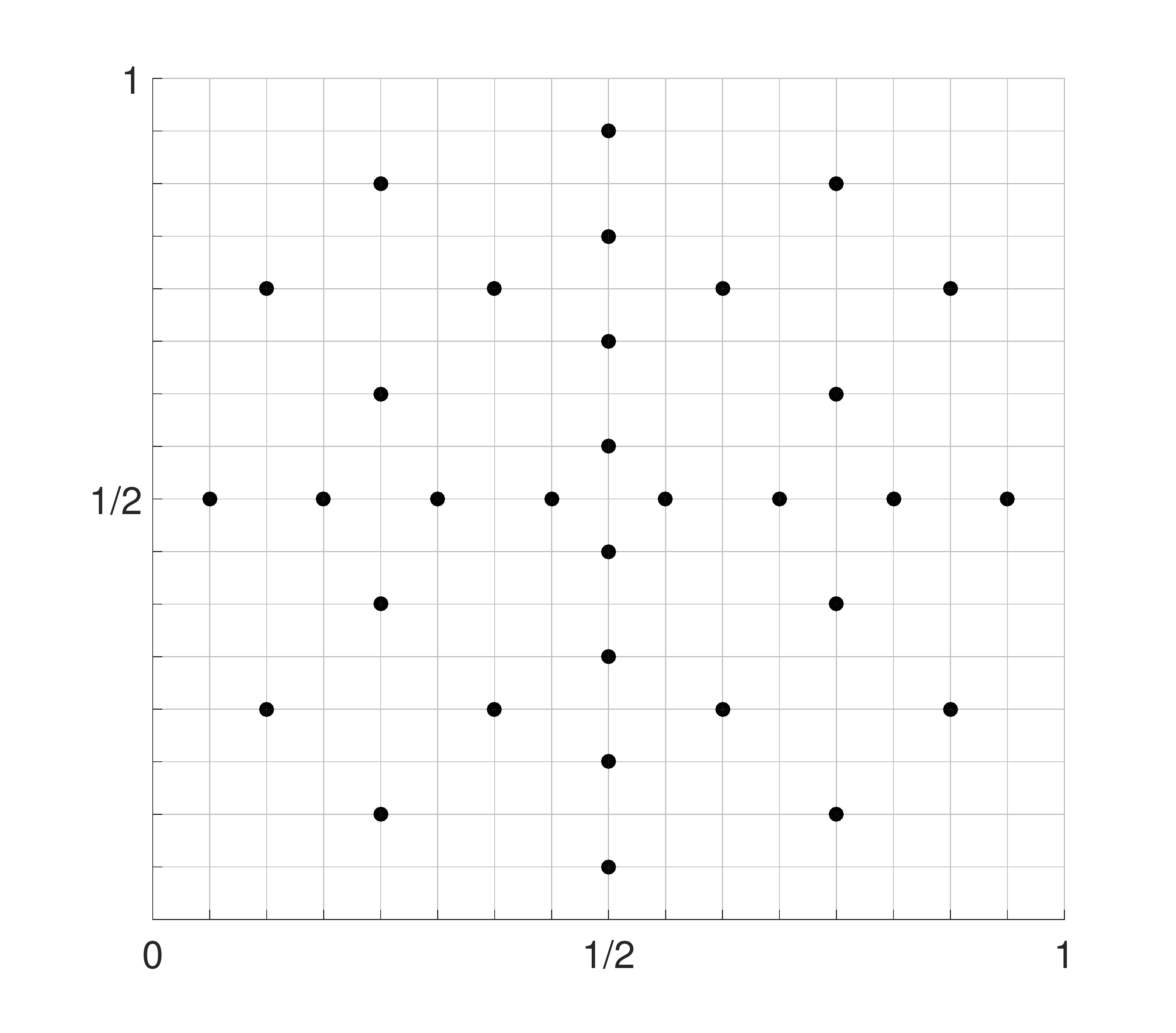}
\end{minipage}
\begin{minipage}{.22\linewidth}
\caption{The Point Set $P(3,2)$.}
\label{Point Set}
  This figure shows the set $P(k,d)$ of order $3$
  in dimension $2$.
  The largest empty box has the volume $1/16$,
  the size of $16$ of the little squares.
  If any of the $32$ points is removed,
  an empty box of volume $1/8$ emerges.
\end{minipage}
\end{figure}

\begin{thm}[\cite{Kr18b}]
\label{thm:dispersion}
 Let $\varepsilon\in (0,1)$ and $k(\varepsilon)=
 \left\lceil \log_2\brackets{\varepsilon^{-1}} \right\rceil -1$.
 For any $d\geq 2$
 the dispersion of the set $P(k(\varepsilon),d)$ is at most $\varepsilon$
 and its cardinality is given by
 \begin{equation*}
  \card\brackets{P(k(\varepsilon),d)} =
  2^{k(\varepsilon)}\, \binom{d+k(\varepsilon)-1}{d-1}.
 \end{equation*}
\end{thm}
The formula for the cardinality of $P(k(\varepsilon),d)$ 
may be simplified.
On the one hand, we have
\begin{equation*}
 \card\brackets{P(k(\varepsilon),d)}
 \leq \varepsilon^{-1}\left\lceil \log_2\brackets{\varepsilon^{-1}} \right\rceil^{d-1},
\end{equation*}
which shows that the size roughly grows linearly in $1/\varepsilon$ for a fixed dimension $d$.
On the other hand,
\begin{equation*}
 \card\brackets{P(k(\varepsilon),d)}
 \leq (2d)^{k(\varepsilon)},
\end{equation*}
which shows that the size grows at most polynomially in $d$
for a fixed error tolerance~$\varepsilon$.
Although very simple, $P(k(\varepsilon),d)$ is the smallest explicitly known
point set in $[0,1]^d$ with dispersion at most $\varepsilon$
for many instances of $\varepsilon$ and $d$, see Section~\ref{Comparison Section}.


\subsection{Proof of Theorem~\ref{thm:dispersion}}

In the following,
we write $[d]=\set{1,\hdots,d}$ for each $d\in\IN$.
The vector $\mathbf e_\ell\in\IR^d$ has entry 1 in the $\ell^{\rm th}$
and 0 in all other coordinates.
We start with computing the number of elements in $P(k,d)$ for $k\in\IN_0$
and $d\in\IN$.

\begin{lemma}
 \begin{equation*}
  \card(P(k,d))=2^k\, \binom{d+k-1}{d-1}.
 \end{equation*}
\end{lemma}

\begin{proof}
 Note that the cardinality of $M_j$ is $2^j$ for all $j\in\IN_0$. 
 The identity
 \begin{multline*}
  \card(P(k,d))=
  \sum_{\abs{\mathbf j}=k}
  \card(M_{j_1}\times\hdots\times M_{j_d})\\
  = \sum_{\abs{\mathbf j}=k}
  2^{j_1+\hdots+j_d}
  = 2^k\, \card\set{\mathbf j\in\IN_0^d\mid \abs{\mathbf j}=k}
 \end{multline*}
 yields the statement of the lemma.
\end{proof}

It follows from \cite[Theorem 2.3]{Te17} 
that the dispersion of $P(k,d)$
decays with order $2^{-k}$ if $d$ is fixed and $k$ tends to infinity.
For our purposes, however, we need to study the dependence of 
the dispersion of $P(k,d)$ on both $k$ and $d$.
In turns out that
the dispersion can be computed precisely.
In dimension $d=1$, it is readily checked that the dispersion
equals $2^{-k}$ for $k\geq 1$ and $1/2$ for $k=0$.
In any other case, we obtain the following.

\begin{lemma}
For any $k\in\IN_0$ and $d\geq 2$, we have 
$$
 \disp(P(k,d))=2^{-(k+1)}.
$$
\end{lemma}

\begin{proof}
We first observe that there are many boxes of volume $2^{-(k+1)}$
which do not intersect with $P(k,d)$. For instance, the box
\begin{equation*}
 (0,2^{-(k+1)})\times(0,1)\times\dots\times(0,1)
\end{equation*}
has these properties.
This yields $\disp(P(k,d))\geq 2^{-(k+1)}$.
To prove the upper bound, 
let $B=I_1\times\dots\times I_d$ be any box in $[0,1]^d$
whose intersection with $P(k,d)$ is empty. 
The set
\begin{equation*}
 P= \bigcup_{m\in\IN} P(m,d)
 = \set{\frac{\alpha}{2^\beta} \,\big\vert\, 
 \beta \in\IN \text{ and } \alpha\in \left[2^\beta-1\right]}^d
\end{equation*}
is dense in $[0,1]^d$.
Without loss of generality, we assume that the interior of $B$ is nonempty.
Therefore, $B$ has nonempty intersection with $P$ and hence with $P(m,d)$ for some~$m\in\IN$.
Let $m$ be minimal with this property.
Since $B$ has empty intersection with $P(k,d)$, we either have $m>k$ or $m<k$.
Let $\mathbf x\in P(m,d)\cap B$.
This means that there is some $\mathbf j\in\IN_0^d$ with
$\abs{\mathbf j}=m$ and
\begin{equation*}
 x_\ell \in M_{j_\ell} \cap I_\ell
\end{equation*}
for all $\ell\in[d]$.
We observe that the numbers
$x_\ell \pm \frac{1}{2^{j_\ell+1}}$ are either
contained in $\set{0,1}$ or in $M_j$ for some $j< j_\ell$.
Hence, they are not contained in $I_\ell$,
because $I_\ell$ is a subset of $(0,1)$ and $m$ is minimal.
We obtain that 
\begin{equation*}
 I_\ell \subset
 \brackets{x_\ell - \frac{1}{2^{j_\ell+1}},x_\ell + \frac{1}{2^{j_\ell+1}}},
\end{equation*}
and hence
\begin{equation*}
 \lambda^d\brackets{B} \leq \prod_{\ell\in[d]} 2^{-j_\ell} = 2^{-m}.
\end{equation*}
In the case $m>k$, this yields the statement.
In the case $m<k$, we observe that the
numbers
$x_\ell \pm \frac{1}{2^{k-m+j_\ell+1}}$
cannot be contained in $I_\ell$ for any $\ell\in[d]$,
since otherwise the points
\begin{equation*}
\mathbf x \pm \frac{\mathbf e_\ell}{2^{k-m+j_\ell+1}}
\end{equation*}
would be both in $B$ and in $P(k,d)$.
This means that
\begin{equation*}
 I_\ell \subset
 \brackets{x_\ell - \frac{1}{2^{k-m+j_\ell+1}},x_\ell + \frac{1}{2^{k-m+j_\ell+1}}}.
\end{equation*}
We obtain
\begin{equation*}
 \lambda^d(B) \leq 
 \prod_{\ell\in[d]} 2^{m-k-j_\ell} 
 = 2^{dm-dk-m}
 \leq 2^{-(k+1)},
\end{equation*}
where we used that $d\geq 2$.
This yields $\disp(P(k,d))\leq 2^{-(k+1)}$.
\end{proof}

Note that the smallest number $k\in\IN_0$
that satisfies $2^{-(k+1)}\leq\varepsilon$
for some fixed $\varepsilon\in (0,1)$ is given by
\begin{equation*}
 k(\varepsilon)=\left\lceil \log_2\brackets{\varepsilon^{-1}} \right\rceil -1.
\end{equation*}
This yields the statement of Theorem~\ref{thm:dispersion}.

\subsection{A Comparison with Known Results}
\label{Comparison Section}

Let $d\geq 2$ be an integer and let $\varepsilon\leq 1/4$ be positive.
Let us call $P\in \mathcal{S}_d$ admissible
if the dispersion of $P$ is at most $\varepsilon$.
In 2017, Aistleitner, Hinrichs and Rudolf~\cite{AHR17} proved
that any admissible point set satisfies
\begin{equation}
\label{lower bound disp}
 \card(P) \geq 
 (4\varepsilon)^{-1} (1-4\varepsilon)\log_2 d.
\end{equation} 
At that time, the smallest known admissible point set
was a finite Halton-Hammersley sequence $H$ of size
\begin{equation}
\label{hammersley}
\card(H) \leq
 \left\lceil 2^{d-1} \pi_d\, \varepsilon^{-1}\right\rceil,
\end{equation}
where $\pi_d$ is the product of the first $(d-1)$ primes.
This was proven by Rote and Tichy~\cite{RT96}, see also 
Dumitrescu an Jiang~\cite{DJ13} for more details.
The cardinality of this set grows as slowly as possible as $\varepsilon$
tends to zero and $d$ is fixed.
However, it grows super-exponentially with $d$.
Larcher realized that there is a $(t,m,d)$-net $N$
which is admissible and satisfies
\begin{equation}
\label{nets}
\card(N) \leq 
 \left\lceil 2^{7d+1} \varepsilon^{-1}\right\rceil.
\end{equation}
The proof is included in~\cite{AHR17}.
This number is smaller than (\ref{hammersley}) for $d\geq 54$.
However, its exponential growth with respect to $d$ for fixed $\varepsilon$
is still far away from the logarithmic growth of the lower bound (\ref{lower bound disp}).
In the beginning of 2017, Rudolf~\cite{Ru18} significantly narrowed
this gap. Based on 
results of Blumer, Ehrenfeucht, Haussler and Warmuth~\cite{BEHW89},
he obtained the existence of an admissible point set with
\begin{equation}
\label{Daniel}
\card(P) \leq
 \left\lfloor 8d\, \varepsilon^{-1} \ln\brackets{33\varepsilon^{-1}} \right\rfloor.
\end{equation}
Quite recently, the remaining gap was closed by Sosnovec~\cite{So18},
who proved the existence of an admissible point set with
\begin{equation}
 \label{Jacub}
 \card(P) \leq
 \left\lfloor q^{q^2+2}(1+4\ln q)\cdot \ln d \right\rfloor,
 \quad
 q=\lceil 1/\varepsilon\rceil.
\end{equation}
This shows that the logarithmic
dependence on the dimension in (\ref{lower bound disp})
is sharp.
On the other hand, the upper bound~\eqref{Jacub}
depends super-exponentially on $1/\varepsilon$.
This was improved by Ullrich and Vyb\'iral \cite{UV18} 
who proved the existence of an admissible point set with
\begin{equation}
\label{eq:UllVyb}
\card(P)\leq
  \left\lceil 2^7\, 
  \varepsilon^{-2} \brackets{1 + \log_2\brackets{\varepsilon^{-1}}}^2  
  \,\log_2 d \right\rceil.
\end{equation}
Up to now, this is the best known
upper bound for the minimal cardinality $\comp(\eps,\mathcal{P}_d)$
of admissible point sets for many parameters $\eps$
and $d$.
We point to the fact that the upper bounds \eqref{Daniel}, \eqref{Jacub}
and \eqref{eq:UllVyb}
are based on the probabilistic method
and only yield the existence of the point set $P$.
On the other hand, it is shown by Ullrich and Vybiral in~\cite{UV19}
that one can construct an admissible point set $P$ with
\begin{equation}
\label{eq:UllVyb constructive}
\card(P)\leq
  \left\lceil C\, 
  \varepsilon^{-6} \brackets{1 + \log_2\brackets{\varepsilon^{-1}}}^6  
  \,\log_2 d^* \right\rceil
\end{equation}
with $d^*=\max\{d,2/\varepsilon\}$ and an absolute constant $C$.
The construction takes a running time which is 
polynomial in $d$ but super-exponential in $\varepsilon^{-1}$.

Here, we provided an admissible point set $P(k(\varepsilon),d)$ with
\begin{equation}
 \label{new}
 \card(P(k(\varepsilon),d)) =
 2^{k(\varepsilon)}\, \binom{d+k(\varepsilon)-1}{d-1},
 \quad
 k(\varepsilon)=\lceil \log_2\brackets{\varepsilon^{-1}}\rceil -1.
\end{equation}
It can be constructed in a running time which is linear in the cardinality.
For many parameters $(\varepsilon,d)$ this cardinality is much smaller than
the cardinalities of the point sets from \eqref{hammersley}, \eqref{nets} and \eqref{eq:UllVyb constructive}.
In some cases, it is even smaller
than the cardinalities resulting from the 
nonconstructive results \eqref{Daniel}, \eqref{Jacub}
and \eqref{eq:UllVyb}.
To illustrate these facts, we consider the dimension
$d\in\set{2,\hdots,100}$ and error tolerance
$\varepsilon\in\set{1/4,1/5,\dots,1/100}$ in Figure~\ref{comparison figure}.

\begin{figure}[ht]
 \includegraphics[width=\textwidth]{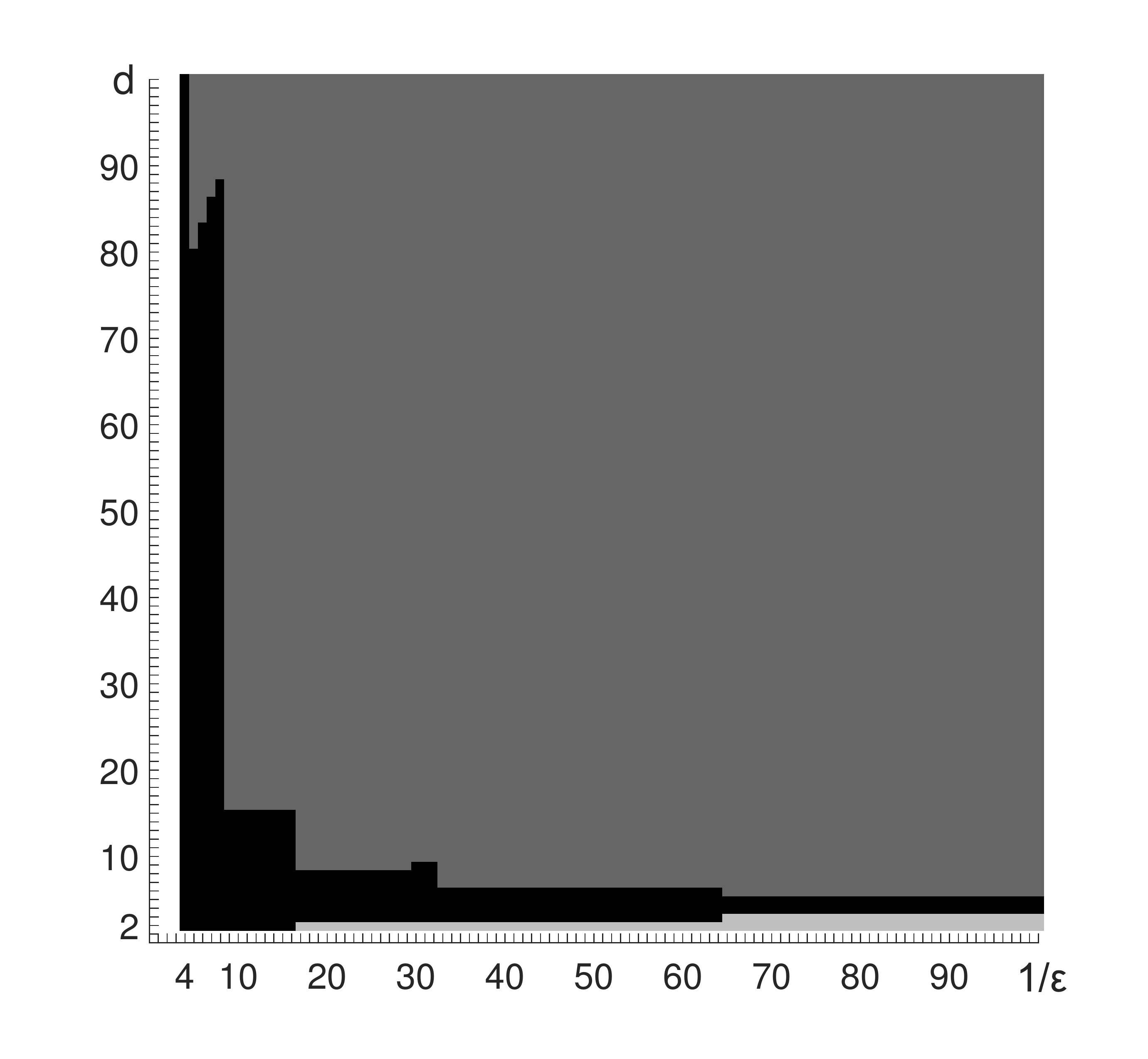}
 \caption{Cardinalities of admissible point sets.
 For the parameters $(\eps^{-1},d)$ 
 in the dark gray area, 
 the author
 does not know an admissible set which is smaller than 
 the sparse grid $P(k(\varepsilon),d)$
 although the existence of such a point set 
 follows form relation \eqref{Daniel}.
 In the black area, 
 it is not even clear whether such a point
 set exists.
 In the light gray area, the Halton-Hammersley set from~\eqref{hammersley} is 
 a smaller admissible set.}
 \label{comparison figure}
\end{figure}


\chapter{Optimal Information versus Random Information}
\label{chap:random info}

In complexity theory, 
we often want to approximate the solution of a linear
problem based on $n$ pieces of information
about the unknown problem instance.
The power of the information is measured by the
minimal worst case error that can be achieved
with the given information.
Usually, we assume that some kind of oracle is available
which grants us this information at our request.
We call the oracle $n$ times to get $n$ pieces of information.
Of course, we try to choose clever questions 
such that the information is most powerful.
We hope to obtain \emph{optimal information}.


Often, however, this model does not match reality.
There is no oracle which we can call at our will.
The information rather comes random 
and we simply have to get along with the information at hand.
Note that this is a standard assumption
in learning theory and uncertainty quantification.
It may also happen that an oracle is available
but we just do not know which questions to ask
to obtain optimal information from the oracle.
Also in this case, we may simply ask random questions. 
What we obtain is \emph{random information}.

In this chapter, we want to compare the power
of optimal information with the 
expected power of random information.
It is clear that random information cannot
be better than optimal information.
But how much do we loose?
We study this question for some basic examples.
Depending on the problem,
the answers will range from \emph{almost nothing}
over \emph{a little} up to \emph{almost everything}.
But before we turn to these examples,
let us state the general question a little more precisely.

A linear problem is given by
a linear solution operator $S$
that maps from a convex and symmetric subset $F$ 
of a normed space to a normed space $G$
and a class $\Lambda$ of continuous linear functionals on $F$,
the class of admissible measurements.
We may think of an integration problem,
where $S(f)$ is the integral of a function $f$,
or a recovery problem,
where $S$ is an embedding.
One wants to approximate the solution $S(f)$
for unknown $f\in F$
based on $n$ of these measurements
such that we can guarantee a small error with
respect to the norm in $G$.
We refer the reader to
Section~\ref{sec:LPs} for more details.
We consider a random family of information mappings
$$
 N_n:F\to \IR^n,
 \quad
 N_n(f)=(L_1(f),\hdots,L_n(f)),
$$
where the random functionals $L_i\in \Lambda$ 
are independent and
identically distributed.
The power or quality of the information mapping is measured by the
radius of information
$\rad(N_n,F,S,G)$.
This is the worst case error of the best algorithm
$A_n=\phi\circ N_n$ based on the information $N_n$,
see Section~\ref{sec:errors}.
The goal is to compare 
$$
 \inf_{N_n}\,\rad(N_n,F,S,G)
 \qquad\text{vs.}\qquad
 \IE\brackets{\rad(N_n,F,S,G)},
$$
the radius of optimal information
and the expected radius of random information.

If the infimum and the expected value are comparable,
this means that there are many good algorithms
based on many different information mappings.
In this case, 
optimal information
and therefore optimal algorithms are
not very special. 
On the other hand, if the infimum is significantly smaller
than the expected value,
this means that optimal information is very special.
It seems to be an interesting characteristic
of a problem whether 
optimal information is special or not.

Of course, the answer to this question
heavily depends on the distribution
of our measurements.
While the question may be interesting for many distributions,
we feel that there often is a natural choice.
Often, the distribution only depends on the class $\Lambda$
of admissible measurements. 
In this case, collecting random information 
might even be a good idea if optimal information is available.
It may happen that we do not loose much in
terms of the radius but gain the following nice properties.
\begin{itemize}
 \item Since the distribution is independent of $n$, 
 it is easy to increase the number of
 measurements if our current approximation is not yet
 satisfactory.
 \item The information can be used for
 many different input classes $F$,
 solution operators $S$ and target spaces $G$.
 It is \emph{universal}.
\end{itemize}
We note that the second property
does not mean that the corresponding algorithm $A_n=\phi\circ N_n$
is universal. The optimal choice of $\phi$ 
usually depends on $F$, $S$ and $G$.

We will study this question for two linear problems.
In both cases, 
there is a rather canonical choice for the
distribution of the measurements.
The first problem is the $L^p$-approximation
of $d$-variate Lipschitz functions from
standard information. We assume that
random information is given by function values
at $n$ random points that are independent and uniformly
distributed on the domain.
This problem is
studied in Section~\ref{sec:random info Lipschitz}.
%
%
The second problem is the
$\ell^2$-approximation
of a point from an $m$-dimensional ellipsoid
by means of $n$ linear measurements,
where we imagine that $m$ is much larger than $n$.
We assume that random information is given by scalar products
in $n$ directions taken independently from the 
uniform distribution on the sphere in $\IR^m$.
This is studied in Section~\ref{sec:random info hilbert},
which is based on~\cite{HKNPU19}.

We point to the fact that several examples of the sort
\emph{random information is good}
can be deduced from various papers that use the probabilistic method.
We refer to \cite{GG84,SW98,HNWW01,UV18}.
In these papers, the authors introduce a random
family of algorithms or point sets
and show that the expected worst case error (respectively
discrepancy or dispersion) is small.
This is used to obtain the existence of good algorithms.
However, it actually implies that \emph{most}
of the algorithms in that family are good.
Therefore, the expected radius of the 
random information that lies on
the bottom of these algorithms must also be small.

\section{Standard Information for Lipschitz Functions}
\label{sec:random info Lipschitz}

Let $\dist$ denote the maximum metric on the $d$-torus,
that is,
$$
 \dist(\mathbf{x},\mathbf{y})=\min_{\mathbf{k}\in\IZ^d} 
 \norm{\mathbf{x}+\mathbf{k}-\mathbf{y}}_\infty
 \quad\text{for}\quad
 \mathbf{x},\mathbf{y}\in[0,1]^d.
$$
We study the problem of $L^p$-approximation for $1\leq p\leq \infty$ on the class 
$$
 F_d=\set{f:[0,1]^d\to \IR \,\big\vert\, \forall\mathbf{x},\mathbf{y}\in[0,1]^d:
 \abs{f(\mathbf{x})-f(\mathbf{y})}
 \leq \dist(\mathbf{x},\mathbf{y})}
$$
of Lipschitz continuous functions on the $d$-torus
with deterministic algorithms based on $n$
pieces of standard information in the worst
case setting.
We note that this is a linear problem
and thus it is enough to consider nonadaptive information
$$
  N_n:F_d\to \IR^n, \quad N_n(f)=(f(\mathbf{x}))_{\mathbf{x}\in P_n}
$$
for point sets $P_n\subset [0,1]^d$ of cardinality $n$,
see Corollary~\ref{cor:adaption does not help}.
The quality of the information mapping $N_n$ is measured
by its radius
$$
  \rad(N_n,\APP,F_d,L^p)
 = \adjustlimits\inf_{\phi:\IR^n\to L^p} \sup_{f\in F_d} \norm{\phi(N_n(f))-f}_p,
$$
which is the worst case error of the best algorithm based on $N_n$,
or alternatively, by its radius at zero
$$
 r_{\mathbf 0}(N_n,\APP,F_d,L^p)
 = \sup_{f\in F_d: f\vert_{P_n}=0} \norm{f}_p.
$$
By Theorem~\ref{thm:radius at zero} the radius at zero coincides 
with the overall radius up to a factor of at most 2.
In this case, we even know that the additional
factor is not necessary and that the optimal algorithm
based on $N_n$ works as follows.

\begin{alg}
\label{alg:optimal alg Lipschitz}
 Given $N_n$ as above and $f\in F_d$, let
 \[
  f^+=\min_{\mathbf{x}\in P_n} \brackets{f(\mathbf x) + \dist(\cdot,\mathbf x)}
  \quad\text{and}\quad
  f^-=\max_{\mathbf{x}\in P_n} \brackets{f(\mathbf x) - \dist(\cdot,\mathbf x)}.
 \]
 We define $A_n(f)=(f^++f^-)/2$.
\end{alg}

Note that $f^+$ and $f^-$ are the maximal and the minimal 
 function in $F_d$ that interpolate $f$ at the points of $P_n$.

\begin{lemma}
For any nonadaptive information $N_n$,
Algorithm~\ref{alg:optimal alg Lipschitz}
satisfies
 $$
 \sup_{f\in F_d} \norm{f-A_n(f)}_p
  =\rad(N_n,\APP,F_d,L^p)=r_{\mathbf 0}(N_n,\APP,F_d,L^p).
 $$
\end{lemma}

\begin{proof}
 Clearly $A_n$ is of the form $\phi\circ N_n$
 for some mapping $\phi:\IR^n\to L^p$.
 The definitions of the radii easily yield 
 the first two inequalities of
 $$
  \sup_{f\in F_d} \norm{f-A_n(f)}_p
  \geq \rad(N_n) \geq r_{\mathbf 0}(N_n)
  \geq \sup_{f\in F_d} \norm{f-A_n(f)}_p.
 $$
 On the other hand, any $f\in F_d$ satisfies the pointwise estimate
 $$
  \abs{f - A_n(f)} \leq \frac{f^+-f^-}{2}.
 $$
 This implies the remaining inequality
 since the right hand side
 is an element of $F_d$ that vanish on $P_n$.
\end{proof}

It is known that \emph{optimal} information satisfies
$$
 \inf_{N_n}\rad(N_n,\APP,F_d,L^p) \asymp n^{-1/d}
$$
for all $1\leq p\leq\infty$.
This follows from the upper bound on the complexity
of uniform approximation as studied in \cite{Su78} 
and the lower bound on the complexity of numerical
integration as studied in \cite{Su79}.
Using the proof technique of the latter, 
we even obtain a precise formula
for the minimal radius if $n=m^d$.

\begin{prop}
 Let $n=m^d$ for some $m\in\IN$.
 Then
 $$
  \inf_{N_n} \rad(N_n,\APP,F_d,L^p)
  = \left\{\begin{array}{ll}
  		\displaystyle
        \frac{1}{2} \sqrt[\leftroot{1}\uproot{2}p]{\frac{d}{d+p}}\,n^{-1/d}
        & \text{if} \quad \displaystyle 1\leq p<\infty,\vspace*{2mm}\\
        \displaystyle
        \frac{1}{2}\,n^{-1/d} & \text{if} \displaystyle \quad p=\infty.
        \end{array}\right.
 $$
The infima are attained for 
 $P_n=\{i/m \mid 0\leq i<m\}^d$.
\end{prop}

\begin{proof}
 We realize that the function $\dist(\cdot,P_n)$
 is contained in $F_d$ and vanishes on $P_n$.
 On the other hand, every other function $f\in F_d$
 that vanishes on $P_n$ must satisfy
 $$
  \abs{f(\mathbf{x})} \leq \dist(\mathbf{x},P_n)
 $$
 for all $\mathbf{x}\in[0,1]^d$. 
 This yields
 $$
  \rad(N_n,\APP,F_d,L^p)
  = \sup_{f\in F_d: f\vert_{P_n}=0}\,\norm{f}_p
  = \norm{\dist(\cdot,P_n)}_p.
 $$
 Let us first consider the case $p=\infty$.
 Since the volume of the union of the balls 
 $B_r^\infty(\mathbf x)$ over $\mathbf x\in P_n$ is smaller than 1
 for all $r<1/(2m)$, there must be some $\mathbf{x}\in [0,1]^d$
 with $\dist(\mathbf{x},P_n)\geq r$ and thus
 $$
  \norm{\dist(\cdot,P_n)}_\infty \geq \frac{1}{2m}.
 $$
 It is easy to see that equality is satisfied for $P_n=\{i/m \mid 0\leq i<m\}^d$.
 Let us now turn to the case $p<\infty$.
 We have the formula
 $$
  \rad(N_n,\APP,F_d,L^p)^p
  = \int_{[0,1]^d} \dist(\mathbf{x},P_n)^p~\d\mathbf{x}
  = \int_0^{\infty} \lambda^d\brackets{\dist(\mathbf{x},P_n)^p\geq t}~\d t.
 $$
 We note that
 $$
  \lambda^d\brackets{\dist(\mathbf{x},P_n)^p\geq t}
  = 1-\lambda^d\brackets{B_{t^{1/p}}^\infty(P_n)}
  \geq 1-n(2t^{1/p})^d,
 $$
 where equality holds if the sets
 $B_{t^{1/p}}^\infty(\mathbf y)$ for $\mathbf y\in P_n$
 are pairwise disjoint.
 This yields
 \begin{multline*}
  \rad(N_n,\APP,F_d,L^p)^p
  \geq \int_0^{(1/2m)^p} \lambda^d\brackets{\dist(\mathbf{x},P_n)^p\geq t}~\d t\\
  \geq \frac{1}{(2m)^p} - 2^d n \int_0^{(1/2m)^p} t^{d/p}~\d t
  = \frac{1}{(2m)^p} \frac{d}{d+p}
 \end{multline*}
 with equality for $P_n=\{i/m \mid 0\leq i<m\}^d$.
 This proves the statement.
\end{proof}

In the following, 
we want to study the quality of an \emph{average} information
mapping with cost $n$.
That is, we ask for the expected radius of the random information
$$
 N_n: F_d\to\IR^n, \quad
 N_n(f)=\brackets{f\brackets{\mathbf{x}^{(1)}},\hdots,f\brackets{\mathbf{x}^{(n)}}},
$$
where the points $\mathbf{x}^{(i)}$ are 
independent and uniformly distributed
in $[0,1]^d$.
If $p$ is finite, the $p^{\rm th}$ moment of the radius
at zero can be computed precisely.

\begin{thm}
 Let $p> 0$ and $n\in\IN$.
 Then
 $$
  \IE \brackets{\rad(N_n,\APP,F_d,L^p)}^p
  = \frac{1}{2^p} \frac{n!}{(p/d+1)\cdots(p/d+n)}.
 $$
 In particular, the following sequences are strongly
 equivalent as $n$ tends to infinity:
 $$
  \sqrt[\leftroot{1}\uproot{2}p]{\IE \brackets{\rad(N_n,\APP,F_d,L^p)}^p}
  \sim
  \frac{1}{2} \sqrt[\leftroot{1}\uproot{2}p]{\Gamma(p/d+1)}\cdot n^{-1/d}.
 $$
\end{thm}

\begin{proof}
 Let $P_n=\{\mathbf{x}^{(1)},\hdots,\mathbf{x}^{(n)}\}$.
 Recall that
 $$
  \rad(N_n,\APP,F_d,L^p)^p=
  \int_{[0,1]^d} \dist(\mathbf{x},P_n)^p~\d\mathbf{x}.
 $$
 Using Tonelli's theorem, we obtain
 $$
  \IE \brackets{\rad(N_n,\APP,F_d,L^p)}^p
  = \int_{[0,1]^d} \IE\brackets{\dist(\mathbf{x},P_n)}^p~\d\mathbf{x}.
 $$
 We will show that the integrand of the latter
 integral is constant.
 Let us fix $\mathbf{x}\in[0,1]^d$ and
 note that $\dist(\mathbf{x},P_n)\in[0,1/2]$.
 For any $t\in[0,1/2]$ we have 
 $$
  \dist(\mathbf{x},P_n) \geq t
  \quad\Leftrightarrow\quad
  \forall i\in\set{1,\hdots,n}:
  \mathbf{x}^{(i)} \not\in B_t^\infty(\mathbf{x})
 $$
 and thus
 $$
  \IP\brackets{\dist(\mathbf{x},P_n) \geq t}
  = \brackets{1-(2t)^d}^n.
 $$
 The substitution $s=1-(2t^{1/p})^d$ and
 integration by parts yields
 \begin{multline*}
  \IE\brackets{\dist(\mathbf{x},P_n)}^p
  =\int_0^{2^{-p}} \IP\brackets{\dist(\mathbf{x},P_n)^p \geq t}~\d t
  =\int_0^{2^{-p}} \brackets{1-(2t^{1/p})^d}^n~\d t\\
  =\frac{p/d}{2^p}\int_0^1 s^n (1-s)^{p/d-1}~\d s
  =\frac{1}{2^p} \frac{n!}{(p/d+1)\cdots(p/d+n)}
 ,\end{multline*}
 which implies the statement of our theorem.
\end{proof}

For $p\geq 1$, the expected radius is bounded
above by its $p^{\rm th}$ moment
and bounded below by the radius of optimal information.
This leads to the following corollary.

\begin{cor}
\label{cor:Lip p<infty}
 Let $1\leq p <\infty$. Then
 $$
  \IE\brackets{ \rad(N_n,\APP,F_d,L^p)}
  \asymp \inf_{N_n}\rad(N_n,\APP,F_d,L^p) 
  \asymp n^{-1/d}.
 $$
\end{cor}

Thus, in the sense of order of convergence,
an average information mapping is already optimal 
for the problem of $L^p$-approximation on $F_d$.

\begin{rem}[Modifications of $F_d$]
 The rates of convergence of the average and
 the optimal radius do not change
 if we replace the maximum metric on the torus
 by some equivalent metric.
 The same holds true if we change the
 Lipschitz constant or if we switch 
 to the nonperiodic setting.
\end{rem}

We now turn to the case $p=\infty$,
the problem of uniform approximation.
In this case,
the expected radius of information 
is closely related to 
the so called coupon collector's problem.
This is the question for
the random number $\tau_\ell$ of coupons that a coupon collector has
to collect to obtain a complete set of $\ell$ distinct coupons.
The following facts on the distribution of $\tau_\ell$ 
are well known. We refer to \cite{LPW09}.
Here $H_\ell=\sum_{k=1}^\ell 1/k$ is the $\ell^{\rm th}$
harmonic number. Note that $H_\ell\sim \ln\ell$
as $\ell\to\infty$. 

\begin{prop}
\label{prop:coupons}
 Let $(Y_i)_{i=1}^\infty$ be a sequence of random
 variables that are uniformly distributed in
 the set $\set{1,\hdots,\ell}$
 and let
 $$
  \tau_\ell=\min\set{n\in\IN \mid \set{Y_1,\hdots,Y_n}=\set{1,\hdots,\ell}}.
 $$
 Then
 $$
  \IE\,\tau_\ell= \ell H_\ell
  \quad\text{and}\quad
  \Var \tau_\ell\leq \ell^2 \sum_{k=1}^\ell 1/k^2
 $$
 and for any $c>0$,
 $$
  \IP\brackets{\tau_\ell > \lceil c\,\ell\ln \ell\rceil} \leq 
  \ell^{-c+1}.
 $$
\end{prop}

\begin{proof}
 For $1\leq i \leq \ell$, let $\nu_i$ be the
 number of coupons that have to be collected
 to get the $i^{\rm th}$ distinct coupon after having
 collected $i-1$ distinct coupons.
 These are independent geometric random variables with
 $$
  \IE\,\nu_i = \frac{\ell}{\ell-i+1}
  \quad\text{and}\quad
  \Var\nu_i= \frac{\ell(i-1)}{(\ell-i+1)^2}
  \leq \frac{\ell^2}{(\ell-i+1)^2}.
 $$
 Now the first two statements follow from $\tau_\ell=\sum_{i=1}^\ell \nu_i$.
 To obtain the tail bound, we consider the events
 $A_i$ that the coupon with number $i$ was not
 collected during the first $\lceil c\,\ell\ln \ell\rceil$ trials.
 Then
 $$
  \IP(A_i)=\brackets{1-1/\ell}^{\lceil c\,\ell\ln\ell\rceil}
  \leq \exp\brackets{-c\ln \ell}
  =\ell^{-c}.
 $$
 This yields
 $$
  \IP\brackets{\tau_\ell > \lceil c\,\ell\ln \ell\rceil}
  = \IP\brackets{\bigcup_{i=1}^\ell A_i}
  \leq \sum_{i=1}^\ell \IP(A_i)
  \leq \ell^{-c+1},
 $$
 as stated in the proposition.
\end{proof}

This leads to the following estimates
of the expected radius for $p=\infty$.

\begin{thm}
 Let $n\in\IN$ and let
 $$
  m_1=\min\set{m\in\IN \mid m^d (H_{m^d}-2) \geq n},
  \quad
  m_2=\max\set{m\in\IN \mid 2 m^d \ln(m^d) \leq n}.
 $$
 Then
 $$
  \frac{1}{4m_1}
  \leq \IE\brackets{ \rad(N_n,\APP,F_d,L^\infty)} \leq
  \frac{2}{m_2}.
 $$
\end{thm}

\begin{proof}
 We decompose $[0,1]^d$ into $\ell=m^d$ subcubes 
 $$
  B_{\mathbf{k}}=\prod_{i=1}^d\left[\frac{k_i-1}{m},\frac{k_i}{m}\right),
  \quad
  \mathbf{k}\in\set{1,2,\hdots,m}^d
 $$
 of equal volume for some $m\in\IN$.
 Recall that the radius at zero is given by
 $$
  \rad(N_n,\APP,F_d,L^\infty)
  = \max_{\mathbf x\in [0,1]^d} \dist(\mathbf x,P_n),
 $$ 
 and therefore bounded above by $1/m$ if every box contains
 a point $\mathbf x^{(i)}\in P_n$,
 and bounded below by $1/(2m)$ if one of the boxes does
 not contain a point.
 Let $A$ be the event that every box contains a point.
 Note that the number of random points $\mathbf x^{(i)}$
 that it takes to hit all the boxes follows the
 distribution of the coupon collector's variable $\tau_\ell$ 
 as defined in Proposition~\ref{prop:coupons}.
 Thus
 $$
  \IP(A) = \IP\brackets{\tau_\ell \leq n}.
 $$
 For the upper bound, we choose $m=m_2$.
 Proposition~\ref{prop:coupons} yields
 $$
  \IP(A^\mathsf{c}) = \IP\brackets{\tau_\ell > n} \leq 1/\ell
 $$
 and hence
 $$
  \IE\brackets{ \rad(N_n,\APP,F_d,L^\infty)}\\
 \leq \IP(A) \cdot \frac{1}{m} + \IP\brackets{A^\mathsf{c}} \cdot 1
 \leq \frac{2}{m}.
 $$
 For the lower bound, we choose $m=m_1$.
 Chebyshev's inequality yields
 $$
  \IP(A) = \IP\brackets{\tau_\ell\leq n}
  \leq \IP\brackets{\tau_\ell \leq \ell H_\ell-2\ell }
  \leq \frac{\Var \tau_\ell}{4\ell^2}
  \leq \frac{1}{2}.
 $$
 We obtain
 $$
  \IE\brackets{ \rad(N_n,\APP,F_d,L^\infty)}
  \geq \IP\brackets{A^\mathsf{c}} \frac{1}{2m} 
  \geq \frac{1}{4m},
 $$
 as it was to be proven.
\end{proof}

Note that both $m_1^d$ and $m_2^d$ are of order $n/\ln n$.
This yields the following corollary.
This corollary is already known from~\cite{BDKKW17},
where the authors study the uniform approximation
of functions on $[0,1]^d$ with bounded $r^{\rm th}$ derivative.
The upper bound for Sobolev spaces on closed manifolds 
can also be found in~\cite{EGO18}.

\begin{cor}[\cite{BDKKW17}]
\label{cor:Lip p=infty}
 For all $n\in\IN$ let $\ell=\lfloor n/\ln(n+1) \rfloor$. Then
 $$
  \IE\brackets{ \rad(N_n,\APP,F_d,L^\infty)}
  \asymp \inf_{N_\ell}\rad(N_\ell,\APP,F_d,L^\infty) 
  \asymp \ell^{-1/d}.
 $$
\end{cor}

Thus, for the problem of uniform approximation on $F_d$,
an average information mapping with cost $n$ is as good 
as an optimal information mapping with cost $n/\ln n$.

\section{Linear Information for \texorpdfstring{$\ell^2$-}{}Approximation}
\label{sec:random info hilbert}

We study random information for $\ell^2$-approximation 
of points from a high or infinite dimensional ellipsoid
and compare it to optimal information. 
The radius of optimal information with cost $n$ is given by 
the length $\sigma_{n+1}$ of the $(n+1)^{\text{st}}$ largest semi-axis
of the ellipsoid. 
The sequence $\sigma$ of semi-axes also determines 
the distribution of the radius $\mathcal R_n$ 
of Gaussian random information. 
We find that 
random information behaves
very differently depending 
on whether $\sigma \in \ell^2$ or not. 
For $\sigma \notin \ell^2$ random information is completely useless
and we have $\IE[\mathcal R_n] = \sigma_1$. 
For $\sigma \in \ell^2$ the expected radius of random information tends to zero at least 
at rate $ o (1/\sqrt{n})$ as $n\to\infty$. 
The case 
 \[
  \sigma_n \asymp n^{-\alpha} \ln^{-\beta}(n+1) , 
 \]
where $\alpha > 0$ and $\beta\in\IR$, 
is very interesting for applications. 
Here we prove 
 \[
\IE[\mathcal R_n] \asymp
  \left\{\begin{array}{cl}
  	\sigma_1 	
        &
        \text{if} \quad \alpha<1/2 \text{\, or \,} \beta\leq\alpha=1/2,
        \vspace*{2mm}
        \\
	  \sigma_{n+1} \, \sqrt{\ln(n+1)}  \quad
        &
        \text{if} \quad \beta>\alpha=1/2,
        \vspace*{2mm}
        \\
       \sigma_{n+1}  
        &
        \text{if} \quad \alpha>1/2.
        \end{array}\right.
 \]
For the proof we use
a comparison result for Gaussian processes \`a la Gordon, 
exponential estimates for sums of chi-squared random variables, 
and estimates for the extreme singular values of (structured) Gaussian random matrices.
This section is based on \cite{HKNPU19}.

\subsection{The Problem}

Let $\sigma$ be a sequence of nonnegative numbers
$\sigma_1\geq\sigma_2\geq\hdots\geq 0$.
We consider the ellipsoid
$$
 F(\sigma)
 =\set{\mathbf x \in \ell^2 \mid 
 \sum_{j=1}^\infty \frac{x_j^2}{\sigma_j^2} \leq 1
 },
$$
where we require that $x_j=0$ whenever $\sigma_j=0$.
For all $n\in\IN$ let $G_n\in\IR^{n\times\infty}$ be a random
matrix with independent standard Gaussian entries $g_{ij}$.
We want to study the distribution of the random variable
\begin{equation}
\label{eq:definition R_n}
 \mathcal{R}_n(\sigma)
 = \sup\set{\norm{\mathbf x}_2 \mid \mathbf x \in F(\sigma),\,
 G_n(\mathbf x) = 0}.
\end{equation}
Of course, the equation
$G_n(\mathbf x) = 0$
requires that the involved series converge at all.
We now give several interpretations of the
random variable $\mathcal{R}_n(\sigma)$.
We start with the case that
\begin{equation}
 \label{eq:sigma finite}
 \exists m\in\IN \colon \sigma_j=0 \Longleftrightarrow j> m. 
\end{equation}
Then $F(\sigma)$ can be regarded as an ellipsoid in $\IR^m$
and $G_n$ can be regarded as an $n\times m$-matrix.
In this case we assume that $n<m$.
In fact, our main interest lies in the case
that $n$ is much smaller than $m$.


\subsubsection{Version 1}

Let $E_n$ be uniformly distributed on the
Grassmannian manifold of $n$-codimensional hyperplanes in $\IR^m$.
The intersection of $E_n$ and $F$ is
an $(m-n)$-dimensional ellipsoid.
We study the circumradius
of the random intersection,
that is,
$$
 \mathcal{R}_n^{(1)}(\sigma)=\rad\brackets{F(\sigma)\cap E_n}.
$$
This is the radius of the smallest Euclidean
ball that contains the intersection ellipsoid,
or equivalently the length of its largest semi-axis.
It is easy to see that the radius
is maximal if $E_n$ contains $\mathbf e_1$.
In this case, it takes the value $\sigma_1$.
The minimal radius on the other hand, is attained
if $E_n$ is the
span of the vectors $\mathbf e_i$ for $i>n$.
It is given by $\sigma_{n+1}$.
Thus, we always have
$$
 \mathcal{R}_n^{(1)}(\sigma)\in [\sigma_{n+1},\sigma_1].
$$
But how large is the radius of a typical intersection?
Is it comparable to the minimal or the
maximal radius or does it behave 
completely different?

\subsubsection{Version 2}

We study the problem of recovering
$\mathbf x\in F(\sigma)$ 
from $n$ pieces of information,
where we want to guarantee a small error
in the Euclidean norm.
The information about $\mathbf x\in F(\sigma)$
is given by coordinates in $n$ directions 
$\mathbf y^{(i)}\in \mathbb{S}_{m-1}$.
This is described by the mapping
$$
 N_n:F(\sigma)\to \IR^n, \quad N_n(\mathbf x)=
 (\langle\mathbf x,\mathbf y^{(i)}\rangle)_{i=1}^n.
$$
The quality of the information mapping is measured
by its radius, which is the worst case error of the best
recovery algorithm based on the information $N_n$,
that is,
$$
  \rad(N_n,\APP,F(\sigma),\ell^2_m)
 = \adjustlimits\inf_{\phi:\IR^n\to \IR^m} 
 \sup_{\mathbf x\in F(\sigma)} \norm{\phi(N_n(\mathbf x))-\mathbf x}_2.
$$
This is a linear problem over Hilbert spaces
as described in Section~\ref{sec:LPs}
since $F(\sigma)$ is the unit ball of 
the Hilbert space $H(\sigma)=\IR^m$
equipped with the scalar product 
$$
 \Xscalar{\mathbf x}{\mathbf y}{H(\sigma)}
 = \sum_{j\leq m} \frac{x_j y_j}{\sigma_j^2}.
$$
The numbers $\sigma_j$ are the singular values
of the embedding of $H(\sigma)$ into $\ell^2_m$.
It is well known that the information is \emph{optimal}
(its radius is minimal)
if the directions $\mathbf y^{(i)}$ coincide with the $n$
largest semi-axes of the ellipsoid, see Theorem~\ref{thm:LPs over Hilbert spaces}.
The quality of optimal information is 
given by 
$$
 \min_{\mathbf y^{(1)},\hdots,\mathbf y^{(n)}\in\mathbb{S}_{m-1}} \rad(N_n,\APP,F(\sigma),\ell^2_m)
 = \sigma_{n+1}.
$$
Here we want to study the typical quality 
of \emph{random} information in
comparison to optimal information.
We ask for the radius
$$
\mathcal{R}_n^{(2)}(\sigma)=\rad(N_n,\APP,F(\sigma),\ell^2_m),
$$
of the random information mapping $N_n$
where the points $\mathbf y^{(i)}$ are independent
and uniformly distributed on the Euclidean sphere $\mathbb{S}_{m-1}$.
Is typical random information much worse than optimal information?

\subsubsection{Version 3}

Like in the previous version,
we study the radius 
of a random information mapping.
This time we consider the Gaussian information $G_n$
from above.
We denote the radius of information by
\begin{equation}
\label{eq:radius random info}
 \mathcal{R}_n^{(3)}(\sigma)=\rad(G_n,\APP,F(\sigma),\ell^2_m)
 = \adjustlimits\inf_{\phi:\IR^n\to \IR^m} 
 \sup_{\mathbf x\in F(\sigma)} \norm{\phi(G_n(\mathbf x))-\mathbf x}_2.
\end{equation}
The following lemma says
that these are indeed merely three versions
of $\mathcal{R}_n(\sigma)$.
Moreover,
the alignment of the ellipsoid
with the standard axes of $\IR^m$
is not a relevant assumption.

\begin{lemma}
 \label{lem:equivalent versions}
 Under the assumption~\eqref{eq:sigma finite},
 the random variables $\mathcal{R}_n^{(1)}(\sigma)$, $\mathcal{R}_n^{(2)}(\sigma)$
 and $\mathcal{R}_n^{(3)}(\sigma)$ have the same distribution
 as $\mathcal{R}_n(\sigma)$,
 which does not change
 if $F(\sigma)$ is replaced by $QF(\sigma)$
 for some orthogonal matrix $Q\in \mathrm{O}(m)$
 in any of their definitions.
\end{lemma}

\begin{proof}
 The orthogonal invariance immediately follows
 from the fact that the distributions of the 
 hyperplane $E_n$, the matrix $N_n$ and the matrix
 $G_n$ are invariant under orthogonal transformations.
 To see that the variables $\mathcal{R}_n^{(i)}(\sigma)$ for $i\leq 3$
 are interchangeable, we need the fact that
 $$
  \rad(A,\APP,F(\sigma),\ell^2_m)
  = \sup\set{\norm{\mathbf x}_2 \mid \mathbf x \in F(\sigma)\cap \ker(A) }
  = \rad\brackets{F(\sigma)\cap \ker(A)}
 $$
 for any matrix $A\in\IR^{n\times m}$,
 which follows from Theorem~\ref{thm:linear algorithms are optimal}.
 Now we only need to notice that the kernels 
 of $N_n$ and $G_n$ are uniformly distributed
 on the Grassmannian
 which follows from the orthogonal invariance of both 
 distributions
 and the uniqueness of the normalized Haar measure on compact groups.
\end{proof}

\begin{rem}
The radius of the section of a symmetric convex body 
with a random lower-dimensional subspace has already been studied in \cite{GM97,GM98} 
and subsequently in many other papers such as \cite{LT00,GMT05,LPT06}. 
However, one cannot expect these bounds to be sharp for the whole class 
of symmetric convex bodies, as has already been pointed out in \cite[Example 2.2]{GM97} 
for the example of ellipsoids with highly incomparable semi-axes. 
Moreover, the focus in these papers was on subspaces of proportional codimension, 
while we focus on subspaces with comparably small codimension
such as $m=2^n$.
\end{rem}

In the infinite-dimensional case, that is,
if \eqref{eq:sigma finite} does not hold,
the interpretations according to Versions~1 and 2 fail.
There is no uniform distribution on the sphere in $\ell^2$
and the Grassmannian for $m=\infty$.
However, $\mathcal{R}_n(\sigma)$ may still be
interpreted as the radius of Gaussian random information:
\begin{itemize}
 \item Let $\sigma \in \ell^2$. Then
 the matrix $G_n$ almost surely defines a bounded
 operator from the Hilbert space
 $$
 H(\sigma)=\Big\{\mathbf x\in\ell^2\,\big\vert\,\sum_{j=1}^\infty \frac{x_j^2}{\sigma_j^2}<\infty\Big\},
 \quad
 \Xscalar{\mathbf x}{\mathbf y}{H(\sigma)}=
  \sum_{j=1}^\infty \frac{x_j y_j}{\sigma_j^2}
 $$
 to $\ell^2_n$. This follows for example from \cite[Theorem~3.1]{BV16}
, see also Lemma~\ref{lem:BH}.
 Since $H(\sigma)$ is a Hilbert
 space and $F(\sigma)$ is its unit ball, we have
 $$
  \mathcal{R}_n(\sigma)=\rad(G_n,\APP,F(\sigma),\ell^2_m)
 $$
 almost surely, see Theorem~\ref{thm:linear algorithms are optimal}.
 \item Let $\sigma \not\in \ell^2$.
 Then the matrix $G_n$ almost surely defines an unbounded operator from
 $H(\sigma)$ to $\ell^2_n$.
 This follows for example from \cite[Corollary~4.1]{LVY18},
 see also Lemma~\ref{lem:smin basic}.
 The mapping $G_n$ need not even be defined for all $\mathbf x \in F(\sigma)$.
 Thus, the definition of the radius 
 $\rad(G_n,\APP,F(\sigma),\ell^2_m)$ according to 
 equation~\eqref{eq:radius random info} makes no sense
 and we need to define the radius in some other way.
 Recall that the radius is supposed to reflect the
 worst case error of the best recovery algorithm based on $G_n$.
 On the one hand, the zero algorithm has the worst case error $\sigma_1$.
 On the other hand, any algorithm based on $G_n$ cannot distinguish the
 elements $\mathbf x\in F(\sigma)$ for which $G_n(\mathbf x)=0$.
 Thus, we must have
 $$
  \mathcal{R}_n(\sigma) \leq \rad(G_n,\APP,F(\sigma),\ell^2_m) \leq \sigma_1
 $$
 for any reasonable definition of the radius.
 It will turn out that $\mathcal{R}_n(\sigma)=\sigma_1$ almost
 surely, which is why the precise definition of the
 radius does not matter.
\end{itemize}

\begin{rem}
Instead of $\ell^2$ we may also consider a separable 
$L^2$-space since both spaces 
are isometrically isomorphic. 
Then we may study a compact embedding 
of a Hilbert space $H$ into $L^2$ and denote the unit ball 
of $H$ by $F(\sigma)$, 
where $\sigma$ is the sequence of singular values
of the embedding.
An important case are Sobolev embeddings, 
where $H$ is a Sobolev space of functions that are defined 
on a bounded domain in $\IR^d$. 
It is well known that then the singular values behave as 
$ 
  \sigma_n \asymp n^{-\alpha} \ln^{-\beta}(n+1),
$  
where $\alpha$ and $\beta$ depend on the smoothness and the dimension $d$.
\end{rem}

\subsection{Results}


We prove the following bounds on the
random variable $\mathcal{R}_n(\sigma)$
which hold with high probability.
We start with upper bounds.

\begin{thm}[\cite{HKNPU19}]
\label{thm:upper bound random section}
Let $\sigma\in\ell^2$ be nonincreasing.
Then, for all $n\in\IN$ and
$c,s\in[1,\infty)$, we have
  $$
   \IP\left[
   \mathcal R_n(\sigma) \geq  \frac{221}{\sqrt{n}} 
   \bigg(\sum_{j\geq\lfloor n/4\rfloor}\sigma_j^2\bigg)^{1/2}\,
   \right]
   \leq 2e^{-n/100}
  $$
 and
  $$
   \IP\left[
   \mathcal R_n(\sigma) \geq 14sn
   \bigg(\sum_{j>n}\sigma_j^2\bigg)^{1/2}\,
   \right]
   \leq e^{-c^2 n}+\frac{c \sqrt{2e}}{s}.
$$
\end{thm}

The first estimate will turn out to be useful 
when we treat polynomially decaying sequences $\sigma$, 
while the second part is better for exponentially decaying $\sigma$.
Our lower bound is given as follows.

\begin{thm}[\cite{HKNPU19}]
\label{thm:lower}
Let $\sigma\in\ell^2$ be nonincreasing, $\varepsilon\in(0,1)$, and $n,k\in\IN$ be such that $\sigma_k\neq 0$ and
\[
\sum_{j>k} \sigma_j^2 \geq \frac{3n\sigma_k^2}{\varepsilon^2}.
\]
Then
\[
\IP\Big[\mathcal R_n(\sigma) \,\le\, \sigma_k(1-\varepsilon)\Big]
\leq 5e^{-n/64}.
\]
\end{thm}

As will become apparent in the proof, 
the lower bound of Theorem~\ref{thm:lower}
already holds for the easier problem of recovering just the $k^{\rm th}$ coordinate
of $\mathbf x\in F(\sigma)$.
As a consequence of the previous theorems, 
we obtain that random information is useful if and only if $\sigma\in\ell^2$.

\begin{cor}[\cite{HKNPU19}]
 \label{cor:l2 not l2}
 If $\sigma\not\in\ell^2$, then $\mathcal R_n(\sigma)=\|\sigma\|_\infty$
  holds almost surely for all $n\in\IN$.
 If $\sigma\in\ell^2$, then
 $$
  \lim_{n\to\infty}\sqrt{n}\,\IE[ \mathcal R_n(\sigma)] = 0.
 $$
\end{cor}

\begin{rem}
The phenomenon that the results very much depend on whether $\sigma\in\ell^2$ or not 
is known from a related problem
that was studied earlier in several papers. 
There $F$ is the unit ball of a reproducing kernel Hilbert space $H$, 
that is, $H\subset L^2(D)$ consists of functions 
on a common domain $D$ and
function evaluation $f \mapsto f(x)$ is 
a continuous functional on $H$ for every $x\in D$. 
The optimal linear information 
$N_n$ for the $L^2$-approximation problem is given by the singular value decomposition 
and has radius $\sigma_{n+1}$. 
This information might be difficult to implement and hence one 
might allow only standard information $N_n$ of the form 
\[
N_n (f) = \big( f(x_1), \dots , f(x_n)\big)\,, \qquad x_i\in D.
\]
The goal is to relate the power of 
function evaluations
to the power of all continuous 
linear functionals.
Ideally one would like to
prove that their power is roughly the same.
Unfortunately, in general
this is \emph{not} true.
In the case $\sigma \notin \ell^2$
the convergence of optimal algorithms 
that may only use standard information
can be arbitrarily slow 
\cite{HNV08}.
The situation is much better if we assume that $\sigma \in \ell^2$.
It was shown in \cite{WW01} and \cite{KWW09}
that function values are almost as good as general 
linear information. 
We refer to \cite[Chapter 26]{NW12} for a presentation of 
these results. 
We must say that we do not fully understand the analogy 
of the two different problems. 
\end{rem} 

Before we present the proofs, 
let us provide some of the results on the expected radius that
follow from our main results for special sequences.
We start with the case of polynomial decay.

\begin{cor}[\cite{HKNPU19}]
 \label{thm:main result random info hilbert}
 Let $\sigma$ be a nonincreasing sequence such that 
 $$
  \sigma_n \asymp n^{-\alpha} \ln^{-\beta}(n+1)
 $$
 for some $\alpha > 0$ and $\beta\in\IR$.
 Then
 $$
  \IE[ \mathcal{R}_n(\sigma^m)] \asymp
  \left\{\begin{array}{cl}
  		1
        &
        \text{if} \quad \alpha<1/2 \text{\, or \,} \beta\leq\alpha=1/2,
        \vspace*{2mm}
        \\
        n^{-\alpha} \ln^{-\beta+1/2}(n+1)\quad
        &
        \text{if} \quad \beta>\alpha=1/2,
        \vspace*{2mm}
        \\
        n^{-\alpha} \ln^{-\beta}(n+1)
        &
        \text{if} \quad \alpha>1/2.
        \end{array}\right.
 $$
\end{cor}

The very same estimates hold in the finite-dimensional case,
that is,
if $\sigma_j$ is replaced by 0 for all $j>m$,
provided that $n$ is small enough in comparison with $m$,
see Corollaries~\ref{cor:slow polynomial}, \ref{cor:medium polynomial}
and \ref{cor:quick polynomial}.
This means that random information is 
just as good as optimal information
if the singular values decay with
a polynomial rate greater than $1/2$.
The size of a typical intersection ellipsoid
is comparable to the size of the smallest 
intersection.
On the other hand,
if the singular values decay too slowly,
random information is completely useless.
A typical intersection ellipsoid
is almost as large as the largest.
There is also an intermediate case
where random information is
worse than optimal information,
but only slightly.
Moreover, we discuss sequences of exponential decay
and obtain the following.

\begin{cor}[\cite{HKNPU19}]
\label{thm:random sections exponential decay}
 Let $\sigma$ be a nonincreasing sequence that satisfies $\sigma_n \asymp a^n $
 for some $a\in(0,1)$. Then
 $$
  a^n
  \preccurlyeq
  \IE[ \mathcal{R}_n(\sigma)]
  \preccurlyeq
  n^2\, a^n.
 $$
\end{cor}

\begin{rem}
We have seen that $\IE[\mathcal{R}_n(\sigma)]\asymp \sigma_{n+1}$
holds for sequences with
sufficiently fast polynomial decay.
It remains open whether the same holds
for sequences of exponential decay.
We note that, despite the gap, 
the result of Corollary~\ref{thm:random sections exponential decay} is even stronger 
than the result of Corollary~\ref{thm:main result random info hilbert}
if considered from the complexity point of view.
Corollary~\ref{thm:main result random info hilbert} states
that there is a constant $c$ such that $cn$ pieces
of random information are at least as good as
$n$ pieces of optimal information.
Corollary~\ref{thm:random sections exponential decay} states that
there is a constant $c$ such that $n+c\ln n$ pieces
of random information are at least as good
as $n$ pieces of optimal information.
\end{rem}



\subsection{Proofs}
\label{sec:Proofs}

We now present the proofs of the results 
that were presented in the previous section.
But first we repeat and extend some of our notation.
Let $\sigma=(\sigma_j)_{j=1}^\infty$ 
be a nonincreasing sequence of nonnegative numbers.
We consider the Hilbert space
$$
 H(\sigma)
 =\Big\{\mathbf x \in \ell^2 \,\Big\vert\, 
 x_j=0 \text{ if } \sigma_j=0,\,
 \sum_{j=1}^\infty \frac{x_j^2}{\sigma_j^2} <\infty
 \Big\}
$$
with scalar product 
$$
 \Xscalar{\mathbf x}{\mathbf y}{H(\sigma)}
 = \sum_{j=1}^\infty \frac{x_j y_j}{\sigma_j^2}.
$$
Note that we write $\sum_{j=1}^\infty$
but only take the sum over all $j\in\IN$
for which $\sigma_j$ is positive.
The unit ball of $H(\sigma)$ is denoted by $F(\sigma)$.
The numbers $g_{ij}$
shall be independent real standard Gaussian variables
for all $i,j\in\IN$.
For index sets $I\subset \IN$ and
$J\subset \IN$, we consider the 
(structured) Gaussian
$I\times J$-matrices
$$
  G_{I,J} =\brackets{g_{ij}}_{i\in I, j\in J}
  \quad\text{and}\quad
  \Sigma_{I,J} =\brackets{\sigma_j g_{ij}}_{i\in I, j\in J}.
$$
Recall that $[k]$ denotes the set of integers
from 1 to $k$ and note that
$G_n=G_{[n],\IN}$.
Moreover, we consider
$$
 H_I(\sigma)=\set{\mathbf x \in H(\sigma) \mid x_j=0 \text{ for all } j\in \IN\setminus I}
$$
as a closed subspace of the Hilbert space $H(\sigma)$
and denote its unit ball by $F_I(\sigma)$.
The projection of $\mathbf x\in H(\sigma)$ onto $H_I(\sigma)$
is denoted by $\mathbf x_I$.
We want to study the distributions of the random variables
$\mathcal{R}_n(\sigma)$ from \eqref{eq:definition R_n}.
 
As mentioned earlier,
a crucial role in our proofs is played by estimates 
for the extreme singular values of random matrices.
So let us recall some basic facts about singular values.
Let $A$ be a real $r\times k$-matrix,
where we allow that $r=\infty$ or $k=\infty$
provided that $A$ describes a compact operator from $\ell^2_k$
to $\ell^2_r$.
For every $j\leq k$,
the $j^{\rm th}$ singular value $s_j(A)$ of this matrix 
can be defined as the square-root of the $j^{\rm th}$ largest
eigenvalue of the symmetric matrix $A^\top A$,
which describes a positive operator on $\ell^2_k$.
Note that $s_j(A)=s_j(A^\top)$ if we have $j\leq \min\{r,k\}$.
Our interest lies in the extreme singular values of $A$.
The largest singular value of $A$ is given by
$$
 s_1(A)=
 \sup_{\mathbf x \in \ell^2_k\setminus\{\mathbf 0\}} 
 \frac{\Vert A\mathbf x\Vert_2}{\Vert\mathbf x\Vert_2}
 =\norm{A:\ell^2_k \to \ell^2_r}.
$$
This number is also called the spectral norm of $A$.
The smallest singular value is given by
$$
  s_k(A)=
  \inf_{\mathbf x \in \ell^2_k\setminus\{\mathbf 0\}} 
  \frac{\Vert A\mathbf x\Vert_2}{\Vert\mathbf x\Vert_2}.
$$
Clearly, we have $s_k(A)=0$ whenever $k>r$.
If $r\leq k$, it also makes sense to talk about 
the $r^{\rm th}$ singular value of $A$.
This number equals the radius of the largest
Euclidean ball that is contained in the image of the unit ball 
of $\ell^2_k$ under $A$, that is,
$$ 
 s_r(A) = \sup \set{ \varrho \geq 0 \mid 
 B_\varrho^2(\mathbf 0) \subset A ( B_1^2(\mathbf 0))}.
$$
These extreme singular values are also defined for noncompact
operators $A$, where $A$ is restricted
to its domain if necessary.

\subsubsection{Proof of Theorem~\ref{thm:upper bound random section}}

We give upper bounds on the radius $\mathcal{R}_n(\sigma)$
in terms of $\sigma$.
Here we always assume that $\sigma\in\ell^2$.
As shown in Corollary~\ref{cor:random info useless},
this is no real restriction.
We start with a pointwise upper bound
in terms of the extreme singular values
of the corresponding (structured) Gaussian matrices.

\begin{prop}[\cite{HKNPU19}]
 \label{prop:radius vs singular values}
 Let $\sigma\in\ell^2$ be nonincreasing
 and let $k\leq n$.
 If $G_{[n],[k]}\in\IR^{n\times k}$ has full rank, then
 $$
 \mathcal{R}_n(\sigma) 
 \leq 
 \sigma_{k+1} + 
 \frac{s_1\brackets{\Sigma_{[n],\IN\setminus[k]}}}{s_k\brackets{G_{[n],[k]}}}.
 $$
\end{prop}

\begin{proof}
We first note that $s_k(G_{[n],[k]})$ is positive
if $G_{[n],[k]}$ has full rank.
Moreover, we may assume that $\mathcal{R}_n(\sigma)>0$
without loss of generality.
Let $\varrho>0$ such that $\varrho< \mathcal{R}_n(\sigma)$.
By the very definition of $\mathcal{R}_n(\sigma)$
there exists some $\mathbf y\in F(\sigma)$ 
such that $\|\mathbf y\|_2=\varrho$ 
and $G_n(\mathbf y)=0$.
The triangle inequality yields
\begin{equation}
\label{eq:rho triangle}
 \varrho = \|\mathbf y\|_2 \le \norm{\mathbf y-\mathbf y_{[k]}}_2 + \norm{\mathbf y_{[k]}}_2.
\end{equation}
The first summand in \eqref{eq:rho triangle} can be bounded by $\sigma_{k+1}$ since
\[
\norm{\mathbf y-\mathbf y_{[k]}}_2^2
= \sum_{j>k} y_j^2 
= \sum_{j>k} \sigma_j^2\Big(\frac{y_j}{\sigma_j}\Big)^2 
\le \sigma_{k+1}^2 \Xnorm{\mathbf y}{H(\sigma)}^2
\leq \sigma_{k+1}^2.
\]
On the other hand, the definition of $s_k(G_{[n],[k]})$ yields
\begin{multline*}
 s_k(G_{[n],[k]}) \cdot \norm{\mathbf y_{[k]}}_2
\leq
\norm{G_{[n],[k]} \brackets{\mathbf y_{[k]}}}_2 
= \norm{G_n \brackets{\mathbf y -\mathbf y_{[k]}}}_2\\
\leq \norm{G_n: H_{\IN\setminus [k]}(\sigma) \to \ell^2_n}\cdot
\Xnorm{\mathbf y -\mathbf y_{[k]}}{H(\sigma)}
\leq \norm{G_n: H_{\IN\setminus [k]}(\sigma) \to \ell^2_n}.
\end{multline*}
Note that we have $G_n=\Sigma_{[n],\IN\setminus[k]} D_k$
as mappings on $H_{\IN\setminus [k]}(\sigma)$,
where 
$$
 D_k: H_{\IN\setminus [k]}(\sigma) \to \ell^2,
 \quad
 \brackets{x_j}_{j=1}^\infty \mapsto
 \brackets{x_{k+j}/\sigma_{k+j}}_{j=1}^\infty.
$$
Since $D_k$ is an isometry, we get
$$
 \norm{G_n: H_{\IN\setminus [k]}(\sigma) \to \ell^2_n }
 = \norm{ \Sigma_{[n],\IN\setminus[k]}: \ell^2 \to \ell^2_n } 
 = s_1(\Sigma_{[n],\IN\setminus[k]}).
$$ 
This means that the second summand in \eqref{eq:rho triangle} can be bounded by
\[
\norm{\mathbf y_{[k]}}_2 
\leq \frac{s_1(\Sigma_{[n],\IN\setminus[k]})}{s_k(G_{[n],[k]})}.
\]
Since these bounds hold for all $\varrho<\mathcal{R}_n(\sigma)$, we obtain
the stated inequality.
\end{proof}


Now the task is to bound the $k^{\rm th}$
singular value of the Gaussian matrix $G_{[n],[k]}$
from below and the largest singular value of the
structured Gaussian matrix $\Sigma_{[n],\IN\setminus[k]}$
from above.
We start with
the largest singular value 
of the latter.
We note that the 
question for the order of the expected value of the 
largest singular value 
of a structured Gaussian matrix has recently been settled by 
Lata\l a, Van Handel, and Youssef \cite{LVY18}.
The result we use here is due to Bandeira and Van Handel~\cite{BV16}.

\begin{lemma}
\label{lem:BH}
Let $\sigma\in\ell^2$ be nonincreasing.
For every $c\geq 1$
and $n,k\in\IN$, 
we have
\[
\IP\left[
s_{1}\brackets{\Sigma_{[n],\IN\setminus[k]}} 
\geq
 \frac{3}{2}\sqrt{\sum\nolimits_{j>k} \sigma_j^2} + 11c\,\sigma_{k+1} \sqrt{n}
\right]
\leq e^{-c^2 n}.
\]
\end{lemma}

\begin{proof}
Without loss of generality, we may 
assume that $\sigma_{k+1}\neq 0$.
Let us first consider the finite matrix 
$$
 A_m=\Sigma_{[n],[m+k]\setminus[k]}\in\IR^{n\times m}
\quad\text{for}\quad
m\in\IN.
$$
and set
$$
 C_m=\frac{3}{2}\Big(\sum_{j=k+1}^{k+m} \sigma_j^2\Big)^{1/2} 
 + \frac{103c}{10}\sigma_{k+1} \sqrt{n},
$$
where $A$ and $C$ denote their infinite dimensional variants.
It is proven in \cite[Corollary 3.11]{BV16} that, 
for every $t\geq 0$ (and $\varepsilon=1/2$),
we have
$$
 \IP\bigg[s_{1}(A_m) \,\ge\, 
 \frac{3}{2}\Bigl(\Big(\sum_{j=k+1}^{k+m} \sigma_j^2\Big)^{1/2}+\sigma_{k+1}\sqrt{n}
 +\frac{5\sqrt{\ln(n)}}{\sqrt{\ln(3/2)}} \sigma_{k+1} \Bigr)+t\bigg]
 \le e^{-t^2/2\sigma_{k+1}^2}.
$$
By setting $t=\sqrt{2}c\sigma_{k+1}\sqrt{n}$,
it follows that
$$
 \IP[s_1(A_m)\geq C_m] \leq e^{-c^2n}.
$$
Turning to the infinite dimensional case,
we note that we have $s_1(A)>C$ if and only if
there is some $m\in\IN$ such that $s_1(A_m)>C$.
This yields
$$
 \IP[s_1(A)> C]
 = \IP\left[\exists m\in\IN: s_1(A_m)> C\right]
 = \lim_{m\to\infty} \IP[s_1(A_m)> C]
 \leq e^{-c^2n}
$$
since $s_1(A_m)$ is increasing in $m$ and $C\geq C_m$.
\end{proof}

Together with Proposition~\ref{prop:radius vs singular values}
this means that the estimate
\begin{equation}
\label{eq:intermediate upper bound}
 \mathcal{R}_n(\sigma) 
 \leq 
 \sigma_{k+1} + 
 \frac{\frac{3}{2}\sqrt{\sum\nolimits_{j>k} \sigma_j^2} + 11c\,\sigma_{k+1} \sqrt{n}}{
 s_k\brackets{G_{[n],[k]}}}
\end{equation}
holds with probability at least $1-e^{-c^2n}$ for all $k\leq n$
and $c\geq 1$.
It remains to bound the $k^{\rm th}$ singular value 
of the Gaussian matrix $G_{[n],[k]}$ from below.
It is known from \cite[Theorem~1.1]{RV09} that this number
typically is of order $\sqrt{n}-\sqrt{k-1}$
for all $n\in\IN$ and $k\leq n$.
To exploit our upper bound to the full extend, 
the number $k\leq n$ may be chosen such
that the right hand side of \eqref{eq:intermediate upper bound}
becomes minimal.
We realize that the term $1/s_k(G_{[n],[k]})$
increases with $k$, whereas all other terms
decrease with $k$.
However, the inverse singular number achieves its
minimal order $n^{-1/2}$ already for
$k= cn$ with some $c<1$.
If $\sigma$ does not decay extremely fast,
this does not lead to a loss 
regarding the other terms of \eqref{eq:intermediate upper bound}.
For instance, we may choose $k=\lfloor n/2\rfloor$
and use the following special case of \cite[Theorem~II.13]{DS01}.

\begin{lemma}
\label{lem:smallest SV rectangular matrix}
Let $n\in\IN$ and $k=\lfloor n/2\rfloor$.
Then
\[
\IP\Big[s_k\brackets{G_{[n],[k]}} \le \sqrt{n}/7 \Big]
\le e^{-n/100}. 
\]
\end{lemma}

\begin{proof}
 It is shown in~\cite[Theorem~II.13]{DS01} that,
 for all $k\leq n$ and $t>0$, we have
 $$
  \IP\left[s_k\brackets{G_{[n],[k]}} \le \sqrt{n}\brackets{1-\sqrt{k/n}-t} \right]
  \leq e^{-n t^2/2}.
 $$
 The statement follows by putting $k=\lfloor n/2\rfloor$ and $t^{-1}=\sqrt{50}$.
\end{proof}

If $\sigma$ decays very fast,
$k=\lfloor n/2\rfloor$
might not be the best choice.
The term $\sigma_{k+1}$ 
in estimate~\eqref{eq:intermediate upper bound}
may be much smaller for $k=n$ 
than for $k=\lfloor n/2\rfloor$.
It is better to choose $k=n$.
In this case,
the inverse singular number is of order $\sqrt{n}$.
We state a result of \cite[Theorem~1.2]{Sz91}.

\begin{lemma}
\label{lem:smallest SV square matrix}
 Let $n\in\IN$ and $t\geq 0$. Then
$$
 \IP\Big[s_n\brackets{G_{[n],[n]}} \le \frac{t}{\sqrt{n}}\Big]
 \le t \sqrt{2e}.
$$
\end{lemma}

This leads to the two different probabilistic
estimates of the radius $\mathcal{R}_n(\sigma)$
as presented in Theorem~\ref{thm:upper bound random section}.
The first is optimized for sequences $\sigma$
with moderate decay,
whereas the second is optimized for
sequences with rapid decay.


\begin{proof}[Proof of Theorem~\ref{thm:upper bound random section}]
 For the first part, let $k=\lfloor n/2\rfloor$.
 We combine Lemma~\ref{lem:smallest SV rectangular matrix} 
 and Lemma~\ref{lem:BH} for $c=1$ with Proposition~\ref{prop:radius vs singular values}
 and obtain that
 $$
  \mathcal{R}_n(\sigma) \leq 78\,\sigma_{k+1} +\frac{21}{2\sqrt{n}}
   \bigg(\sum_{j>\lfloor n/2\rfloor}\sigma_j^2\bigg)^{1/2}
 $$
 with probability at least $1-e^{-n}-e^{-n/100}$.
 The statement follows if we take into account that
 $$
  \sigma_{k+1}^2 \leq \frac{4}{n} 
  \sum_{j=\lfloor n/4\rfloor}^{\lfloor n/2\rfloor}\sigma_j^2.
 $$
 For the second part, set $t=c/s$.
 We combine Lemma~\ref{lem:smallest SV square matrix} 
 and Lemma~\ref{lem:BH} with Proposition~\ref{prop:radius vs singular values}
 and obtain that
 $$
  \mathcal{R}_n(\sigma) \leq \sigma_{n+1} + \frac{1}{t}
  \brackets{\frac{3\sqrt{n}}{2}\bigg(\sum_{j>n}\sigma_j^2\bigg)^{1/2}
  + 11c\,n\, \sigma_{n+1} }
 $$
 with probability at least $1-e^{-c^2 n}-t\sqrt{2e}$.
 The rough estimates $\sigma_{n+1}^2 \leq \sum_{j>n} \sigma_j^2$
 and $3\sqrt{n}/2\leq 2cn$ and $1\leq sn$ yield the statement.
\end{proof}

%
%

\subsubsection{Proof of Theorem~\ref{thm:lower}}

We want to give lower bounds on the radius of information
$$
 \mathcal{R}_n(\sigma)
 = \sup\set{\norm{\mathbf x}_2 \mid \mathbf x \in F(\sigma),\, G_n(\mathbf x) = 0},
$$
which corresponds to the difficulty of recovering
an unknown element $\mathbf x\in F(\sigma)$ from the information
$G_n(\mathbf x)$ in $\ell^2$.
In fact, our lower bounds already hold for the smaller quantity
$$
 \mathcal{R}_{n,k}(\sigma)
 = \sup\set{\abs{x_k} \mid \mathbf x \in F(\sigma),\, G_n(\mathbf x) = 0},
$$
which corresponds to the difficulty of recovering
just the $k^{\rm th}$ coordinate of $\mathbf x$. 
Before we come to our bound which holds with high probability, 
we shall prove the following pointwise estimate.

\begin{prop}[\cite{HKNPU19}]
\label{prop:lower}
Let $\sigma\in\ell^2$ be nonincreasing. 
For all $n,k\in\IN$ with $\sigma_k\neq 0$
we almost surely have
\[
\mathcal{R}_{n,k}(\sigma) \,\ge\,
\sigma_k \brackets{1-\frac{\norm{(g_{ik})_{i=1}^n}_2}{
\sigma_k^{-1} s_n\brackets{\Sigma_{[n],\IN\setminus\{k\}}} + \norm{(g_{ik})_{i=1}^n}_2
}}.
\]
\end{prop}

\begin{proof}
We may assume that 
$\mathbf g=(g_{ik})_{i=1}^n$ is nonzero
since this happens almost surely. 
Let
$$
 s_n :=
 s_n\brackets{\Sigma_{[n],\IN\setminus\{k\}}}
 = \sup \set{ \varrho \geq 0 \mid 
 B_\varrho^2(\mathbf 0) 
 \subset \Sigma_{[n],\IN\setminus\{k\}}\brackets{B_1^2(\mathbf{0})}}.
$$
Since we have
$$
 \Sigma_{[n],\IN\setminus\{k\}}\brackets{B_1^2(\mathbf{0})}
 =
 G_n\brackets{F_{\IN\setminus\{k\}}(\sigma)},
$$
the image of $F_{\IN\setminus\{k\}}(\sigma)$ under $G_n$
contains a Euclidean ball of radius $s_n$.
Let $\mathbf e_k$ be the $k^{\rm th}$ standard unit vector in $\ell^2$. 
We find an element $\bar{\mathbf y}$ of $F_{\IN\setminus\{k\}}(\sigma)$ 
such that 
$$
 G_n \bar{\mathbf y} =
 \frac{s_n\cdot G_n \mathbf e_k}{\norm{G_n \mathbf e_k}_2}.
$$ 
Our statement is trivial if $s_n=0$, so let $s_n>0$.
For $\mathbf y= s_n^{-1} \norm{G_n \mathbf e_k}_2 \bar{\mathbf y}$ 
we obtain $G_n\mathbf y = G_n \mathbf e_k=\mathbf g$ and
\[
\Xnorm{\mathbf y}{H(\sigma)}
=s_n^{-1} \| G_n \mathbf e_k \|_2 \Xnorm{\bar{\mathbf y}}{H(\sigma)} 
\le s_n^{-1} \|\mathbf g\|_2.
\]
Then the vector $\mathbf z:=\mathbf e_k-\mathbf y$ satisfies $G_n \mathbf z=0$ 
and $z_k=1$ as well as  
\[
\Xnorm{\mathbf z}{H(\sigma)}
\le \Xnorm{\mathbf e_k}{H(\sigma)}+\Xnorm{\mathbf y}{H(\sigma)}
\le \sigma_k^{-1} + s_n^{-1} \|\mathbf g\|_2.
\]
The statement is obtained
if we insert the $H(\sigma)$-normalization of $\mathbf z$ 
into the very definition of $\mathcal{R}_{n,k}(\sigma)$.
\end{proof}

It remains to bound the $n^{\rm th}$ singular value of $\Sigma_{[n],\IN\setminus\{k\}}$ 
and the norm of the Gaussian vector $(g_{ik})_{i=1}^n$
with high probability.
For both estimates, we use the following concentration result from \cite[Lemma 1]{LM00}.

\begin{lemma}\label{lem:vector}
Let $u_j$ be independent centered Gaussian variables
with variance $a_j$ for $1\leq j\leq m$.
Then, for any $0<\delta\leq 1$, we have
\begin{align*}
 \IP\bigg[\sum_{j=1}^m u_j^2 \,\le\, (1-\delta) \sum_{j=1}^m a_j\bigg]
 &\leq \exp\brackets{-\frac{\delta^2\norm{\mathbf a}_1}{4\norm{\mathbf a}_\infty}},\\
 \IP\bigg[\sum_{j=1}^m u_j^2 \,\ge\, (1+\delta) \sum_{j=1}^m a_j\bigg]
 &\leq \exp\brackets{-\frac{\delta^2\norm{\mathbf a}_1}{16\norm{\mathbf a}_\infty}}.
\end{align*}
\end{lemma}

\begin{proof}
By~\cite[Lemma 1]{LM00} we have for all $t>0$ that
\begin{align*}
 &\IP\left[\sum_{j=1}^m u_j^2 \leq 
 \norm{\mathbf a}_1 - 2\norm{\mathbf a}_2 t\right] 
 \leq e^{-t^2}, \\
 &\IP\left[\sum_{j=1}^m u_j^2 \geq 
 \norm{\mathbf a}_1 + 2\norm{\mathbf a}_2 t + 2\norm{\mathbf a}_\infty t^2 \right] 
 \leq e^{-t^2}.
\end{align*}
The formulation of Lemma~\ref{lem:vector} follows
if we put 
\[
t=\frac{\delta \norm{\mathbf a}_1}{2 \norm{\mathbf a}_2},
\qquad\text{respectively}\qquad
t=\min\set{
\frac{\delta \norm{\mathbf a}_1}{4 \norm{\mathbf a}_2}, 
\sqrt{\frac{\delta \norm{\mathbf a}_1}{4 \norm{\mathbf a}_\infty}}
}.
\]
The desired probability estimate then follows by using 
$\|\mathbf a\|_2^2\le \|\mathbf a\|_1 \|\mathbf a\|_\infty$.
\end{proof}

In particular,
the norm of the Gaussian vector $(g_{ik})_{i=1}^n$
concentrates around $\sqrt{n}$.
In order to bound the $n^{\rm th}$ singular value of $\Sigma_{[n],\IN\setminus\{k\}}$
we shall use Gordon's min-max theorem.
Let us state Gordon's theorem \cite[Lemma 3.1]{Go88} 
in a form that can be found in \cite{HOT15}.

\begin{thm}[\cite{HOT15}]
\label{thm:Gordon min-max}
Let $n,m\in\IN$ and let $S_1\subset\IR^n$, $S_2\subset\IR^m$ be compact sets. 
Assume that $\psi:S_1\times S_2\to\IR$ is a continuous mapping. 
Let $G\in\IR^{m\times n}$, $\mathbf u\in\IR^m$, and $\mathbf v\in\IR^n$ be independent random objects 
with independent standard Gaussian entries. Moreover, define
\begin{eqnarray*}
\Phi_1(G) & := & \min_{\mathbf x\in S_1}\max_{\mathbf y\in S_2}
 \Big( \langle \mathbf y,G\mathbf x \rangle_2 + \psi(\mathbf x,\mathbf y)\Big), \cr
\Phi_2(\mathbf u,\mathbf v)& := & \min_{\mathbf x\in S_1}\max_{\mathbf y\in S_2} 
 \Big(\|\mathbf x\|_2\langle \mathbf u,\mathbf y \rangle_2 + \|\mathbf y\|_2\langle \mathbf v,\mathbf x \rangle_2 
 + \psi(\mathbf x,\mathbf y) \Big).
\end{eqnarray*}
Then, for all $c\in\IR$, we have
\[
\IP\big[\Phi_1(G)< c\big] \leq 2\, \IP\big[\Phi_2(\mathbf u,\mathbf v) \leq c\big].
\]
\end{thm}

This yields the following lower bound
on the smallest singular value
of structured Gaussian matrices.
Note that this is a generalization
of Lemma~\ref{lem:smallest SV rectangular matrix}.

\begin{lemma}
\label{lem:smin basic}
Let $A\in\IR^{m\times n}$ be a random matrix
with $m\geq n$
whose entries $a_{ij}$ are centered Gaussian
variables with variance $a_i$ for all $i\leq m$ and $j\leq n$.
Then, for all $0<\delta<1$, we have
$$
\IP\left[s_n(A) \leq 
 \sqrt{(1-\delta)\norm{\mathbf a}_1}
 - \sqrt{(1+\delta) n \norm{\mathbf a}_\infty} \right] 
 \leq 4 \exp\left(-\frac{\delta^2}{16} \min\set{n, 
 \frac{\norm{\mathbf a}_1}{\norm{\mathbf a}_\infty}}\right).
$$
\end{lemma}

\begin{proof}
Note that the statement is trivial if $m\leq n$.
We may assume that the $a_i$ are positive
since an additional row of zeros does neither change
$s_n(A)$ nor the norms of the vector $\mathbf a$.
We have the identity $A=DG$ where $G\in\IR^{m\times n}$ is
a random matrix with independent standard Gaussian entries
and $D\in\IR^{m\times m}$ is the diagonal matrix 
$$
 D=\diag\brackets{\sqrt{a_1},\hdots,\sqrt{a_m}}.
$$
We want to apply Gordon's theorem for the matrix $G$
and $\psi=0$,
where $S_1$ is the sphere in $\ell^2_n$ and
$S_2$ is the image of the sphere in $\ell^2_m$ under $D$.
Then we have
\begin{multline*}
 \Phi_1(G) = \min_{\mathbf x\in S_1}\max_{\mathbf y\in S_2}
    \langle \mathbf y,G\mathbf x \rangle_2
 = \min_{\norm{\mathbf x}_2=1}\max_{\norm{\mathbf z}_2=1}
    \langle D\mathbf z,G\mathbf x \rangle_2\\
 = \min_{\norm{\mathbf x}_2=1}\max_{\norm{\mathbf z}_2=1}
    \langle \mathbf z,A\mathbf x \rangle_2
 = \min_{\norm{\mathbf x}_2=1} \norm{A \mathbf x}_2
 = s_n(A).
\end{multline*}
On the other hand, if $\mathbf u\in\IR^n$ and $\mathbf v\in\IR^m$
are standard Gaussian vectors, 
the choice of $\mathbf z= D\mathbf u / \Vert D\mathbf u\Vert_2$ yields
\begin{multline*}
 \Phi_2(\mathbf u,\mathbf v)=
 \min_{\mathbf x\in S_1} \max_{\mathbf y\in S_2}
   \Big(\scalar{\mathbf u}{\mathbf y}_2+ \norm{\mathbf y}_2 \scalar{\mathbf v}{\mathbf x}_2\Big)\\
 =\min_{\norm{\mathbf x}_2=1}\max_{\norm{\mathbf z}_2=1}
   \Big(\scalar{\mathbf u}{D\mathbf z}_2+ \norm{D\mathbf z}_2 \scalar{\mathbf v}{\mathbf x}_2\Big)\\
 \geq \min_{\norm{\mathbf x}_2=1}
   \Big(\norm{D\mathbf u}_2+ \frac{\norm{D^2\mathbf u}_2}{\norm{D\mathbf u}_2} 
   \scalar{\mathbf v}{\mathbf x}_2\Big)\\
 = \norm{D\mathbf u}_2- \frac{\norm{D^2\mathbf u}_2}{\norm{D\mathbf u}_2} \norm{\mathbf v}_2
 \geq \norm{D\mathbf u}_2- \sqrt{\norm{\mathbf a}_\infty} \norm{\mathbf v}_2.
\end{multline*}
Theorem~\ref{thm:Gordon min-max} implies for all $c\in\IR$ that
$$
 \IP\Big[s_n(A)< c\Big]
 \leq 2 \IP\Big[ \Phi_2(\mathbf u,\mathbf v) \leq c \Big]
 \leq 2 \IP\Big[ \norm{D\mathbf u}_2- \sqrt{\norm{\mathbf a}_\infty} \norm{\mathbf v}_2 
      \leq c \Big].
$$
To obtain the statement of our lemma,
we set $c=\sqrt{(1-\delta)\Vert\mathbf a\Vert_1} - \sqrt{(1+\delta) n \Vert\mathbf a\Vert_\infty}$.
By Lemma~\ref{lem:vector}, we have 
\[
\IP\Big[\norm{D\mathbf u}_2 \,\le\, \sqrt{(1-\delta) \norm{\mathbf a}_1}\Big]
\,\le\, \exp\brackets{- \frac{\delta^2\norm{\mathbf a}_1}{4\norm{\mathbf a}_\infty}}
\]
and 
\[
\IP\Big[\norm{\mathbf v}_2 \,\ge\, \sqrt{(1+\delta) n}\Big]
\,\le\, \exp\brackets{-\frac{\delta^2 n}{16}}.
\]
Now the statement is obtained from a union bound.
\end{proof}

We need the statement of Lemma~\ref{lem:smin basic}
for matrices with infinitely many rows,
which is obtained from a simple limit argument.

\begin{lemma}
\label{lem:smin basic infinite}
 The estimate in Lemma~\ref{lem:smin basic} also holds for $m=\infty$
 if $\mathbf a\in \ell^1$.
\end{lemma}

\begin{proof}
Again, we may assume that $\mathbf a$ is strictly positive.
For $m\in\IN$ let $A_m$ be the sub-matrix consisting
of the first $m$ rows of $A$ and let $\mathbf a^{(m)}$
be the sub-vector consisting of the first $m$ entries of $\mathbf a$.
We use the notation
\begin{align*}
 c_m(\delta)&=\sqrt{(1-\delta)\Vert\mathbf a^{(m)}\Vert_1} 
 - \sqrt{(1+\delta) n \Vert\mathbf a^{(m)}\Vert_\infty}, \\
 p_m(\delta)&=4 \exp\left(-\frac{\delta^2}{16} \min\set{n, 
 \frac{\norm{\mathbf a^{(m)}}_1}{\norm{\mathbf a^{(m)}}_\infty}}\right),
\end{align*}
where $c(\delta)$ and $p(\delta)$ correspond to the case $m=\infty$.
For any $\varepsilon>0$ with $\varepsilon<\delta/2$
we can choose $m\geq n$ such that 
$c(\delta) \leq c_m(\delta-\varepsilon)$
and $p_m(\delta-\varepsilon) \leq p(\delta-2\varepsilon)$.
Note that we have $s_n(A)\geq s_n(A_m)$
and thus
\begin{multline*}
 \IP\left[s_n(A) \leq c(\delta) \right]
 \leq \IP\left[s_n(A_m) \leq c(\delta) \right]
 \leq \IP\left[s_n(A_m) \leq c_m(\delta-\varepsilon)\right]\\
 \leq p_m(\delta-\varepsilon) 
 \leq p(\delta-2\varepsilon). 
\end{multline*}
Letting $\varepsilon$ tend to zero yields the statement.
\end{proof}

We arrive at our main lower bound. 

\begin{lemma}
\label{lem:lower}
Let $\sigma\in\ell^2$ be nonincreasing and let $n,k\in\IN$ be such that $\sigma_k\neq 0$.
Define
\[
C_k := C_k(\sigma) = \sigma_k^{-2}\sum_{j>k} \sigma_j^2\,. 
\]
Then, for all $\delta\in(0,1)$, 
we have
\[
\IP\left[ \mathcal R_n^{(k)}(\sigma) \,\le\, 
\sigma_k \left(1\,-\,\sqrt{\frac{(1+\delta)n}{(1-\delta)C_k}}\,\right)\right]
\,\leq\, 
5 \exp\brackets{-(\delta/4)^2\, \min\set{n, C_k}}.
\]
\end{lemma}

\begin{proof}
First note that, in the setting of Proposition~\ref{prop:lower}, 
the matrix $\Sigma_{[n],\IN\setminus[k]}^\top$ and the vector $(g_{ik})_{i=1}^n$ are independent. 
Lemma~\ref{lem:vector} and Lemma~\ref{lem:smin basic infinite} yield
\begin{align*}
 \|(g_{ik})_{i=1}^n\|_2 \,&\le\, \sqrt{1+\delta}\,\sqrt{n}
 \qquad\text{and}\\
 s_n\brackets{\Sigma_{[n],\IN\setminus[k]}^\top}
 \,&\ge\, \sqrt{1-\delta}\, \sigma_k \sqrt{C_k} \,-\, \sqrt{1+\delta}\, \sigma_{k+1} \sqrt{n} 
\end{align*}
with probability at least $1-5 \exp(-(\delta/4)^2\, \min\{n, C_k\})$. 
Note that we have
$$
 s_n\brackets{\Sigma_{[n],\IN\setminus\{k\}}}
 = s_n\brackets{\Sigma_{[n],\IN\setminus\{k\}}^\top}
 \geq s_n\brackets{\Sigma_{[n],\IN\setminus[k]}^\top}
$$
since erasing rows can only shrink the smallest singular value.
In this case, we have
\begin{align*}
 \frac{\norm{(g_{ik})_{i=1}^n}_2}{\sigma_k^{-1} 
    s_n\brackets{\Sigma_{[n],\IN\setminus\{k\}}} + \norm{(g_{ik})_{i=1}^n}}
 & \leq \frac{\sqrt{1+\delta}\sqrt{n}}{ 
    \sqrt{1-\delta} \sqrt{C_k} - (\sigma_{k+1}/\sigma_k)\sqrt{1+\delta} \sqrt{n} 
    + \sqrt{1+\delta}\sqrt{n}} \cr
 & \leq \frac{\sqrt{1+\delta}\sqrt{n}}{ 
    \sqrt{1-\delta} \sqrt{C_k}}.
\end{align*}
Now the statement is obtained from Proposition~\ref{prop:lower}.
\end{proof}

The proof of Theorem~\ref{thm:lower} is completed by choosing $\delta=1/2$.

\subsubsection{Proofs of Corollaries~\ref{cor:l2 not l2}, \ref{thm:main result random info hilbert} and \ref{thm:random sections exponential decay}}

In order to optimize the lower bound of Theorem~\ref{thm:lower},
we may choose $k\in\IN$ such that the
right-hand side of our lower bound becomes maximal.
If the Euclidean norm of $\sigma$ is large, we simply choose $k=1$.
Taking into account that $\mathcal R_n(\sigma)$
is decreasing in $n$, we immediately arrive at the following result.

\begin{lemma}
 \label{cor:random info useless}
 Let $\sigma\in\ell^2$ be a nonincreasing sequence of nonnegative numbers
 and let 
  $$
   n_0=\left\lfloor \frac{\varepsilon^2}{3 \sigma_1^2} \sum_{j=2}^\infty  \sigma_j^2\right\rfloor,
   \quad \varepsilon\in(0,1).
  $$ 
  Then $\mathcal R_n(\sigma)\geq \sigma_1(1-\varepsilon)$
  for all $n\leq n_0$
  with probability at least $1-5e^{-n_0/64}$.
\end{lemma}

We can now prove that random information is useful if and only if $\sigma\in\ell^2$.

\begin{proof}[Proof of Corollary~\ref{cor:l2 not l2}]
 We first consider the case that $\sigma\in\ell^2$.
 Since $\mathcal R_n(\sigma)\leq \sigma_1$, Theorem~\ref{thm:upper bound random section} yields
 $$
  \IE[ \mathcal R_n(\sigma)] \leq 2e^{-n/100} \cdot \sigma_1
  + \frac{156}{\sqrt{n}} 
   \bigg(\sum_{j\geq\lfloor n/4\rfloor}\sigma_j^2\bigg)^{1/2}.
 $$
 Now the statement is implied by the fact that $\sigma\in\ell^2$.
 
 For the case that $\sigma\not\in\ell^2$,
 let $0<\varepsilon<1$.
 For $m\in\IN$ let $\sigma^{(m)}$ be the sequence
 obtained from $\sigma$ by replacing $\sigma_j$
 with zero for all $j>m$.
 For any $N\geq n$, we can choose $m\in\IN$
 such that
 $$
  \frac{\varepsilon^2}{3 \sigma_1^2} \sum_{j=2}^m  \sigma_j^2
  \geq N
 $$
 since $\sigma\not\in\ell^2$.
 Lemma~\ref{cor:random info useless} yields that
 \begin{align*}
  \IP\left[\mathcal R_n(\sigma)\geq \sigma_1(1-\varepsilon)\right]
  & \geq \IP\left[\mathcal R_n(\sigma^{(m)})\geq \sigma_1(1-\varepsilon)\right]\cr
  & \geq \IP\left[\mathcal R_N(\sigma^{(m)})\geq \sigma_1(1-\varepsilon)\right]
   \geq 1-5\exp\brackets{-N/64}.
 \end{align*}
 Since this holds for any $N\geq n$, we get that
 the event $\mathcal R_n(\sigma)\geq \sigma_1(1-\varepsilon)$
 happens with probability 1 for any $\varepsilon\in(0,1)$.
 This yields the statement 
 since the event $\mathcal R_n(\sigma)\geq \sigma_1$
 is the intersection of countably many such events.
\end{proof}

We now apply our general estimates for $\mathcal R_n(\sigma)$ to 
specific sequences $\sigma$ and give a proof
of Corollaries~\ref{thm:main result random info hilbert} and \ref{thm:random sections exponential decay}. 
Note that the first part of Corollary~\ref{thm:main result random info hilbert}
which is concerned with slowly decaying sequences
is already proven by Corollary~\ref{cor:l2 not l2}.
We add a finite dimensional version of this statement.
%
%

\begin{cor}[\cite{HKNPU19}]
\label{cor:slow polynomial}
Let $m,n\in\IN$ and consider the sequence $\sigma$ with
$$
 \sigma_j=\left\{\begin{array}{cl}
  	\min\set{1,j^{-\alpha}\brackets{1+\ln j}^{-\beta}}	
        &
        \text{for} \quad j\leq m,
        \\
	  0
        &
        \text{for} \quad j>m.
        \end{array}\right.
$$
where $0\leq \alpha \leq 1/2$ and $\beta\in\IR$
with $\beta>0$ for $\alpha=0$ and $\beta\leq 1/2$ for $\alpha=1/2$.
Then, for any $0<\varepsilon<1$, 
we have with probability at least $1-5\exp(-n_0/64)$
for all $n\leq n_0$ that
$$
 1-\varepsilon \leq \mathcal{R}_n(\sigma)\leq 1
$$ 
if we put
\[
 n_0=\left\{\begin{array}{cl}
        \displaystyle \left\lfloor \frac{\varepsilon^2(m-2)m^{-2\alpha}}{3 (1+\ln m)^{\max\{2\beta,0\}}} \right\rfloor
        &
        \text{for} \quad \alpha<1/2,\vspace*{2mm}
        \\
       \displaystyle \left\lfloor \frac{\varepsilon^2(\ln m-1)}{3 (1+\ln m)^{\max\{2\beta,0\}}} \right\rfloor
        &
        \text{for} \quad \alpha=1/2,\, \beta<1/2,\vspace*{2mm}
        \\
        \displaystyle \left\lfloor \frac{\varepsilon^2(\ln\ln m-1)}{3} \right\rfloor
        &
        \text{for} \quad \alpha=\beta=1/2.
        \end{array}\right.
\]
\end{cor}

We now present a result for sequences on the edge of $\ell^2$.
This result shows that random information may be worse than
optimal information even if $\sigma\in\ell^2$.

\begin{cor}[\cite{HKNPU19}]
\label{cor:medium polynomial}
Let $\beta>1/2$ and consider the sequence $\sigma$ with
$$
 \sigma_j = \left\{\begin{array}{cl}
  	j^{-1/2} (1+\ln j)^{-\beta}	
        &
        \text{for} \quad j\leq m,
        \\
	  0
        &
        \text{for} \quad j>m.
        \end{array}\right.
$$
Then there exist constants $c_\beta, C_\beta>0$ such that
for all $n\in\IN$ and $m\in\IN\cup\{\infty\}$ with $m>n^2$
we have with probability at least $1- 7 e^{-n/100}$ that
\[
 c_\beta n^{-1/2} (1+\ln n)^{1/2-\beta}
 \leq \mathcal{R}_n(\sigma)
 \leq C_\beta n^{-1/2} (1+\ln n)^{1/2-\beta}.
\]
\end{cor}

\begin{proof}
 Note that we have for any $1< k< m<\infty$ that
 $$ 
  \sum_{j=k+1}^m \sigma_j^2 
  = \sum_{j=k+1}^m j^{-1} (1+\ln j)^{-2\beta}
  \asymp \ln^{1-2\beta}(k) - \ln^{1-2\beta}(m),
 $$
 where $\asymp$ means that the both sides of
 the equation are bounded by a constant multiple
 of the other side, where the constant depends only on $\beta$.
 Now the upper bound follows from the first part 
 of Theorem~\ref{thm:upper bound random section}
 and the lower bound follows from
 the second part of Theorem~\ref{thm:lower}
 with $k = \lceil c_\beta' n/(1+\ln n)\rceil$
 for some $c'_\beta >0$.
\end{proof}

If $\sigma$ decays with a polynomial rate
strictly larger than $1/2$,
then random information is up to a constant as good
as optimal information.

\begin{cor}[\cite{HKNPU19}]
\label{cor:quick polynomial}
Let $\alpha>1/2$ and $\beta\in\IR$ 
and consider the sequence $\sigma$ with
$$
 \sigma_j = \left\{\begin{array}{cl}
  	\min\set{1, j^{-\alpha} (1+\ln j)^{-\beta}}
        &
        \text{for} \quad j\leq m,
        \\
	  0
        &
        \text{for} \quad j>m.
        \end{array}\right.
$$
Then there exists a constant $C_{\alpha,\beta}>0$ such that
for all $n\in\IN$ and $m\in\IN\cup\{\infty\}$ with $n<m$
we have with probability at least $1- 2 e^{-n/100}$ that
\[
 \sigma_{n+1}
 \leq \mathcal{R}_n(\sigma)
 \leq C_{\alpha,\beta}\,\sigma_{n+1}.
\]
\end{cor}

\begin{proof}
 The lower bound is trivial, it
 holds for every realization of $\mathcal{R}_n(\sigma)$.
 The upper bound is a consequence of Theorem~\ref{thm:upper bound random section}, 
 since for large $n$ we have
 $$ 
  \sum_{j\geq\lfloor n/4\rfloor} \sigma_j^2 
  = \sum_{j\geq\lfloor n/4\rfloor} j^{-2\alpha}(1+\ln j)^{-2\beta} 
  \leq C n^{1-2\alpha}(1+\ln n)^{-2\beta}
 $$
 with a constant $C$ depending only on $\alpha$ and $\beta$.
\end{proof}

Corollaries~\ref{cor:slow polynomial}, \ref{cor:medium polynomial}
and \ref{cor:quick polynomial} form a proof of
Corollary~\ref{thm:main result random info hilbert}:

\begin{proof}[Proof of Corollary~\ref{thm:main result random info hilbert}]
 It suffices to consider the sequences from Corollaries~\ref{cor:slow polynomial},
 \ref{cor:medium polynomial} and \ref{cor:quick polynomial}
 since $\sigma\le C\sigma'$ implies $\mathcal R_n(\sigma)\le C\mathcal R_n(\sigma')$ for all $n$. 
 Since we have $0\leq \mathcal{R}_n(\sigma) \leq \sigma_1$
 almost surely, the statements for the expected value hold
 if the corresponding lower bounds hold at least with a constant positive probability
 and if the corresponding upper bounds hold with probability 
 at least $1-c\sigma_{n+1}$ for some constant $c>0$.
 This is shown in the corollaries.
\end{proof}

\begin{rem}
 The case $\sigma_n \asymp n^{-\alpha} \ln^{-\beta}(n+1)$ with $\alpha>1/2$
 can be extended to $\sigma_n \asymp n^{-\alpha} \phi(n)$ 
 for any slowly varying function $\phi$. 
 In this case, random information is up to a constant 
 as powerful as optimal information, 
 i.e., $\IE[ \mathcal R_n(\sigma)] \asymp \sigma_{n+1}$.
\end{rem}

We turn to the case of exponentially decaying singular values


\begin{proof}[Proof of Corollary~\ref{thm:random sections exponential decay}]
 The lower bound is implied by the trivial estimate 
 $\mathcal{R}_n(\sigma) \ge \sigma_{n+1}$.
 To prove the upper bound, we use the second part
 of Theorem~\ref{thm:upper bound random section}.
 Without loss of generality, we may assume that 
 $\sigma_j = a^{j-1}$ for all $j \in\IN$.
 The general case follows from the fact that $\sigma\le C\sigma'$ implies $\mathcal R_n(\sigma)\le C\mathcal R_n(\sigma')$ for all $n$.
 We choose $c\geq 1$ such that $e^{-c^2}\leq a$.
 Note that there is some $b>0$ such that
 $$
  \bigg(\sum_{j>n}\sigma_j^2\bigg)^{1/2} = \frac{b\,a^n}{14}
 $$
 for all $n\in\IN$.
 The theorem yields for all $t\geq b n a^n$ that
 $$
  \IP[\mathcal{R}_n(\sigma)\geq t]
  \leq a^n + \frac{b\,n\,a^n\,c\sqrt{2e}}{t}.
 $$
 This yields that
 $$
  \IE[\mathcal{R}_n(\sigma)]
  = \int_0^1 \IP[\mathcal{R}_n(\sigma)\geq t]~\d t
  \leq a^n + b n a^n + n a^n \int_{b n a^n}^1 \frac{b c\sqrt{2e}}{t}~\d t
  \preccurlyeq n^2 a^n,
 $$
 as it was to be proven.
\end{proof}

\chapter*{Symbols}
\addcontentsline{toc}{chapter}{Symbols}

\pagestyle{scrplain}

\begin{tabularx}{\textwidth}{cX}
\multicolumn{2}{X}{\textbf{General}
\bigskip}\\
	$\IN_0$, $\IN$
		& the set of natural numbers with and without zero\\
	$[k]$
		& the set of natural numbers from 1 to $k\in\IN$\\
	$\IZ$, $\IQ$, $\IR$, $\IC$
		& the sets of integers, rational, real and complex numbers\\
	$\lfloor a \rfloor$ 
		& the largest integer smaller than or equal to $a \in  \IR$\\
	$\lceil a \rceil$ 
		& the smallest integer larger than or equal to $a \in  \IR$\\
	$\ln(x)$
		& natural logarithm of $x>0$\\
	$\log_a(x)$
		& logarithm of $x>0$ in base $a>0$\\
	$\card(A)$
		& cardinality of a set~$A$; number of elements if $A$ is finite\\
	$A\subset B$
		& set inclusion, equality allowed\\
	$\dist(f,g)$
		& distance of $f$ and $g$ in a metric space\\
	$\Xnorm{f}{G}$
		& norm of $f$ in a normed space $G$\\
	$\Xscalar{f}{g}{H}$
		& scalar product of $f$ and $g$ in a pre-Hilbert space $H$\\
	$\rad(M)$
		& radius of a set $M$ in a metric space, see~\eqref{eq:radius of sets}\\
	$\mathbb T$
		& a circle, usually represented by $[0,1]$ where $0$ and $1$
		are identified\\
	$\mathbb T^d$
		& the $d$-torus, usually represented by $[0,1]^d$\\
\phantom{void}& \\
\multicolumn{2}{X}{\textbf{Vectors and Sequences of Real Numbers}
\bigskip}\\
	$\mathbf x = (x_1,\hdots,x_m)$
		& vector in~$ \IC^m$ with entries $x_i$\\
	$\mathbf x = (x_1,x_2,\hdots)$
		& vector in~$ \IC^\IN$ with entries $x_i$\\
	$\mathbf x_J$
		& sub-vector $(x_i)_{i \in J}$ of $\mathbf x$ for some index set $J$\\
	$\mathbf 0$
		& vector with all the entries set to~$0$ \\
	$\mathbf 1$
		& vector with all the entries set to~$1$ \\
	$\mathbf e_i$
		& vector with the $i^{\rm th}$ entry set to~$1$ and all other entries set to~$0$ \\
	$[\mathbf x,\mathbf y]$
		& set of vectors $\mathbf z$ with entries $z_i$ 
		between $x_i\in\IR$ and $y_i\in\IR$\\
	$I_J$
		& Cartesian product of intervals $I_j$ over $j\in J$\\
	$\langle \mathbf x, \mathbf y \rangle$
		& Euclidean scalar product, 
		that is, \mbox{$\langle \mathbf x, \mathbf y \rangle=\sum_i x_i \bar{y_i}$}\\
	$\norm{\mathbf x}_p$
		& $p$-norm of a vector, that is, 
		\mbox{$\|\mathbf  x\|_p = \left(\sum_i |x_i|^p\right)^{1/p}$}
			for \mbox{$1 \leq p < \infty$} and
			\mbox{$\|\mathbf x\|_{\infty} = \sup_i |x_i|$} for \mbox{$p = \infty$}\\
	$\vert\mathbf x\vert$
		& sometimes used instead of~$\norm{\mathbf x}_1$,
		mainly if $\mathbf x\in\IZ^d$\\
	$\ell^p_m$
		& $\IR^m$ equipped with the $p$-norm;
		in some contexts $\IC^m$\\
	$\ell^p$
		& space of all vectors in $\IR^\IN$ with finite $p$-norm
		equipped with the $p$-norm;
		in some contexts $\IC^\IN$\\
	$B_r^p(\mathbf x)$
		& open ball within~$\ell^p$ or $\ell^p_m$ with radius $r\geq 0$ and center $\mathbf x$ \\
	$B_r^p(M)$
		& union of the balls $B_r^p(\mathbf x)$ over $\mathbf x\in M$ 
		for $M\subset \ell^p$ or $M\subset \ell^p_m$\\
	$\mathbb S_{m-1}$
		& Euclidean sphere in $\IR^m$\\
	$c_{00}$
		& set of finite sequences, that is, $c_{00}=\bigcup_{n\in\IN_0} \IR^n$\\
\phantom{void}& \\
\multicolumn{2}{X}{\textbf{Comparison of Sequences of Positive Numbers}
\bigskip}\\
	$x_n \preccurlyeq y_n$
		& there is a constant $c>0$ and a threshold $n_0\in\IN$ 
		such that $x_n \leq c y_n$ for all $n\geq n_0$\\
	$x_n \succcurlyeq y_n$
		& there is a constant $c>0$ and a threshold $n_0\in\IN$ 
		such that $x_n \geq c y_n$ for all $n\geq n_0$;
		equivalent to $y_n \preccurlyeq x_n$\\
	$x_n \asymp y_n$
		& $x_n \preccurlyeq y_n$ and $y_n \preccurlyeq x_n$;
		weak equivalence of sequences\\
	$x_n \varlesssim y_n$
		& for every constant $c>1$ there is a threshold $n_0\in\IN$ 
		such that $x_n \leq c y_n$ for all $n\geq n_0$\\
	$x_n \vargtrsim y_n$
		& for every constant $c<1$ there is a threshold $n_0\in\IN$ 
		such that $x_n \geq c y_n$ for all $n\geq n_0$;
		equivalent to $y_n \varlesssim x_n$\\
	$x_n \sim y_n$
		& $x_n \varlesssim y_n$ and $y_n \vargtrsim x_n$ or
		equivalently $\lim_{n\to\infty} x_n/y_n =1$;
		strong equivalence of sequences\\
\phantom{void}& \\
\multicolumn{2}{X}{\textbf{Matrices and Operators}
\bigskip}\\
	$\diag(\mathbf{x})$
		& square matrix with main diagonal $\mathbf{x}\in\IR^m$
		and all other entries set to~$0$\\
	$A^{-1}$
		& inverse of a square matrix\\
	$A^{\top}$
		& transpose of a matrix\\
	$A^{-\top}$
		& transpose of the inverse of a square matrix\\
	$\det(A)$
		& determinant of a square matrix\\
	$\ker(A)$
		& kernel of a matrix\\
	$\norm{T:X\to Y}$
		& operator norm of a bounded linear operator $T$ 
		between normed spaces $X$ and $Y$, that is,
		$\sup\{\Xnorm{Tx}{Y}\mid x\in X, \Xnorm{x}{X}=1\}$\\
	$\mathcal L(X,Y)$
		& space of bounded linear operators  
		between $X$ and $Y$ equipped with the operator norm\\
	$\norm{A}_p$
		& operator norm of the matrix $A\in \IR^{n\times m}$
		in $\mathcal L(\ell^p_m,\ell^p_n)$\\
	$X \hookrightarrow Y$
		& embedding, $X$~is identified with a subset of~$Y$,
			$f \mapsto f$\\
\phantom{void}& \\
\multicolumn{2}{X}{\textbf{Functions and Derivatives}
\bigskip}\\
	$f:D\to \IR$
		& real valued function on a domain $D\subset \IR^d$,
		mapping a point $\mathbf x\in D$ to a number $f(\mathbf x)\in\IR$\\
	$\sup f$
		&  supremum of $f$, that is, $\sup f= \sup\{f(\mathbf x)\mid \mathbf x\in D\}$\\
	$\supp f$
		& support of $f$; closure of the set $\{\mathbf x\in D \mid f(\mathbf x)\neq 0\}$ \\
	$f\vert_P$
		& restriction of $f$ to the set $P\subset D$\\
	$f^{(r)}$
		&  the $r^{\rm th}$ weak derivative of $f$ in the case $D\subset \IR$;
		if possible, $f^{(r)}$ is identified with a continuous function\\
	$\partial_{\theta} f$
		& directional (weak) derivative of $f$
		in the direction $\theta\in \mathbb S_{d-1}$ \vspace*{1mm} \\
	$\displaystyle\frac{\partial f}{\partial x_i}$
		& partial (weak) derivative with respect to $x_i$;
			equivalently~\mbox{$\partial_{\mathbf e_i} f$}\\
	$\diff^\alpha f$
		& partial (weak) derivative 
		$\displaystyle\frac{\partial^{\abs{\alpha}} f}{\partial x_1^{\alpha_1}
		\cdots\partial x_d^{\alpha_d}}$
		of order $\alpha\in\IN_0^d$\\
	$\diff \Psi$
		& Jacobian matrix of a function $\Psi:D\to \IR^d$\\
	$\abs{\diff \Psi}$
		& absolute value of the determinant of $\diff \Psi$\\
	$f_1\otimes\hdots\otimes f_d$
		& tensor product of the functions $f_i:D_i\to\IC$;
		maps $\mathbf x\in \prod_{i=1}^d D_i$ to
		$\prod_{i=1}^d f_i(x_i)\in \IC$\\
	$f_J$
		& tensor product of the functions $f_i$ over $i\in J$\\
\phantom{void}& \\
\multicolumn{2}{X}{\textbf{Measures and Function Spaces}
\bigskip}\\
	$(D,\mathcal{A},\mu)$
		& measure space\\
	$\|f\|_p$
		& $p$-norm of a measurable function $f:D\to\IC$
		(with respect to $\mathcal{A}$
		and the Borel $\sigma$-algebra on $\IC$),
		that is, \mbox{$\|f\|_p = \left(\int |f|^p ~\d\mu \right)^{1/p}$}
		for \mbox{$1 \leq p < \infty$}
		and \mbox{$\|f\|_{\infty} = \esssup_{ x \in D} |f( x)|$}
		for \mbox{$p = \infty$};\\
	$L^p(D,\mathcal{A},\mu)$
		& the space of measurable 
		functions~\mbox{$f:D \rightarrow  \IR$}
		with finite $p$-norm;
		functions that are
		equal $\mu$-almost everywhere are identified;
		sometimes $\IC$ instead of $\IR$\\
	$\scalar{f}{g}$
		& scalar product in $L^2(D,\mathcal{A},\mu)$,
		that is, \mbox{$\scalar{f}{g}= \int f\bar{g}~\d\mu$};\\
	$L^p(D)$
		& short for $L^p(D,\mathcal{A},\mu)$
		if $D$ is a domain in $\IR^d$, $\mathcal{A}$ is the
		Borel $\sigma$-algebra and $\mu$ is the
		$d$-dimensional Lebesgue measure\\
	$\lambda^d$
		& $d$-dimensional Lebesgue measure\\
	$\mu(Statement)$
		& measure of the set of all $x\in D$
			for which \emph{Statement} is true \\
	$(\Omega,\mathcal{F},\IP)$
		& usually used instead of $(D,\mathcal{A},\mu)$
		if $\mu(D)=1$; probability space\\
	$\IE$
		& expectation, that is, \mbox{$\IE X=\int X~\d\IP$} for $X\in L^1(\Omega,\mathcal{F},\IP)$\\
	$A^\mathsf{c}$
		& complement of $A\subset\Omega$, 
		that is, $A^\mathsf{c}=\Omega\setminus A$\\
	$\mathcal B(D)$
		& bounded real valued functions on a set $D$\\
	$\C(D)$
		& continuous real valued functions on 
		a topological space $D$\\
	$\C_c(D)$
		& continuous real valued functions on $D$ with compact support\\
	$\C^r(D)$
		& $r$ times continuously differentiable real valued functions
		on a domain $D\subset \IR^d$\\
	$\C^\infty(D)$
		& infinitely differentiable real valued functions on $D$\\
	$W_p^r(D)$
		& Sobolev space of functions $f:D\to\IR$
		whose weak derivatives $D^\alpha f$
		exist and are in $L^p(D)$ 
		for all $\alpha\in\IN_0^d$ with $\abs{\alpha}\leq r$\\
	$H^r(D)$
		& equal to $W_2^r(D)$\\
	$W^r_{p,\rm{mix}}(D)$
		& Sobolev space of functions $f:D\to\IR$
		whose weak derivatives $D^\alpha f$
		exist and are in $L^p(D)$ 
		for all $\alpha\in\IN_0^d$ with $\norm{\alpha}_\infty\leq r$\\
	$H^r_{\rm mix}(D)$
		& equal to $W^r_{2,\rm{mix}}(D)$
\end{tabularx}

\newpage
\renewcommand{\refname}{Bibliography}

\newpage
\cleardoublepage
\chapter*{Ehrenw\"ortliche Erkl\"arung}
\thispagestyle{empty}

Hiermit erkl\"are ich,
\begin{itemize}
	\item dass mir die Promotionsordnung der Fakult\"at bekannt ist,
	\item dass ich die Dissertation selbst angefertigt habe,
		keine Textabschnitte oder Ergebnisse eines Dritten
		oder eigenen Pr\"ufungsarbeiten ohne Kennzeichnung \"ubernommen
		und alle von mir benutzten Hilfsmittel, pers\"onliche Mitteilungen
		und Quellen in meiner Arbeit angegeben habe,
	\item dass ich die Hilfe eines Promotionsberaters nicht in Anspruch genommen habe
		und dass Dritte weder unmittelbar noch mittelbar geldwerte Leistungen von mir
		f\"ur Arbeiten erhalten haben, die im Zusammenhang mit dem Inhalt der vorgelegten
		Dissertation stehen,
	\item dass ich die Dissertation noch nicht als Pr\"ufungsarbeit f\"ur eine staatliche
		oder andere wissenschaftliche Pr\"ufung eingereicht habe.
\end{itemize}

\noindent
Bei der Auswahl und Auswertung des Materials sowie bei der Herstellung des Manuskripts
wurde ich durch Prof.~Dr.~Erich Novak unterst\"utzt.

\bigskip

\noindent
Ich habe weder die gleiche, noch eine in wesentlichen Teilen \"ahnliche
oder andere Abhandlung bei einer anderen Hochschule als Dissertation eingereicht.

\vspace*{\fill}

\noindent
Jena, 05.\ Februar 2019
\hfill David Krieg

%


\begin{thebibliography}{99.}
\bibliographystyle{plain}
\addcontentsline{toc}{chapter}{Bibliography}
\raggedright{



      
\bibitem[AHR17]{AHR17} 
      \textsc{C.\,Aistleitner, A.\,Hinrichs, D.\,Rudolf}.
      \newblock On the size of the largest empty box amidst a point set.
      \newblock {\em Discrete Applied Mathematics}, 
      230:146--150, 2017.  
      
      

\bibitem[Bab60]{Ba60} 
      \textsc{K.I.\,Babenko}.
      \newblock About the approximation of 
	    periodic functions of many variable trigonometric polynomials.
      \newblock {\em Dokladi Akademii Nauk SSR}, 32:247--250, 1960.     
      
\bibitem[Bak59]{Ba59}
      \textsc{N.S.\,Bakhvalov}.
      \newblock On the approximate calculation of multiple integrals.
      \newblock {\em Vestnik Moskovskogo Universiteta, 
      Seriya Matematiki, Mehaniki, Astronomi, Fiziki, Himii},
	  4:3--18, 1959. In Russian.
      \newblock English translation:
      \newblock {\em Journal of Complexity}, 
	  31(4):502--516, 2015.
	 

\bibitem[Bak62]{Ba62} 
      \textsc{N.S.\,Bakhvalov}.
      \newblock On a rate of convergence of indeterministic integration processes
	  within the functional classes $W_p^{(l)}$.
      \newblock {\em Theory of Probability and its Applications}, 
	  7:227, 1962.         
      
      
\bibitem[BDDG14]{BDDG14} 
      \textsc{M.\,Bachmayr, W.\,Dahmen, R.\,DeVore and L.\,Grasedyck}.
      \newblock Approximation of high-dimensional rank one tensors.
      \newblock {\em Constructive Approximation}, 39:385--395, 2014.
      
\bibitem[BDKKW17]{BDKKW17}
      \textsc{B.\,Bauer, L.\,Devroye, M.\,Kohler, A.\,Krzyzak, H.\,Walk}.
      \newblock Nonparametric estimation of a function 
      		from noiseless observations at random points.
      \newblock {\em Journal of Multivariate Analysis}, 
      160:93--104, 2017.   

\bibitem[BEHW89]{BEHW89} 
      \textsc{A.\,Blumer, A.\,Ehrenfeucht, D.\,Haussler, M.\,Warmuth}.
      \newblock Learnability  and  the  Vapnik-Chervonenkis dimension.
      \newblock {\em Journal of the Association for Computing Machinery}, 
      36(4):929--965, 1989.   
      	  
\bibitem[BV16]{BV16}
      \textsc{A.S.\,Bandeira, R.\,van Handel}.
      \newblock Sharp nonasymptotic bounds on the norm of random matrices with 
	  independent entries.
      \newblock {\em Annals of Probability}, 
      44(4):2479--2506, 2016.

      
\bibitem[CD16]{CD16} 
      \textsc{A.\,Chernov, D.\,D\~ung}.
      \newblock New explicit-in-dimension estimates for the 
	  cardinality of high-dimensional
	  hyperbolic crosses and approximation of functions 
	  having mixed smoothness. 
      \newblock{\em Journal of Complexity}, 32:92--121, 2016.

\bibitem[CDL13]{CDL13} 
      \textsc{A.\,Cohen, M.A.\,Davenport, D.\,Leviatan}.
      \newblock On the stability and accuracy of least squares approximations.
      \newblock {\em Foundation of Computational Mathematics}, 13:819--834, 2013.
      
\bibitem[CK91]{CK91} 
      \textsc{E.~Cheney, D.~Kincaid}.
      \newblock Numerical analysis: Mathematics of scientific computing.
      \newblock Brooks/Cole, Pacific Grove, California, 1991.

\bibitem[CM17]{CM17} 
      \textsc{A.\,Cohen, G.\,Migliorati}.
      \newblock Optimal weighted least-squares methods.
      \newblock {\em SMAI-Journal of Computational Mathematics}, 3:181--203, 2017.

\bibitem[CW17a]{CW17} 
      \textsc{J.\,Chen, H.\,Wang}.
      \newblock Preasymptotics and asymptotics of approximation 
	  numbers of anisotropic Sobolev embeddings.
      \newblock {\em Journal of Complexity}, 39:94--110, 2017.

\bibitem[CW17b]{CW17b} 
      \textsc{J.\,Chen, H.\,Wang}.
      \newblock Approximation numbers of Sobolev and Gevrey type embeddings 
      on the sphere and on the ball -- Preasymptotics, 
      asymptotics, and tractability.
      \newblock {\em Journal of Complexity}, 50:1--24, 2019.


\bibitem[DJ13]{DJ13} 
      \textsc{A.\,Dumitrescu, M.\,Jiang}.
      \newblock On the largest empty axis-parallel box amidst $n$ points.
      \newblock {\em Algorithmica}, 
      66(2):225--248, 2013.   
      
\bibitem[DS01]{DS01} 
      \textsc{K.\,Davidson, S.\,Szarek}.
      \newblock Local operator theory, random matrices and Banach spaces.
      \newblock Handbook on the geometry of Banach spaces. Volume~1.
      \newblock 317--366, Elsevier Science B.V., Amsterdam, 2001.

\bibitem[DTU18]{DTU18} 
      \textsc{D.\,D\~ung, V.N.\,Temlyakov, T.\,Ullrich}.
      \newblock Hyperbolic cross approximation.
      \newblock Advanced Courses in Mathematics -- 
      CRM Barcelona. Birkhäuser/Springer, 2018.
      
\bibitem[EGO18]{EGO18}
      \textsc{M.\,Ehler, M.\,Gräf, C.J.\,Oates}.
      \newblock Optimal Monte Carlo integration on closed manifolds.
      \newblock {\em Statistics and Computing}, to appear,
      available on arXiv:1707.04723 [math.NA].

\bibitem[Fro76]{Fr76} 
      \textsc{K.K.\,Frolov}.
      \newblock Upper error bounds for 
	  quadrature formulas on function classes.
      \newblock {\em Soviet Mathematics Doklady}, 17(6):1665--1669, 1976.
      
\bibitem[GG84]{GG84} 
      \textsc{A.Yu.\,Garnaev, E.D.\,Gluskin}.
      \newblock The widths of a Euclidean ball.
      \newblock {\em Soviet Mathematics Doklady}, 30:200--204, 1984.
      
\bibitem[GM97]{GM97}
      \textsc{A.A.\,Giannopoulos, V.D.\,Milman}.
      \newblock On the diameter of proportional sections of a symmetric convex body.
      \newblock {\em International Mathematics Research Notices}, 1:5--19, 1997.
      
\bibitem[GM98]{GM98}
      \textsc{A.A.\,Giannopoulos, V.D.\,Milman}.
      \newblock Mean width and diameter of proportional sections of a
              symmetric convex body.
      \newblock {\em Journal f\"{u}r die Reine und Angewandte Mathematik}, 
            497:113--139, 1998.
            
\bibitem[GMT05]{GMT05}
      \textsc{A.A.\,Giannopoulos, V.D.\,Milman, A.\,Tsolomitis}.
      \newblock Asymptotic formulas for the diameter of sections of symmetric
              convex bodies.
      \newblock {\em Journal of Functional Analysis}, 
            1:86--108, 2005.
      
\bibitem[Gor88]{Go88}
      \textsc{Y.\,Gordon}.
      \newblock On Milman's inequality and random subspaces which escape through a mesh in~${\bf R}^n$.
      \newblock Geometric aspects of functional analysis.
      \newblock Lecture Notes in Mathematics~1317, 84--106, Springer, 1988.

\bibitem[GT01]{GT01} 
      \textsc{D.\,Gilbarg, N.\,Trudinger}.
      \newblock Elliptic partial differential equations of second order.
      \newblock Springer, Berlin Heidelberg, 2001. 

\bibitem[GW11]{GW11} 
      \textsc{M.\,Gnewuch, H.\,Woźniakowski}.
      \newblock Quasi-polynomial tractability.
      \newblock {\em Journal of Complexity}, 27:312--330, 2011.
      
      
      
\bibitem[Hei94]{He94} 
      \textsc{S.\,Heinrich}.
      \newblock Random approximation in numerical analysis.
      \newblock {\em Proceedings of the Conference Functional Analysis}, 
      Essen~(Germany), 123--171, Marcel Dekker, 1994.

\bibitem[Hei01]{He01} 
      \textsc{S.\,Heinrich}.
      \newblock Multilevel Monte Carlo methods.
      \newblock {\em Proceedings of the Third International Conference 
	  on Large-Scale Scientific Computing}, 
      \newblock Sozopol~(Bulgaria), 58--67, Springer, 2001.
    
\bibitem[Hei08]{He08} 
      \textsc{S.\,Heinrich}.
      \newblock Randomized approximation of Sobolev embeddings.
      \newblock {\em Proceedings of the Seventh International Conference on 
	  Monte Carlo and Quasi-Monte Carlo Methods in Scientific Computing}, 
      \newblock Ulm~(Germany), 445--459, Springer, 2008.
            
\bibitem[HKNPU19]{HKNPU19}
      \textsc{A.\,Hinrichs, D.\,Krieg, E.\,Novak, J.\,Prochno, M.\,Ullrich}.
      \newblock Random sections of ellipsoids and the power of random information.
      \newblock Preprint, available on arXiv:1901.06639 [math.FA].
            
\bibitem[HNUW17]{HNUW17} 
      \textsc{A.\,Hinrichs, E.\,Novak, M.\,Ullrich, H.\,Woźniakowski}.
      \newblock Product rules are optimal for numerical integration 
	  in classical smoothness spaces.
      \newblock {\em Journal of Complexity}, 38:39--49, 2017.

\bibitem[HNV08]{HNV08} 
      \textsc{A.\,Hinrichs, E.\,Novak, J.\,Vybíral}.
      \newblock Linear information versus function evaluations for $L_2$-approximation.
      \newblock {\em Journal of Approximation Theory}, 153:97--107, 2008.
            
\bibitem[HNWW01]{HNWW01} 
      \textsc{S.\,Heinrich, E.\,Novak, G.W.\,Wasilkowski, H.\,Woźniakowski}.
      \newblock The inverse of the star-discrepancy depends linearly on the dimension.
      \newblock {\em Acta Arithmetica}, 96:279--302, 2001.

\bibitem[HOT15]{HOT15}
      \textsc{B.\,Hassibi, C.\,Thrampoulidis, S.\,Oymak}.
      \newblock The Gaussian min-max theorem in the presence of convexity.
      \newblock E-print, arXiv:1408.4837 [cs.IT], 2015. 
            
\bibitem[HPU18]{HPU18} 
      \textsc{A.\,Hinrichs, J.\,Prochno, M.\,Ullrich}.
      \newblock The curse of dimensionality for numerical 
	  integration on general domains.
      \newblock {\em Journal of Complexity}, 50:25--42, 2019.

\bibitem[HT08]{HT08} 
      \textsc{D.D.\,Haroske, H.\,Triebel}.
      \newblock Distributions, Sobolev spaces, elliptic equations.
      \newblock European Mathematical Society, Zürich, 2008.

   
\bibitem[Jer67]{Je67} 
      \textsc{J.W.\,Jerome}.
      \newblock On the $L_2$ n-width of 
	  certain classes of functions of several variables.
      \newblock {\em Journal of Mathematical Analysis and Applications}, 
	  20:110--123, 1967.
	  
      
\bibitem[KMU16]{KMU16} 
      \textsc{T.\,Kühn, S.\,Mayer, T.\,Ullrich}.
      \newblock Counting via entropy: New preasymptotics for the approximation
	  numbers of Sobolev embeddings.
      \newblock {\em SIAM Journal on Numerical Analysis}, 54(6):3625--3647, 2016.
      
\bibitem[KN17]{KN17} 
      \textsc{D.\,Krieg, E.\,Novak}.
      \newblock A universal algorithm
	    for multivariate integration.
      \newblock {\em Foundation of Computational Mathematics}, 17(4):895--916, 2017.
      
\bibitem[Koc00]{Ko00} 
      \textsc{H.\,Koch}.
      \newblock Number theory: algebraic numbers and functions.
      \newblock Graduate studies in mathematics. 
	  American Mathematical Society, Providence, 2000.
	        
\bibitem[Koe84]{Ko84} 
      \textsc{H.\,König}.
      \newblock On the tensor stability of $s$-number ideals.
      \newblock {\em Mathematische Annalen}, 269:77--93, 1984.
	  
\bibitem[KOU17]{KOU17} 
      \textsc{C.\,Kacwin, J.\,Oettershagen, T.\,Ullrich}.
      \newblock On the orthogonality of the Chebyshev-Frolov lattice and applications.
      \newblock {\em Monatshefte f\"ur Mathematik}, 184(3):425--441, 2017.
      
\bibitem[KR19]{KR19}
      \textsc{D.\,Krieg, D.\,Rudolf}.
      \newblock Recovery algorithms for high-dimensional rank one tensors.
      \newblock {\em Journal of Approximation Theory}, 237:17--29, 2019.
      
\bibitem[Kri16]{Kr16}
      \textsc{D.\,Krieg}.
      \newblock On the randomization of Frolov’s algorithm for multivariate integration.
      \newblock Master thesis, Friedrich-Schiller-Universit\"at Jena, 2016,
      available on arXiv:1603.04637 [math.NA].
   
\bibitem[Kri18a]{Kr18} 
      \textsc{D.\,Krieg}.
      \newblock Tensor power sequences and the approximation of tensor product operators.
      \newblock {\em Journal of Complexity}, 44:30--51, 2018. 

\bibitem[Kri18b]{Kr18b} 
      \textsc{D.\,Krieg}.
      \newblock On the dispersion of sparse grids.
      \newblock {\em Journal of Complexity}, 45:115--119, 2018.
      
\bibitem[Kri18c]{Kr18c} 
      \textsc{D.\,Krieg}.
      \newblock Optimal Monte Carlo methods for $L^2$-approximation.
      \newblock {\em Constructive Approximation}, 2018.
      https://doi.org/10.1007/s00365-018-9428-4
      
\bibitem[Kri19]{Kr19}
      \textsc{D.\,Krieg}.
      \newblock Uniform recovery of high-dimensional $C^r$-functions.
      \newblock {\em Journal of Complexity}, 50:116--126, 2019.      
      
\bibitem[KSU15]{KSU15} 
      \textsc{T.\,Kühn, W.\,Sickel, T.\,Ullrich}.
      \newblock Approximation of mixed order Sobolev functions 
	  on the $d$-torus -- asymptotics, preasymptotics and $d$-dependence.
      \newblock {\em Constructive Approximation}, 42:353--398, 2015.
      
\bibitem[Kun17]{Ku17}
      \textsc{R.J.\,Kunsch}.
      \newblock High-Dimensional Function Approximation: Breaking the Curse with Monte Carlo Methods.
      \newblock Dissertation, Friedrich-Schiller-Universit\"at Jena, 2017, available on arXiv:1704.08213 [math.NA]. 
      
\bibitem[KWW09]{KWW09} 
      \textsc{F.Y.\,Kuo, G.W.\,Wasilkowski, H.\,Wo\'{z}niakowski}.
      \newblock On the power of standard information for multivariate
              approximation in the worst case setting.
      \newblock {\em Journal of Approximation Theory}, 158(1):97--125, 2009.
      
\bibitem[LM00]{LM00}
      \textsc{B.\,Laurent, P.\,Massart}.
      \newblock Adaptive estimation of a quadratic functional by model selection.
      \newblock {\em The Annals of Statistics}, 28(5):1302--1338, 2000.


\bibitem[LPT06]{LPT06} 
      \textsc{A.E.\,Litvak, A.\,Pajor, N.\,Tomczak-Jaegermann}.
      \newblock Diameters of sections and coverings of convex bodies.
      \newblock {\em Journal of Functional Analysis}, 231(2):438--457, 2006.
      
\bibitem[LPW09]{LPW09}
      \textsc{D.A.\,Levin, Y.\,Peres, E.\,L\,Wilmer}.
      \newblock Markov chains and mixing times.
      \newblock American Mathematical Society, Providence, 2009.
      
\bibitem[LT00]{LT00}
      \textsc{A.E.\,Litvak, N.\,Tomczak-Jaegermann}.
      \newblock Random aspects of high-dimensional convex bodies.
      \newblock Geometric aspects of functional analysis.
      \newblock Lecture Notes in Mathematics~1745, 169--190, Springer, 2000.
      
\bibitem[LVY18]{LVY18} 
      \textsc{R.\,Latała, R.\,van Handel, P.\,Youssef}.
      \newblock The dimension-free structure of nonhomogeneous random matrices.
      \newblock {\em Inventiones Mathematicae}, 214(3):1031--1080, 2018.
      
      
\bibitem[Mat91]{Ma91} 
      \textsc{P.\,Mathé}.
      \newblock Random approximation of Sobolev embeddings.
      \newblock {\em Journal of Complexity}, 7:261--281, 1991.
      
\bibitem[Mit62]{Mi62} 
      \textsc{B.S.\,Mityagin}.
      \newblock Approximation of functions in $L^p$ and $C$ on the torus.
      \newblock {\em Mathematical Notes}, 58:397--414, 1962.
      
\bibitem[MUV15]{MUV15}
      \textsc{S.\,Mayer, T.\,Ullrich, J.\,Vybíral}.
      \newblock Entropy and sampling numbers of classes of ridge functions.
      \newblock {\em Constructive Approximation}, 42(2):231--264, 2015.
      
\bibitem[Nik74]{Ni74} 
      \textsc{N.S.\,Nikol’skaya}.
      \newblock Approximation of differentiable functions of several 
	  variables by Fourier sums in the $L_p$-metric.
      \newblock {\em Sibirskii Matematicheskii Zhurnal}, 15:395--412, 1974.

\bibitem[Nov88]{No88} 
      \textsc{E.\,Novak}.
      \newblock Deterministic and stochastic error bounds in 
	  numerical analysis.
      \newblock Lecture Notes in Mathematics~1349, Springer, 1988.
      
\bibitem[Nov92]{No92} 
      \textsc{E.\,Novak}.
      \newblock Optimal linear randomized methods for linear 
	   operators in Hilbert spaces.
      \newblock {\em Journal of Complexity}, 8:22--36, 1992.
      

\bibitem[NR97]{NR97} 
      \textsc{E.\,Novak, K.\,Ritter}.
      \newblock The curse of dimension and a universal method for numerical integration.
      \newblock In G.\,Nürnberger, J.\,W.\,Schmidt, G.\,Walz (eds): Multivariate approximation and splines.
      International Series of Numerical Mathematics~125, 177--188, Birkhäuser, 1997.
      
\bibitem[NR16]{NR16} 
      \textsc{E.\,Novak, D.\,Rudolf}.
      \newblock Tractability of the approximation of high-dimensional rank one tensors.
      \newblock {\em Constructive Approximation}, 
      43:1--13, 2016.   
      
\bibitem[NW08]{NW08} 
      \textsc{E.\,Novak, H.\,Woźniakowski}.
      \newblock Tractability of multivariate problems.
	  Volume~I: Linear information.
      \newblock European Mathematical Society, Zürich, 2008. 
      
\bibitem[NW09]{NW09} 
      \textsc{E.\,Novak, H.\,Woźniakowski}.
      \newblock Approximation of infinitely differentiable multivariate 
	  functions is intractable.
      \newblock {\em Journal of Complexity}, 25:398--404, 2009.

\bibitem[NW10]{NW10} 
      \textsc{E.\,Novak, H.\,Woźniakowski}.
      \newblock Tractability of multivariate problems. 
	  Volume~II: Standard information for functionals.
      \newblock European Mathematical Society, Zürich, 2010. 
      
\bibitem[NW12]{NW12} 
      \textsc{E.\,Novak, H.\,Woźniakowski}.
      \newblock Tractability of multivariate problems.
	  Volume~III: Standard information for operators.
      \newblock European Mathematical Society, Zürich, 2012. 
      
\bibitem[Pie78]{Pi78}
      \textsc{A.\,Pietsch}.
      \newblock Operator ideals.
      \newblock VEB Deutscher Verlag der Wissenschaften, Berlin, 1978,
      and North-Holland, Amsterdam, 1980. 

\bibitem[Pie82]{Pi82} 
      \textsc{A.\,Pietsch}.
      \newblock Tensor products of sequences, functions, and operators.
      \newblock {\em Archiv der Mathematik}, 38:335--344, 1982.
      
\bibitem[PW10]{PW10} 
      \textsc{A.\,Papageorgiou, H.\,Woźniakowski}.
      \newblock Tractability through increasing smoothness.
      \newblock {\em Journal of Complexity}, 26:409--421, 2010.
      
\bibitem[RT96]{RT96} 
      \textsc{G.\,Rote, R.F.\,Tichy}.
      \newblock Quasi-Monte Carlo methods and the dispersion of point sequences.
      \newblock {\em Mathematical Computational Modeling}, 
      23(8-9):9--23, 1996. 
      
      
      
\bibitem[Rud18]{Ru18} 
      \textsc{D.\,Rudolf}.
      \newblock An upper bound of the minimal dispersion via delta covers.
      \newblock {\em Contemporary Computational Mathematics - 
	    A Celebration of the 80th Birthday of Ian Sloan}, 
      1099--1108, Springer, 2018. 
      
\bibitem[RV09]{RV09} 
      \textsc{M.\,Rudelson, R.\,Vershynin}.
      \newblock Smallest singular value of a random rectangular matrix.
      \newblock {\em Communications on Pure and Applied Mathematics}, 
      62(12):1707--1739, 2009.

\bibitem[Skr94]{Sk94} 
      \textsc{M.M.\,Skriganov}.
      \newblock Constructions of uniform distributions in terms of geometry of numbers.
      \newblock {\em Algebra i Analiz}, 6:200--230, 1994. 
      
\bibitem[Sos18]{So18} 
      \textsc{J.\,Sosnovec}.
      \newblock A note on the minimal dispersion of point sets
	  in the unit cube.
      \newblock {\em European Journal of Combinatorics}, 
      69:255--259, 2018.     
      
\bibitem[SU09]{SU09}
      \textsc{W.\,Sickel, T.\,Ullrich}.
      \newblock Tensor products of Sobolev-Besov spaces and 
      applications to approximation from the hyperbolic cross.
      \newblock {\em Journal of Approximation Theory}, 161(2):748--786, 2009.
      
\bibitem[SU10]{SU10} 
      \textsc{W.\,Sickel, T.\,Ullrich}.
      \newblock Spline interpolation on sparse grids.
      \newblock {\em Applicable Analysis}, 90:337--383, 2010.
      
\bibitem[Suk78]{Su78} 
      \textsc{A.G.\,Sukharev}.
      \newblock Optimal method of constructing best uniform approximations
	  for functions of a certain class.
      \newblock {\em USSR Computational Mathematics and Mathematical Physics}, 18(2):21--31, 1978.
      
\bibitem[Suk79]{Su79} 
      \textsc{A.G.\,Sukharev}.
      \newblock Optimal numerical integration formulas for some classes of
		functions of several variables.
      \newblock {\em Soviet Mathematics Doklady}, 
      20:472--475, 1979.
      
\bibitem[SW98]{SW98} 
      \textsc{I.H.\,Sloan, H.\,Woźniakowski}.
      \newblock When are quasi-Monte Carlo algorithms
	  efficient for high dimensional integrals?
      \newblock {\em Journal of Complexity}, 14:1--33, 1998.
      
      
\bibitem[Sza91]{Sz91}
      \textsc{S.J.\,Szarek}.
      \newblock Condition numbers of random matrices.
      \newblock {\em Journal of Complexity}, 7(2):131--149, 1991.
      
\bibitem[Sze39]{Sz39}
      \textsc{G.\,Szeg\H{o}}.
      \newblock Orthogonal polynomials.
      \newblock Colloquium publications. American Mathematical Society, Providence, 1939. 
      

\bibitem[Tem86]{Te86} 
      \textsc{V.N.\,Temlyakov}.
      \newblock Approximation of functions with bounded mixed derivative.
      \newblock {\em Trudy MIAN}, 178:1--112, 1986.
      English translation in {\em Proceedings of the 
	Steklov Institute of Mathematics}~1, 1989.

\bibitem[Tem93]{Te93} 
      \textsc{V.N.\,Temlyakov}.
      \newblock Approximation of periodic functions.
      \newblock Computational mathematics and analysis series. 
      Nova Science Publishers, New York, 1993. 

\bibitem[Tem03]{Te03} 
      \textsc{V.N.\,Temlyakov}.
      \newblock Cubature formulas, discrepancy, 
	  and nonlinear approximation.
      \newblock {\em Journal of Complexity}, 19:352--391, 2003.
      
\bibitem[Tem17]{Te17} 
      \textsc{V.N.\,Temlyakov}.
      \newblock Universal discretization.
      \newblock {\em Journal of Complexity}, 
      47:97--109, 2018. 

\bibitem[Tho96]{Th96} 
      \textsc{C.\,Thomas-Agnan}.
      \newblock Computing a family of reproducing kernels 
	  for statistical applications.
      \newblock {\em Numerical Algorithms}, 13:21--32, 1996.
      
\bibitem[Tri05]{Tr05} 
      \textsc{H.\,Triebel}.
      \newblock Sampling numbers and embedding constants.
      \newblock {\em Proceedings of the Steklov Institute of Mathematics}, 
      248:268–277, 2005.  
      
\bibitem[TW80]{TW80} 
      \textsc{J.F.\,Traub, H.\,Woźniakowski}.
      \newblock A general theory of optimal algorithms.
      \newblock Academic Press, 1980.
      
\bibitem[TWW88]{TWW88} 
      \textsc{J.F.\,Traub, G.W.\,Wasilkowski, H.\,Woźniakowski}.
      \newblock Information-based complexity.
      \newblock Academic Press, 1988.   

\bibitem[Ull16]{Ul16} 
      \textsc{M.\,Ullrich}.
      \newblock On "Upper error bounds for quadrature 
	  formulas on function classes" by K.K.\,Frolov.
      \newblock {\em Proceedings of the MCQMC 2014}, 
	  Leuven~(Belgium), 571--582, Springer, 2016.

      
\bibitem[Ull17]{Ul17} 
      \textsc{M.\,Ullrich}.
      \newblock A Monte Carlo method for integration of
	  multivariate smooth functions.
      \newblock {\em SIAM Journal on Numerical Analysis}, 55(3):1188–1200, 2017.
      
\bibitem[UV18]{UV18} 
      \textsc{M.\,Ullrich, J.\,Vyb\'iral}.
      \newblock An upper bound on the minimal dispersion.
      \newblock {\em Journal of Complexity}, 
      45:120--126, 2018. 
      
\bibitem[UV19]{UV19}
      \textsc{M.\,Ullrich, J.\,Vyb\'iral}.
      \newblock Deterministic constructions of high-dimensional sets with small dispersion.
      \newblock Preprint, available on arXiv:1901.06702 [cs.CC].
      
\bibitem[Vyb14]{Vy14} 
      \textsc{J.\,Vybíral}.
      \newblock Weak and quasi-polynomial tractability of 
	  approximation of infinitely differentiable functions.
      \newblock {\em Journal of Complexity}, 30(2):48--55, 2014.
      
\bibitem[Was84]{Wa84} 
      \textsc{G.W.\,Wasilkowski}.
      \newblock Some nonlinear problems are as easy as the approximation problem.
      \newblock {\em Computers \& Mathematics with Applications}, 10:351--363, 1984.     
      
\bibitem[Wei12]{We12}
      \textsc{M.\,Weimar}.
      \newblock The complexity of linear tensor product problems in
	  (anti)symmetric Hilbert spaces.
      \newblock {\em Journal of Approximation Theory}, 164(10):1345--1368, 2012. 
      
\bibitem[Woz18]{Wo18} 
      \textsc{H.\,Woźniakowski}.
      \newblock ABC on IBC.
      \newblock {\em Journal of Complexity}, in press.
      \newblock https://doi.org/10.1016/j.jco.2018.05.001 
      
\bibitem[WW01]{WW01}
      \textsc{G.W.\,Wasilkowski, H.\,Woźniakowski}.
      \newblock On the power of standard information for weighted approximation.
      \newblock {\em Foundations of Computational Mathematics}, 1(4):417--434, 2001.        

\bibitem[WW04]{WW04}
      \textsc{G.W.\,Wasilkowski, H.\,Woźniakowski}.
      \newblock Finite-order weights imply tractability of linear multivariate problems.
      \newblock {\em Journal of Approximation Theory}, 130(1):57--77, 2004.  
      
\bibitem[WW06]{WW06} 
      \textsc{G.W.\,Wasilkowski, H.\,Woźniakowski}.
      \newblock The power of standard information for multivariate approximation
	  in the randomized setting.
      \newblock {\em Mathematics of Computation}, 76:965--988, 2006.

\bibitem[Yse10]{Ys10} 
      \textsc{H.\,Yserentant}.
      \newblock Regularity and approximability of
	  electronic wave functions.
      \newblock Lecture Notes in Mathematics~2000,
	  Springer, 2010.
      

   
}

\end{thebibliography}
\end{document}